\documentclass{amsart}

\usepackage{latexsym}
\usepackage{linearb}
\usepackage[LGR, OT1]{fontenc}
\usepackage{amsthm}
\usepackage{amssymb}
\usepackage{upgreek}
\usepackage[greek,english]{babel}
\usepackage{teubner}
\usepackage{MnSymbol, wasysym}
\usepackage{graphicx}
\usepackage{marvosym}
\usepackage[margin=1.25in,dvips]{geometry}
\usepackage{color}
\usepackage{comment}
\usepackage{mathrsfs}

\usepackage{mathtools}

\usepackage{amsfonts}

\usepackage{multicol}

\usepackage[toc,page]{appendix}

\usepackage{physics}

\usepackage{enumitem}

\usepackage{amsaddr}

\usepackage{xcolor}

\usepackage{hyperref}
\hypersetup{%
  colorlinks = true,
  linkcolor  = black
}
\usepackage{slashed}

\newtheorem{innercustomTheorem}{Theorem}
\newenvironment{customTheorem}[1]
  {\renewcommand\theinnercustomTheorem{#1}\innercustomTheorem}
  {\endinnercustomTheorem}
\newtheorem{innercustomCorollary}{Corollary}
\newenvironment{customCorollary}[1]
  {\renewcommand\theinnercustomCorollary{#1}\innercustomCorollary}
  {\endinnercustomCorollary}

\newtheorem{innercustomDefinition}{Definition}
\newenvironment{customDefinition}[1]
{\renewcommand\theinnercustomDefinition{#1}\innercustomDefinition}
{\endinnercustomDefinition}

\DeclareMathOperator\supp{supp}

\newcommand{\nabb}{\mbox{$\nabla \mkern-13mu /$\,}}

\makeindex

\numberwithin{equation}{section}

\newtheorem{definition}{Definition}[section]

\newtheorem{remark}{Remark}[section]
\newtheorem{lemma}{Lemma}[section]
\newtheorem{theorem}{Theorem}[section]
\newtheorem{proposition}{Proposition}[section]

\newtheorem{corollary}{Corollary}[section]

\newtheorem*{rough version}{Rough Version}

\newtheorem*{theorem*}{Theorem}

\newtheorem*{corollary*}{Corollary}

\newenvironment{sketch proof}{\proof}{\endproof}

\title[Relatively non-degenerate estimates]{Relatively non-degenerate integrated decay estimates\\ on subextremal Kerr--de~Sitter}

\author{Georgios Mavrogiannis}
\address{Department of Mathematics, Rutgers University, New Brunswick, NJ 08903 USA.}
\email{gm758@math.rutgers.edu}

\date\today

\begin{document}

\begin{abstract}
We study the Klein--Gordon equation~$\Box\psi-\mu^2_{\textit{KG}}\psi=0$ on subextremal Kerr--de~Sitter black hole backgrounds with parameters~$(a,M,l)$, where~$l^2=\frac{3}{\Lambda}$. We prove a `relatively non-degenerate integrated decay estimate' assuming an appropriate mode stability statement for real frequency solutions of Carter’s radial ode. Our results, in particular, apply unconditionally in the very slowly rotating case~$|a|\ll M,l$, and in the case where~$\psi$ is axisymmetric. Exponential decay for~$\psi$ to a constant is a consequence of this estimate.

To prove our result, we introduce a novel pseudodifferential commutation operator~$\mathcal{G}$ that generalizes our previous purely physical space commutation~\cite{mavrogiannis} and we use it in conjunction with the Morawetz estimate of our companion~\cite{mavrogiannis4}. This pseudodifferential operator is defined using Fourier decomposition with respect to time frequencies $\omega$ and azimuthal frequencies $m$, but does not require Carter's full separation. 
\end{abstract}

\maketitle

{
  \hypersetup{linkcolor=black}
  \tableofcontents
}

\section{Introduction}\label{sec: intro}

Let~$(\mathcal{M},g_{a,M,l})$, with~$l^2=\frac{3}{\Lambda}$, be a Kerr--de~Sitter black hole spacetime, where the metric in Boyer--Lindquist coordinates~(see~\cite{Carter2}) takes the form
\begin{equation}\label{eq: prototype metric in BL coordinates}
    \begin{aligned}
        g_{a,M,l} =& \frac{\rho^2}{\Delta}dr^{2}+\frac{\rho^2}{\Delta_{\theta}}d \theta^{2}+\frac{\Delta_{\theta}(r^{2}+a^{2})^{2}-\Delta a^{2}\sin^{2}\theta }{\Xi ^{2}\rho^2}\sin^{2}\theta d \varphi^{2}-2\frac{\Delta_{\theta}(r^{2}+a^{2})-\Delta}{\Xi\rho^2} a \sin^{2}\theta d\varphi dt\\
        &	\qquad -\frac{\Delta -\Delta_{\theta}a^{2}\sin^{2}\theta}{\rho^2}dt^{2},
    \end{aligned}
\end{equation}
with~$\Xi=1+\frac{a^2}{l^2},~\Delta_\theta= 1+\frac{a^2}{l^2}\cos^2\theta,~\rho^2=r^2+a^2\cos^2\theta$, where 
\begin{equation}\label{eq: prototype Delta}
    \Delta=(r^2+a^2)\left(1-\frac{r^2}{l^2}\right)-2Mr.
\end{equation}
In the non-rotating case~$a=0$ the Kerr--de~Sitter metric~\eqref{eq: prototype metric in BL coordinates} reduces to the Schwarzschild--de~Sitter metric
\begin{equation}\label{eq: prototype metric of SdS}
	g_{M,l}=-\left(1-\frac{2M}{r}-\frac{\Lambda}{3}r^2\right) dt^2 +\left(1-\frac{2M}{r}-\frac{\Lambda}{3}r^2\right)^{-1} dr^2 + r^2 d\sigma_{\mathbb{S}^2},
\end{equation}
where~$d\sigma_{\mathbb{S}^2}$ is the standard metric of the unit sphere.

In our previous~\cite{mavrogiannis} we studied the wave equation
\begin{equation}\label{eq: sec: intro, eq 1}
\Box_{g_{M,l}}\psi=0
\end{equation}
on the Schwarzschild--de~Sitter black hole spacetime~\eqref{eq: prototype metric of SdS}. Specifically, we proved a `relatively non-degenerate integrated energy decay estimate' by commuting~\eqref{eq: sec: intro, eq 1} with the vector field~$\mathcal{G}=r\sqrt{1-\frac{2M}{r}-\frac{\Lambda}{3}r^2}\partial_r$, here written in specially chosen hyperboloidal coordinates. Note that Holzegel--Kauffman~\cite{gustav} initially studied an analogue of the vector field~$\mathcal{G}$ in the~$\Lambda=0$ Schwarzschild case, and proved a Morawetz estimate for the wave equation with small first order terms. Recently they extended their work in the subextremal Kerr case~\cite{gustav2}. A relatively non degenerate integrated energy decay estimate controls a spacetime bulk term with degeneration at top order from an initial boundary flux term with the exact same degeneration, see already estimate~\eqref{eq: SdS relatively non degenerate} in Section~\ref{subsec: sec: intro, subsec 1}. Thus, the bulk control is non-degenerate relative to the boundary control. To prove this estimate we used in addition a Morawetz estimate of Dafermos--Rodnianski~\cite{DR3}.

In the present paper, we prove a `relatively non-degenerate' integrated energy decay estimate~(Theorem~\ref{main theorem, relat. non-deg}) in the Kerr--de~Sitter context. Specifically, we study solutions of the Klein--Gordon equation 
\begin{equation}\label{eq: kleingordon}
    \Box_{g_{a,M,l}} \psi -\mu_{\textit{KG}}^2\psi=0,
\end{equation}
with mass~$\mu_{\textit{KG}}^2\geq 0$ (the case~$\mu_{\textit{KG}}=0$ will be called simply the `wave equation') on a Kerr--de~Sitter spacetime~\eqref{eq: prototype metric in BL coordinates} with subextremal parameters~$(a,M,l)$, under the following condition
\begin{equation*}
	\begin{aligned}
		&\text{(MS):~mode stability on the real axis for Carter's radial ode holds on a curve in subextremal}\\
		&	\qquad\quad \text{parameter space connecting}~(a,M)~\text{to the subextremal Schwarzschild--de~Sitter family.}
	\end{aligned}
\end{equation*}
The very slowly rotating case~$|a|\ll M,l$ is in fact included in the condition~(MS), as shown in our companion~\cite{mavrogiannis4}, where the smallness depends on~$\mu_{KG}$. Theorem~\ref{main theorem, relat. non-deg} also applies unconditionally to axisymmetric solutions of the Klein--Gordon equation~\eqref{eq: kleingordon} for all~$\mu^2_{KG}\geq 0$, in the full subextremal range of parameters.

To prove the theorem we introduce a novel pseudodifferential commutation operator~$\mathcal{G}$, see already Section~\ref{subsec: sec: intro: subsec 3}, for the entire subextremal range of Kerr--de~Sitter black holes, thus generalizing our previous work~\cite{mavrogiannis} on Schwarzschild--de~Sitter, and use it in conjunction with a Morawetz estimate proved in our companion~\cite{mavrogiannis4}. An immediate Corollary of our~`relatively non-degenerate' estimate~(Theorem~\ref{rough: thm relat. non-deg}) is the exponential decay for the solutions of the Klein--Gordon equation~(Corollary~\ref{rough: cor: thm relat. non-deg, cor 1}).

In contrast to our previous~\cite{mavrogiannis}, the commutation with~$\mathcal{G}$ is frequency dependent, see Section~\ref{sec: G}, i.e. the operator~$\mathcal{G}$ is pseudodifferential and is defined applying the Fourier transform. The pseudodifferential aspect of~$\mathcal{G}$ only depends on the time frequency~$\omega$ and the aximuthal frequency~$m$, which correspond to the Killing vector fields~$\partial_t,\partial_{\varphi}$ respectively, but it does not require Carter's full separation. As a result, in addition to initial boundary fluxes on the right hand side, the `relatively non-degenerate' estimate of Theorem~\ref{rough: thm relat. non-deg} contains an averaged bulk term, which is again however suitably localized in spacetime. There terms arise from cutoffs and pseudodifferential commutation error terms. Because~$\mathcal{G}$ is defined without recourse to Carter's separation, these terms take a form which can be bounded by standard commutator estimates.

In addition, we prove the analogue of Theorem~\ref{rough: thm relat. non-deg} for an inhomogeneous wave equation (Theorem~\ref{thm: main thm extended region}). In our forthcoming~\cite{mavrogiannis3} we use Theorem~\ref{thm: main thm extended region} to give an elementary proof of stability and exponential decay of the solutions of appropriate quasilinear wave equations on a Kerr--de~Sitter black hole background.

\subsection{Previous works}

For a discussion of decay results for linear and non-linear wave equations on (anti)-de~Sitter or asymptotically flat black hole backgrounds see our~\cite{mavrogiannis,mavrogiannis1}. We simply list some important references here.

For the study of various linear equations on black hole backgrounds with $\Lambda>0$ see~\cite{bony,zworski,Dyatlov1,Dyatlov2,vasy1,vasy2,DR2,petersen2021wave,casals2021hidden,hintz2021mode,volker,AllenFangLinear}. For the wave and Klein--Gordon equations in the case $\Lambda\le 0$ see for instance~\cite{whiting,Shlapentokh_Rothman_2014_mode_stability,tatarutohaneanuKerr,LindbladTohaneanuKerr,holzegel1,holzegel2,holzegel3,holzegel4,vasy6}.

Note that we will refer to the works~\cite{gustav2,petersen2021wave,casals2021hidden,Dyatlov2,hintz2021mode} further down the paper.

\subsection{The `relatively non-degenerate' estimate of~\cite{mavrogiannis} on Schwarzschild--de~Sitter}\label{subsec: sec: intro, subsec 1}

The present paper generalizes the commutation with
\begin{equation}\label{eq: prototype vector field SdS}
\mathcal{G}=r\sqrt{1-\frac{2M}{r}-\frac{\Lambda}{3}r^2}\frac{\partial}{\partial r}
\end{equation}
of our previous~\cite{mavrogiannis}, which we here review. The vector field~\eqref{eq: prototype vector field SdS} is written in `special' hyperboloidal coordinates~$(\bar{t},r,\theta,\varphi)$ on Schwarzschild--de~Sitter which were specifically chosen to make the form of the vector field~\eqref{eq: prototype vector field SdS} simple. In particular, with our choice,~$\mathcal{G}$ is orthogonal to the timelike Killing field~$\partial_{\bar{t}}$ at~$r=3M$ and vanishes continuously (but not differentiably) at the horizons~$\mathcal{H}^+,~\bar{\mathcal{H}}^+$. Note that in the standard Schwarzschild--de~Sitter coordinates~$(t,r,\theta,\varphi)$ the vector field~\eqref{eq: prototype vector field SdS} takes the form 
\begin{equation}\label{eq: prototype vector field SdS 2}
	\mathcal{G}=G_1(r)\partial_{r^\star}+ G_2(r)\partial_t,
\end{equation}
where~$r^\star$ is the tortoise coordinate and
\begin{equation*}
G_1(r)=\frac{r}{\sqrt{1-\frac{2M}{r}-\frac{\Lambda}{3}r^2}},\qquad G_2(r)= \frac{r}{\sqrt{1-\frac{2M}{r}-\frac{\Lambda}{3}r^2}}\left(1-\frac{3M}{r}\right)\frac{1}{\sqrt{1-9M^2\Lambda}}\sqrt{1+\frac{6M}{r}},
\end{equation*}
see~\cite{mavrogiannis}.

In~\cite{mavrogiannis} we studied the wave equation~(\eqref{eq: kleingordon} with~$\mu^2_{\textit{KG}}=0$) on the Schwarzschild--de~Sitter spacetime. Specifically, we commuted the wave equation with the vector field~$\mathcal{G}$, see~\eqref{eq: prototype vector field SdS}, which produced only one top order term, which had a `good' sign. We applied an energy estimate for the equation satisfied by~$\mathcal{G}\psi$, with multiplier~$\partial_{\bar{t}}\mathcal{G}\psi$, and used the Morawetz estimate of Dafermos--Rodnianski, see~\cite{DR2}. In view of this `good' sign, this led to a `relatively non-degenerate' energy estimate 
\begin{equation}\label{eq: SdS relatively non degenerate}
\begin{aligned}
\int_{\{\bar{t}=\tau_2\}}\mathcal{E}(\mathcal{G}\psi,\psi) +\int\int_{\{\tau_1\leq \bar{t}\leq \tau_2\}} \mathcal{E}(\mathcal{G}\psi,\psi)  \lesssim \int_{\{\bar{t}=\tau_1\}} \mathcal{E}(\mathcal{G}\psi,\psi),
\end{aligned}
\end{equation}
for all~$0\leq\tau_1\leq \tau_2$, where 
\begin{equation}
\begin{aligned}
\mathcal{E}(\mathcal{G}\psi,\psi)\sim (\partial_{\bar{t}}\mathcal{G}\psi)^2+\left(1-\frac{2M}{r}-\frac{\Lambda}{3}r^2\right)(\partial_r\mathcal{G}\psi)^2+|\slashed{\nabla}\mathcal{G}\psi|^2+(\partial_{\bar{t}}\psi)^2+(\partial_r\psi)^2+|\slashed{\nabla}\psi|^2.
\end{aligned}
\end{equation}
The energy estimate of~\eqref{eq: SdS relatively non degenerate} is `relatively non-degenerate' in the sense that in view of the coarea formula~$	\int\int_{\{\tau_1 \leq \bar{t}\leq \tau_2\}} \mathcal{E}(\mathcal{G}\psi,\psi)\sim \int_{\tau_1}^{\tau_2}d\tau\int_{\{\bar{t}=\tau\}}\mathcal{E}(\mathcal{G}\psi,\psi)$ the boundary terms on the right hand side are precisely comparable to the bulk term on the left hand side.

Exponential decay followed immediately as a Corollary of~\eqref{eq: SdS relatively non degenerate}. 

Note that in~\cite{gustav}, Holzegel--Kauffman introduced a vector field conciding with~\eqref{eq: prototype vector field SdS 2} in the limit~$\Lambda=0$, and used it to prove energy estimates for the wave equation with small first order error terms.

\subsection{A review of the results of our companion~\cite{mavrogiannis4}}\label{subsec: sec: intro, subsec 2}

In our companion~\cite{mavrogiannis4} we prove boundedness and a degenerate Morawetz estimate for solutions of the Klein--Gordon equation~\eqref{eq: kleingordon}, after assuming~(MS). We here review the Morawetz estimate of our companion~\cite{mavrogiannis4}, where for the precise statement see already Theorem~\ref{main theorem 1}.

The projection into Fourier space allows us to capture the phenomena of trapping and superadiance at the level of modes and moreover allows us to rigorously formulate our mode stability condition~(MS). Moreover, we discuss Carter's separation of variables for Kerr--de~Sitter.

\subsubsection{The~\texorpdfstring{$(t,\varphi)$}{g}~Fourier projections}\label{subsubsec: subsec: sec: intro, subsec 2, subsubsec 0}

Let~$(a,M,l)$ be subextremal. Let~$\Psi$ be a sufficiently regular with~$\text{supp}\Psi\subset \{0\leq t^\star<\infty\}$, where for any~$\tau\geq 0$ the leaves~$\{t^\star=\tau\}$ connect the event horizon~$\mathcal{H}^+$ with the cosmological horizon~$\bar{\mathcal{H}}^+$, see Section~\ref{subsec: sec: preliminaries, subsec 2}. We define the Fourier projection
\begin{equation}\label{eq: subsubsec: subsec: sec: intro, subsec 2, subsubsec 0, eq 1}
	\mathcal{F}_{\omega,m}(\Psi)(r,\theta) 	= \int_{\mathbb{R}}\int_0^{2\pi}e^{-i\omega t}e^{im\varphi} \Psi (t,r,\theta,\varphi) d\varphi dt,
\end{equation}
where
\begin{equation}\label{eq: subsubsec: subsec: sec: intro, subsec 2, subsubsec 0, eq 2}
	\omega\in\mathbb{R},\qquad m\in\mathbb{Z}
\end{equation}
are the Fourier frequencies with respect to the Boyer--Lindquist coordinates~$t,\varphi$ respectively, see already Section~\ref{sec: carter separation, radial}.

\subsubsection{The Hawking--Reall vector fields}

We denote as
\begin{equation}
	K_+=\partial_t+\frac{a\Xi}{r_+^2+a^2}\partial_{\varphi},\qquad \bar{K}_+=\partial_t+\frac{a\Xi}{\bar{r}_+^2+a^2}\partial_{\varphi}
\end{equation}
the Hawking--Reall vector fields of the event horizon~$\mathcal{H}^+$ and the cosmological horizon~$\bar{\mathcal{H}}^+$ respectively.

\subsubsection{The superradiant frequencies}

Τhe~$\bar{K}^+$ energy flux along the event horizon~$\mathcal{H}^+$ is given by
\begin{equation}\label{eq: subsubsec: subsec: sec: intro, subsec 2, subsubsec 1.1, eq 1}
	\int_{\mathcal{H}^+} K\psi \overline{\bar{K}^+\psi} =\int_{\mathcal{H}^+} \left(\partial_t\psi+\omega_+\partial_{\varphi}\psi\right) \cdot  \overline{\left(\partial_t\psi+\bar{\omega}_+\partial_{\varphi}\psi\right)}.
\end{equation}
In particular, if we consider a solution of the form~$\psi=e^{-i\omega t}e^{im\varphi} \psi_0(r,\theta)$ then the sign of~\eqref{eq: subsubsec: subsec: sec: intro, subsec 2, subsubsec 1.1, eq 1} is determined by the sign of
\begin{equation}
	\left(\omega-\frac{am\Xi}{r_+^2+a^2}\right) \left(\omega-\frac{am\Xi}{\bar{r}_+^2+a^2}\right).
\end{equation}

The superradiant frequencies of Kerr--de~Sitter are the following 
\begin{equation}\label{eq: subsubsec: subsec: sec: intro, subsec 2, subsubsec 2.1, eq 1}
	\mathcal{SF}=\Bigl\{(\omega,m):~am\omega\in \left(\frac{a^2m^2\Xi}{\bar{r}_+^2+a^2},\frac{a^2m^2\Xi}{r_+^2+a^2}\right)\Bigr\},~~\text{or equivalently}~~	\left(\omega-\frac{am\Xi}{r_+^2+a^2}\right)\left(\omega-\frac{am\Xi}{\bar{r}_+^2+a^2}\right) < 0,
\end{equation}
where~$r_+,\bar{r}_+$ correspond respectively to the event horizon~$\mathcal{H}^+$ and the cosmological horizon~$\bar{\mathcal{H}}^+$.

\subsubsection{The radial ode in Kerr--de~Sitter}\label{subsubsec: subsec: sec: intro, subsec 2, subsubsec 1}

Following~\cite{Carter}, also see~\cite{holzegel3}, we decompose the solution of the Klein--Gordon equation~\eqref{eq: kleingordon} into Fourier modes~\eqref{eq: subsubsec: subsec: sec: intro, subsec 2, subsubsec 0, eq 2} via~\eqref{eq: subsubsec: subsec: sec: intro, subsec 2, subsubsec 0, eq 1} and then further into spheroidal harmonics. Remarkably, Carter's separation of variables for the wave operator implies that the radial part~$u$ of the separated solution, for the Klein--Gordon equation~\eqref{eq: kleingordon}, satisfies the ode
\begin{equation}\label{eq: subsec: sec: intro, subsec 2, eq 1}
	u^{\prime\prime}+(\omega^2-V)u=H
\end{equation}
where~$^\prime=\frac{d}{dr^\star}$, and~$r^\star$ is the tortoise coordinate, see Section~\ref{subsec: tortoise coordinate} and~$H$ is a cut-off inhomogeneity. The potential~$V$ is real and takes the form
\begin{equation}\label{eq: subsec: sec: intro, subsec 2, eq 2}
V=V_0+V_{\textit{SL}}+V_{\mu_{\textit{KG}}},\qquad V_0-\omega^2=\frac{\Delta}{(r^2+a^2)^2}\left(\lambda^{(a\omega)}_{m\ell}+(a\omega)^2-2m\omega a\Xi\right)-\left(\omega-\frac{am\Xi}{r^2+a^2}\right)^2,
\end{equation}
where~$V_{\textit{SL}},V_{\mu_{\textit{KG}}}$ are frequency independent terms (the latter is associated with the Klein--Gordon mass) see Section~\ref{sec: carter separation, radial}, and~$\lambda^{(a\omega)}_{m\ell}\in\mathbb{R}$ are the eigenvalues of the spherical part of Carter's separation of variables, which are indexed by~$\ell$.

\subsubsection{Energy identity and mode stability in Kerr--de~Sitter}\label{subsubsec: subsec: sec: intro, subsec 2, subsubsec 1.1}

We say that a classical solution~$u$ of the radial inhomogeneous ode~\eqref{eq: subsec: sec: intro, subsec 2, eq 1} satisfies outgoing boundary conditions at~$r^\star=-\infty$,~$r^\star=+\infty$ if the following hold respectively
\begin{equation}\label{eq: subsec: sec: intro, subsec 2, eq 2.1}
	u^\prime=-i\left(\omega-\frac{am\Xi}{r_+^2+a^2}\right) u,\qquad u^\prime=i\left(\omega-\frac{am\Xi}{\bar{r}_+^2+a^2}\right) u,
\end{equation}
see already Section~\ref{subsec: energy identity}.

For a classical solution~$u$ of~\eqref{eq: subsec: sec: intro, subsec 2, eq 1} that satisfies both the outgoing boundary conditions~\eqref{eq: subsec: sec: intro, subsec 2, eq 2.1} we multiply the radial ode~\eqref{eq: subsec: sec: intro, subsec 2, eq 1} with~$\bar{u}$ and after integration by parts we obtain
\begin{equation}\label{eq: subsec: sec: intro, subsec 2, eq 3}
\left(\omega-\frac{am\Xi}{\bar{r}_+^2+a^2}\right)|u|^2(\infty)+\left(\omega-\frac{am\Xi}{r_+^2+a^2}\right)|u|^2(-\infty)=\Im (\bar{u}H).
\end{equation}

For the superradiant frequencies~\eqref{eq: subsubsec: subsec: sec: intro, subsec 2, subsubsec 2.1, eq 1}, the left hand side of the energy identity~\eqref{eq: subsec: sec: intro, subsec 2, eq 3} is not coercive. Therefore, we cannot conclude from~\eqref{eq: subsec: sec: intro, subsec 2, eq 3} that there are no modes on the real axis. At present, it is not known whether mode stability holds for the solutions of the wave equation~$\mu^2_{KG}=0$, or the conformally invariant wave equation~$\mu^2_{KG}=\frac{2\Lambda}{3}$, on Kerr--de~Sitter in the full subextremal case in the spirit of~\cite{whiting,Shlapentokh_Rothman_2014_mode_stability}. Therefore, we will need to appeal to condition~(MS)~(we do not expect that mode stability holds for all~$\mu^2_{KG}>0$, see~\cite{Shlapentokh_Rothman_2014}). The very slowly rotating case~$|a|\ll M$ is in fact included in~(MS), where the smallness depends on~$\mu_{KG}$. Note, moreover, the partial mode stability result of Casals--Teixeira~da~Costa~\cite{casals2021hidden} and the recent result of Hintz~\cite{hintz2021mode}, which we will also discuss later.

\subsubsection{The boundedness and Morawetz estimates of~\cite{mavrogiannis4}}\label{subsubsec: subsec: sec: intro, subsec 2, subsubsec 3}

To prove our Morawetz estimate of~\cite{mavrogiannis4}, quoted here as Theorem~\ref{main theorem 1}, the work~\cite{mavrogiannis4} followed the strategy of~\cite{DR2} and constructed fixed frequency multipliers for the radial ode~\eqref{eq: subsec: sec: intro, subsec 2, eq 1}. Note that two main features of the geometry of Kerr--de~Sitter, trapping of null geodesics and superradiance, were also the main difficulties for proving the frequency localized estimates of~\cite{mavrogiannis4}.

In Theorem~\ref{main theorem 1}, we control from initial data the following quantity 
\begin{equation}\label{eq: subsubsec: subsec: sec: intro, subsec 2, subsubsec 3, eq 1}
	\begin{aligned}
		&	\int_{\mathbb{R}}d\omega\sum_{m,\ell} \Big( \int_{r_+}^{r_++\epsilon}\frac{1}{\Delta^2}|u^\prime+i(\omega-\omega_+m)u|^2dr +\int_{\bar{r}_+-\epsilon}^{\bar{r}_+}\frac{1}{\Delta^2}|u^\prime-(\omega-\bar{\omega}_+m)u|^2dr \\
		&	\qquad\qquad\qquad+ \int_{r_+}^{\bar{r}_+} |u^\prime|^2dr+\left(1-\frac{r_{\textit{trap}}(\omega,m,\ell)}{r}\right)^2(\omega+\lambda^{(a\omega)}_{m\ell})|u|^2dr\Big)
	\end{aligned}
\end{equation}
for a sufficiently small~$\epsilon>0$, where~$\omega_+=\frac{a\Xi}{r_+^2+a^2},~\bar{\omega}_+=\frac{a\Xi}{\bar{r}_+^2+a^2}$ and where~$u$ is the fixed frequency solution of the radial ode~\eqref{eq: subsec: sec: intro, subsec 2, eq 1}.

Note that the integrand of~\eqref{eq: subsubsec: subsec: sec: intro, subsec 2, subsubsec 3, eq 1} admits a degeneration in accordance with the obstructions of trapped null geodesics, see~\cite{sbierski,ralston}. However, note that the precise degeneration of the integrand of~\eqref{eq: subsubsec: subsec: sec: intro, subsec 2, subsubsec 3, eq 1} is at the frequency dependent values
\begin{equation}\label{eq: subsec: sec: introduction, sec 0, eq 1}
	r_{\textit{trap}}(\omega,m,\ell). 
\end{equation}

Furthermore, the boundedness estimate for solutions of the Klein--Gordon equation~\eqref{eq: kleingordon} follows by using the Morawetz estimate in conjunction with some extra Fourier based arguments.

\subsection{The operator~\texorpdfstring{$\mathcal{G}$}{g}}\label{subsec: rough version of theorem 3, exponential decay}\label{subsec: sec: intro: subsec 3}

Having reviewed the necessary background, we now turn to discuss the generalization of the vector field~\eqref{eq: prototype vector field SdS} to the Kerr--de~Sitter spacetime, which is an operator that we also denote as~$\mathcal{G}$. Our generalized~$\mathcal{G}$ must be here pseudodifferential, depending on the frequencies~$\omega,m$~(as it also must capture the phenomenon of trapping on Kerr--de~Sitter) but independent of the Carter frequency~$\ell$.

Let~$(\omega,m)\in\mathbb{R}\times\mathbb{Z}$. First, we define 
\begin{equation}\label{eq: subsec: rough version of theorem 3, g2, 2}
\mathcal{T}(\omega,m,r)=\frac{(r^2+a^2)^2}{\Delta}\left(\omega-\frac{am\Xi}{r^2+a^2}\right)^2
\end{equation}
and 
\begin{equation}\label{eq: subsec: rough version of theorem 3, g2, 3}
\begin{aligned}
g_1(r^\star(r))\:\dot{=}\:\frac{r^2+a^2}{\sqrt{\Delta}},\qquad \left(g_2(\omega,m,r)\right)^2		= \mathcal{T}(\omega,m,r)-\min_{r\in[r_+,\bar{r}_+]}\mathcal{T}(\omega,m,r),\\
\end{aligned}
\end{equation}
where for~$\Delta$ see~\eqref{eq: prototype Delta}, with~$\Xi=1+\frac{a^2}{l^2}$. Note that for the superradiant frequencies~$(\omega,m)\in \mathcal{SF}$ we have~$\min_{r\in [r_+,\bar{r}_+]}\mathcal{T}(\omega,m,r)=0$.

Now, for any~$(\omega,m)\in\mathbb{R}\times\mathbb{Z}$ we denote as 
\begin{equation}\label{eq: subsec: rough version of theorem 3, g2, 4}
	r_{\textit{crit}}(\omega,m).
\end{equation}
to be the unique critical point of~$\mathcal{T}(\omega,m,r)$, see already Propositions~\ref{prop: subsec: sec: G, subsec 1, prop 0},~\ref{prop: subsec: sec: G, subsec 1, prop -1}. Note that for the borderline non-superradiant frequencies
\begin{equation}
	am\omega=\frac{a^2m^2\Xi}{r_+^2+a^2},\qquad am\omega=\frac{a^2m^2\Xi}{\bar{r}_+^2+a^2},
\end{equation}
we have~$r_{\textit{crit}}(\omega,m)=r_+,r_{\textit{crit}}(\omega,m)=\bar{r}_+$.

\begin{remark}(Connection of~$\mathcal{G}$ with trapping)
 We note that when a trapped null geodesic exists, for certain conserved quantities~$E,L$ associated with~$\omega$ and~$m$ respectively, then the $r$-coordinate of the geodesic is~$r_{crit}(E,L)$, where for~$r_{crit}(\cdot,\cdot)$ see~\eqref{eq: subsec: rough version of theorem 3, g2, 4}. Therefore, our operator~$\mathcal{G}$ will indeed be connected to trapping. Specifically, note that in the Schwarzschild--de~Sitter case we have~$r_{crit}(\omega,m)=3M$ for all~$(\omega,m)\in \mathbb{R}\times\mathbb{Z}$. 
\end{remark}

\begin{definition}\label{def: sec: G, def 0.1}
	We define a signed square root~$g_2(\omega,m,r)$ of~$g_2^2(\omega,m,r)$ defined by~\eqref{eq: subsec: rough version of theorem 3, g2, 3} as follows. For the frequencies~$am\omega>\frac{a^2m^2\Xi}{r_+^2+a^2}$ we define 
	\begin{equation}\label{eq: subsec: rough version of theorem 3, g2, 4.1}
		g_2(\omega,m,r)=
		\begin{cases}
			-\sqrt{g_2^2(\omega,m,r)},	\quad r\geq r_{\textit{crit}}(\omega,m)\\
			+\sqrt{g_2^2(\omega,m,r)}, \quad r\leq r_{\textit{crit}}(\omega,m).
		\end{cases}
	\end{equation}
	For the superradiant frequencies~$\mathcal{SF}$ we define 
	\begin{equation}\label{eq: subsec: rough version of theorem 3, g2, 4.2}
		g_2(\omega,m,r)=-\sqrt{g_2^2} =-\left|\omega-\frac{am\Xi}{r^2+a^2}\right|\frac{r^2+a^2}{\sqrt{\Delta}}.
	\end{equation}
	For the frequencies~$am\omega<\frac{a^2m^2\Xi}{\bar{r}_+^2+a^2}$ we define 
	\begin{equation}\label{eq: subsec: rough version of theorem 3, g2, 4.3}
		g_2(\omega,m,r)=
		\begin{cases}
			+\sqrt{g_2^2(\omega,m,r)},	\quad r\geq r_{\textit{crit}}(\omega,m)\\
			-\sqrt{g_2^2(\omega,m,r)}, \quad r\leq r_{\textit{crit}}(\omega,m).
		\end{cases}
	\end{equation}
	For the borderline non-superradiant frequencies~$am\omega=\frac{a^2m^2\Xi}{r_+^2+a^2},am\omega=\frac{a^2m^2\Xi}{\bar{r}_+^2+a^2}$ we define respectively
	\begin{equation}\label{eq: subsec: rough version of theorem 3, g2, 4.4}
		g_2(\omega,m,r)=-\sqrt{g_2^2}= -|am\Xi| \frac{r^2-r_+^2}{(r_+^2+a^2)\sqrt{\Delta}},\qquad g_2(\omega,m,r)= -\sqrt{g_2^2}=-|am\Xi| \frac{\bar{r}_+^2-r^2}{(\bar{r}_+^2+a^2)\sqrt{\Delta}}. 
	\end{equation}	
	In the Schwarzschild--de~Sitter case~$a=0$ we define~$g_2=g_1\frac{1}{\sqrt{1-9M^2\Lambda}}\left(1-\frac{3M}{r}\right)\sqrt{1+\frac{6M}{r}}$, as in our previous~\cite{mavrogiannis}. For a graphic presentation of~$g_2(\omega,m,r)$ see already Figure~\ref{fig: g2}.
\end{definition}

We now present the rough definition of the operator~$\mathcal{G}$, which is a generalization of the vector field~$\mathcal{G}$ of~\cite{mavrogiannis}. For the precise definition of~$\mathcal{G}$ see already Section~\ref{sec: G}.

\begin{customDefinition}{\ref{def: sec: G, def 1}}[rough version]\label{rough def: g2}
Let~$g_1(r),~g_2(\omega,m,r)$ be as in~\eqref{eq: subsec: rough version of theorem 3, g2, 3} and Definition~\ref{def: sec: G, def 0.1}. We define an operator~$\mathcal{G}$ of the form 
\begin{equation}\label{eq: subsec: sec: intro, 2, eq 0}
\begin{aligned}
\mathcal{G}& \:\dot{=}\:g_1(r^\star)\partial_{r^\star}+ i\textit{Op}\left(g_2\right),
\end{aligned}
\end{equation}
where~$\textit{Op}(g_2)$ is a Fourier multiplier operator defined as follows 
\begin{equation}
	\left(\textit{Op}(g_2)\Psi\right)(r;t,\varphi,\theta)=\frac{1}{\sqrt{2\pi}}\int_{\mathbb{R}}d\omega \sum_{m\in \mathbb{Z}} e^{i\omega t} e^{-im\varphi} g_2(\omega,m,r) \mathcal{F}_{\omega,m}(\Psi)(r,\theta),
\end{equation}
where for~$\mathcal{F}_{\omega,m}$ see~\eqref{eq: subsubsec: subsec: sec: intro, subsec 2, subsubsec 0, eq 1}. 
\end{customDefinition}

\begin{remark}\label{rem: subsec: rough version of theorem 3, g2, rem 1}
For the superradiant frequencies~$\mathcal{SF}$ we note that there exists an~$r_s\in (r_+,\bar{r}_+)$ such that~$\omega=\frac{am\Xi}{r_s^2+a^2}$.  Then, the function~$g_2(\frac{am\Xi}{r_s^2+a^2},m,\cdot)$ is not differentiable at~$r=r_s$. For a graphic representation of this non-differentiability see already Figure~\ref{fig: g2}. Furthermore, for any fixed~$r\in (r_+,\bar{r}_+)$ and for any any fixed~$m$ such that~$(\omega,m)\in \mathcal{SF}$ we have that~$g_2(\cdot,m,r)$ is not a differentiable function either. Therefore, in order to be able to use pseudodifferential commutations, we will approximate~$\mathcal{G}$ by a more regular version, which we denote as~$\widetilde{\mathcal{G}}$, see already Definition~\ref{def: proof thm 3 exp decay, def 1}. The pseudodifferential operator~$\widetilde{\mathcal{G}}$ belongs in the class of operators satisfying the assumptions of the classical Coifman--Meyer commutation Lemma~\ref{lem: sec: appendix, pseudodifferential com, lem 0}. Introducing the more regular operator~$\widetilde{\mathcal{G}}$ is essential in the pseudodifferential estimates of Section~\ref{sec: proof of Thm: rel-nondeg}. 
\end{remark}

\begin{remark}
The operator~$\mathcal{G}$ of Definition~\ref{def: sec: G, def 1} can also be defined in the asymptotically flat case~$\Lambda=0$. In the high frequency limit our relevant~$\Lambda=0$ operator is related to but differs from the relevant commutator employed in the recent work of Holzegel--Kauffman~\cite{gustav2}. 
\end{remark}

We note finally that our commutation with~$\mathcal{G}$, see~\eqref{eq: subsec: sec: intro, 2, eq 0}, in the high frequency limit, connects with the previous works~\cite{zworski,burq2,ikawa,nonnenmacher,hintz4,Dyatlov3,Dyatlov4}.

\subsection{The function~\texorpdfstring{$g_2(\omega,m,r)$}{g} graphically}\label{subsec: sec: intro: subsec 3.1}

In Figure~\ref{fig: g2} we graph the function~$g_2(\omega,m,r)$, see Definition~\ref{rough def: g2}, for the values 
\begin{equation}\label{eq: subsec: sec: intro: subsec 3.1, eq 1}
	am\omega> \frac{a^2m^2\Xi}{r_+^2+a^2},\quad am\omega= \frac{a^2m^2\Xi}{r_+^2+a^2},\quad am\omega\in \left(\frac{a^2m^2\Xi}{\bar{r}_+^2+a^2},\frac{a^2m^2\Xi}{r_+^2+a^2}\right),\quad am\omega=\frac{a^2m^2\Xi}{\bar{r}_+^2+a^2},\quad am\omega<\frac{a^2m^2\Xi}{\bar{r}_+^2+a^2}
\end{equation}
respectively, where for~$r_{\textit{crit}}(\omega,m)$ see~\eqref{eq: subsec: rough version of theorem 3, g2, 4}.
\begin{figure}[htbp]
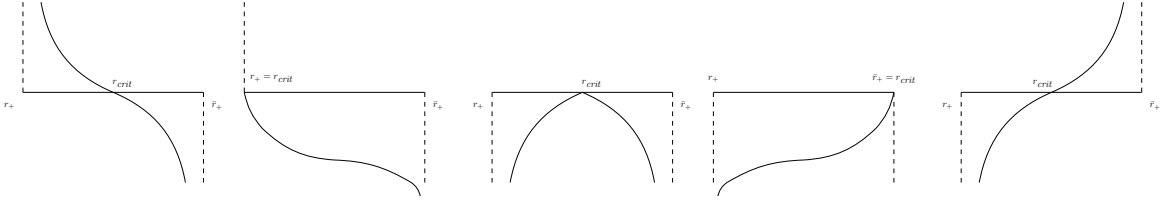

\begin{multicols}{5}
\includegraphics[scale=0.6]{NS1.g2.jpg}\par
\includegraphics[scale=0.6]{S1.g2.jpg}\par 
\includegraphics[scale=0.6]{S2.g2.jpg}\par
\includegraphics[scale=0.6]{S3.g2.jpg}\par
\includegraphics[scale=0.6]{NS2.g2.jpg}\par
\end{multicols}
\caption{The function~$g_2(\omega,m,r)$ for the values~\eqref{eq: subsec: sec: intro: subsec 3.1, eq 1} respectively}
\label{fig: g2}
\end{figure}

\subsection{A bulk integral associated with~\texorpdfstring{$g_1\frac{d}{dr^\star}+ig_2$}{g}}\label{subsec: sec: intro: subsec 3.2}

As mentioned already, we note that in the definition of the pseudodifferential operator~$\mathcal{G}$ only refers to frequencies~$\omega,m$ and not the full Carter separation of variables~\eqref{eq: subsec: sec: intro, subsec 2, eq 1}. We will employ extensively in this paper, however, the formalism of the Carter separation, since we can use the fixed freqency multiplier formalism of~\cite{mavrogiannis4} and borrow all relevant formulas (see already~\eqref{eq: subsec: sec: intro: subsec 3.2, eq 0}). In addition, we need the Morawetz estimate with the sharp degeneration at~$r_{trap}(\omega,m,\ell)$~(and not simply with a degeneration at an open interval), where~$\ell$ is the Carter frequency, see Theorem~\ref{main theorem 1}. Furthermore, the Carter separation also allows us to formulate mode stability.

Let~$g_1,g_2$ be as in~\eqref{eq: subsec: rough version of theorem 3, g2, 3} and Definition~\ref{def: sec: G, def 0.1}, and let
\begin{equation}\label{eq: subsec: sec: intro: subsec 3.2, eq 0}
v=\left(g_1\frac{d}{dr^\star}+ig_2\right)u
\end{equation}
where~$u$ is the fixed frequency Fourier projection of~$\sqrt{r^2+a^2}\chi\psi$ and~$\chi$ is an appropriate cut-off. We have the following Parseval identity
\begin{equation}
	\int_{r_+}^{\bar{r}_+}\int_\mathbb{R}  \sum_{m,\ell}|v|^2drd\omega= 	\int_{r_+}^{\bar{r}_+}\int_\mathbb{R}\int_{\mathbb{S}^2} |\mathcal{G}\left(\sqrt{r^2+a^2}\chi\psi\right)|^2 r^2dr dt d\sigma,
\end{equation}
where~$d\sigma$ is the standard metric of the unit sphere.

The main point of our choice of~$\mathcal{G}$ is to control the following fixed frequency expression
\begin{equation}\label{eq: subsec: sec: intro: subsec 3.2, eq 1}
\int_{r_+}^{\bar{r}_+} \left(\frac{1}{\Delta}\left( \omega-\frac{am\Xi}{r^2+a^2} \right)^2 |v|^2+\tilde{\lambda}|v|^2+ |v^\prime|^2 +  \frac{1}{\Delta}|v^\prime+i(\omega-\omega_+m) v|^2+\frac{1}{\Delta}|v^\prime-(\omega-\bar{\omega}_+m) v|^2\right) dr,
\end{equation}
where~$\omega_+=\frac{a\Xi}{r_+^2+a^2},~\bar{\omega}_+=\frac{a\Xi}{\bar{r}_+^2+a^2}$ and~$\tilde{\lambda}$ is associated with the eigenvalues of the angular ode, see already Section~\ref{sec: carter separation, radial}. Note that in contrast to the expression~\eqref{eq: subsubsec: subsec: sec: intro, subsec 2, subsubsec 3, eq 1}, the expression~\eqref{eq: subsec: sec: intro: subsec 3.2, eq 1} does not degenerate at top order at trapping~$r=r_{\textit{trap}}$ \underline{with respect to~$v$.}

The key ingredient in controlling the quantity~\eqref{eq: subsec: sec: intro: subsec 3.2, eq 1} is to prove that passing to~$v$ defined by the commutation~\eqref{eq: subsec: sec: intro: subsec 3.2, eq 0} produces a favorable term for all~$(\omega,m)\in\mathbb{R}\times\mathbb{Z}$. Motivated by the relevant approach of our previous~\cite{mavrogiannis} and in view of the fact that in Kerr--de~Sitter there does not exist a global timelike Killing vector field, for non-superradiant frequencies~$(\omega,m)\in(\mathcal{SF})^c$ we multiply the commuted equation, the equation satisfied by~$v$, with the multiplier~$(|\omega|+|am\Xi|)\bar{v}$ and we generate the term 
\begin{equation}
	2\frac{g_2^\prime}{g_1}|v|^2
\end{equation}
which controls the first term of~\eqref{eq: subsec: sec: intro: subsec 3.2, eq 1}. Also, note from Figure~\ref{fig: g2} that for the non-superradiant frequencies the function~$g_2$ is monotonic. To generate the remaining terms of~\eqref{eq: subsec: sec: intro: subsec 3.2, eq 1} we use the ode satisfied by~$v$, see already Lemma~\ref{prop: equation of v}. 

The treatment of the superradiant frequencies~$(\omega,m)\in\mathcal{SF}$ is slightly more elaborate. Note that for~$(\omega,m)\in \mathcal{SF}$ there exists an~$r_s(\omega,m)\in (r_+,\bar{r}_+)$ such that~$am\omega=\frac{a^2m^2\Xi}{r_s^2+a^2}$. Specifically, we multiply the equation satisfied by~$v$ with a multiplier of the form~$f_s(\omega,m,r) |am\Xi| \bar{v}$, where the smooth function~$f_s(\omega,m,r)$ degenerates as follows~$f_s(\omega,m,r_s)=0$. Then, we obtain an estimate where the LHS degenerates at the value~$r_s$. In order to `cover' that degeneration we need to appeal to a quantitative manifestation of the fact that superradiant frequencies are not trapped, see also our companion~\cite{mavrogiannis4}, and also see~\cite{DR2} for the manifestation of that phaenomenon in the asymptotically flat Kerr geometry.

We emphasize here that although we will use the formalism of Carter's separation to produce our multiplier estimates, all of our multipliers will be independent of the Carter frequency~$\ell$.

All lower order terms produced are controlled by the Morawetz estimate of our companion~\cite{mavrogiannis4}, see Theorem~\ref{main theorem 1}, since this estimate controls the quantity~\eqref{eq: subsubsec: subsec: sec: intro, subsec 2, subsubsec 3, eq 1}, in conjunction with a Poincare type estimate, see already Lemma~\ref{lem: control of the law order derivatives with v}.

\subsection{The time cut-offs and the energy density}\label{subsec: sec: intro: subsec 4}

Let~$(t^\star,r,\theta^\star,\varphi^\star)$ denote the Kerr--de~Sitter star coordinates, for the precise definition see already Section~\ref{subsec: sec: preliminaries, subsec 2}, where note that the leaves~$\{t^\star=\tau\}$ connect the event horizon~$\mathcal{H}^+$ with the cosmological horizon~$\bar{\mathcal{H}}^+$.

In view of the pseudodifferential nature of the operator~$\mathcal{G}$, see Definition~\ref{rough def: g2}, to make the statement of the main Theorem~\ref{main theorem, relat. non-deg} to be of local nature we need to first apply appropriate cut-offs. Let~$T>0$ and~$T< T+\tau_1< 2T+\tau_1< \tau_2$. We choose a smooth cut-off that satisfies the following
\begin{equation}\label{eq: subsec: sec: intro, 3, eq 1}
\chi^{(T)}_{\tau_1,\tau_2}(t^\star)=
\begin{cases}
\text{1,} &\quad\{\tau_1+T\leq t^\star\leq \tau_2-T\}\\
\text{0,} &\quad\{t^\star\leq \tau_1\}\cup\{t^\star\geq \tau_2\}, \\
\end{cases}
\end{equation}
which we often simply denote as~$\chi_{\tau_1,\tau_2}$.

Let~$T\geq 3$, let
\begin{equation}\label{eq: subsec: sec: intro, 3, eq 1.1}
	\tilde{T}=T^2-2T-1
\end{equation}
and let~$\tilde{T}\leq \tau_1$. We define
\begin{equation}\label{eq: subsec: sec: intro, 3, eq 1.2}
\chi_+ (t^\star)=\chi^{(T)}_{\tau_1,\tau_1+T^2}(t^\star),\qquad \chi_{-}(t^\star)=\chi^{(T)}_{\tau_1,\tau_1+T^2}\left(t^\star+\tilde{T}\right),
\end{equation}
also see Figure~\ref{fig: cut-offs -1}, where note that we supress the dependence of~$\chi_\pm$ on~$T$ and~$\tau_1$. 
\begin{figure}[htbp]
	\centering
	\includegraphics[scale=1]{functions0.jpg}
	\caption{The cut-offs~\eqref{eq: subsec: sec: intro, 3, eq 1.2}}
	\label{fig: cut-offs -1}
\end{figure}

Let~$\mathcal{G}_\chi=\chi\mathcal{G}\chi$, where~$\mathcal{G}$ is as in~\eqref{eq: subsec: sec: intro, 2, eq 0} and~$\chi$ is either~$\chi_+$ or~$\chi_-$, see~\eqref{eq: subsec: sec: intro, 3, eq 1.2}. We define the energy density
\begin{equation}\label{eq: new energy}
\mathcal{E}\left(\mathcal{G}_{\chi}\psi,\psi\right)\:\dot{=}\:\mathbb{T}(\partial_{t^\star}+\frac{a\Xi}{r^2+a^2}\partial_{\varphi^\star},n)[\mathcal{G}_{\chi}\psi]+\mathbb{T}(n,n)[\psi],
\end{equation}
where~$\mathbb{T}$ is the energy momentum tensor of the Klein--Gordon equation, see already Section~\ref{sec: preliminaries}, and~$n$ is the unit normal to the foliation~$\{t^\star=\tau\}$. The vector field
\begin{equation}
\partial_{t^\star}+\frac{a\Xi}{r^2+a^2}\partial_{\varphi^\star}
\end{equation}
is globally timelike except at the horizons~$\mathcal{H}^+,\bar{\mathcal{H}}^+$ where it is null, see already Lemma~\ref{lem: causal vf E,1}. Moreover, note that the energy density~\eqref{eq: new energy} satisfies
\begin{equation}
\mathcal{E}(\mathcal{G}_\chi\psi,\psi)\sim \left|\partial_{t^\star}\mathcal{G}_\chi\psi\right|^2+\Delta \left|Z^\star\mathcal{G}_\chi\psi\right|^2+|\slashed{\nabla}\mathcal{G}_\chi\psi|^2+|\partial_{t^\star}\psi|^2+|Z^\star\psi|^2+|\slashed{\nabla}\psi|^2,
\end{equation}
where for~$\Delta$ see~\eqref{eq: prototype Delta}, and for the regular vector field~$Z^\star$ see already Section~\ref{subsec: boldsymbol partial r}.

\subsection{The rough version of Theorem~\ref{main theorem, relat. non-deg}}\label{subsec: sec: intro: subsec 5}

We present the rough version of our `relatively non-degenerate' estimate, where for the detailed version of Theorem~\ref{main theorem, relat. non-deg} see already Section~\ref{sec: main theorems}.

\begin{customTheorem}{\ref{main theorem, relat. non-deg}}[rough version]\label{rough: thm relat. non-deg}

Let~$(a,M,l)$ be subextremal Kerr--de~Sitter black hole parameters and let~$\mu^2_{KG}\geq 0$. Let~$T\geq 3$ be sufficiently large and let~$\tilde{T}< \tau_1<\tau_1+T< \tau_2$, where for~$\tilde{T}$ see~\eqref{eq: subsec: sec: intro, 3, eq 1.1}. Assume that
\begin{equation*}
	\begin{aligned}
		&\text{(MS):~mode stability on the real axis for Carter's radial ode holds on a curve in subextremal}\\
		&	\qquad\quad \text{parameter space connecting}~(a,M)~\text{to the subextremal Schwarzschild--de~Sitter family},
	\end{aligned}
\end{equation*}
and let~$\psi$ be a solution of the Klein--Gordon equation~\eqref{eq: kleingordon} on the subextremal Kerr-de~Sitter black hole with parameters~$(a,M,l)$.

Then, the following `relatively non-degenerate' integrated decay energy estimate holds
\begin{equation}\label{eq: rough: thm relat. non-deg, eq 1}
\begin{aligned}
& \int_{\{t^\star=\tau_2\}} \mathbb{T}(n,n)[\psi]+\int\int_{\{\tau_1+T\leq t^\star\leq \tau_2\}} \mathcal{E}\left(\mathcal{G}_{\chi_+}\psi,\psi\right) \lesssim \int_{\{t^\star=\tau_1\}} \mathbb{T}(n,n)[\psi]+\frac{1}{T}\int\int_{\{\tau_1\leq t^\star\leq \tau_1+T\}} \mathcal{E}(\mathcal{G}_{\chi_{-}}\psi,\psi),\\
\end{aligned}    
\end{equation}
where for the smooth cut-offs~$\chi_+,\chi_{-}$ see~\eqref{eq: subsec: sec: intro, 3, eq 1.2} and Figure~\ref{fig: cut-offs -1}, and for the operator~$\mathcal{G}$ see~\eqref{eq: subsec: sec: intro, 2, eq 0}. The constant implicit in~\eqref{eq: rough: thm relat. non-deg, eq 1} is independent of~$T$.

For axisymmetric solutions~$\psi$ of the Klein--Gordon equation we obtain unconditionally a physical space version of~\eqref{eq: rough: thm relat. non-deg, eq 1} for all subextremal black hole parameters~$(a,M,l)$ and all masses~$\mu_{KG}^2\geq 0$. 
\end{customTheorem}

 \begin{remark}\label{rem: subsec: sec: intro: subsec 5, rem 1}
 	For fixed~$(a,M,l)$, the constants implicit in~$\lesssim$, in the RHS of the estimate~\eqref{eq: rough: thm relat. non-deg, eq 1}, blow up in the limit~$\mu_{KG}\rightarrow 0$ as well as in the limit~$a\rightarrow 0$. Note, however, that $\mu_{KG}=0$ or~$a=0$ are indeed allowed in the theorem, i.e. the constant for $\mu_{KG}=0$ or for~$a=0$ are finite. It might be possible to modify our proof, see specifically Section~\ref{subsec: sec: proof of Theorem 2, subsec 8}~(for the issue in the limit~$\mu^2_{KG}\rightarrow 0$) and Section~\ref{subsec: sec: proof of Theorem 2, subsec 5.5}~(for the issue in the limit~$a\rightarrow 0$), in order to obtain uniform estimates near~$\mu^2_{KG}=0$ and near~$a=0$, which would be more desirable. 
 \end{remark}

 \begin{remark}
 	As mentioned in our companion~\cite{mavrogiannis4} the very slowly rotating case~$|a|\ll M,l$ is in fact included in the condition~(MS), where the smallnes depends on~$\mu_{KG}$. Therefore, Theorem~\ref{main theorem, relat. non-deg} holds for very slowly rotating Kerr--de~Sitter black holes, as well. In view of the recent work of Hintz~\cite{hintz2021mode}, for~$\mu^2_{KG}=0$, we have that subetremal Kerr--de~Sitter black hole parameters with sufficiently small cosmological constant~$0<|a|<M\ll l$ are in fact included in the set~(MS), where the constant implicit in~$\ll$ depends on the difference~$M-|a|$. Therefore, the result of Theorem~\ref{main theorem, relat. non-deg} holds also in this case.
 \end{remark}

\begin{remark}
In fact we can control the stronger quantity~\eqref{eq: subsec: sec: intro: subsec 3.2, eq 1}, summed over frequencies, on the left hand side of our main estimate~\eqref{eq: rough: thm relat. non-deg, eq 1}.
\end{remark}

\begin{remark}
Unlike the `relatively non-degenerate' estimate~\eqref{eq: SdS relatively non degenerate} in Schwarzschild--de~Sitter, the estimate of Theorem~\ref{rough: thm relat. non-deg} exhibits an additional averaged bulk term on the right hand side. This relates to the fact that the operator~$\mathcal{G}$, see Definition~\ref{rough def: g2}, is of pseudodifferential nature. Nonetheless, because of the time translational relation of the cut-offs~$\chi_-,~\chi_+$ and the weight~$\frac{1}{T}$, the estimate of Theorem~\ref{rough: cor: thm relat. non-deg, cor 1} can similarly be immediately iterated in consecutive spacetime slabs of fixed time length in order to prove exponential decay in energy for the solutions of the Klein--Gordon equation, see Corollary~\ref{rough: cor: thm relat. non-deg, cor 1}. 
\end{remark}

Exponential decay is an immediate Corollary by iterating the result of Theorem~\ref{rough: thm relat. non-deg} in consecutive spacetime slabs of fixed time length.

\begin{customCorollary}{\ref{cor: main theorem, relat. non-deg, cor 3}}[rough version]\label{rough: cor: thm relat. non-deg, cor 1}
Let the assumptions of Theorem~\ref{rough: thm relat. non-deg} be satisfied.

For~$T\geq 0$ sufficiently large and for~$\chi_{\textit{data}}(t^\star)=\chi_{0,T^2}(t^\star)$ we have
\begin{equation}
\begin{aligned}
&	\int_{\{t^\star=\tau\}} \mathbb{T}(n,n)[\psi]  \lesssim \Bigg(\int_{\{t^\star=0\}}  \mathbb{T}(n,n)[\psi]+   \int_{\supp \chi_{\textit{data}}}d\tau^\prime  \int_{\{t^\star=\tau^\prime\}} \mathcal{E}(\mathcal{G}_{\chi_{\textit{data}}}\psi,\psi)\Bigg) \cdot e^{-c\tau},\\
&\qquad\qquad\qquad\qquad\qquad\qquad\qquad \sup_{\{t^\star=\tau\}}|\psi-\underline{\psi}|	\lesssim \sqrt{E}\cdot e^{-c \tau},
\end{aligned}
\end{equation}
where~$E$ is an appropriate integral quantity defined on~$\{t^\star=0\}$ and~$\underline{\psi}$ is a constant satisfying~$|\underline{\psi}|\leq  |\underline{\psi}(0)|+C\sqrt{E}$ if~$\mu_{KG}=0$ and~$\underline{\psi}\equiv 0$ if~$\mu^2_{KG}>0$. 
\end{customCorollary}

\begin{proof}
	See Section~\ref{sec: cor: main theorem, relat. non-deg, cor 3}. 
\end{proof}

\begin{remark}
	Note that we do not explicitly give the dependence of~$\underline{\psi}$ on initial data. Specifically, we note that the constant~$\underline{\psi}$ is obtained from a limiting argument, see already Section~\ref{sec: cor: main theorem, relat. non-deg, cor 3}. 
\end{remark}

\begin{remark}
For a `relatively non-degenerate' estimate for the inhomogeneous wave equation see already Theorem~\ref{thm: main thm extended region} in Section~\ref{sec: main theorems}. Note that the estimate of Theorem~\ref{thm: main thm extended region} is used in our forthcoming~\cite{mavrogiannis3} to prove global non-linear stability and exponential decay for the solutions of the quasilinear wave equation on the subextremal Kerr--de~Sitter black hole background, if the condition~(MS) holds. 
\end{remark}

\subsection{Acknowledgments} The author would like to thank Mihalis Dafermos for introducing him to this problem, for his invaluable suggestions and for his uncountably infinite patience. Preliminary version of the results of the current paper were outlined in my thesis~\cite{mavrogiannisthesis}. Previous versions of the operator~$\mathcal{G}$ were inspired by conversations with G.~Moschidis, Y.~Shlapentokh-Rothman and D.~Gajic, to whom the author acknowledges their vital assistance. Moreover, the author would like to thank G. Moschidis for numerous conversations on various mathematical topics and for very enlightening discussions on the estimates of Section~\ref{sec: proof of Thm: rel-nondeg}. Finally, the author would like to acknowledge many helpful discussions with Renato Velozo Ruiz, particularly on issues of trapping.

\section{Preliminaries}\label{sec: preliminaries}

In order for the present paper to be read independently, we repeat some necessary definitions and known results that have already been presented in our companion~\cite{mavrogiannis4}.

\subsection{The subextremal parameters}\label{subsec: delta polynomial}

We define the subextremal set of black hole parameters. 

\begin{definition}\label{def: subextremality, and roots of Delta}
	Let~$l>0$. We define the set
	\begin{equation}\label{eq: subextremal set}
	\mathcal{B}_{l}\:\dot{=}\:\{ (a,M)\in \mathbb{R}\times(0,\infty):~ \Delta ~\textrm{attains four distinct real roots} \},
	\end{equation}
	where for~$\Delta$ see~\eqref{eq: prototype Delta}.

	If $l>0$ and $(a,M)\in\mathcal{B}_l$ we say that the black hole parameters $(a,M,l)$ correspond to a subextremal Kerr--de~Sitter black hole and denote as 
	\begin{equation}
	\bar{r}_-(a,M,l)\quad<\quad 0\quad\leq \quad r_-(a,M,l)\quad<\quad r_+(a,M,l)\quad<\quad \bar{r}_+(a,M,l)
	\end{equation}
	the four distinct real roots of the $\Delta$ polynomial. Note that in the Schwarzschild--de~Sitter case~$(a=0)$ we have~$r_-=0$. 
\end{definition}

For~$(a,M,l)$ be subextremal Kerr--de~Sitter black hole parameters, see Definition~\ref{def: subextremality, and roots of Delta}. Then, we have the following graphic representation of the polynomial~$\Delta$. 
\begin{figure}[htbp]
	\centering
	\includegraphics[scale=0.6]{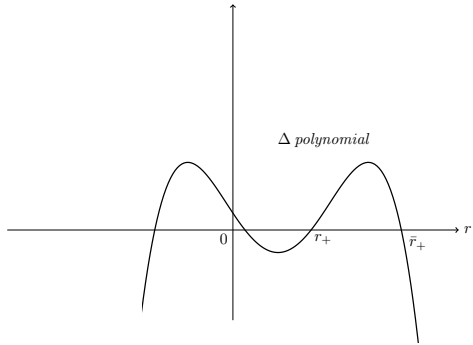}
	\caption{The $\Delta$ polynomial with 4 roots}
	\label{fig: Delta}
\end{figure}

We have the following

\begin{lemma}\label{lem: sec: properties of Delta, lem 1, derivatives of Delta}
	Let~$l>0$ and~$(a,M)\in\mathcal{B}_l$. Then, the following hold 
	\begin{equation}\label{eq: lem: sec: properties of Delta, lem 1, derivatives of Delta, eq 1}
	\begin{aligned}
	&   \frac{d\Delta}{dr}=2r-\frac{4}{l^{2}}r^{3}-a^{2}\frac{2}{l^{2}}r-2M,\qquad \frac{d^2\Delta}{dr^2}=2-a^{2}\frac{2}{l^{2}}-\frac{12}{l^{2}}r^{2},\qquad \frac{d^3\Delta}{dr^3}=-\frac{24}{l^2}r \\  
	\end{aligned}
	\end{equation}
	\begin{equation}\label{eq: lem: sec: properties of Delta, lem 1, derivatives of Delta, eq 3}
	\frac{d\Delta}{dr}(r_+)>0,\quad \frac{d\Delta}{dr}(\bar{r}_+)<0,\quad \frac{d}{dr}\left(\frac{\Delta}{(r^2+a^2)^2}\right)(r_+)>0,\quad \frac{d}{dr}\left(\frac{\Delta}{(r^2+a^2)^2}\right)(\bar{r}_+)<0,
	\end{equation}
	where for~$\Delta$ see~\eqref{eq: prototype Delta}. 
\end{lemma}

\subsection{The Kerr--de~Sitter manifold}\label{subsec: sec: preliminaries, subsec 2}

We define the Kerr--de~Sitter manifold to be 
\begin{equation}
\mathcal{M}=[r_+,\bar{r}_+]_r\times\mathbb{R}_{t^\star}\times\mathbb{S}^2_{(\theta^\star,\varphi^\star)},
\end{equation}
where we term the coordinates~$(t^\star,r,\theta^\star,\varphi^\star)$ Kerr--de~Sitter star coordinates, see our companion~\cite{mavrogiannis4} for their definition. We define the event horizon and the cosmological horizon respectively as 
\begin{equation}
	\mathcal{H}^+=\{r=r_+,~t^\star> -\infty\},\qquad \bar{\mathcal{H}}^+=\{r=\bar{r}_+,~t^\star> -\infty\}.
\end{equation}
Note that the hypersurfaces~$\mathcal{H}^+,\bar{\mathcal{H}}^+$ are boundaries of the manifold~$\mathcal{M}$. The spacelike hypersurfaces~$\{t^\star=c\}$ connect~$\mathcal{H}^+$ with~$\bar{\mathcal{H}}^+$. We refer the reader to the Carter--Penrose diagram of Figure~\ref{fig: The penroseKdS}.

\begin{figure}[htbp]
	\centering
	\includegraphics[scale=0.9]{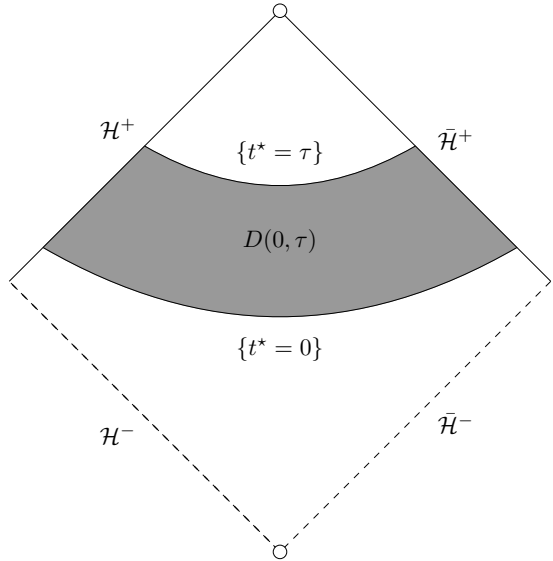}
	\caption{Carter-Penrose diagram of Kerr--de~Sitter}
	\label{fig: The penroseKdS}
\end{figure}

\subsection{The metric in Boyer--Lindquist coordinates}\label{subsec: boyer Lindquist coordinates}

To define the metric on $\mathcal{M}$ it is useful to recall first the Boyer Lindquist coordinates~$(t,r,\theta,\varphi)$ on the Kerr--de~Sitter manifold from our companion~[Section 2,~\cite{mavrogiannis4}]. Specifically, note that~$r$ coincides with the~$r$ of the Kerr--de~Sitter star coordinates above and~$\theta=\theta^\star$. We need the following definition.

\begin{definition}\label{def: delta+, and other polynomials}
	Let~$l>0$ and~$(a,M)\in \mathcal{B}_{l}$. Then, we define
	\begin{equation}
		\begin{aligned}
			\rho^2 \:\dot{=}&\:r^{2}+a^{2}\cos^{2}\theta,\qquad \Delta_{\theta}\:\dot{=}\:1+\frac{a^{2}}{l^{2}} \cos^{2}\theta,\qquad \Xi\:\dot{=}\:1+\frac{a^{2}}{l^{2}}. 
		\end{aligned}
	\end{equation}
\end{definition}

We now define the Kerr--de~Sitter metric in Boyer--Lindquist coordinates 

\begin{definition}\label{def: metric in boyer lindquist coordinates}
	Let~$l>0$ and~$(a,M)\in\mathcal{B}_l$. We define the Kerr--de~Sitter metric as 
\begin{equation}\label{eq: metric in boyer lindquist coordinates}
	\begin{aligned}
		g_{a,M,l} &	= \frac{\rho^2}{\Delta}dr^{2}+\frac{\rho^2}{\Delta_{\theta}}d \theta^{2}+\frac{\Delta_{\theta}(r^{2}+a^{2})^{2}-\Delta a^{2}\sin^{2}\theta }{\Xi ^{2}\rho^2}\sin^{2}\theta d \varphi^{2}   -2\frac{\Delta_{\theta}(r^{2}+a^{2})-\Delta}{\Xi\rho^2} a \sin^{2}\theta d\varphi dt\\
		&\qquad\qquad -\frac{\Delta -\Delta_{\theta}a^{2}\sin^{2}\theta}{\rho^2}dt^{2},
	\end{aligned}
\end{equation}
where for $\Delta$ see~\eqref{eq: prototype Delta}, and for~$\Xi,\rho^2,\Delta_\theta$ see Definition~\ref{def: delta+, and other polynomials}. The metric~\eqref{eq: metric in boyer lindquist coordinates} extends smoothly to~$\mathcal{M}$ and moreover the metric depends smoothly on~$a,M$ for fixed~$l$.

The vector field~$W=\partial_t+\frac{a\Xi}{r^2+a^2}\partial_{\varphi}$ is timelike away from the horizons, see already Lemma~\ref{lem: causal vf E,1}. We time orient~$\mathcal{M}$ by the vector~$W$.  
\end{definition}

\begin{remark}
	We have that~$\partial_t=\partial_{t^\star}$ and~$\partial_\varphi=\partial_{\varphi^\star}$. Thus, these Boyer Lindquist coordinate vector fields extend regularly at~$\mathcal{H}^+$ and~$\bar{\mathcal{H}}^+$, despite the fact that the coordinates themselves do not. 
\end{remark}

\subsection{A causal vector field for the entire subextremal Kerr--de~Sitter family}

We present Lemma~\ref{lem: causal vf E,1}, proved in our companion~\cite{mavrogiannis4}, which states that the vector field
\begin{equation}
	W=\partial_t+\frac{a\Xi}{r^2+a^2}\partial_{\varphi}
\end{equation}
is timelike in~$\mathcal{M}$ except at the horizons~$\mathcal{H}^+,\bar{\mathcal{H}}^+$ where it is null. By Definition~\ref{def: metric in boyer lindquist coordinates}, this then is future directed with respect to the time orientation it determines.

\begin{lemma}\label{lem: causal vf E,1}
	Let~$l>0$ and~$(a,M)\in\mathcal{B}_l$. The vector field
	\begin{equation}\label{eq: lem: causal vf E,1, eq 1}
	W\:\dot{=}\:\partial_{t}+\frac{a\Xi}{r^2+a^2} \partial_{\varphi},  
	\end{equation}
	is timelike in~$\{r_+<r<\bar{r}_+\}$ and null on the event horizon~$\mathcal{H}^+$ and on the cosmological horizon~$\bar{\mathcal{H}}^+$. Note that~\eqref{eq: lem: causal vf E,1, eq 1} is also timelike for~$\Lambda=0$. 
	
	Moreover, for axisymmetric solutions~$\psi$ of the Klein--Gordon equation~\eqref{eq: kleingordon} the following holds 
	\begin{equation}
	\rm{K}^{\partial_{t^\star}+\frac{a\Xi}{r^2+a^2} \partial_{\varphi^\star}}[\psi]=0,
	\end{equation}
	where for~$\rm{K}^X$ see already Section~\ref{sec: divergence theorem}.
\end{lemma}
\begin{proof}
	See our companion~\cite{mavrogiannis4}.
\end{proof}

\subsection{The regular vector field \texorpdfstring{$Z^\star$}{Z}}\label{subsec: boldsymbol partial r}

Let~$l>0$ and~$(a,M)\in\mathcal{B}_l$. Let~$\partial_r|_{star}$ be the coordinate vector field of the modified Kerr--de~Sitter star coordinates~$(t^\star,r,\theta^\star,\varphi^\star)$, see~\cite{mavrogiannis4}, and~$\partial_r|_{\textit{BL}}$ be the coordinate vector field of the Boyer--Lindquist coordinates~$(t,r,\theta,\varphi)$, see Definition~\cite{mavrogiannis4}.

For a sufficiently small
\begin{equation}
\epsilon_{\textit{hyb}}(a,M,l)>0,
\end{equation}
we choose a smooth cut-off that satisfies the following
\begin{equation}
\chi_{\textit{hyb}}(r)=
\begin{cases}
0,\:& r\in (r_+,r_++\epsilon_{\textit{hyb}})\cup (\bar{r}_+-\epsilon_{\textit{hyb}},\bar{r}_+),\\
1,\: & r\in(r_++2\epsilon_{\textit{hyb}},\bar{r}_+-2\epsilon_{\textit{hyb}}). 
\end{cases}    
\end{equation}

Then, we define the hybrid vector field 
\begin{equation}\label{eq: subsec: boldsymbol partial r, eq 1}
Z^\star= (1-\chi_{\textit{hyb}}(r))\partial_r|_{\textit{star}}+\chi_{\textit{hyb}}(r) \partial_{r}|_{\textit{BL}}.
\end{equation}

The vector field $Z^\star$ coincides with the coordinate vector field $\partial_r|_{BL}$, of Boyer--Lindquist coordinates, in the region~$[R_-,R_+]$, where for~$R_\pm$ see Theorem~\ref{main theorem 1}.  The significance of~$Z^\star$ will be that the derivatives~$(Z^\star\psi)^2$ do \texttt{not} degenerate in the integrand of the Morawetz estimate of Theorem~\ref{main theorem 1}.

\subsection{Hawking--Reall vector fields}\label{subsec: Hawking-Reall v.f.}

We define the null normals of the event horizon~$\mathcal{H}^+$ and the cosmological horizon~$\bar{\mathcal{H}}^+$ respectively as the Hawking--Reall vector fields
\begin{equation}\label{eq: Hawking vector field on the event horizon}
K^+=\partial_{t^\star}+\frac{a\Xi}{r_{+}^{2}+a^{2}}\partial_{\varphi^\star},\qquad \bar{K}^+=\partial_{t^\star}+\frac{a\Xi}{\bar{r}_{+}^{2}+a^{2}}\partial_{\varphi^\star}.
\end{equation}

\subsection{Volume forms and normals}\label{subsec: admissible hypersurfaces}

We denote by~$n$ the normals of the foliation~$\{t^\star=\tau\}$. Moreover, when convenient we denote the normal vector fields~\eqref{eq: Hawking vector field on the event horizon} also as~$n$. The context will make it clear what is meant.

We use the notation~$\slashed{g},~\slashed{\nabla},~d\sigma_{\mathbb{S}^2}$ to denote respectively the induced~$g_{a,M,l}$ metric on the~$\mathbb{S}^2$ factors of~$\mathcal{M}$, the covariant derivative with respect to~$\slashed{g}$ and the standard metric of the unit sphere. Then, we have
\begin{equation}
\slashed{g}= r^2v(r,\theta)d\sigma_{\mathbb{S}^2},
\end{equation} 
with~$v(r,\theta)\sim 1$ in~$\mathcal{M}$, where the unnecessary weights in~$r$ are included to compare with the asymptotically flat~$\Lambda=0$ Kerr black hole.

The spacetime volume form and the volume form of hypersurfaces~$\{t^\star=c\}$ are given respectively by
\begin{equation}\label{eq: subsec: volume forms of spacelike hypersurfaces, eq 0}
dg= v(r,\theta)dt^\star drd\slashed{g}=v(r,\theta)\frac{\Delta}{r^{2}+a^{2}}dt^\star dr^\star d\slashed{g}=v(r,\theta)dtdr d\slashed{g},\qquad dg_{\{t^\star=\tau\}}= v(r,\theta) dr d\theta^\star d\varphi^\star,
\end{equation}
where~$\rho^2=r^2+a^2\cos^2\theta$ and for~$\Delta$ see~\eqref{eq: prototype Delta}, with~$v(r,\theta)\sim 1$. We define the volume forms of the null hypersurfaces~$\mathcal{H}^+,\bar{\mathcal{H}}^+$ respectively as 
\begin{equation}
dg=(K^+)^\flat\wedge dg_{\mathcal{H}^+},\qquad dg=(\bar{K}^+)^\flat \wedge dg_{\bar{\mathcal{H}}^+},
\end{equation}
where~$\flat$ here is the flat-musical isomorphism with respect to the Kerr--de~Sitter metric, see~\eqref{eq: prototype metric in BL coordinates}, and~$K^+,\bar{K}^+$ are the Hawking--Reall vector fields of Section~\ref{subsec: Hawking-Reall v.f.}.

\subsection{Causal domains}\label{subsec: causal domains}

We define the following causal spacetime domains.

\begin{definition}\label{def: causal domain, def 1}
	We define
	\begin{equation}
	D(\tau_1,\tau_2)\:\dot{=}\: \{\tau_1\leq t^\star\leq \tau_2\} ,\qquad     D(\tau_1,\infty)\:\dot{=}\: \{t^\star\geq \tau_1\},
	\end{equation}
	where the spacelike hypersurfaces~$\{t^\star=\tau\}$ connect the event horizon~$\mathcal{H}^+$ with the cosmological horizon~$\bar{\mathcal{H}}^+$, see Figure~\ref{fig: The penroseKdS}, or~\cite{mavrogiannis4} for their precise definition. 
\end{definition}

In our Theorem \ref{thm: main thm extended region} we will need the following extended manifold

\begin{definition}\label{def: causal domain, def 2}
	Let~$l>0$ and~$(a,M)\in\mathcal{B}_l$. Let~$\delta>0$ be sufficiently small. We define the extended manifold 
	\begin{equation}
	\mathcal{M}_{\delta}=[r_+-\delta,\bar{r}_++\delta]_r\times \mathbb{R}_{t^\star}\times\mathbb{S}^2_{(\theta^\star,\varphi^\star)},
	\end{equation}
with respect to the Kerr--de~Sitter star coordinates~$(t^\star,r,\theta^\star,\varphi^\star)$ see Section~\ref{subsec: sec: preliminaries, subsec 2}. The Kerr--de~Sitter metric~\eqref{eq: metric in boyer lindquist coordinates} extends smoothly to~$\mathcal{M}_\delta$ by its analytic expression near~$r_+,\bar{r}_+$ in Kerr--de~Sitter star coordinates~(see~\cite{mavrogiannis4}). Note that the hypersurfaces~$\{t^\star=c\}$ are spacelike. We define the spacelike boundaries of the manifold~$\mathcal{M}_\delta$ to be 
	\begin{equation}
	\mathcal{H}^+_\delta=\{r=r_+-\delta,\:t^\star> -\infty\},\qquad \bar{\mathcal{H}}^+_\delta=\{r=\bar{r}_++\delta,\: t^\star> -\infty\}.
	\end{equation}
	Now, we define the causal domains
	\begin{equation}
	D_\delta(\tau_1,\tau_2)\:\dot{=}\: \{\tau_1 \leq t^\star\leq \tau_2\}\subset \mathcal{M}_\delta,\qquad  D_\delta(\tau_1,+\infty)\:\dot{=}\: \{t^\star\geq \tau_1\}\subset \mathcal{M}_\delta. 
	\end{equation}
\end{definition}

\begin{remark}
	We will use the~$\delta$-extended manifolds of Definition~\ref{def: causal domain, def 2} in Theorem~\ref{thm: main thm extended region}. 
\end{remark}

\subsection{Coarea formula}\label{subsec: coarea formula}

Let~$f$ be a continuous non-negative function. Then, note the following holds 
\begin{equation}\label{eq: subsec: coarea formula, eq 1}
\int\int_{D(\tau_1,\tau_2)} \frac{1}{r}\: f dg \sim  \int _{\tau_1}^{\tau_2}d\tau \int_{\{t^\star=\tau\}} f dg_{\{t^\star=\tau\}}. 
\end{equation}
(We have included the~$r$ factor so that the constant implicit in~$\sim$ does not degenerate as~$l=\infty$).

\subsection{The energy momentum tensor and the divergence theorem}\label{sec: divergence theorem}

\begin{definition}\label{def: sec: preliminaries, def 1}
	Let $g$ be a smooth Lorentzian metric. The energy momentum tensor associated to a solution of the inhomogeneous Klein--Gordon equation
	\begin{equation}
		\Box_g\psi-\mu^2_{\textit{KG}}\psi=F,
	\end{equation}
	where~$F$ is a sufficiently regular function, is defined as
	\begin{equation}
	\mathbb{T}_{\mu\nu}[\psi]\:\dot{=}\:\partial_{\mu}\psi\partial_{\nu}\psi-\frac{1}{2}g_{\mu\nu}\Big(|\nabla\psi|_g^2 +\mu_{\textit{KG}}^2|\psi|^2 \Big),
	\end{equation}
	where~$|\nabla\psi|^2_g=g^{\gamma\delta}\partial_{\gamma}\psi\partial_{\delta}\psi$. The energy current, with respect to a vector field~$X$, is defined as
	\begin{equation}
	J^{X}_{\mu}[\psi]\:\dot{=}\: \mathbb{T}_{\mu\nu}[\psi]X^{\nu}.
	\end{equation}
	Note that for~$X,N$ everywhere causal future directed vector fields, the following holds 
	\begin{equation}
	J^X_\mu[\psi]N^\mu=\mathbb{T}(X,N)[\psi]\geq 0.
	\end{equation}
	
	The divergence of the energy current is 
	\begin{equation}
	\nabla^{\mu}J^{X}_{\mu}[\psi] =X(\psi)\cdot F +\rm{K}^X[\psi], 
	\end{equation}
	where 
	\begin{equation}
	\begin{aligned}
	\rm{K}^X[\psi]	= \frac{1}{2}\mathbb{T}_{\mu\nu}\:^{(X)}\pi^{\mu\nu}=\frac{1}{2}\left(\partial_\mu\psi\partial_\nu\psi-\frac{1}{2}g_{\mu\nu}|\nabla\psi|^2\right)\pi^{\mu\nu} -\frac{1}{2}\mu^2_{\textit{KG}}|\psi|^2 \Tr \: ^{(X)}\pi,\qquad 
	^{(X)}\pi^{\mu\nu}		=\frac{1}{2}\big( \nabla^{\mu}X^{\nu}+\nabla^{\nu}X^{\mu}  \big).
	\end{aligned}
	\end{equation}
\end{definition}

The following is the divergence theorem.

\begin{lemma}\label{lem: divergence theorem}
	Let~$l>0$ and~$(a,M)\in\mathcal{B}_l$. Let~$0\leq \tau_1\leq \tau_2$. Let~$\psi$ satisfy the inhomogeneous Klein--Gordon equation 
	\begin{equation}
	\Box_g\psi-\mu_{\textit{KG}}^2\psi=F,
	\end{equation}
	on~$D(\tau_1,\tau_2)$ where~$g=g_{a,M,l}$.

	Then, the following holds
	\begin{equation}\label{eq: lem: divergence theorem, eq 1}
	\begin{aligned}
	&	\int_{\{t^\star=\tau_2\}}J^{X}_{\mu}[\psi]n^{\mu}+\int_{\mathcal{H}^{+}\cap D(\tau_1,\tau_2)}J^{X}_{\mu}[\psi]n^{\mu}+\int_{\bar{\mathcal{H}}^{+}\cap D(\tau_1,\tau_2)}J^{X}_{\mu}[\psi]n^{\mu}   +\int\int_{D(\tau_1,\tau_2)}\rm{K}^X[\psi]\\
	&	\quad\quad\quad\quad\quad=\int_{\{t^\star=\tau_1\}} J^{X}_{\mu}[\psi]n^{\mu}-\int\int_{D(\tau_1,\tau_2)} X\psi\cdot F.
	\end{aligned}
	\end{equation}
	Note that the integrals and normals of~\eqref{eq: lem: divergence theorem, eq 1} are to be understood with respect to the volume forms and normals of Sections~\ref{subsec: admissible hypersurfaces},~\ref{subsec: Hawking-Reall v.f.}. A similar statement holds for~$D_\delta(\tau_1,\tau_2)$. 
\end{lemma}

For a further study of currents related to partial differential equations, see the monograph of Christodoulou~\cite{christodoulou}.

\subsection{Time cutoffs}\label{subsec: sec: carter separation, subsec 1}

Let~$T\geq 3$ and~$T\leq T+\tau_1\leq \tau_2-T\leq \tau_2$. We choose smooth cut-offs that will satisfy the following
\begin{equation}\label{eq: subsec: sec: carter separation, subsec 1, eq 2}
\chi^{(T)}_{\tau_1,\tau_2}(t^\star) \:=\: 
\begin{cases}
\text{1,} &\quad\{\tau_1+T\leq t^\star\leq \tau_2-T\}\\
\text{0,} &\quad\{t^\star\leq \tau_1\}\cup\{t^\star\geq \tau_2\}, \\
\end{cases}
\qquad \eta^{(T)}_{\tau_1}(t^\star) \:=\: 
\begin{cases}
\text{1,} &\quad\{t^\star\geq T+\tau_1\}\\
\text{0,} &\quad\{t^\star\leq \tau_1\}\\
\end{cases}
\end{equation}
which we often simply denote as~$\chi_{\tau_1,\tau_2},~\eta_{\tau_1}$ respectively. We also require that~$|Z^\star\eta|+|\partial_{t^\star}\eta|+|\slashed{\nabla}\eta|\leq  \frac{B}{T}$ and~$|\partial_{t^\star}^2\eta|\leq \frac{B}{T^2}$, with respect to the Kerr--de~Sitter--star coordinates, for a constant~$B$ independent of~$T,\tau_1,\tau_2$. Also see Figure~\ref{fig: cut-off}. 

\begin{figure}[htbp]
	\centering
	\includegraphics[scale=1]{functions-1.jpg}
	\caption{The cut-off~\eqref{eq: subsec: sec: carter separation, subsec 1, eq 2}}
	\label{fig: cut-off}
\end{figure}

We define the cut-offed functions 
\begin{equation}\label{eq: subsec: sec: carter separation, subsec 1, eq 3}
\psi_{\tau_1,\tau_2}(t^\star,r,\varphi^\star,\theta)\dot{=}\:\chi^{(T)}_{\tau_1,\tau_2}(t^\star)\psi(t^\star,r,\varphi^\star,\theta), \qquad \psi_{\tau_1}(t^\star,r,\varphi^\star,\theta)\dot{=}\:\eta^{(T)}_{\tau_1}(t^\star)\psi(t^\star,r,\varphi^\star,\theta).
\end{equation}

For~$T\geq 3$, let~$\tilde{T}\leq \tau_1$, where for~$\tilde{T}$ see~\eqref{eq: subsec: sec: intro, 3, eq 1.1}, and define
\begin{equation}
\chi_+=\chi^{(T)}_{\tau_1,\tau_1+T^2}(t^\star),\qquad \chi_{-}(t^\star)=\chi^{(T)}_{\tau_1,\tau_1+T^2}\left(t^\star+\tilde{T}\right),
\end{equation}
also see Figure~\ref{fig: cut-offs -1}. We also require that~$|\frac{d}{dt^\star}\chi_+|+|\frac{d}{dt^\star}\chi_-|\leq  \frac{B}{T}$.

\subsection{Tortoise coordinate}\label{subsec: tortoise coordinate}
We define the tortoise coordinate 
\begin{equation}\label{eq: tortoise coordinate}
\frac{dr^\star}{dr}=\frac{r^{2}+a^{2}}{\Delta(r)},
\end{equation}
where for~$\Delta(r)$ see~\eqref{eq: prototype Delta}, with~$r^\star(r_{\Delta,\textit{frac}})=0$, where 
\begin{equation}
r_{\Delta,\textit{frac}}\in (r_+,\bar{r}_+)
\end{equation}
is the unique local maximum of
\begin{equation}
\frac{\Delta}{(r^2+a^2)^2},
\end{equation}
also see our companion~\cite{mavrogiannis4}. Note that~$r^\star(r)$ is a diffeomorphism mapping~$(r_+,\bar{r}_+)\rightarrow (-\infty,+\infty)$

Note that we will often use the primed notation
\begin{equation}
^\prime=\frac{d}{dr^\star}
\end{equation}
for the~$r^\star$-derivative. Finally, for a value~$\alpha\in[r_+,\bar{r}_+]$, we will often use the notation 
\begin{equation}
\alpha^\star=r^\star (\alpha),
\end{equation}
where note~$r_+^\star=-\infty,~\bar{r}_+^\star=+\infty$.

\subsection{Poincare--Wirtinger inequality}\label{subsec: poincare wirtigner}

Note the following Lemma

\begin{lemma}
	Assume that~$\psi$ is a smooth function on the manifold~$\mathcal{M}$. Then, for any~$\tau\geq 0 $ we have the following 
	\begin{equation}
	\int_{\{t^\star=\tau\}} |\psi-\underline{\psi}(\tau)|^2\leq B \int_{\{t^\star=\tau\}} J^n_\mu[\psi]n^\mu 
	\end{equation}
	where~$\underline{\psi}(\tau)=\frac{1}{|\{t^\star=\tau\}|}\int_{\{t^\star=\tau\}}\psi$ and~$|\{t^\star=\tau\}|=\int_{\{t^\star=\tau\}} 1$ with respect to the volume forms of Section~\ref{subsec: admissible hypersurfaces}.
\end{lemma}

\section{The~\texorpdfstring{$(t,\varphi)$}{g} Fourier projections and Carter's separation in Kerr--de~Sitter}\label{sec: carter separation, radial}

Let~$l>0$,~$(a,M)\in\mathcal{B}_l$ and~$\mu^2_{KG}\geq0$. In the present Section we present the Fourier projections and Carter's separation of variables we will use in the present paper.

\subsection{Fourier decomposition in the~$t,\varphi$ directions}\label{subsec: sec: carter separation, radial, subsec 1}

Let~$\omega\in \mathbb{R},m\in\mathbb{Z}$ and let~$(t,r,\theta,\varphi)$ be the Boyer--Lindquist coordinates. Suppose that~$\Psi:\mathcal{M}\rightarrow\mathbb{R}$ is a smooth function with~$\supp \Psi \subset \{0\leq t^\star<\infty\}$. We define the following Fourier transform
\begin{equation}\label{eq: subsec: sec: carter separation, radial, eq -2}
	\begin{aligned}
		\mathcal{F}_{\omega,m}(\Psi)(r,\theta) &	= \int_{\mathbb{R}}\int_0^{2\pi}e^{-i\omega t} e^{im\varphi}\Psi (t,r,\theta,\varphi) d\varphi dt,
	\end{aligned}
\end{equation}
where note that the integral~\eqref{eq: subsec: sec: carter separation, radial, eq -2} converges pointwise.

Throughout this paper we will use the Fourier trasform~\eqref{eq: subsec: sec: carter separation, radial, eq -2} with respect to the Boyer--Lindquist coordinates~$t,\varphi$, for test functions~$\Psi$ whose support is a subset of~$\{0\leq t^\star<\infty\}$. For the Kerr--de~Sitter star coordinate~$t^\star$ see Section~\ref{subsec: sec: preliminaries, subsec 2}. The Fourier decomposition in the~$t,\varphi$ directions is also used to define the pseudodifferential operator~$\mathcal{G}$, see Definition~\ref{rough def: g2}, that is the central object of the present paper.

\subsection{The superradiant frequencies}\label{subsec: superradiant frequencies}

An explicit computation shows that the~$\bar{K}^+$ energy flux along the event horizon~$\mathcal{H}^+$ is given by
\begin{equation}\label{eq: subsec: superradiant frequencies, eq 0}
	\int_{\mathcal{H}^+} K_+\psi \overline{\bar{K}^+\psi} =\int_{\mathcal{H}^+} \left(\partial_t\psi+\omega_+\partial_{\varphi}\psi\right) \cdot  \overline{\left(\partial_t\psi+\bar{\omega}_+\partial_{\varphi}\psi\right)},
\end{equation}
where for the Hawking--Reall vector fields~$K^+,\bar{K}^+$ see Section~\ref{subsec: Hawking-Reall v.f.}. In particular, if we consider a solution of the form~$\psi=e^{-i\omega t}e^{im\varphi} \psi_0(r,\theta)$ then the sign of~\eqref{eq: subsec: superradiant frequencies, eq 0} is formally determined by the sign of
\begin{equation}
	\left(\omega-\frac{am\Xi}{r_+^2+a^2}\right) \left(\omega-\frac{am\Xi}{\bar{r}_+^2+a^2}\right).
\end{equation}

Therefore, the superradiant frequencies are defined to be the frequencies for which~\eqref{eq: subsec: superradiant frequencies, eq 0} is negative. Namely, the superradiant frequencies are 
\begin{equation}\label{eq: subsec: superradiant frequencies, eq 1}
	\mathcal{SF}:=\Bigl\{(\omega,m):~\left(\omega-\frac{am\Xi}{r_+^2+a^2}\right) \left(\omega-\frac{am\Xi}{\bar{r}_+^2+a^2}\right) < 0 \Bigr\},
\end{equation}
where note that for the frequencies~\eqref{eq: subsec: superradiant frequencies, eq 1} the left hand side of~\eqref{eq: subsec: energy identity, eq 1} is not coercive. 

We note that the superradiant frequency set~\eqref{eq: subsec: superradiant frequencies, eq 1} is also equal to
\begin{equation}
	\Bigl\{(\omega,m):~ a m \omega\in \left(\frac{a^2m^2\Xi}{\bar{r}_+^2+a^2},\frac{a^2m^2\Xi}{r_+^2+a^2}\right) \Bigr\}.
\end{equation}

\subsection{Carter's separation}\label{subsec: sec: carter separation, radial, subsec 3}

Let~$\ell\in \mathbb{Z}_{\geq |m|}$ and let~$S^{(a\omega),\mu_{KG}}_{m\ell},\lambda^{(a\omega),\mu_{KG}}_{m\ell}$ be respectively the spheroidal harmonics and their eigenvalues, as discussed in our companion~\cite{mavrogiannis4}. For brevity we denote them simply as
\begin{equation}\label{eq: subsec: sec: carter separation, radial, eq -1}
	S^{(a\omega)}_{m\ell},\qquad \lambda^{(a\omega)}_{m\ell},
\end{equation}
dropping the dependence on~$\mu_{KG}$.

We need the following definition

\begin{definition}\label{def: subsec: sec: carter separation, radial, def 1}
	Let~$l>0$,~$(a,M)\in\mathcal{B}_l$ and~$\mu^2_{KG}\geq 0$. We fix~$0 < \tau_1<\tau_2<\infty$. Let~$\chi(t^\star)$ be any smooth cut-off such that 
	\begin{equation}\label{eq: subsec: sec: carter separation, radial, eq 0}
		\supp \chi \subset \{\tau_1\leq t^\star \leq \tau_2\}. 
	\end{equation}
	
	Let~$F$ be a sufficiently regular function. Let~$\psi$ be smooth solution of the inhomogeneous Klein--Gordon equation 
	\begin{equation}
			\Box_{g_{a,M,l}}\psi-\mu^2_{KG}\psi =F
	\end{equation}
	holds. Then, for
\begin{equation}\label{eq: subsec: sec: carter separation, radial, eq 0.1}
	\psi_{\chi}=\chi \psi
\end{equation}
we have
\begin{equation}\label{eq: subsec: sec: carter separation, radial, eq 1}
	\Box_{g}\psi_{\chi}-\mu_{\textit{KG}}^2\psi_{\chi}=F_0,\qquad F_0 \:\dot{=}\:2\nabla^{b}\chi  \nabla_{b}\psi+\psi \Box_{g}(\chi)+\chi F,
\end{equation}
and for all~$(\omega,m,\ell)$ we define the following smooth Fourier transformations
\begin{equation}\label{eq: def: subsec: sec: carter separation, radial, def 1, eq 3}
	\begin{aligned}
		\Psi^{(a\omega)}_{m\ell} (r)	&	= \int_{\mathbb{R}}\int_{\mathbb{S}^2} e^{-i\omega t}e^{im\varphi}S^{(a\omega)}_{m\ell} \cdot \psi_\chi(t,r,\theta,\varphi)d\sigma_{\mathbb{S}^2}dt,\\
		(F_0)^{(a\omega)}_{m\ell} (r)	&	= \int_{\mathbb{R}}\int_{\mathbb{S}^2} e^{-i\omega t}e^{im\varphi}S^{(a\omega)}_{m\ell} \cdot F_0(t,r,\theta,\varphi)d\sigma_{\mathbb{S}^2}dt,
	\end{aligned}
\end{equation}
where~$d\sigma_{\mathbb{S}^2}$ is the usual volume form on the unit sphere. The integrals on the RHS of~\eqref{eq: def: subsec: sec: carter separation, radial, def 1, eq 3} converge in a pointwise sense for~$r_+<r<\bar{r}_+$. 
\end{definition}

We have the following Proposition

\begin{proposition}[Carter's separation]\label{prop: Carters separation, radial part}
	
Let~$l>0$ and~$(a,M)\in\mathcal{B}_l$ and~$\mu^2_{KG}\geq 0$. Let~$\psi,F$ be smooth functions in~$\{0\leq t^\star<\infty\}$ and such that the inhomogeneous Klein--Gordon equation
\begin{equation}
	\Box_{g_{a,M,l}}\psi-\mu^2_{KG}\psi=F
\end{equation}
holds.

Let~$\chi,\psi_{\chi},F_0,\Psi^{(a\omega)}_{m\ell}$ be as in Definition~\ref{def: subsec: sec: carter separation, radial, def 1}, where specifically recall that~$\supp\chi \subset \{0\leq t^\star<\infty\}$. Then, for all~$(\omega,m,\ell)\in\mathbb{R}\times\mathbb{Z}\times\mathbb{Z}_{\geq |m|}$ the rescaled radial fixed frequency function 
	\begin{equation}
	u^{(a\omega)}_{m\ell}(r)=\sqrt{r^2+a^2}\Psi^{(a\omega)}_{m\ell}(r)
	\end{equation}
	is smooth in~$r\in (r_+,\bar{r}_+)$, and moreover it satisfies Carter's fixed frequency radial ode
	\begin{equation}\label{eq: ode from carter's separation}
	[u^{(a\omega)}_{ml}(r)]^{\prime\prime}+(\omega^{2}-V^{(a\omega)}_{ml}(r))u^{(a\omega)}_{ml}(r)=H^{(a\omega)}_{ml}(r), \qquad H^{(a\omega)}_{ml}=\frac{\Delta}{(r^2+a^2)^{3/2}}\left(\rho^2 F_0\right)^{(a\omega)}_{ml}(r),
	\end{equation}
with~$^\prime=\frac{d}{dr^\star}$, see Section~\ref{subsec: tortoise coordinate}.
	
	 The potential~$V$ can we written as 
	\begin{equation}\label{eq: the potential V}
	V\:\dot{=}\:V^{(a\omega)}_{m\ell}(r)=V_{\textit{SL}}(r)+\left(V_0\right)^{(a\omega)}_{m\ell}(r) +V_{\mu_{\textit{KG}}}(r),
	\end{equation}
	where
	\begin{equation}\label{eq: the potentials V0,VSl,Vmu}
	\begin{aligned}
	V_{\textit{SL}}
	&   =(r^{2}+a^{2})^{-1/2}\left(\frac{d}{dr^\star}\right)^2(\sqrt{r^{2}+a^{2}})\\ &=-\Delta^{2}\frac{3r^{2}}{(r^{2}+a^{2})^{4}}+\Delta\frac{-5\frac{r^{4}}{l^{2}}+3r^{2}(1-\frac{a^{2}}{l^{2}})-4Mr+a^{2}}{(r^{2}+a^{2})^{3}}, \\
	V_0 &   =\frac{\Delta(\lambda^{(a\omega)}_{m\ell}+\omega^{2}a^{2})-\Xi^{2}a^{2}m^{2}-2m\omega a\Xi(\Delta -(r^{2}+a^{2}))}{(r^{2}+a^{2})^{2}}\\
	&=\frac{\Delta}{(r^2+a^2)^2}(\lambda^{(a\omega)}_{m\ell}+a^2\omega^2-2am\omega\Xi)+\omega^2-\left(\omega-\frac{am\Xi}{r^2+a^2}\right)^2,\\
	V_{\mu_{\textit{KG}}}	&=\mu_{\textit{KG}}^2\frac{\Delta}{r^2+a^2}. 
	\end{aligned}
	\end{equation}

	Moreover, the following boundary conditions hold 
	\begin{equation}\label{eq: BC}
		\begin{aligned}
			&   \frac{d u}{d r^\star}=-i(\omega-\omega_+m) u,\qquad \frac{d u}{d r^\star}=i(\omega-\bar{\omega}_+ m)u,
		\end{aligned}
	\end{equation}
	at~$r^\star=-\infty,~r^\star=+\infty$ respectively, in the sense
	\begin{equation}
		\lim_{r^\star\rightarrow-\infty}\left(u^\prime (r^\star) +i(\omega-\omega_+m) u(r^\star) \right)=0,\qquad 	\lim_{r^\star\rightarrow +\infty}\left(u^\prime (r^\star) -i(\omega-\bar{\omega}_+m) u(r^\star) \right)=0 
	\end{equation}
	 where~$\omega_+=\frac{a\Xi}{r_+^2+a^2},~\bar{\omega}_+=\frac{a\Xi}{\bar{r}_+^2+a^2}$. Finally, we note that~$|u|^2(\pm \infty)=\lim_{r^\star\rightarrow\pm\infty}|u|^2(r^\star)$ are well defined.
\end{proposition}

\begin{remark}
	Note that in Proposition~\ref{prop: Carters separation, radial part} we may also cut-off the solution of the Klein--Gordon equation with a cut-off that does not vanish in the future of~$\{t^\star=\tau_2\}$. However, in order to do that we need to restrict to a class of solutions to Klein--Gordon that are `sufficiently integrable' in an~$L^2$ sense, see our companion~\cite{mavrogiannis4}. 
\end{remark}

We need the following definition 
\begin{definition}\label{def: subsec: sec: frequencies, subsec 3, def 0}
	Let~$l>0$,~$(a,M)\in\mathcal{B}_l$ and~$\mu^2_{KG}\geq 0$. Let~$\omega\in \mathbb{R},m\in\mathbb{Z},\ell\in \mathbb{Z}_{\geq |m|}$. Let~$\lambda^{(a\omega)}_{m\ell}$ be as in~\eqref{eq: subsec: sec: carter separation, radial, eq -1}. Then, we define the following
	\begin{equation}\label{eq: def: subsec: sec: frequencies, subsec 3, def 0, eq 1}
	\tilde{\lambda}^{(a\omega)}_{m\ell}=\lambda^{(a\omega)}_{m\ell}+a^2\omega^2.
	\end{equation}
	We will often drop the frequency dependence and denote~\eqref{eq: def: subsec: sec: frequencies, subsec 3, def 0, eq 1} simply as~$\tilde{\lambda}$.
\end{definition}

\subsection{Parseval identities}\label{subsec: sec: carter separation, radial, subsec 4}

We have the following Proposition

\begin{proposition}
	Let~$l>0$,~$(a,M)\in\mathcal{B}_l$. Morever, let~$\chi,\psi,\psi_\chi$ be as in Definition~\ref{def: subsec: sec: carter separation, radial, def 1}. Then, we have the following Parseval identities
	\begin{equation}
		\begin{aligned}
			&  \int_{\mathbb{S}^2}\int_\mathbb{R}|\partial_t(\psi_\chi)|^2 \sin\theta d\theta d\varphi dt=\int_\mathbb{R}\sum_{m,\ell}\omega^2|\Psi^{(a\omega)}_{m,\ell}|^2 d\omega,\\
			&  \int_{\mathbb{S}^2}\int_\mathbb{R}|\partial_{\varphi}(\psi_\chi)|^2 \sin\theta d\theta d\varphi dt=\int_\mathbb{R}\sum_{m,\ell}m^2|\Psi^{(a\omega)}_{m,\ell}|^2 d\omega,\\
			&  \int_{\mathbb{S}^2}\int_\mathbb{R}|\partial_{r^\star}(\sqrt{r^2+a^2}\psi_\chi)|^2 \sin\theta d\theta d\varphi dt=\int_\mathbb{R}\sum_{m,\ell}|\partial_{r^\star}\left(\sqrt{r^2+a^2}\Psi^{(a\omega)}_{m,\ell}\right)|^2 d\omega,\\
			& \int_{\mathbb{S}^2}\int_\mathbb{R} |\psi_\chi|^2 \sin\theta d\theta d\varphi dt =\int_{\mathbb{R}}\sum_{m,\ell}|\Psi^{(a\omega)}_{m\ell}|^2 d\omega,
		\end{aligned}
	\end{equation}
	for any~$r_+< r < \bar{r}_+$, and
	\begin{equation}
		\begin{aligned}
			\int_{\mathbb{R}}\sum_{m,\ell}\lambda^{(a\omega)}_{m\ell}|\Psi^{(a\omega)}_{m,\ell}|^2d\omega =& \int_{\mathbb{S}^2}\int_\mathbb{R}|^{d\sigma_{\mathbb{S}^2}}\nabla(\psi_\chi)|_{d\sigma_{\mathbb{S}^2}}^2\sin\theta d\theta d\varphi dt -a^2\int_{\mathbb{S}^2}\int_\mathbb{R}|\partial_t (\psi_\chi)|^2\cos^2\theta\sin\theta d\theta d\varphi dt,
		\end{aligned}
	\end{equation}
	for any~$r_+< r < \bar{r}_+$, where~$^{d\sigma_{\mathbb{S}^2}}\nabla$ is the covariant derivative of~$d\sigma_{\mathbb{S}^2}$ and~$d\sigma_{\mathbb{S}^2}$ is the standard metric of the unit sphere. 
\end{proposition}

\subsection{Energy identity for smooth outgoing solutions of~\eqref{eq: ode from carter's separation}}\label{subsec: energy identity}

Note the following

\begin{proposition}
	Let~$l>0$,~$(a,M)\in \mathcal{B}_l$ and~$\mu^2_{KG}\geq 0$. Let~$\omega\in \mathbb{R},m\in\mathbb{Z},\ell\in \mathbb{Z}_{\geq |m|}$. Let~$u$ be a smooth solution of Carter's radial ode~\eqref{eq: ode from carter's separation} which moreover satisfies the boundary conditions~\eqref{eq: BC}. Then, we have the following
	\begin{equation}\label{eq: subsec: energy identity, eq 1}
	\left(\omega-\frac{am\Xi}{\bar{r}_+^2+a^2}\right)|u|^2(\infty)+\left(\omega-\frac{am\Xi}{r_+^2+a^2}\right)|u|^2(-\infty)=\Im (\bar{u}H).
	\end{equation}
\end{proposition}
\begin{proof}
This can also be found in our companion~\cite{mavrogiannis}. We multiply the radial ode~\eqref{eq: ode from carter's separation} with~$\bar{u}$ and after integration by parts we obtain the desired result, in view of the outgoing boundary conditions~\eqref{eq: BC}. 
\end{proof}

\subsection{The fixed frequency current~\texorpdfstring{$Q^h$}{h}}

We will need the following current

\begin{definition}\label{def: currents}
	Let~$l>0$,~$(a,M)\in\mathcal{B}_l$ and~$\mu^2_{KG}\geq 0$. Let~$\omega\in \mathbb{R},m\in\mathbb{Z},\ell\in \mathbb{Z}_{\geq |m|}$. Let~$v:\mathbb{R}\rightarrow \mathbb{C}$ be a~$C^0(-\infty,+\infty)$ function and let~$v,v^\prime$ be piecewise~$C^1(-\infty,\infty)$. Moreover, assume that~$u$ is a distributional solution of the inhomogeneous Carter's radial ode~$v^{\prime\prime}(r^\star)+(\omega^2-V)v(r^\star)=H$, where for the potential~$V$ see Proposition~\ref{prop: Carters separation, radial part}, and the inhomogeneity~$H$ is a distribution.

	Then, given a smooth function~$h:\mathbb{R}\rightarrow \mathbb{R}$, we define the current
	\begin{equation}\label{eq: def: currents, eq 1}
	\begin{aligned}
	Q^{h}[v] &\:\dot{=}\:h \Re(v^\prime\bar{v})-\frac{1}{2}h^\prime|v|^2. \\
	\end{aligned}
	\end{equation} 
\end{definition}

The following identity holds in a distributional sense
	\begin{equation}
	\begin{aligned}
	(Q^{h}[v])^\prime &= h|v^\prime|^2  + \Big( h(V-\omega ^{2}) -\frac{1}{2} h^{\prime\prime}\Big)|v|^2+h \Re(v\bar{H}). \\
	\end{aligned}
	\end{equation}

\section{The Morawetz estimates of~\cite{mavrogiannis4}}\label{sec: morawetz estimate}

Before presenting the main results of~\cite{mavrogiannis4} that we need in the present paper, see already Theorems~\ref{main theorem 1}, we need first to discuss some preliminary definitions and notation that have already been presented in~\cite{mavrogiannis4}.

\subsection{Preliminaries}

We need the following definition

\begin{definition}
	Let~$l>0$ and~$\mu^2_{KG}\geq 0$.
	We define the sets
	\begin{equation}\label{eq: sec: main theorems, eq 1.2}
		\widetilde{\mathcal{MS}}_{l,\mu_{KG}}=\{(a,M)\in\mathcal{B}_l: \omega\in \mathbb{R},m\in\mathbb{Z},\ell\in\mathbb{Z}_{\geq |m|},\omega\neq \omega_+ m,\omega\neq \bar{\omega}_+ m ~~|\mathcal{W}^{-1}(a,M,l,\mu^2_{KG},\omega,m,\ell)|<+\infty\},
	\end{equation}
	where~$\omega_+ =\frac{a\Xi}{r_+^2+a^2},\bar{\omega}_+=\frac{a\Xi}{\bar{r}_+^2+a^2}$, and 
	\begin{equation}
		\mathcal{B}_{0,l} = \{(0,M)\in\mathcal{B}_l\} = \{(0,M):~0<\frac{M^2}{l^2}<\frac{1}{27}\},
	\end{equation}
	where recall from our companion~\cite{mavrogiannis4} that~$\mathcal{B}_{0,l}\subset \widetilde{\mathcal{MS}}_l$.
	
	We define the following set
	\begin{equation}\label{eq: sec: main theorems, eq 2}
		\begin{aligned}
			\mathcal{MS}_{l,\mu_{KG}}	
		\end{aligned}
	\end{equation}
	to be the connected component of~$\mathcal{B}_{0,l}$ in the set~$\widetilde{\mathcal{MS}}_{l,\mu_{KG}}$ with the standard euclidean topology.
\end{definition}

\subsection{The Morawetz estimate of~\cite{mavrogiannis4}}\label{subsec: sec: morawetz estimate, subsec 2}

Let~$T,\tau_1\geq 0$ and let~$\chi:\mathcal{M}\rightarrow \mathbb{R}$ be as follows 
\begin{equation}\label{eq: sec: main theorems, eq 3}
\chi=\eta^{(T)}_{\tau_1}\chi^2_{\tau_1,\tau_1+T^2}
\end{equation}
where for the smooth cut-offs~$\eta^{(T)}_{\tau_1},\chi_{\tau_1,\tau_1+T^2}$ see Section~\ref{subsec: sec: carter separation, subsec 1}. For convenience we will denote~$\eta_{\tau_1}^{(T)}$ as~$\eta$ and~$\chi_{\tau_1,\tau_1+T^2}$ as~$\chi_+$.

\begin{remark}
	Note we do not require that the cut-off~\eqref{eq: sec: main theorems, eq 3} vanishes in the future of~$\{t^\star\geq \tau_1+T^2\}$.
\end{remark}

We have

\begin{theorem}[Main Theorem of~\cite{mavrogiannis4}]\label{main theorem 1}
	Let~$l>0$,~$\mu^2_{\textit{KG}}\geq 0$ and
	\begin{equation*}
		(a,M)\in \mathcal{MS}_{l,\mu_{KG}},
	\end{equation*}
	where for~$\mathcal{MS}_{l,\mu_{KG}}$ see~\eqref{eq: sec: main theorems, eq 2}. 
	
	Then, there exist constants
		\begin{equation}
		C=C(a,M,l,\mu^2_{\textit{KG}})>0,\qquad r_+<R^-\leq R^+<\bar{r}_+
	\end{equation}
	and a frequency dependent parameter
	\begin{equation}
		r_{trap}(\omega,m,\ell)\in (R^-,R^+)\cup \{0\}
	\end{equation}
	such that the following holds.
	
	Let~$T>0,1\leq \tau_1<\tau_1+1\leq\tau_2$ and let~$\psi$ be a smooth solution of the inhomogeneous Klein--Gordon equation~$\Box_{g_{a,M,l}}\psi-\mu^2_{KG}\psi=F$ in~$D(\tau_1,\tau_2)$, and let~$u$ be the projection in frequencies of~$\sqrt{r^2+a^2}\chi\psi$, as in Proposition~\ref{prop: Carters separation, radial part}. For the cut-off~$\chi$ see~\eqref{eq: sec: main theorems, eq 3} and also see~[Remark 7.1,\cite{mavrogiannis4}].

	Then, we have the following 
	\begin{equation}\label{eq: main theorem 1, eq 1}
	\begin{aligned}
		&	\int_{\mathbb{R}}d\omega\sum_{m,
		\ell} \Bigg( \int_{r_+}^{r_++\epsilon}\frac{1}{\Delta^2}|u^\prime+i(\omega-\omega_+m)u|^2dr +\int_{\bar{r}_+-\epsilon}^{\bar{r}_+}\frac{1}{\Delta^2}|u^\prime-(\omega-\bar{\omega}_+m)u|^2dr \\
		&	\qquad\qquad\qquad+ \int_{r_+}^{\bar{r}_+} \left(1_{\{|m|>0\}}|u|^2+ \mu^2_{\textit{KG}}|u|^2+ |u^\prime|^2+\left(1-\frac{r_{\textit{trap}}(\omega,m,\ell)}{r}\right)^2(\omega^2+\lambda^{(a\omega)}_{m\ell})|u|^2\right)dr\Bigg) \\
		&	\qquad\qquad\leq C \int_{\{t^\star=\tau_1\}} J^n_{\mu}[\psi]n^{\mu}+C\int\int_{D(\tau_1,\tau_2)} \epsilon^\prime|\partial_t\psi|^2 +\epsilon^\prime|\partial_{\varphi}\psi|^2+(\epsilon^\prime)^{-1}|F|^2,
	\end{aligned}
	\end{equation}
	\begin{equation}\label{eq: main theorem 1, eq 2}
		\int_{\{t^\star=\tau_2\}} J^n_{\mu}[\psi] n^{\mu} \leq C\int_{\{t^\star=\tau_1\}}  J^n_{\mu} [\psi]n^{\mu}+C\int\int_{D(\tau_1,\tau_2)} \epsilon^\prime|\partial_t\psi|^2 +\epsilon^\prime|\partial_{\varphi}\psi|^2 +(\epsilon^\prime)^{-1}|F|^2,
	\end{equation}
	\begin{equation}\label{eq: main theorem 1, eq 3}
		\int_{\{t^\star=\tau_2\}}  J^n_{\mu}[\psi]n^{\mu}+|\psi|^2\leq C\int_{\{t^\star=\tau_1\}}  J^n_{\mu}[\psi]n^{\mu}+|\psi|^2+C\int\int_{D(\tau_1,\tau_2)} \epsilon^\prime|\partial_t\psi|^2 +\epsilon^\prime|\partial_{\varphi}\psi|^2 +(\epsilon^\prime)^{-1}|F|^2, 
	\end{equation}
	\begin{equation}\label{eq: main theorem 1, eq 4}
		\begin{aligned}
			&	\int_{\mathcal{H}^+\cap D(\tau_1,\infty)} J^N_\mu[\psi]n^\mu_{\mathcal{H}^+}+\int_{\bar{\mathcal{H}}^+\cap D(\tau_1,\infty)} J^N_\mu[\psi]n^\mu_{\bar{\mathcal{H}}^+}\\
			&	\qquad\qquad\leq  C\int_{\{t^\star=\tau_1\}}  
			J^n_\mu[\psi]n^\mu_{\{t^\star=0\}}+C\int\int_{D(\tau_1,\tau_2)} \epsilon^\prime|\partial_t\psi|^2 +\epsilon^\prime|\partial_{\varphi}\psi|^2+(\epsilon^\prime)^{-1}|F|^2 , 
		\end{aligned}
	\end{equation}
		\begin{equation}\label{eq: main theorem 1, eq 5}
		\begin{aligned}
			&	\int_{\mathcal{H}^+\cap D(\tau_1,\infty)} J^N_\mu[\chi\psi]n^\mu_{\mathcal{H}^+}+\int_{\bar{\mathcal{H}}^+\cap D(\tau_1,\infty)} J^N_\mu[\chi\psi]n^\mu_{\bar{\mathcal{H}}^+}\\
			&	\qquad\qquad\leq  C\int_{\{t^\star=\tau_1\}}  
			J^n_\mu[\psi]n^\mu_{\{t^\star=0\}}+|\psi|^2+C\int\int_{D(\tau_1,\tau_2)} \epsilon^\prime|\partial_t\psi|^2 +\epsilon^\prime|\partial_{\varphi}\psi|^2 +(\epsilon^\prime)^{-1}|F|^2, 
		\end{aligned}
	\end{equation}
	for a sufficiently small~$\epsilon(a,M,l)>0$ and for any~$0<\epsilon^\prime<1$, where~$\omega_+=\frac{a\Xi}{r_+^2+a^2},~\bar{\omega}_+=\frac{a\Xi}{\bar{r}_+^2+a^2}$. 
\end{theorem}

\begin{remark}
	The last estimate~\eqref{eq: main theorem 1, eq 5} was not included in our companion~\cite{mavrogiannis4}, but can easily be inferred from the estimates~\eqref{eq: main theorem 1, eq 1}-\eqref{eq: main theorem 1, eq 4}. 
\end{remark}

\subsection{A higher order Corollary}

A Corollary of Theorem~\ref{main theorem 1}, as proved in our companion~\cite{mavrogiannis4}, is the following

\begin{corollary}\label{cor: main theorem 1, cor 1}
	Let the assumptions of Theorem~\ref{main theorem 1} be satisfied. Then, for any~$j\geq 1$ there exists a constant
	\begin{equation}
		C=C(j,a,M,l,\mu^2_{\textit{KG}})>0,
	\end{equation}
	and there exists a sufficiently small~$\delta(a,M,l,\mu_{KG},j)>0$ such that we obtain
	\begin{equation}\label{eq: cor: main theorem 1, cor 1, eq 1}
		\begin{aligned}
			&	\int\int_{D_\delta(\tau_1,\infty)}  \sum_{1\le i_1+i_2\le j}|\nabb^{i_1}(\partial_{t^\star})^{i_2}\mathcal{P}_{trap}(\chi\psi)|^2+\int\int_{D_\delta(\tau_1,\infty)\setminus D(\tau_1,\infty)}\sum_{1\leq i_1+i_2+i_3\leq j} |\slashed{\nabla}^{i_1}\partial_t^{i_2}(Z^\star)^{i_3}\psi|^2 \\
			& \qquad\qquad +\int\int_{D_\delta(\tau_1,\infty)} \sum_{1\le i_1+i_2+i_3\le j-1}
			\left(|\nabb^{i_1}(\partial_{t^\star})^{i_2}(Z^\star)^{i_3+1}\psi|^2+|\nabb^{i_1}(\partial_{t^\star})^{i_2}(Z^\star)^{i_3}\psi|^2\right)  \\
			&	\qquad\qquad\qquad \leq C\int_{\{t^\star=\tau_1\}}  \sum_{0\leq i_1+i_2+i_3 \leq j}|\nabb^{i_1}(\partial_{t^\star})^{i_2}(Z^\star)^{i_3}\psi|^2
			\\
			&	\qquad\qquad\qquad\qquad +C\sum_{1\leq i_1\leq j}\int\int_{D_\delta(\tau_1,\infty)} \epsilon^\prime|\partial_t^{i_1}\psi|^2 +(\epsilon^\prime)^{-1} |\partial_t^{i_1-1}F|^2+ \epsilon^\prime|\partial_{\varphi}^{i_1}\psi|^2 + (\epsilon^\prime)^{-1}|\partial_{\varphi}^{i_1-1}F|^2,
		\end{aligned}
	\end{equation}
	\begin{equation}\label{eq: cor: main theorem 1, cor 1, eq 2}
		\begin{aligned}
			&	\int_{\mathcal{H}^+\cap D_\delta(\tau_1,\infty)} \sum_{0\leq i_1+i_2+i_3 \leq j}|\nabb^{i_1}(\partial_{t^\star})^{i_2}(Z^\star)^{i_3}\psi|^2+\int_{\bar{\mathcal{H}}^+\cap D_\delta(\tau_1,\infty)}\sum_{0\leq i_1+i_2+i_3 \leq j}|\nabb^{i_1}(\partial_{t^\star})^{i_2}(Z^\star)^{i_3}\psi|^2\\
			&	\qquad \leq  C\int_{\{t^\star=\tau_1\}} \sum_{0\leq i_1+i_2+i_3 \leq j}|\nabb^{i_1}(\partial_{t^\star})^{i_2}(Z^\star)^{i_3}\psi|^2 \\
			&	\qquad\qquad +C\sum_{1\leq i_1\leq j}\int\int_{D_\delta(\tau_1,\infty)} \epsilon^\prime|\partial_t^{i_1}\psi|^2 +(\epsilon^\prime)^{-1} |\partial_t^{i_1-1}F|^2+ \epsilon^\prime|\partial_{\varphi}^{i_1}\psi|^2 + (\epsilon^\prime)^{-1}|\partial_{\varphi}^{i_1-1}F|^2,
		\end{aligned}
	\end{equation}
	\begin{equation}\label{eq: cor: main theorem 1, cor 1, eq 3}
		\begin{aligned}
			&	\int_{\{t^\star=\tau_2\}} \sum_{0\leq i_1+i_2+i_3 \leq j}|\nabb^{i_1}(\partial_{t^\star})^{i_2}(Z^\star)^{i_3}\psi|^2\\
			&	\qquad \leq C\int_{\{t^\star=\tau_1\}} \sum_{0\leq i_1+i_2+i_3 \leq j}|\nabb^{i_1}(\partial_{t^\star})^{i_2}(Z^\star)^{i_3}\psi|^2\\
			&	\qquad\qquad+C\sum_{1\leq i_1\leq j}\int\int_{D_\delta(\tau_1,\infty)} \epsilon^\prime|\partial_t^{i_1}\psi|^2 +(\epsilon^\prime)^{-1} |\partial_t^{i_1-1}F|^2+ \epsilon^\prime|\partial_{\varphi}^{i_1}\psi|^2 + (\epsilon^\prime)^{-1}|\partial_{\varphi}^{i_1-1}F|^2, 
		\end{aligned}
	\end{equation}
	for any~$2\leq1+\tau_1\leq \tau_2$. We used the operator
	\begin{equation}
		\begin{aligned}
			\mathcal{P}_{\textit{trap}}[\Psi] &	= \frac{1}{\sqrt{2\pi}} \int_{\mathbb{R}}\sum_{m,\ell}\left|1-\frac{r_{\textit{trap}}}{r}\right|e^{i\omega t}\Psi^{(a\omega)}_{m\ell}(r)S^{(a\omega)}_{m\ell}(\theta)e^{- i m\phi}d\omega. 
		\end{aligned}
	\end{equation}
\end{corollary}

\section{The operator~\texorpdfstring{$\mathcal{G}$}{G}}\label{sec: G}

The main definition of this section is Definition~\ref{def: sec: G, def 1} where we define the operator~$\mathcal{G}$. However, first we need several preliminary definitions and lemmata to justify that~$\mathcal{G}$ is well defined.

\subsection{The definitions of~\texorpdfstring{$\mathcal{T}(\omega,m,r),g_1(r),g_2^2(\omega,m,r)$}{g}}

Note the following definition

\begin{definition}\label{def: sec: G, def 0}
	Let~$l>0$ and~$(a,M)\in\mathcal{B}_l$. Let~$(\omega,m)\in\mathbb{R}\times \mathbb{Z}$. We define 
	\begin{equation}\label{eq: sec: G, eq 0}
	\mathcal{T}(\omega,m,r)=\frac{(r^2+a^2)^2}{\Delta}\left(\omega-\frac{am\Xi}{r^2+a^2}\right)^2,
	\end{equation}
	and moreover we define 
	\begin{equation}\label{eq: sec: G, eq 1}
	\begin{aligned}
	g_1(r)\:\dot{=}\:\frac{r^2+a^2}{\sqrt{\Delta}},\qquad \left(g_2(\omega,m,r)\right)^2		\: \dot{=}\:  & \mathcal{T}(\omega,m,r)-\min_{r\in[r_+,\bar{r}_+]}\mathcal{T}(\omega,m,r).
	\end{aligned}
	\end{equation}
\end{definition}

\begin{remark}
	The function~$\mathcal{T}$, see~\eqref{def: sec: G, def 0}, has also been studied by Dyatlov for the case of the slowly rotating Kerr--de~Sitter black hole, see~\cite{Dyatlov4}, and more recently by Petersen--Vasy for the entire subextremal range of Kerr--de~Sitter black hole paremeters, see~\cite{petersen2021wave}. 
\end{remark}

In order to define~$g_2(\omega,m,r)$, see already Definition~\ref{def: subsec: sec: G, subsec 1, def 1}, we need first Propositions~\ref{prop: subsec: sec: G, subsec 1, prop 0},~\ref{prop: subsec: sec: G, subsec 1, prop -1}.

\subsection{The critical point~$r_{crit}(\omega,m)$}\label{subsec: sec: morawetz estimate, subsec 3}

In this Section, amongst other things, we will prove that for any~$(\omega,m)\in\mathbb{R}\times\mathbb{Z}$ the function~$\mathcal{T}(\omega,m,\cdot)$, see~\eqref{eq: sec: G, eq 0}, attains a unique critical point~$r_{crit}(\omega,m)$, a minimum.

First, we study~$\mathcal{T}(\omega,m,r)$, see~\eqref{eq: sec: G, eq 0}, for the superradiant frequencies.

\begin{proposition}\label{prop: subsec: sec: G, subsec 1, prop 0}
	Let~$l>0$ and~$(a,M)\in \mathcal{B}_l$. For the superradiant frequencies
	\begin{equation}\label{eq: lem: subsec: sec: G, subsec 1, lem 0, eq 0}
		(\omega,m)\in\mathcal{SF}
	\end{equation}
	the function~$\mathcal{T}(\omega,m,r)$, see~\eqref{eq: sec: G, eq 0}, satisfies~$\mathcal{T}(\omega,m,\cdot)\in C^\infty (r_+,\bar{r}_+)$, and it attains a unique critical point
	\begin{equation}
		r_{\textit{crit}}(\omega,m)\in (r_+,\bar{r}_+),
	\end{equation}
	a minimum where~$\frac{d\mathcal{T}}{dr}(r_{crit})=0$. The following holds
	\begin{equation*}
		r_{crit}(\omega,m)=r_s(\omega,m),
	\end{equation*}
 where~$am\omega=\frac{a^2m^2\Xi}{r_s^2+a^2}$.

	Moreover, we have that 
	\begin{equation}\label{eq: lem: subsec: sec: G, subsec 1, lem 0, eq 1}
		\left|\frac{d^2}{dr^2}\Big|_{r=r_{\textit{crit}}} \mathcal{T}\right| >0.
	\end{equation}	
\end{proposition}
\begin{proof}
	
For~$(\omega,m)\in\mathcal{SF}$ we note that there exists an~$r_s\in(r_+,\bar{r}_+)$ such that 
\begin{equation}\label{eq: proof: lem: subsec: sec: G, subsec 1, lem 0, eq 1.0}
am\omega =\frac{a^2m^2\Xi}{r_s^2+a^2}. 
\end{equation}
Therefore, the derivative of~$\mathcal{T}$, see~\eqref{eq: sec: G, eq 0}, can be written as 
\begin{equation}\label{eq: proof: lem: subsec: sec: G, subsec 1, lem 0, eq 1}
\frac{d\mathcal{T}}{dr}=\frac{d}{dr}\left(\frac{(r^2+a^2)^2}{\Delta}\left(\omega-\frac{a\Xi}{r^2+a^2}m\right)^2\right)=(r^2-r_s^2)\frac{4r\Delta -(r^2-r_s^2)\frac{d\Delta}{dr}}{2\Delta\sqrt{\Delta}}\:\dot{=}\:(r^2-r_s^2)\frac{N_2(r;r_s)}{2\Delta\sqrt{\Delta}},
\end{equation}
where note that for any~$r_s\in (r_+,\bar{r}_+)$ the function~$N_2(r;r_s)$ is a polynomial of degree~$3$, namely
\begin{equation}\label{eq: proof: lem: subsec: sec: G, subsec 1, lem 0, eq 1.01}
N_2(r;r_s)= 2\left((1-\frac{a^2}{l^2}-2\frac{r_s^2}{l^2})r^3 -3M r^2 +(2a^2+r_s^2-\frac{a^2}{l^2}r_s^2)r-2Mr_s^2 \right).
\end{equation}

From the definition of~$N_2(r;r_s)$, see~\eqref{eq: proof: lem: subsec: sec: G, subsec 1, lem 0, eq 1}, we note that  
\begin{equation}\label{eq: proof: lem: subsec: sec: G, subsec 1, lem 0, eq 1.1}
N_2(r_+;r_s)=-\frac{d\Delta}{dr}(r_+)(r_+^2-r_s^2)>0,\qquad N_2(\bar{r}_+;r_s)= - \frac{d\Delta}{dr}(\bar{r}_+)(\bar{r}_+^2-r_s^2)>0.
\end{equation}

Moreover, we note that for~$r_s=\bar{r}_+$ the following hold
\begin{equation}\label{eq: proof: lem: subsec: sec: G, subsec 1, lem 0, eq 2}
N_2(\bar{r}_+;\bar{r}_+)=0,\quad \frac{d N_2}{dr}(\bar{r}_+;\bar{r}_+)=2\bar{r}_+ \frac{d\Delta}{dr}(\bar{r}_+)<0,\quad \frac{d^2 N_2}{dr^2}(\bar{r}_+;\bar{r}_+)= 6\frac{d\Delta}{dr}(\bar{r}_+)<0.
\end{equation}
Therefore, in view of~\eqref{eq: proof: lem: subsec: sec: G, subsec 1, lem 0, eq 1.1},~\eqref{eq: proof: lem: subsec: sec: G, subsec 1, lem 0, eq 2}, we conclude that the 3rd degree polynomial~$N_2(r;\bar{r}_+)$ attains three real roots~$r_1,r_2,\bar{r}_+$ that satisfy
\begin{equation}\label{eq: proof: lem: subsec: sec: G, subsec 1, lem 0, eq 3}
r_1\leq r_+ <\bar{r}_+< r_2
\end{equation}
and 
\begin{equation}\label{eq: proof: lem: subsec: sec: G, subsec 1, lem 0, eq 4}
	N_2(r;\bar{r}_+)\geq 0,~r\in [r_+,\bar{r}_+]
\end{equation}
where the only zero at~$[r_+,\bar{r}_+]$ is attained at~$r=\bar{r}_+$. Also, see Figure~\ref{fig: N2-0}.

\begin{figure}[htbp]
	\includegraphics[scale=0.6]{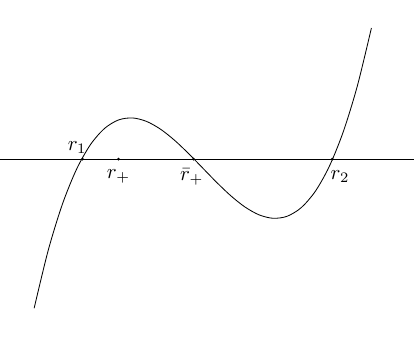}
	\caption{The function~$N_2(r;\bar{r}_+)$}
		\label{fig: N2-0}
\end{figure}

\underline{\texttt{For the sake of contradiction}} we assume that for some~$r_s=r_s(\omega,m)\in (r_+,\bar{r}_+)$, see~\eqref{eq: proof: lem: subsec: sec: G, subsec 1, lem 0, eq 1.0}, the polynomial~$N_2(r;r_s)$ attains two roots
\begin{equation}
r_1\leq r_2
\end{equation}
in the interval~$(r_+,\bar{r}_+)$. We note that~$N_2(r;r_s)$ cannot attain only one root in~$(r_+,\bar{r}_+)$ because then one of the inequalities~\eqref{eq: proof: lem: subsec: sec: G, subsec 1, lem 0, eq 1.1} would not be satisfied. In view of~\eqref{eq: proof: lem: subsec: sec: G, subsec 1, lem 0, eq 1.1} we note Figure~\ref{fig: N2-1} for the possible graphs of the function~$N_2(r;r_s)$.

\begin{figure}[htbp]
	\begin{multicols}{2}
		\includegraphics[scale=0.6]{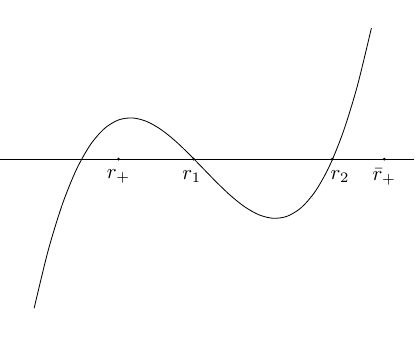}\par
		\includegraphics[scale=0.6]{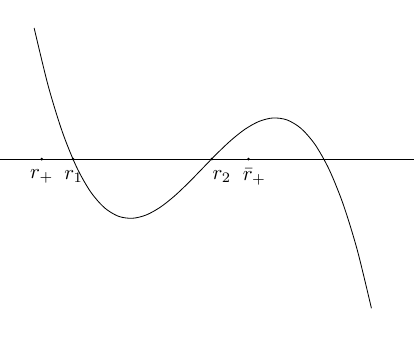}\par 
	\end{multicols}
	\caption{The possible functions $N_2(r;r_s)$ with two roots~$r_1<r_2$}
	\label{fig: N2-1}
\end{figure}

First, we note the following. Suppose that there exists a~$\tilde{r}\in (r_+,\bar{r}_+)$ such that~$N_2(\tilde{r};r_s)<0$,~(note that for~$r\in (r_1,r_2)$ we have~$N_2(r;r_s)<0$) and moreover~$\frac{d\Delta}{dr}(\tilde{r})\leq 0$. Then, in view of 
\begin{equation}
	N_2(\tilde{r};r_s) = 4\tilde{r}\Delta (\tilde{r})-\tilde{r}^2 \frac{d\Delta}{dr}(\tilde{r})+r_s^2 \frac{d\Delta}{dr}(\tilde{r})
\end{equation}
we note that for any~$r_s<\tilde{r}_s< \bar{r}_+$ we have that~$N_2(\bar{r}_+;\tilde{r}_s)<0$, which is a contradiction to~\eqref{eq: proof: lem: subsec: sec: G, subsec 1, lem 0, eq 2}, also see Figure~\ref{fig: N2-1}. Therefore, we have that for any~$	\tilde{r}\in (r_1,r_2) $
\begin{equation}\label{eq: proof: lem: subsec: sec: G, subsec 1, lem 0, eq 4.9}
\frac{d\Delta}{dr}(\tilde{r})>0.
\end{equation}

Then, in view of the properties of the function~$N_2(r;\bar{r}_+)$, see Figure~\ref{fig: N2-0}, and since the function~$N(\cdot;\cdot)$ is continuous in both its variables we obtain that there exists an~$r_{s,h}$ with~$r_s<r_{s,h}<\bar{r}_+$ such that the roots~$r_1(r_{s,h}),~r_2(r_{s,h})$ of the polynomial~$N_2(r;r_{s,h})$ satisfy either of the following
\begin{equation}\label{eq: proof: lem: subsec: sec: G, subsec 1, lem 0, eq 5}
	\begin{aligned}
		&	(1)~r_1 (r_{s,h})=r_2(r_{s,h})~\text{with}~r_1\in [r_+,\bar{r}_+]\\
		&	(2)~r_+<r_1(r_{s,h})<\bar{r}_+<r_2(r_{s,h})
	\end{aligned}
\end{equation}
also see Figure~\ref{fig: N2-2} for the possible graphs. 

\begin{figure}[htbp]
	\begin{multicols}{3}
		\includegraphics[scale=0.6]{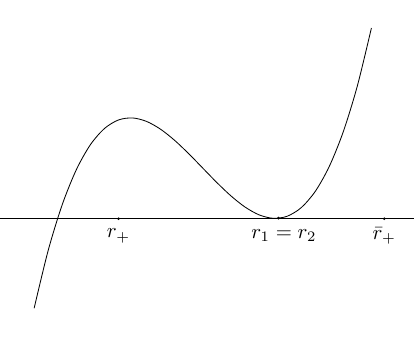}\par
		\includegraphics[scale=0.6]{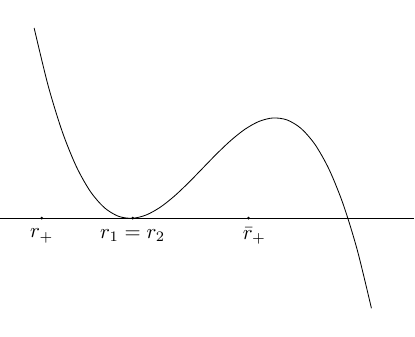}\par
		\includegraphics[scale=0.6]{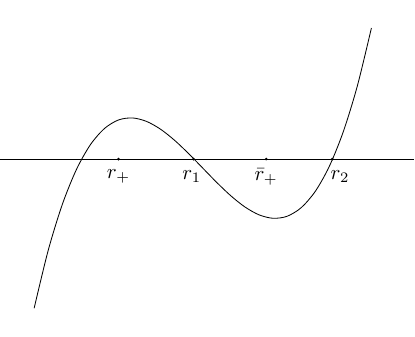}\par 
	\end{multicols}
	\caption{The function $N_2(r;r_{s,h})$ for the cases~$(1),~(2)$ of ~\eqref{eq: proof: lem: subsec: sec: G, subsec 1, lem 0, eq 5} respectively}
	\label{fig: N2-2}
\end{figure}

\boxed{For~the ~case~$(2)$} of~\eqref{eq: proof: lem: subsec: sec: G, subsec 1, lem 0, eq 5} we proceed as follows. We note from~\eqref{eq: proof: lem: subsec: sec: G, subsec 1, lem 0, eq 1.1} that~$N_2(\bar{r}_+;r_s)>0$ which is a \underline{\texttt{contradiction}}. So, in the case~$(2)$ of~\eqref{eq: proof: lem: subsec: sec: G, subsec 1, lem 0, eq 5} the polynomial~$N_2(r;r_s)$ attains no roots in~$(r_+,\bar{r}_+)$ for any~$r_s\in (r_+,\bar{r}_+)$.

\boxed{For ~the ~case~$(1)$} of~\eqref{eq: proof: lem: subsec: sec: G, subsec 1, lem 0, eq 5} we proceed as follows. We consider the sub-cases
\begin{equation}\label{eq: proof: lem: subsec: sec: G, subsec 1, lem 0, eq 6}
	\begin{aligned}
		&	(1.1)~\frac{d\Delta}{dr}(r_{s,h})\geq 0\\
		&	(1.2)~\frac{d\Delta}{dr}(r_{s,h})< 0,~~	1-\frac{a^2}{l^2}-2\frac{r_{s,h}^2}{l^2} \leq  0 \\
		&	(1.3)~\frac{d\Delta}{dr}(r_{s,h})< 0,~~	1-\frac{a^2}{l^2}-2\frac{r_{s,h}^2}{l^2} >  0 \\
	\end{aligned}
\end{equation}
and we will prove a contradiction for all of them.

$\bullet$ We consider the \underline{case~$(1.1)$ of~\eqref{eq: proof: lem: subsec: sec: G, subsec 1, lem 0, eq 6}}. Then, we calculate from the definition of~$N_2(r;r_s)$, see~\eqref{eq: proof: lem: subsec: sec: G, subsec 1, lem 0, eq 1}, that the following holds 
\begin{equation}
	\frac{dN_2}{dr}(r_{s,h})=4 \Delta (r_{s,h})+ 4r_s \frac{d\Delta}{dr}(r_{s,h})>0
\end{equation}
But then, also see the Figure~\ref{fig: N2-2}, the following holds~$r_{s,h}>r_1$. Therefore, we have that
\begin{equation}
	0=N_2(r_1;r_{s,h})= \left(4r\Delta(r)-(r^2-r_{s,h}^2)\frac{d\Delta}{dr}(r)\right)(r=r_1)>0
\end{equation}
where we also used~\eqref{eq: proof: lem: subsec: sec: G, subsec 1, lem 0, eq 4.9}. \underline{\texttt{This is a contradiction.}}

$\bullet$ We consider the \underline{case~$(1.2)$ of~\eqref{eq: proof: lem: subsec: sec: G, subsec 1, lem 0, eq 6}}. In view of the definition of~$N_2(r;r_{s,h})$, see~\eqref{eq: proof: lem: subsec: sec: G, subsec 1, lem 0, eq 1.01}, the graph of~$N_2(r;r_{s,h})$ is to be thought of as the second of Figure~\ref{fig: N2-2}. In this case we note that for all~$r_{s,h}<r_{s,h,2}\leq \bar{r}_+$ we have that~$1-\frac{a^2}{l^2}-2\frac{r_{s,h,2}^2}{l^2}<0$, which is a \underline{\texttt{contradiction}} to~\eqref{eq: proof: lem: subsec: sec: G, subsec 1, lem 0, eq 2}, also see Figure~\ref{fig: N2-0}.

$\bullet$ We consider the \underline{case~$(1.3)$ of~\eqref{eq: proof: lem: subsec: sec: G, subsec 1, lem 0, eq 6}}. In view of the definition of~$N_2(r;r_{s,h})$, see~\eqref{eq: proof: lem: subsec: sec: G, subsec 1, lem 0, eq 1.01}, the graph of~$N_2(r;r_{s,h})$ is to be thought of as the first graph of Figure~\ref{fig: N2-2}.

Then, in view of the subextremality, see also Figure~\ref{fig: Delta}, and since~$\frac{d\Delta}{dr}(\tilde{r})\geq 0$, see~\eqref{eq: proof: lem: subsec: sec: G, subsec 1, lem 0, eq 4.9}, we obtain that 
\begin{equation}\label{eq: proof: lem: subsec: sec: G, subsec 1, lem 0, eq 7}
	\tilde{r}\leq r_{s,h}.
\end{equation}
Moreover, we calculate
\begin{equation}
	\frac{d^2N_2}{dr^2}(r_{s,h})=\left(6\frac{d\Delta}{dr}\right)(r=r_{s,h}) <0.
\end{equation}
Therefore, we obtain also that~$r_{s,h}<\tilde{r}$, also see Figure~\ref{fig: N2-2}. This is a \underline{\texttt{contradiction}} to~\eqref{eq: proof: lem: subsec: sec: G, subsec 1, lem 0, eq 7}.

Therefore, via a contradiction argument we have concluded that the polynomial~$N_2(r;r_s)$ can only attain the zeros 
\begin{equation}
	r=r_+,\qquad r=\bar{r}_+
\end{equation}
respectively for the cases~$r_s=r_+,~r_s=\bar{r}_+$. We conclude that in~$[r_+,\bar{r}_+]$ the only critical point of the function~$\mathcal{T}(r,\omega,m)$ is attained at
\begin{equation}
	r=r_s.
\end{equation}

It is immediate to conclude inequality~\eqref{eq: lem: subsec: sec: G, subsec 1, lem 0, eq 1} for~$a m \omega\in \left(\frac{a^2 m^2\Xi}{\bar{r}_+^2+a^2},\frac{a^2 m^2\Xi}{r_+^2+a^2}\right)$ since~$\frac{d\mathcal{T}}{dr}$ attains a unique zero with multiplicity 1.
\end{proof}

Second, we study~$\mathcal{T}(\omega,m,r)$ for the non-superradiant frequencies

\begin{proposition}\label{prop: subsec: sec: G, subsec 1, prop -1}
	Let~$l>0$ and~$(a,M)\in \mathcal{B}_l$. For the non-superradiant frequencies
	\begin{equation}\label{eq: lem: subsec: sec: G, subsec 1, lem -1, eq 0}
		(\omega,m)\in (\mathcal{SF})^c
	\end{equation}
	the function~$\mathcal{T}(\omega,m,r)$, see~\eqref{eq: sec: G, eq 0}, satisfies~$\mathcal{T}(\omega,m,\cdot)\in C^\infty (r_+,\bar{r}_+)$. Moreover, the function~$\mathcal{T}(\omega,m,\cdot)$ attains a unique critical point
	\begin{equation}
		r_{\textit{crit}}(\omega,m)\in [r_+,\bar{r}_+],
	\end{equation}
	where~$\frac{d\mathcal{T}}{dr}(r_{crit})=0$. The cases~$r_{crit}=r_+,r_{crit}=\bar{r}_+$ correspond to the borderline superradiant cases~$am\omega=\frac{a^2m^2\Xi}{r_+^2+a^2}$ and~$am\omega=\frac{a^2m^2\Xi}{\bar{r}_+^2+a^2}$ respectively, in which case we have~$\mathcal{T}(\omega,m,\cdot)\in C^1[r_+,\bar{r}_+)$ and $\mathcal{T}(\omega,m,\cdot)\in C^1(r_+,\bar{r}_+]$ respectively.

	Moreover, for any~$(\omega,m)\in (\mathcal{SF})^c\setminus \{am\omega=\frac{a^2m^2\Xi}{r_+^2+a^2},am\omega=\frac{a^2m^2\Xi}{\bar{r}_+^2+a^2}\}$ we have that~$\mathcal{T}(\omega,m,\cdot) \in C^2(r_+,\bar{r}_+)$ and the following holds
	\begin{equation}\label{eq: lem: subsec: sec: G, subsec 1, lem -1, eq 1}
		\left|\frac{d^2}{dr^2}\Big|_{r=r_{\textit{crit}}} \mathcal{T}\right| >0.
	\end{equation} 
	For the borderline non superradiant frequencies~$am\omega=\frac{a^2m^2\Xi}{r_+^2+a^2}$ or~$am\omega=\frac{a^2m^2\Xi}{\bar{r}_+^2+a^2}$ the function~$\frac{d \mathcal{T}}{dr}$ is not differentiable in~$r$ at the points~$r=r_+$ and~$r=\bar{r}_+$, respectively. 
\end{proposition}
\begin{proof}
First, it is clear from~\eqref{eq: sec: G, eq 0} that~$\mathcal{T}\in C^\infty(r_+,\bar{r}_+)$. For the borderline non superradiant frequencies~$am\omega=\frac{a^2m^2\Xi}{r_+^2+a^2},am\omega=\frac{a^2m^2\Xi}{\bar{r}_+^2+a^2}$ we have respectively
	\begin{equation}\label{eq: proof lem: subsec: sec: proof of exp decay, sign of g2, lem 1, eq 1}
		\begin{aligned}
			\frac{d\mathcal{T}}{dr}&	=(r^2-r_+^2)\frac{4r\Delta -(r^2-r_+^2)\frac{d\Delta}{dr}}{2\Delta\sqrt{\Delta}}\:\dot{=}\:(r^2-r_+^2)\frac{N_2(r;r_+)}{2\Delta\sqrt{\Delta}},\\
			\frac{d\mathcal{T}}{dr}&	=(r^2-\bar{r}_+^2)\frac{4r\Delta -(r^2-r_s^2)\frac{d\Delta}{dr}}{2\Delta\sqrt{\Delta}}\:\dot{=}\:(r^2-\bar{r}_+^2)\frac{N_2(r;\bar{r}_+)}{2\Delta\sqrt{\Delta}}.\\
		\end{aligned}
	\end{equation}
	It is easy to adapt the study of the polynomial~$N_2(r;r_s)$, of Proposition~\ref{prop: subsec: sec: G, subsec 1, prop 0}, to prove that~$N_2(r;r_+),N_2(r;\bar{r}_+)$ attain unique zeroes respectively at~$r_+,\bar{r}_+$. Moreover, it is obvious from~\eqref{eq: proof lem: subsec: sec: proof of exp decay, sign of g2, lem 1, eq 1} that the function~$\frac{d\mathcal{T}}{dr}$ is not differentiable for the frequencies~$am\omega=\frac{a^2m^2\Xi}{r_+^2+a^2},am\omega=\frac{a^2m^2\Xi}{\bar{r}_+^2+a^2}$ at the values~$r_+,\bar{r}_+$ respectively, but it is continous at~$r_+,\bar{r}_+$.

	We note that if~$\omega=0$ then
	\begin{equation}
		\mathcal{T}(0,m,r)= \frac{(am\Xi)^2}{\Delta},
	\end{equation} 
	and therefore the Lemma follows by the properties of~$\Delta$, see Lemma~\ref{lem: sec: properties of Delta, lem 1, derivatives of Delta} and the subextremality condition of Definition~\ref{def: subextremality, and roots of Delta}.  Therefore, in what follows we assume that~$\omega\not{=}0$.

	We study
	\begin{equation}
		(\mathcal{T}(\omega,m,r))^{-1}= \frac{\Delta}{(r^2+a^2)^2\left(\omega-\frac{am\Xi}{r^2+a^2}\right)^2}.
	\end{equation}

	For the non-superradiant frequencies~\eqref{eq: lem: subsec: sec: G, subsec 1, lem -1, eq 0} we note that for all~$r\in [r_+,\bar{r}_+]$ we have
	\begin{equation}\label{eq: proof lem: subsec: sec: proof of exp decay, sign of g2, lem 1, eq 2}
		(\omega(r^2+a^2)-am\Xi)^2>0. 
	\end{equation}
	We obtain that 
	\begin{equation}\label{eq: proof lem: subsec: sec: proof of exp decay, sign of g2, lem 1, eq 3}
		\begin{aligned}
			\frac{d}{dr}\frac{\Delta}{(r^2+a^2)^2\left(\omega-\frac{am\Xi}{r^2+a^2}\right)^2} &	= \frac{\omega \left((r^2+a^2)\frac{d\Delta}{dr}-4r\Delta\right)-am\Xi\frac{d\Delta}{dr}}{\left((r^2+a^2)\omega-am\Xi\right)^3}\\
			&	=2\frac{\left(-\omega\Xi+\frac{2am\Xi}{l^2}\right)r^3+3M\omega r^2+\left(am\Xi\frac{a^2}{l^2}-a^2\Xi\omega-am\Xi\right)r-M\omega a^2 +am\Xi M +am\Xi \frac{a^2}{l^2}}{\left((r^2+a^2)\omega-am\Xi\right)^3}\\
			& \dot{=}\:2\frac{N_1(r)}{\left((r^2+a^2)\omega-am\Xi\right)^3}.
		\end{aligned}
	\end{equation}
	First, from~\eqref{eq: proof lem: subsec: sec: proof of exp decay, sign of g2, lem 1, eq 3} we note that 
	\begin{equation}\label{eq: proof lem: subsec: sec: proof of exp decay, sign of g2, lem 1, eq 3.1}
		\frac{d^2N_1}{dr^2}= \omega \left(-6\frac{d\Delta}{dr}+\frac{\omega (r^2+a^2)-am\Xi}{\omega}\frac{d^3\Delta}{dr^3} \right).
	\end{equation}
	
	As mentioned earlier, in the non-superradiant frequencies~\eqref{eq: lem: subsec: sec: G, subsec 1, lem -1, eq 0} the inequality~\eqref{eq: proof lem: subsec: sec: proof of exp decay, sign of g2, lem 1, eq 2} holds, and therefore for all~$r\in[r_+,\bar{r}_+]$ the function
	\begin{equation}
		\omega (r^2+a^2)-am\Xi 
	\end{equation}
	has a sign. Therefore, by the first line of~\eqref{eq: proof lem: subsec: sec: proof of exp decay, sign of g2, lem 1, eq 3} and by the properties of the polynomial~$\Delta$, see Lemma~\ref{lem: sec: properties of Delta, lem 1, derivatives of Delta}, we conclude that 
	\begin{equation}\label{eq: proof lem: subsec: sec: proof of exp decay, sign of g2, lem 1, eq 4}
		N_1(r_+)\cdot N_1(\bar{r}_+)<0.
	\end{equation}
	Moreover, we calculate the following 
	\begin{equation}\label{eq: proof lem: subsec: sec: proof of exp decay, sign of g2, lem 1, eq 4.01}
		\frac{dN_1}{dr}(r)= 3(-\omega\Xi +2am\Xi\frac{1}{l^2})r^2 +6M\omega r -a^2\Xi\omega -am\Xi +am\Xi\frac{a^2}{l^2}, 
	\end{equation}
	where the discriminant of~\eqref{eq: proof lem: subsec: sec: proof of exp decay, sign of g2, lem 1, eq 4.01} is 
	\begin{equation}\label{eq: proof lem: subsec: sec: proof of exp decay, sign of g2, lem 1, eq 4.02}
		\text{discriminant of }\frac{dN_1}{dr}= (6M\omega)^2 -12 a\Xi^2 \left(\omega-2\frac{am}{l^2}\right)\left(a\omega +m-m\frac{a^2}{l^2}\right).
	\end{equation}
	Therefore, by using the discriminant~\eqref{eq: proof lem: subsec: sec: proof of exp decay, sign of g2, lem 1, eq 4.02} we note that the roots of~$\frac{dN_1}{dr}$ are the following two
	\begin{equation}\label{eq: proof lem: subsec: sec: proof of exp decay, sign of g2, lem 1, eq 4.03} 
		\begin{aligned}
			\tilde{r}_\pm &= \frac{6M\omega \mp \sqrt{(6M\omega)^2 -12 a\Xi^2 \left(\omega-2\frac{am}{l^2}\right)\left(a\omega +m-m\frac{a^2}{l^2}\right)}}{6\Xi\left(\omega-\frac{2ma}{l^2}\right)}\\
			&	=\frac{\omega}{\omega-\frac{2a m}{l^2}}\frac{M}{\Xi} \mp \frac{\sqrt{(6M\omega)^2 -12 a\Xi^2 \left(\omega-2\frac{am}{l^2}\right)\left(a\omega +m-m\frac{a^2}{l^2}\right)}}{6\Xi\left(\omega-\frac{2 a m}{l^2}\right)}.
		\end{aligned}
	\end{equation}
	We note from~\eqref{eq: proof lem: subsec: sec: proof of exp decay, sign of g2, lem 1, eq 4.03} that if~$\frac{\omega}{\omega-\frac{2a m}{l^2}}\leq 0$ then at least one of~$\tilde{r}_\pm$ is negative, and therefore the function~$\frac{dN_1}{dr}$ attains at most one root in~$[r_+,\bar{r}_+]$. It also follows that~\eqref{eq: lem: subsec: sec: G, subsec 1, lem -1, eq 1} holds since otherwise~$\frac{dN_1}{dr}$ would attain at least~$2$ roots in~$[r_+,\bar{r}_+]$.

	Therefore, without loss of generality we may assume that 
	\begin{equation}\label{eq: proof lem: subsec: sec: proof of exp decay, sign of g2, lem 1, eq 4.04}
		\frac{\omega}{\omega-\frac{2a m}{l^2}}>0.
	\end{equation}
	We study the two disjoint frequency ranges
	\begin{equation}\label{eq: proof lem: subsec: sec: proof of exp decay, sign of g2, lem 1, eq 4.1}
		\omega\left(\omega(r^2+a^2)-am\Xi\right)>0,\qquad \omega\left(\omega(r^2+a^2)-am\Xi\right)< 0
	\end{equation}
	which cover the entire range of non-superradiant frequencies, except for~$\omega=0$.

	In the first case of~\eqref{eq: proof lem: subsec: sec: proof of exp decay, sign of g2, lem 1, eq 4.1} we proceed as follows. In view of the assumption~\eqref{eq: proof lem: subsec: sec: proof of exp decay, sign of g2, lem 1, eq 4.04} we note from~\eqref{eq: proof lem: subsec: sec: proof of exp decay, sign of g2, lem 1, eq 3.1} and from Lemma~\ref{lem: sec: properties of Delta, lem 1, derivatives of Delta}~(properties of~$\Delta$) that the following holds
	\begin{equation}\label{eq: proof lem: subsec: sec: proof of exp decay, sign of g2, lem 1, eq 4.2}
		\omega\frac{d^2 N_1}{dr^2}(r_+)<0.
	\end{equation}
	Therefore, inequality~\eqref{eq: proof lem: subsec: sec: proof of exp decay, sign of g2, lem 1, eq 4} in conjunction with~\eqref{eq: proof lem: subsec: sec: proof of exp decay, sign of g2, lem 1, eq 4.2} implies that the third degree polynomial function~$N_1(r)$ attains exactly one root in~$(r_+,\bar{r}_+)$. Again, in view of that~$N_1(r)$ is a polynomial of degree~$3$ we also conclude~\eqref{eq: lem: subsec: sec: G, subsec 1, lem -1, eq 1}. Therefore, we conclude the result of the Lemma under the first case of~\eqref{eq: proof lem: subsec: sec: proof of exp decay, sign of g2, lem 1, eq 4.1}.

	In the second case of~\eqref{eq: proof lem: subsec: sec: proof of exp decay, sign of g2, lem 1, eq 4.1} we proceed as follows. In view of the assumption~\eqref{eq: proof lem: subsec: sec: proof of exp decay, sign of g2, lem 1, eq 4.04} we note from~\eqref{eq: proof lem: subsec: sec: proof of exp decay, sign of g2, lem 1, eq 3.1} and from Lemma~\ref{lem: sec: properties of Delta, lem 1, derivatives of Delta} that the following holds
	\begin{equation}\label{eq: proof lem: subsec: sec: proof of exp decay, sign of g2, lem 1, eq 4.3}
		\omega\frac{d^2 N_1}{dr^2}(\bar{r}_+) > 0.
	\end{equation}
	Therefore, inequality~\eqref{eq: proof lem: subsec: sec: proof of exp decay, sign of g2, lem 1, eq 4} in conjunction with~\eqref{eq: proof lem: subsec: sec: proof of exp decay, sign of g2, lem 1, eq 4.3} implies that the third degree polynomial function~$N_1(r)$ attains exactly one root in~$(r_+,\bar{r}_+)$. Again, in view of that~$N_1(r)$ is a polynomial of degree~$3$ we also conclude~\eqref{eq: lem: subsec: sec: G, subsec 1, lem -1, eq 1}. Therefore, we conclude the result of the Proposition under the second case of~\eqref{eq: proof lem: subsec: sec: proof of exp decay, sign of g2, lem 1, eq 4.1}. 
\end{proof}

\subsection{The definition of~\texorpdfstring{$g_2(\omega,m,r)$}{g} }\label{subsec: sec: morawetz estimate, subsec 4}

In view of Propositions~\ref{prop: subsec: sec: G, subsec 1, prop -1},~\ref{prop: subsec: sec: G, subsec 1, prop 0} the following is well defined

\begin{definition}\label{def: subsec: sec: G, subsec 1, def 1}
	Let~$l>0$ and~$(a,M)\in\mathcal{B}_l$. Let~$(\omega,m)\in\mathbb{R}\times\mathbb{Z}$. Let~$g_1(r),~g_2^2(\omega,m,r)$ be as in Definition~\ref{def: sec: G, def 0}. Moreover, let~$r_{\textit{crit}}(\omega,m)$ be the unique critical point of the function~$g_2^2(\omega,m,\cdot)$, as proved in Propositions~\ref{prop: subsec: sec: G, subsec 1, prop 0},~\ref{prop: subsec: sec: G, subsec 1, prop -1}.

	For the frequencies~$am\omega>\frac{a^2m^2\Xi}{r_+^2+a^2}$ we define 
	\begin{equation}
		g_2(\omega,m,r)=
		\begin{cases}
			-\sqrt{g_2^2},	\quad r\geq r_{\textit{crit}}(\omega,m)\\
			+\sqrt{g_2^2}, \quad r\leq r_{\textit{crit}}(\omega,m).
		\end{cases}
	\end{equation}

	For the superradiant frequencies~$(\omega,m)\in\mathcal{SF}$ we define 
	\begin{equation}
		g_2(\omega,m,r)=-\sqrt{g_2^2} =-\left|\omega-\frac{am\Xi}{r^2+a^2}\right|\frac{r^2+a^2}{\sqrt{\Delta}}.
	\end{equation}
	
	For the frequencies~$am\omega<\frac{a^2m^2\Xi}{\bar{r}_+^2+a^2}$ we define 
		\begin{equation}
	g_2(\omega,m,r)=
	\begin{cases}
	+\sqrt{g_2^2},	\quad r\geq r_{\textit{crit}}(\omega,m)\\
	-\sqrt{g_2^2}, \quad r\leq r_{\textit{crit}}(\omega,m).
	\end{cases}
	\end{equation}
	For the borderline non-superradiant frequencies~$am\omega=\frac{a^2m^2\Xi}{r_+^2+a^2},~am\omega=\frac{a^2m^2\Xi}{\bar{r}_+^2+a^2}$ we define respectively
	\begin{equation}
		g_2(\omega,m,r)= -\sqrt{g_2^2}=-|am\Xi| \frac{r^2-r_+^2}{(r_+^2+a^2)\sqrt{\Delta}},\qquad g_2(\omega,m,r)= -\sqrt{g_2^2}=-|am\Xi| \frac{\bar{r}_+^2-r^2}{(\bar{r}_+^2+a^2)\sqrt{\Delta}}. 
	\end{equation}	
	In the Schwarzschild--de~Sitter case~$a=0$ we define~$g_2=g_1\frac{-i\omega}{\sqrt{1-9M^2\Lambda}}\left(1-\frac{3M}{r}\right)\sqrt{1+\frac{6M}{r}}$, as in our previous~\cite{mavrogiannis}. For the function~$g_2(\omega,m,r)$ graphically see Figure~\ref{fig: g2}. 	
\end{definition}

\subsection{The definition of~\texorpdfstring{$\mathcal{G}$}{g}}\label{subsec: sec: morawetz estimate, subsec 5}

We define here the operator~$\mathcal{G}$, discussed in the introduction, as follows

\begin{customDefinition}{1}\label{def: sec: G, def 1}
Let~$l>0$ and~$(a,M)\in\mathcal{B}_l$. Let~$(\omega,m)\in\mathbb{R}\times\mathbb{Z}$.  Let the functions~$g_1,~g_2$ be as in Definitions~\ref{def: sec: G, def 0},~\ref{def: subsec: sec: G, subsec 1, def 1}.  For any~$r_+<r<\bar{r}_+ $ we define the operator
\begin{equation}
\mathcal{G}=g_1(r)\partial_{r^\star}+i\textit{Op}( g_2),
\end{equation}
such that the Fourier multiplier operator~$\textit{Op}(g_2)$ is defined as 
\begin{equation}
\left(\textit{Op}(g_2)\Psi\right)(r;t,\varphi,\theta)=\frac{1}{\sqrt{2\pi}}\int_{\mathbb{R}}d\omega \sum_{m\in \mathbb{Z}} e^{i\omega t} e^{-im\varphi} g_2(\omega,m,r) 
\mathcal{F}_{\omega,m}(\Psi)(r,\theta), 
\end{equation}
for any smooth~$\Psi:(r_+,\bar{r}_+)\times\mathbb{R} \times [0,2\pi)\times [0,\pi)\rightarrow \mathbb{C}$, where moreover~$\Psi$ is sufficiently integrable. For~$\mathcal{F}_{\omega,m}$ see~\eqref{eq: subsec: sec: carter separation, radial, eq -2}. For~$r\leq r_+$ and for~$r\geq \bar{r}_+$ we define~$\mathcal{G}\equiv 0$, where~$0$ here is an operator. 
\end{customDefinition}

For the Schwarzschild--de~Sitter case~$a=0$ we note that the pseudodifferential operator induced from Definition~\ref{def: sec: G, def 1} corresponds to the physical space vector field~$\mathcal{G}$ of our previous~\cite{mavrogiannis}.

If the solution of the Klein--Gordon equation~\eqref{eq: kleingordon} is axisymmetric, then we will define the alternative physical space vector field 
\begin{equation}
	\mathcal{G}:= g_1(r)\partial_{r^\star}+G_2(r)\partial_{t^\star},
\end{equation}
which we also denote as~$\mathcal{G}\big|_{m=0}$, where 
\begin{equation}
	G_2=
	\begin{cases}
		\sqrt{G_2^2},\qquad r\geq r_{\Delta,\textit{frac}}\\
		-\sqrt{G_2^2},\qquad r\leq r_{\Delta,\textit{frac}}
	\end{cases}
\end{equation}
with~$G_2^2=\frac{(r^2+a^2)^2}{\Delta}-\min_{[r_+,\bar{r}_+]}\frac{(r^2+a^2)^2}{\Delta}$, and~$r_{\Delta,\textit{frac}}$ is the unique critical point of the function~$\frac{\Delta}{(r^2+a^2)^2}$.

\subsection{The properties of the function~\texorpdfstring{$g_2(\omega,m,r)$}{g}}\label{subsec: sec: morawetz estimate, subsec 6}

In the following Lemma we discuss the differentiability of the function~$g_2(\omega,m,r)$, see Definition~\ref{def: subsec: sec: G, subsec 1, def 1}. We will use the following Lemma in the ode estimates of Section~\ref{sec: fixed frequency ode estimates}.

\begin{proposition}\label{prop: subsec: sec: G, subsec 1, prop 1}
	Let~$l>0$ and~$(a,M)\in\mathcal{B}_l$. Let the functions~$g_1,g_2$ be as in Definitions~\ref{def: sec: G, def 0},~\ref{def: subsec: sec: G, subsec 1, def 1} respectively. Then, the following hold.
	
	\begin{enumerate}
		\item For fixed~$(\omega,m)\in \mathbb{R}\times\mathbb{Z}$ we have~$g_2(\omega,m,\cdot)\in C^0(r_+,\bar{r}_+)$\\
		\item For fixed~$(\omega,m)\in (\mathcal{SF})^c$ we have that~$g_2(\omega,m,\cdot)\in C^1 (r_+,\bar{r}_+)$.
		\item For fixed~$(\omega,m)\in \mathcal{SF}$ we have that~$g_2(\omega,m,\cdot) \not\in C^1(r_+,\bar{r}_+)$. Specifically, the function~$g_2^\prime (\omega,m,\cdot)$ is piecewise~$C^0(r_+,\bar{r}_+)$ and is not differentiable only at the point~$r_s$ where~$am\omega =\frac{a^2m^2\Xi}{r_s^2+a^2}$, see Figure~\ref{fig: g2}.  
		\item For fixed~$m\in\mathbb{Z}$ and $r\in(r_+,\bar{r}_+)$ and for any~$\omega$ such that~$(\omega,m)\in (\mathcal{SF})^c$ we have~$g_2(\cdot,m,r)$ is differentiable. \\
		\item For fixed~$r\in(r_+,\bar{r}_+)$ and fixed~$m\in\mathbb{Z}$ the function~$g_2(\cdot,m,r)$ is not differentiable in~$\omega$ at the point~$\omega=\frac{am\Xi}{r^2+a^2}$.\\
		\item For a smooth function~$u$ satisfying the outgoing boundary conditions~\eqref{eq: BC} we obtain 
		\begin{equation}
			\left(g_1\frac{d}{dr^\star}+ig_2(\omega,m,r)\right)\Big|_{r^\star=\pm \infty}u=0.
		\end{equation}
			\item For the non-superradiant frequencies~$(\omega,m)\in (\mathcal{SF})^c$ we have 
		\begin{equation}\label{eq: lem: subsec: sec: G, subsec 1, lem 1, eq 5}
			\frac{1}{g_1}\left|\frac{d g_2}{d r^\star}\right| \geq b\left(\left|\omega-\frac{am\Xi}{r^2+a^2}\right|+\Delta \left(|\omega|+\frac{|am\Xi|}{r^2}\right) \right) ,
		\end{equation}
	\end{enumerate}
\end{proposition}

\begin{proof}
	We obtain the following:
	\begin{enumerate}
			\item Follows from the definition of~$g_2$, see Definition~\ref{def: subsec: sec: G, subsec 1, def 1}. \\
			\item Follows from the definition of~$g_2$, see Definition~\ref{def: subsec: sec: G, subsec 1, def 1} for the non superradiant frequencies~$(\mathcal{SF})^c$. Specifically, since~$g_2^2(\omega,m,\cdot)\in C^\infty (r_+,\bar{r}_+)$, see Proposition~\ref{prop: subsec: sec: G, subsec 1, prop -1}, we use Taylor's theorem around the point~$r=r_{crit}$, for the cases~$r_{crit}\neq r_+,\bar{r}_+$ which correspond to the borderline superradiant cases~$am\omega=\frac{a^2m^2\Xi}{r_+^2+a^2},am\omega=\frac{a^2m^2\Xi}{\bar{r}_+^2+a^2}$ respectively. We obtain
			\begin{equation}
				g_2(\omega,m,r)= \pm (r-r_{crit}) \sqrt{ \frac{1}{2}\big| \frac{d^2}{dr^2}\big|_{r=r_{crit}}\left(g_2^2(\omega,m,r)\right)\big| +\mathcal{O}(r-r_{crit})}.
			\end{equation}
			Since~$\big|\frac{d^2}{dr^2}\big|_{r=r_{crit}}(g_2^2)\big|>0$, see Proposition~\ref{prop: subsec: sec: G, subsec 1, prop -1}, we conclude that indeed~$g_2$ is indeed~$C^1$ in a neighborgood of~$r_{crit}$. Away from~$r_{crit}$ the function~$g_2$ is obviously~$C^1$ from its definition.\\
			\item Follows from the definition of~$g_2$, see Definition~\ref{def: subsec: sec: G, subsec 1, def 1}. Specifically, we see that the function~$g_2(\omega,m,r)$ is not differentiable at the point~$r=r_{crit}=r_s$, where~$am\omega=\frac{a^2m^2\Xi}{r_s^2+a^2}$, see Proposition~\ref{prop: subsec: sec: G, subsec 1, prop 0}.  \\
		\item Follows from the definition of~$g_2$, see Definition~\ref{def: subsec: sec: G, subsec 1, def 1}. \\
		\item  Follows from the definition of~$g_2$, see Definition~\ref{def: subsec: sec: G, subsec 1, def 1}.	\\
		\item  Follows from the definitions of~$g_1,g_2$, see Definition~\ref{def: subsec: sec: G, subsec 1, def 1}, by also using the outgoing boundary condition~\eqref{eq: BC}. 
		\\
			\item We first note that for the borderline non-superradiant frequencies
		\begin{equation*}
			am\omega=\frac{a^2m^2\Xi}{r_+^2+a^2},\qquad am\omega=\frac{a^2m^2\Xi}{\bar{r}_+^2+a^2}
		\end{equation*}
		we have respectively 
		\begin{equation*}
			g_2(\omega,m,r)= -|am\Xi| \frac{r^2-r_+^2}{(r_+^2+a^2)\sqrt{\Delta}},\qquad g_2(\omega,m,r)= -|am\Xi| \frac{\bar{r}_+^2-r^2}{(\bar{r}_+^2+a^2)\sqrt{\Delta}}. 
		\end{equation*}
		see Definition~\ref{def: subsec: sec: G, subsec 1, def 1}, for which~\eqref{eq: lem: subsec: sec: G, subsec 1, lem 1, eq 5} indeed holds after a direct computation. Which indeed satisfies the lower bound~\eqref{eq: lem: subsec: sec: G, subsec 1, lem 1, eq 5}. To prove that the remaining frequencies satisfy the bound~\eqref{eq: lem: subsec: sec: G, subsec 1, lem 1, eq 5} is a trivial computation. The reader should also see Figure~\ref{fig: g2} for more intuition. 
	\end{enumerate}
\end{proof}

\subsection{Parseval identities with the operator~\texorpdfstring{$\mathcal{G}$}{g}}\label{subsec: sec: morawetz estimate, subsec 6.1}

Let~$g_1(r),g_2(\omega,m,r),\mathcal{G},\Psi$ be as in Definition~\ref{def: sec: G, def 1}. Let~$v$ be as follows 
\begin{equation*}
	v=(g_1\partial_{r^\star}+ig_2(\omega,m,r))u,
\end{equation*}
where~$u=\sqrt{r^2+a^2}\mathcal{F}_{\omega,m}(\Psi)$. For~$\mathcal{F}_{\omega,m}$ see~\eqref{eq: subsec: sec: carter separation, radial, eq -2}.

Then, we have the following Parseval identities
\begin{equation}
	\begin{aligned}
		\int_\mathbb{R}  \sum_{m,\ell}|v|^2d\omega &= 	\int_\mathbb{R}\int_{\mathbb{S}^2} |\mathcal{G}\left(\sqrt{r^2+a^2}\Psi\right)|^2  dt d\sigma,\\
		\int_\mathbb{R}  \sum_{m,\ell}\omega^2|v|^2d\omega&= \int_\mathbb{R}\int_{\mathbb{S}^2} |\partial_t\mathcal{G}\left(\sqrt{r^2+a^2}\Psi\right)|^2  dt d\sigma,\\
		\int_\mathbb{R}  \sum_{m,\ell}m^2|v|^2d\omega&= 	\int_\mathbb{R}\int_{\mathbb{S}^2} |\partial_\varphi\mathcal{G}\left(\sqrt{r^2+a^2}\Psi\right)|^2 dt d\sigma,\\
		\int_\mathbb{R}  \sum_{m,\ell}\tilde{\lambda}|v|^2d\omega&=\int_\mathbb{R}\int_{\mathbb{S}^2} |^{d\sigma}\nabla\mathcal{G}\left(\sqrt{r^2+a^2}\Psi\right)|^2 dt d\sigma,\\
		\int_{\mathbb{R}}\sum_{m,\ell}(r-r_+)^{-1}\left|v^\prime+i(\omega-\omega_+m) v\right|^2d\omega	&=\int_{\mathbb{R}}\int_{\mathbb{S}^2}(r-r_+)^{-1}\left|(\partial_{r^\star}-(\partial_t+\omega_+\partial_\varphi))\mathcal{G}(\sqrt{r^2+a^2}\Psi)\right|^2dtd\sigma,\\
		\int_{\mathbb{R}}\sum_{m,\ell}(r-\bar{r}_+)^{-1}\left|v^\prime-i(\omega-\bar{\omega}_+m) v\right|^2d\omega	&=\int_{\mathbb{R}}\int_{\mathbb{S}^2}(r-\bar{r}_+)^{-1}\left|(\partial_{r^\star}+(\partial_t+\bar{\omega}_+\partial_\varphi))\mathcal{G}(\sqrt{r^2+a^2}\Psi)\right|^2dtd\sigma,
	\end{aligned}
\end{equation}
for any~$r_+< r < \bar{r}_+$, where~$d\sigma$ is the standard metric of the unit sphere, and~$^{d\sigma}\nabla$ is the covariant derivative of~$d\sigma$. Recall that~$\omega_+=\frac{a\Xi}{r_+^2+a^2},\bar{\omega}_+=\frac{a\Xi}{\bar{r}_+^2+a^2}$.

\subsection{\textbf{The operator~\texorpdfstring{$\widetilde{\mathcal{G}}$}{g}} and its properties}\label{subsec: sec: morawetz estimate, subsec 7}

We define a more regular version of~$\mathcal{G}$, which we call~$\widetilde{\mathcal{G}}$, see already Definition~\ref{def: proof thm 3 exp decay, def 1}. The reason for introducing~$\widetilde{\mathcal{G}}$ is to be able to appeal to pseudodifferential estimates which require operators with sufficiently regular symbol, see Lemma~\ref{lem: sec: appendix, pseudodifferential com, lem 0}.

Specifically, recall from the definition of~$g_2(\omega,m,r)$, see Definition~\ref{def: subsec: sec: G, subsec 1, def 1}, that for superradiant frequencies we have
\begin{equation}
	g_2(\omega,m,r)=-\left|\omega-\frac{am\Xi}{r^2+a^2}\right| \frac{r^2+a^2}{\sqrt{\Delta}},
\end{equation}
and therefore, for fixed~$m$ and~$r$ the function~$g_2(\cdot,m,r)$ is not differentiable in~$\omega$.

\subsubsection{The function~$\widetilde{g_2}$}\label{subsubsec: subsec: sec: morawetz estimate, subsec 7, subsubsec 1}

Let~$l>0$,~$(a,M)\in\mathcal{B}_l$. We define a function
\begin{equation}\label{eq: subsubsec: subsec: sec: morawetz estimate, subsec 7, subsubsec 1, eq 1}
\tilde{\epsilon}:\mathbb{R}\times \mathbb{Z}\rightarrow\mathbb{R},
\end{equation}
such that~$\tilde{\epsilon}(\cdot,m)$ is smooth for all~$m\in\mathbb{Z}$, and the following hold:
\begin{equation}\label{eq: subsubsec: subsec: sec: morawetz estimate, subsec 7, subsubsec 1, eq 0}
	\tilde{\epsilon}(\omega,m)	 \equiv 0,\quad (\omega,m)\in (\mathcal{SF})^c,
\end{equation}
\begin{equation}\label{eq: subsubsec: subsec: sec: morawetz estimate, subsec 7, subsubsec 1, eq 0.1}
	0<|\tilde{\epsilon}(\omega,m)|\leq C,\quad  (\omega,m)\in\mathcal{SF},
\end{equation}
\begin{equation}\label{eq: subsubsec: subsec: sec: morawetz estimate, subsec 7, subsubsec 1, eq 0.2}
	|\partial_\omega^{i_1}\Delta_m^{i_2}\tilde{\epsilon}(\omega,m)| \leq  C(1+|\omega|)^{-i_1-i_2}(1+|m|)^{-i_1-i_2},\qquad (\omega,m)\in \mathcal{SF},
\end{equation}
for all~$0\leq i_1+i_2\leq 1$, where~$\Delta_m$ is the finite difference operator. The constant~$C=C(a,M,l)$ is independent of~$\omega,m,i_1,i_2$. 

Note that in the Schwarzschild--de~Sitter case~$a=0$ we have that~$\tilde{\epsilon}(\omega,m)\equiv 0$.

\begin{remark}
	We here give an example of a function that satisfies~\eqref{eq: subsubsec: subsec: sec: morawetz estimate, subsec 7, subsubsec 1, eq 0},~\eqref{eq: subsubsec: subsec: sec: morawetz estimate, subsec 7, subsubsec 1, eq 0.1},~\eqref{eq: subsubsec: subsec: sec: morawetz estimate, subsec 7, subsubsec 1, eq 0.2}. Let~$0\leq \lambda_1\leq \lambda_2$. We define the smooth bump function
	\begin{equation}
		f_{a,\lambda_1,\lambda_2}(x)=
		\begin{cases}
			(\lambda_2-\lambda_1) e^{\frac{1}{(x-\lambda_1)(x-\lambda_2)}},\quad \lambda_1\leq x\leq \lambda_2\\
			0,\quad \text{otherwise}.
		\end{cases}
	\end{equation}
	In the interval~$[\lambda_1,\lambda_2]$ we have the derivative
	\begin{equation}
		f^\prime_{\lambda_1,\lambda_2}(x)=-\frac{\left(\lambda _1-\lambda _2\right) e^{\frac{1}{\left(x- \lambda _1\right) \left(x-\lambda _2\right)}} \left( \left(\lambda _1+\lambda _2\right)-2 x\right)}{\left(x- \lambda _1\right){}^2 \left(x- \lambda _2\right){}^2}.
	\end{equation}

	Now, we give an example of a function that satisfies~\eqref{eq: subsubsec: subsec: sec: morawetz estimate, subsec 7, subsubsec 1, eq 0},~\eqref{eq: subsubsec: subsec: sec: morawetz estimate, subsec 7, subsubsec 1, eq 0.1},~\eqref{eq: subsubsec: subsec: sec: morawetz estimate, subsec 7, subsubsec 1, eq 0.2}. Let~$l>0$ and~$(a,M)\in \mathcal{B}_l$. We define the function
	\begin{equation}
		\tilde{\epsilon}(\omega,m)=
		\begin{cases}
			f_{\lambda_1,\lambda_2}\left(\left|\frac{\omega}{m}\right|\right),\qquad(\omega,m)\in\mathcal{SF}\\
			0,\qquad (\omega,m)\in (\mathcal{SF})^c
		\end{cases}
	\end{equation}
	where~$\lambda_1=\frac{|a|\Xi}{\bar{r}_+^2+a^2},~\lambda_2=\frac{|a|\Xi}{r_+^2+a^2}$. For the case~$m=0$ it is clear we can simply define~$\tilde{\epsilon}\equiv 0$. 
\end{remark}

Now, we define the following

\begin{definition}\label{def: proof thm 3 exp decay, def 0}
	Let~$l>0$ and~$(a,M)\in\mathcal{B}_l$. Let~$\tilde{\epsilon}(\cdot,\cdot)$ be as in~\eqref{eq: subsubsec: subsec: sec: morawetz estimate, subsec 7, subsubsec 1, eq 1}. We define the following function
	\begin{equation*}
		\begin{aligned}
			\widetilde{g_2}(\omega,m,r)	&	\equiv g_2(\omega,m,r),\qquad (\omega,m)\in (\mathcal{SF})^c,\\
			\widetilde{g_2}(\omega,m,r)	&	=	-\sqrt{g_1^2\left(\omega-\frac{a\Xi}{r^2+a^2}m\right)^2+\tilde{\epsilon}(\omega,m)},\qquad (\omega,m)\in\mathcal{SF}
		\end{aligned}
	\end{equation*}
	where for~$g_1(r),~g_2(\omega,m,r)$ see Definition~\ref{def: subsec: sec: G, subsec 1, def 1}. In the Schwarzschild--de~Sitter case~$a=0$ we define~$\widetilde{g}_2\equiv g_2$.
\end{definition}

\subsubsection{The definition of~\texorpdfstring{$\widetilde{\mathcal{G}}$}{g}}

In view of Definition~\ref{def: proof thm 3 exp decay, def 0} we define the more regular version of~$\mathcal{G}$, see Definition~\ref{def: sec: G, def 1}, as follows

\begin{definition}\label{def: proof thm 3 exp decay, def 1}
Let~$l>0$ and~$(a,M)\in \mathcal{B}_l$. Let~$g_1,\widetilde{g_2}$ be as in Definition~\ref{def: proof thm 3 exp decay, def 0}. For any~$r\in(r_+,\bar{r}_+)$ we define the following operator 
	\begin{equation}\label{eq: proof thm 3 exp decay, eq 1}
		\widetilde{\mathcal{G}}= g_1\partial_{r^\star}+i\textit{Op}(\widetilde{g_2}),
	\end{equation}
	such that the Fourier multiplier operator~$Op(\widetilde{g_2})$ is defined as 
	\begin{equation}
		\left(\textit{Op}(\widetilde{g_2})\Psi\right)(r;t,\varphi,\theta)=\frac{1}{\sqrt{2\pi}}\int_{\mathbb{R}}d\omega \sum_{m\in \mathbb{Z}} e^{i\omega t} e^{-im\varphi} \widetilde{g_2}(\omega,m,r) \mathcal{F}_{\omega,m}(\Psi)(r,\theta),
	\end{equation}
	for any smooth~$\Psi:(r_+,\bar{r}_+)\times\mathbb{R} \times [0,2\pi)\times [0,\pi)\rightarrow \mathbb{C}$, where moreover~$\Psi$ is sufficiently integrable. For~$\mathcal{F}_{\omega,m}$ see~\eqref{eq: subsec: sec: carter separation, radial, eq -2}. For~$r\leq r_+$ and for~$r\geq \bar{r}_+$ we define~$\widetilde{\mathcal{G}}\equiv 0$, where~$0$ here is an operator. 
\end{definition}

\subsubsection{Properties of the function~\texorpdfstring{$\widetilde{g_2}$}{PDFstring}}

We prove here several properties of~$\widetilde{g_2}$.

\begin{lemma}\label{lem: proof thm 3 exp decay, lem 0}
	Let~$l>0$ and~$(a,M)\in\mathcal{B}_l$. Let~$(\omega,m)\in\mathbb{R}\times\mathbb{Z}$. Let~$g_1,\widetilde{g_2}$ be as in Definition~\ref{def: proof thm 3 exp decay, def 0}. Then, the following hold 
	
	\begin{enumerate}
		\item For any~$(\omega,m)\in\mathbb{R}\times\mathbb{Z}$ we have~$\widetilde{g_2}(\omega,m,\cdot)\in C^1(r_+,\bar{r}_+)$.\\
		\item For any~$r\in (r_+,\bar{r}_+)$ and for any~$m\in\mathbb{Z}$, we have that~$\widetilde{g_2}(\cdot, m,r)\in C^1(\mathbb{R})$.  
		\item For a smooth function~$u$ that satisfies the outgoing boundary conditions~\eqref{eq: BC}  we have
		\begin{equation*}
			\left(g_1(r)\frac{d}{dr^\star}+i\widetilde{g_2}(\omega,m,r)\right)\Big|_{r^\star=\pm\infty} u=0. 
		\end{equation*}
		\item  For any~$r\in[r_+,\bar{r}_+]$ the symbols
		\begin{equation}\label{eq: lem: proof thm 3 exp decay, lem 0, eq 1}
			- g_1(r)i\left(\omega-\frac{am\Xi}{r_+^2+a^2}\right) + i\widetilde{g_2}(\omega,m,r),\qquad 	g_1(r)i\left(\omega-\frac{am\Xi}{\bar{r}_+^2+a^2}\right) + i\widetilde{g_2}(\omega,m,r)
		\end{equation} 
		belong in the classical symbol class 
		\begin{equation*}
			S^1=\{\sigma:\mathbb{R}\times \mathbb{Z}\rightarrow \mathbb{R}:\: \exists C>0~\forall 0\leq i_1+i_2\leq 1~ \left|\partial^{i_1}_\omega\Delta^{i_2}_m \sigma(\omega,m) \right|\leq C\left(1+|\omega|\right)^{1-i_1-i_2}\left(1+|m|\right)^{1-i_1-i_2} \},
		\end{equation*}
		where the constant~$C>0$ is independent of~$r$ and \underline{blows up} in the limit~$a\rightarrow 0$. Note that for~$a=0$ the symbols~\eqref{eq: lem: proof thm 3 exp decay, lem 0, eq 1} also belong in the symbol class~$S^1$, see the definition of~$\widetilde{g}_2$. See already Appendix~\ref{sec: appendix, pseudodifferential commutation} for the Coifman--Meyer commutation for operators in the above symbol class. \\
		\item For any~$r\in (r_+,\bar{r}_+)$ the symbol~$\widetilde{g_2}$ belongs in the symbol class~$S^1$, again for constants~$C(a)$ independent of~$r$, and \underline{blow up} in the limit~$a\rightarrow 0$. For~$a=0$ the symbol~$\widetilde{g}_2$ belongs in the symbol class~$S^1$\\
		\item There exists a constant~$C(a,M,l)>0$, such that~$0<C(0,M,l)<\infty$, such that for any~$r\in (r_+,\bar{r}_+)$ and for any~$(\omega,m)\in \mathbb{R}\times \mathbb{Z}$ we have 
		\begin{equation}\label{eq: lem: proof thm 3 exp decay, lem 0, eq 2}
			|\widetilde{g_2}(\omega,m,r)-g_2(\omega,m,r)|\leq C. 
		\end{equation}
	\end{enumerate}
	
\end{lemma}

\begin{proof}
	First, we discuss the fifth statement. It is easy to see from the definition of~$\widetilde{g}_2$, see Definition~\ref{def: proof thm 3 exp decay, def 0}, that indeed it belongs in the symbol class~$S^1$. Furthermore, the constant in the Coifman Meyer commutations blows up in the limit~$a\rightarrow 0$ since we have 
	\begin{equation}
		\partial_\omega \widetilde{g}_2 =-\frac{2g_1^2\left(\omega-\frac{a\Xi m}{r^2+a^2}\right)+\partial_\omega \tilde{\epsilon}(\omega,m)}{2\sqrt{g_1^2\left(\omega-\frac{a\Xi}{r^2+a^2}m\right)^2+\tilde{\epsilon}(\omega,m)}},
	\end{equation}
	where we need to recall the definition of~$\tilde{\epsilon}(\omega,m)$ from Section~\ref{subsubsec: subsec: sec: morawetz estimate, subsec 7, subsubsec 1}. Note, however, that for~$a=0$ we have~$\widetilde{g}_2\equiv g_2$ and therefore we have that in the case~$a=0$ the symbol~$\widetilde{g}_2$ belongs in the symbol class~$S^1$ for~$r\in(r_+,\bar{r}_+)$.

	Now, we will address the result that the symbols~\eqref{eq: lem: proof thm 3 exp decay, lem 0, eq 1} belong in the symbol class~$S^1$, namely the second statement of the present Lemma. Then, we will prove the estimate~\eqref{eq: lem: proof thm 3 exp decay, lem 0, eq 2}. The remaining statements follow easily from the definition of~$\widetilde{g_2}$, see Definition~\ref{def: proof thm 3 exp decay, def 0}.

	To prove that the operators \eqref{eq: lem: proof thm 3 exp decay, lem 0, eq 1} belong in the symbol class~$S^1$ we proceed as follows. We only discuss the first symbol of~\eqref{eq: lem: proof thm 3 exp decay, lem 0, eq 1} as the proof for the second is the same.

	We write 
	\begin{equation}\label{eq: proof lem: proof thm 3 exp decay, lem 0, eq 1}
		\begin{aligned}
			&	- g_1(r)i\left(\omega-\frac{am\Xi}{r_+^2+a^2}\right) + i\widetilde{g_2}(\omega,m,r) = g_1(r)i R_1(\omega,m,r).
		\end{aligned}
	\end{equation}

	We recall from Propositions~\ref{prop: subsec: sec: G, subsec 1, prop 0},~\ref{prop: subsec: sec: G, subsec 1, prop -1} that for any frequency~$(\omega,m)\in\mathbb{R}\times\mathbb{Z}$ we have a unique critical point~$r_{crit}(\omega,m)$ of the function
	\begin{equation}
		\mathcal{T}(\omega,m,r)= \frac{(r^2+a^2)^2}{\Delta}\left(\omega-\frac{am\Xi}{r^2+a^2}\right)^2.
	\end{equation}

	In view of Definition~\ref{def: proof thm 3 exp decay, def 0}, for the non-superradiant frequencies~$am\omega > \frac{a^2m^2\Xi}{r_+^2+a^2}$ and~$am\omega < \frac{a^2m^2\Xi}{\bar{r}_+^2+a^2}$ we have respectively
	\begin{equation}\label{eq: proof lem: proof thm 3 exp decay, lem 0, eq 2}
		R_1(\omega,m,r)	=	 -\left(\omega-\frac{am\Xi}{r_+^2+a^2}\right)  + 
		\begin{cases}
			-\sqrt{\left(\omega-\frac{am\Xi}{r^2+a^2}\right)^2-\frac{\Delta}{(r^2+a^2)^2} \min_r \mathcal{T}(\omega,m,r)},~~r\geq r_{crit}\\
			+\sqrt{\left(\omega-\frac{am\Xi}{r^2+a^2}\right)^2-\frac{\Delta}{(r^2+a^2)^2}\min_r \mathcal{T}(\omega,m,r)},~~r\leq r_{crit}
		\end{cases}
	\end{equation}
	and
	\begin{equation}\label{eq: proof lem: proof thm 3 exp decay, lem 0, eq 3}
		\begin{aligned}
			R_1(\omega,m,r)	&	=	-\left(\omega-\frac{am^2\Xi}{r_+^2+a^2}\right)+
			\begin{cases}
				+\sqrt{\left(\omega-\frac{am\Xi}{r^2+a^2}\right)^2-\frac{\Delta}{(r^2+a^2)^2} \min_r \mathcal{T}(\omega,m,r)},	\quad r\geq r_{\textit{crit}}\\
				-\sqrt{\left(\omega-\frac{am\Xi}{r^2+a^2}\right)^2-\frac{\Delta}{(r^2+a^2)^2} \min_r \mathcal{T}(\omega,m,r)}, \quad r\leq r_{\textit{crit}}.
			\end{cases}
		\end{aligned}
	\end{equation}
	and for the enlarged superradiant frequencies~$am\omega \in [\frac{a^2m^2\Xi}{\bar{r}_+^2+a^2},\frac{a^2m^2\Xi}{r_+^2+a^2}]$ we have 	
	\begin{equation}\label{eq: proof lem: proof thm 3 exp decay, lem 0, eq 3.1}
		\begin{aligned}
			R_1(\omega,m,r) &	   = -\left(\omega-\frac{am\Xi}{r_+^2+a^2}\right)  -\sqrt{\left(\omega-\frac{a\Xi}{r^2+a^2}m\right)^2+ \frac{\Delta}{(r^2+a^2)^2}Α(\omega,m)},\\
			A(\omega,m)	&	=-\min_{r\in[r_+,\bar{r}_+]}\left( g_1^2\left(\omega-\frac{a\Xi}{r^2+a^2}m\right)^2\right)+\tilde{\epsilon}(\omega,m).
		\end{aligned}
	\end{equation}
	
	For the non-superradiant frequencies~$am\omega > \frac{a^2m^2\Xi}{r_+^2+a^2}$ and~$am\omega < \frac{a^2m^2\Xi}{\bar{r}_+^2+a^2}$ the result that~\eqref{eq: lem: proof thm 3 exp decay, lem 0, eq 1} belong in the symbol class~$S^1$ is immediate by simply differentiating~\eqref{eq: proof lem: proof thm 3 exp decay, lem 0, eq 2} and~\eqref{eq: proof lem: proof thm 3 exp decay, lem 0, eq 3}.

	For the enlarged superradiant frequencies~$am \omega\in [\frac{a^2m^2\Xi}{\bar{r}_+^2+a^2},\frac{a^2m^2\Xi}{r_+^2+a^2}]$ we have 
	\begin{equation}\label{eq: proof lem: proof thm 3 exp decay, lem 0, eq 3.2}
		\partial_\omega R_1 (\omega,m,r)	=	 -\frac{\frac{\Delta (r) \partial_\omega A(\omega ,m)}{\left(a^2+r^2\right)^2}+ 2\left(\omega-\frac{ a m \Xi }{a^2+r^2}\right)}{2 \sqrt{\frac{\Delta (r) A(\omega ,m)}{\left(a^2+r^2\right)^2}+\left(\omega -\frac{a m \Xi }{a^2+r^2}\right)^2}}-1.
	\end{equation}

	We note that there exists a constant~$C(a,M,l)$ such that for any~$r\in [r_+,\bar{r}_+]$ and for any~$am \omega\in [\frac{a^2m^2\Xi}{\bar{r}_+^2+a^2},\frac{a^2 m^2\Xi}{r_+^2+a^2}]$ we have 
	\begin{equation}\label{eq: proof lem: proof thm 3 exp decay, lem 0, eq 3.3}
		| \partial_\omega R_1 (\omega,m)| \leq C(a) ,\qquad | R_1 (\omega,m+1) -  R_1 (\omega,m)|\leq C(a)
	\end{equation}
	in view also of the definition of~$\tilde{\epsilon}(\omega,m)$, see Definition~\ref{def: proof thm 3 exp decay, def 0}. We conclude that for the enlarged superradiant frequencies~$am\omega \in [\frac{a^2m^2\Xi}{\bar{r}_+^2+a^2},\frac{a^2m^2\Xi}{r_+^2+a^2}]$ we indeed have that~$R_1 \in S^1$. It is easy to prove that for any~$r\in [r_+,\bar{r}_+]$ we have that~$g_1(r)iR_1(\cdot,\cdot)\in S^1$. Note that the constant~$C$ of~\eqref{eq: proof lem: proof thm 3 exp decay, lem 0, eq 3.3} blows up in the limit~$a\rightarrow 0$, in view of the definition of~$\epsilon(\omega,m)$, see Section~\ref{subsubsec: subsec: sec: morawetz estimate, subsec 7, subsubsec 1}.

	Now, we discuss~\eqref{eq: lem: proof thm 3 exp decay, lem 0, eq 2}. For~$(\omega,m)\in (\mathcal{SF})^c$ the result is trivial since~$\widetilde{g_2}\equiv g_2$. Therefore, we study the case~$(\omega,m)\in \mathcal{SF}$. For brevity we will use the notation~$\Omega(r)= \frac{a\Xi}{r^2+a^2}$. We have the following identity
	\begin{equation}\label{eq: proof lem: proof thm 3 exp decay, lem 0, eq 4}
		\begin{aligned}
				|\widetilde{g_2}-g_2| &	= \left| \sqrt{g_1^2(\omega-\Omega(r)m)^2+\tilde{\epsilon}(\omega,m)}-|g_1(\omega-\Omega(r)m)|\right|\\
				&	=	1_{|g_1(\omega-\Omega(r)m)|\leq 10} \left| \sqrt{g_1^2(\omega-\Omega(r)m)^2+\tilde{\epsilon}(\omega,m)}-|g_1(\omega-\Omega(r)m)|\right|\\
				&	\quad+ 1_{|g_1(\omega-\Omega(r)m)|\geq 10}\left| \sqrt{g_1^2(\omega-\Omega(r)m)^2+\tilde{\epsilon}(\omega,m)}-|g_1(\omega-\Omega(r)m)|\right|\\
				&	=1_{|g_1(\omega-\Omega(r)m)|\leq 10} \left| \sqrt{g_1^2(\omega-\Omega(r)m)^2+\tilde{\epsilon}(\omega,m)}-|g_1(\omega-\Omega(r)m)|\right|\\
				&	\quad +1_{|g_1(\omega-\Omega(r)m)|\geq 10}
				\left|g_1(\omega-\Omega(r)m)\right| \left(\sqrt{1+\frac{\tilde{\epsilon}(\omega,m)}{g_1^2(\omega-\Omega(r)m)^2}}-1\right).
		\end{aligned}
	\end{equation}
	
	Now, since~$\frac{\tilde{\epsilon}}{g_1^2(\omega-\Omega(r)m)}<1$ in the last term of~\eqref{eq: proof lem: proof thm 3 exp decay, lem 0, eq 4} we use Taylor's theorem and we write~\eqref{eq: proof lem: proof thm 3 exp decay, lem 0, eq 4} as follows
	\begin{equation}\label{eq: proof lem: proof thm 3 exp decay, lem 0, eq 5}
		\begin{aligned}
			|\widetilde{g_2}-g_2| &	=1_{|g_1(\omega-\Omega(r)m)|\leq 10} \left| \sqrt{g_1^2(\omega-\Omega(r)m)^2+\tilde{\epsilon}(\omega,m)}-|g_1(\omega-\Omega(r)m)|\right|\\
			&	\quad +1_{|g_1(\omega-\Omega(r)m)|\geq 10}
			\left|g_1(\omega-\Omega(r)m)\right| \left(\frac{1}{2}\frac{\tilde{\epsilon}(\omega,m)}{g_1^2(\omega-\Omega(r)m)^2}+\mathcal{O}\left(\frac{\tilde{\epsilon}(\omega,m)}{g_1^2(\omega-\Omega(r)m)^2}\right)^2\right)\\
			&	=1_{|g_1(\omega-\Omega(r)m)|\leq 10} \left| \sqrt{g_1^2(\omega-\Omega(r)m)^2+\tilde{\epsilon}(\omega,m)}-|g_1(\omega-\Omega(r)m)|\right|\\
			&	\quad +1_{|g_1(\omega-\Omega(r)m)|\geq 10}
		 \left(\frac{1}{2}\frac{\tilde{\epsilon}(\omega,m)}{g_1|\omega-\Omega(r)m|}+	\left|g_1(\omega-\Omega(r)m)\right|\mathcal{O}\left(\frac{\tilde{\epsilon}(\omega,m)}{g_1^2(\omega-\Omega(r)m)^2}\right)^2\right),
		\end{aligned}
	\end{equation}
	where there exists a constant~$C(a,M,l)>0$ such that for any~$r,\omega,m$ such that~$|g_1(r)(\omega-\Omega(r)m)|\geq 10$ we have that 
	\begin{equation}
		\mathcal{O}\left(\frac{\tilde{\epsilon}(\omega,m)}{g_1^2(\omega-\Omega(r)m)^2}\right)^2\leq  C \frac{\tilde{\epsilon}^2(\omega,m)}{g_1^2(\omega-\Omega(r)m)^2}. 
	\end{equation}

	Now, the result~\eqref{eq: lem: proof thm 3 exp decay, lem 0, eq 2} is immediate from~\eqref{eq: proof lem: proof thm 3 exp decay, lem 0, eq 4} and~\eqref{eq: proof lem: proof thm 3 exp decay, lem 0, eq 5} in view of the definition of the function~$\tilde{\epsilon}(\omega,m)$, see~\eqref{eq: subsubsec: subsec: sec: morawetz estimate, subsec 7, subsubsec 1, eq 1}. We conclude the lemma. 
\end{proof}

\section{The main Theorems}\label{sec: main theorems}

In this Section we present our main theorems.

\begin{customTheorem}{1}[detailed version]\label{main theorem, relat. non-deg}

Let~$l>0$,~$\mu^2_{\textit{KG}}\geq 0$ and
\begin{equation*}
	(a,M)\in \mathcal{MS}_{l,\mu_{KG}},
\end{equation*}
where for~$\mathcal{MS}_{l,\mu_{KG}}$ see~\eqref{eq: sec: main theorems, eq 2}.

Then, there exists a constant 
\begin{equation}
	C(a,M,l,\mu^2_{\textit{KG}})>0
\end{equation}
such that for any~$T\geq 3$ and for any~$\tau_1,\tau_2$ such that~$\tilde{T}\leq \tau_1<\tau_1+T< \tau_2$, where for~$\tilde{T}$ see~\eqref{eq: subsec: sec: intro, 3, eq 1.1}, then the following hold. (Recall from~\eqref{eq: subsec: sec: intro, 3, eq 1.1} that~$\tilde{T}=T^2-2T-1$).

Let~$\psi$ satisfy the Klein--Gordon equation~\eqref{eq: kleingordon} in~$D(\tau_1-\tilde{T},\tau_2)=\supp \chi_+ \cup \supp \chi_-$, where for the cut-offs
\begin{equation}
	\chi_+(t^\star)=\chi_{\tau_1,\tau_1+T^2}(t^\star),\qquad \chi_{-}(t^\star)=\chi_{\tau_1,\tau_1+T^2}(t^\star+\tilde{T})
\end{equation}
see Section~\ref{subsec: sec: carter separation, subsec 1}, and for their graphic representation see Figure~\ref{fig: cut-offs 1}. 

Then, the following estimate holds
\begin{equation}\label{eq: main theorem, relat. non-deg, eq 1}
\begin{aligned}
& \int_{\{t^\star=\tau\}} J^n_\mu[\psi]n^\mu+\int\int_{D(\tau_1+T,\tau_2)}\frac{1}{\Delta}\left|W\mathcal{G}_{\chi_+}\psi\right|^2+ \mathcal{E}(\mathcal{G}_{\chi_+}\psi,\psi) \\
&	\qquad\qquad \leq C\int_{\{t^\star=\tau_1\}} J^n_\mu[\psi]n^\mu+\frac{C}{T^2}\int\int_{D(\tau_1,\tau_1+T)} \mathcal{E}(\mathcal{G}_{\chi_{-}}\psi,\psi),\\
\end{aligned}    
\end{equation}
for any~$\tau_1\leq \tau\leq \tau_2$, where for the operator~$\mathcal{G}$ see Section~\ref{sec: G}. We used the notation~$\mathcal{G}_\chi=\chi\mathcal{G}\chi$. Note that~$W=\partial_t+\frac{a\Xi}{r^2+a^2}\partial_{\varphi}$. For the energy density~$\mathcal{E}(\cdot,\cdot)$ see~\eqref{eq: new energy}. 
\end{customTheorem}
\begin{proof}
See Section~\ref{sec: proof of Thm: rel-nondeg}.
\end{proof}

\begin{remark}
	For fixed~$(a,M,l)$ the constant in the RHS of~\eqref{eq: main theorem, relat. non-deg, eq 1} blow up in the limit~$a\rightarrow 0$. Note, however, that~$a=0$ is indeed allowed in the theorem, i.e. the constant for~$a=0$ is finite. This bad behaviour of our constant stems from the non-differentiability of the~$g_2$, specifically see Lemma~\ref{lem: proof thm 3 exp decay, lem 0} where we study the smoothed version~$\widetilde{g}_2$. 
\end{remark}

For a graphic representation of the cut-offs~$\chi_+,\chi_-$ of Theorem~\ref{main theorem, relat. non-deg} see Figure~\ref{fig: cut-offs 1}.
\begin{figure}[htbp]
	\centering
	\includegraphics[scale=1]{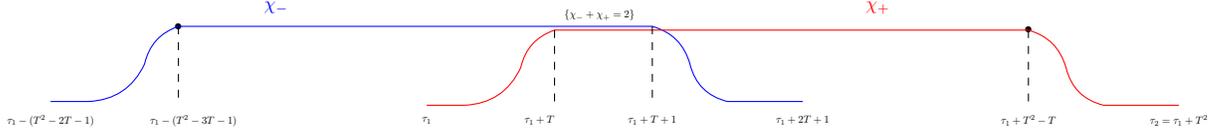}
	\caption{The cut-offs~$\chi_+,\chi_-$~(not to scale)}
	\label{fig: cut-offs 1}
\end{figure}

Exponential decay follows as a Corollary of Theorem~\ref{main theorem, relat. non-deg}.

\begin{customCorollary}{1.1}[detailed version]\label{cor: main theorem, relat. non-deg, cor 3}

Let the assumptions of Theorem~\ref{main theorem, relat. non-deg} be satisfied, where now~$\psi$ satisfies the Klein--Gordon equation~\eqref{eq: kleingordon} in the domain~$D(0,\infty)$ instead. Then, there exist constants
\begin{equation}
	C=C(a,M,l,\mu_{KG}),\qquad c=c(a,M,l,\mu_{KG})
\end{equation}
independent of~$\psi$, such that for any~$T\geq 3$ sufficiently large, then the following hold.

Let
\begin{equation}
	\chi_{\textit{data}}(t^\star)=\chi_{0,T^2}(t^\star).
\end{equation}
Then, for any~$\tau\geq 0$ we have the following
\begin{equation}\label{eq: cor: main theorem, relat. non-deg, cor 3, eq 1}
	\begin{aligned}
		&\int_{\{t^\star=\tau\}} J^n_\mu[\psi]n^\mu \leq C \Bigg(\int_{\{t^\star=0\}}  J^n_\mu[\psi]n^\mu+ \int_{\supp \chi_{\textit{data}}}d\tau^\prime  \int_{\{t^\star=\tau^\prime\}} \mathcal{E}(\mathcal{G}_{\chi_{\textit{data}}}\psi,\psi)\Bigg) \cdot e^{-c\tau}.\\
	\end{aligned}
\end{equation}

Furthermore, we have the following pointwise decay
\begin{equation}\label{eq: cor: main theorem, relat. non-deg, cor 3, eq 2}
	\begin{aligned}
		\sup_{\{t^\star=\tau\}}|\psi-\underline{\psi}|		\leq C\sqrt{E}e^{-c \tau}
	\end{aligned}
\end{equation}
where~$E=\int_{\{t^\star=0\}} J^n_\mu [\psi]n^\mu+J^n_\mu [n\psi]n^\mu +J^n[n^2\psi]n^\mu+J^n[n^3\psi]n^\mu$ and the constant~$\underline{\psi}$ is either~$\underline{\psi}=0$ if~$\mu_{\textit{KG}}^2>0$, or~$|\underline{\psi}|\leq |\underline{\psi}(0)|+C\sqrt{E}$, if~$\mu_{\textit{KG}}=0$. 
\end{customCorollary}

\begin{proof}
	See Section~\ref{sec: cor: main theorem, relat. non-deg, cor 3}. 
\end{proof}

Note the following Remarks

\begin{remark}
	If the solution of the Klein--Gordon equation~\eqref{eq: kleingordon} is in addition axisymmetric, then we obtain the result of Theorem~\ref{main theorem, relat. non-deg} for all subextremal parameters~$l>0$,$(a,M)\in\mathcal{B}_l$ and all Klein--Gordon masses~$\mu^2_{KG}\geq 0$. See Section~\ref{sec: proof of Thm: rel-nondeg, axisymmetry}. 
\end{remark}

\begin{remark}\label{rem: main theorem, relat. non-deg, rem 1}
Note that the integral quantities on the right hand side of~\eqref{eq: cor: main theorem, relat. non-deg, cor 3, eq 1} can be controlled by a Sobolev energy on the initial data hypersurface~$\{t^\star=0\}$, as follows 
\begin{equation}\label{eq: rem: main theorem, relat. non-deg, rem 1, eq 1}
	\begin{aligned}
		&	\int_{\{t^\star=\tau\}}  J^n_\mu[\psi]n^\mu+ \int_{\supp \chi_{\textit{data}}}d\tau^\prime  \int_{\{t^\star=\tau^\prime\}}  \mathcal{E}(\mathcal{G}_{\chi_{\textit{data}}}\psi,\psi) \\
		&	\qquad\qquad \leq B \int_{\{t^\star=0\}}   J^n_\mu[\psi]n^\mu+ B \int\int_{\supp\chi_{data}} \mathcal{E}(\mathcal{G}_{\chi_{\textit{data}}}\psi,\psi) \leq  B \int_{\{t^\star=0\}}  J^n_\mu[\psi]n^\mu+J^n_\mu[n\psi]n^\mu +J^n_\mu[n^2\psi]n^\mu,
	\end{aligned}
\end{equation}
where in~\eqref{eq: rem: main theorem, relat. non-deg, rem 1, eq 1} we only appeal to Parseval and pointwise bounds, without using that~$\psi$ is a solution of the Klein--Gordon equation. To prove~\eqref{eq: rem: main theorem, relat. non-deg, rem 1, eq 1} we appealed to the coarea formula, see Section~\ref{subsec: coarea formula}, and standard finite in time energy estimates for the Klein--Gordon equation~\eqref{eq: kleingordon}, see our companion~\cite{mavrogiannis4}.   
\end{remark}

We also have the following `relatively non degenerate' estimate for the inhomogeneous Klein--Gordon equation in a $\delta$-extended region bounded between the spacelike hypersurfaces~$\mathcal{H}^+_\delta,\bar{\mathcal{H}}^+_\delta$ and the hypersurfaces~$\{t^\star=\tau_1\},~\{t^\star=\tau_2\}$, see Definition~\ref{def: causal domain, def 2}.

\begin{customTheorem}{2}\label{thm: main thm extended region}

 Let~$l>0$,~$\mu^2_{\textit{KG}}\geq 0$ and
 \begin{equation*}
 	(a,M)\in \mathcal{MS}_{l,\mu_{KG}},
 \end{equation*}
 where for~$\mathcal{MS}_{l,\mu_{KG}}$ see~\eqref{eq: sec: main theorems, eq 2}.

 Then, there exist constants
 \begin{equation}
 	C(a,M,l,\mu^2_{\textit{KG}})>0, 	\qquad \delta(a,M,l,\mu^2_{\textit{KG}})>0,
 \end{equation}
 such that for any~$T\geq 3$ and for any~$\tau_1,\tau_2$ such that~$\tilde{T}\leq \tau_1<\tau_1+T< \tau_2$ then the following hold.

  Let~$\psi$ satisfy the inhomogeneous Klein--Gordon equation
 \begin{equation}\label{eq: inhomogeneous equation, KdS}
 	\Box_{g_{a,M,l}}\psi -\mu^2_{KG} \psi =F 
 \end{equation}
 on~$D_\delta(\tau_1-\tilde{T},\tau_2)$, where~$F$ is a sufficiently regular function.

Let~$\{t^\star=\tau\}$ be as in Definition~\ref{def: causal domain, def 2}. Then, the following holds
\begin{equation}
\begin{aligned}
& \int_{\{t^\star=\tau\}} J^n_\mu[\psi]n^\mu \\
&	+\int\int_{D_\delta(\tau_1+T,\tau_2)}\Bigg(\frac{1}{r}\frac{r^2+a^2}{\Delta}\left|W\mathcal{G}_{\chi_+}\psi\right|^2+ \frac{1}{r}\frac{\Delta}{(r^2+a^2)}|Z^\star\mathcal{G}_{\chi_+}\psi|^2+\frac{1}{r}\frac{r^2+a^2}{\Delta}|W\mathcal{G}_{\chi_+}\psi|^2+\frac{1}{r}|\slashed{\nabla}\mathcal{G}_{\chi_+}\psi|^2\\
&	\qquad\qquad\qquad\qquad\qquad+ J^n_\mu[\psi] n^\mu\Bigg)\\
\quad &\leq C(\delta)\int_{\{t^\star=\tau_1\}} J^n_\mu[\psi]n^\mu+\frac{C(\delta)}{T}\int\int_{D_\delta(\tau_1,\tau_1+T)} J_\mu^{W}[\mathcal{G}_{\chi_{-}}\psi]n^\mu \\
&	\qquad + C(\delta)\int\int_{\mathcal{M}_\delta} \Delta|\mathcal{G}\left(\eta\chi_+^2 F\right)|^2+ |\eta\chi_+^2F|^2,
\end{aligned}    
\end{equation} 
for any~$\tau_1\leq \tau\leq \tau_2$, where for~$\mathcal{G}$ see Definition~\ref{def: sec: G, def 1}, and we used the notation~$\mathcal{G}_\chi=\chi\mathcal{G}\chi$. Moreover, for the cut-offs $\eta=\eta^{(T)}_{\tau_1},~\chi_+=\chi_{\tau_1,\tau_1+T^2},~\chi_{-}(t^\star)=\chi_{\tau_1,\tau_1+T^2}(t^\star+\tilde{T})$, see Section~\ref{subsec: sec: carter separation, subsec 1} and for the graphic representation of the latter two cut-offs see Figure~\ref{fig: cut-offs 1}. 
\end{customTheorem}
\begin{proof}
See Section~\ref{sec: proof of Theorem in extended region}. 
\end{proof}

We have the following higher order Corollary

\begin{customCorollary}{2.1}\label{cor: thm: main thm extended region, cor 2}
	Let the assumptions of Theorem~\ref{thm: main thm extended region} hold. Then, for any~$j\geq 3$ there exists a constant~$C=C(j,a,M,l)>0$ and there exists a sufficiently small~$\delta(a,M,l,\mu_{KG},j)>0$, such that for any~$T\geq 3$, we have
\begin{equation}\label{eq: cor: higher order G estimate, eq 1}
	\begin{aligned}
	&	 E_{j-1}[\psi](\tau)+\int\int_{D_\delta(\tau_1+T,\tau_2)}        \frac{1}{\Delta}\sum_{0\leq i_1+i_2\leq j-2}\left|\slashed{\nabla}^{i_1}W^{1+i_2}\mathcal{G}_{\chi_+}\psi\right|^2+\int_{\tau_1+T}^{\tau_2}d\tau^\prime E_{\mathcal{G}_{\chi_+},j}[\psi] d\tau^\prime\\
	&   \qquad \leq C  E_{j-1}[\psi](\tau_1)+\frac{C}{T}\int_{\tau_1}^{\tau_1+T} E_{\mathcal{G}_{\chi_{-}},j}[\psi](\tau^\prime)d\tau^\prime  \\
	&	\qquad\qquad +C\int\int_{\mathcal{M}_\delta} \Delta\sum_{0\leq i_1+i_2\leq j-2}\left|\slashed{\nabla}^{i_1}W^{i_2}\mathcal{G}\left(\eta\chi_+^2 F\right)\right|^2 +\sum_{0\leq i_1+i_2+i_3\leq j-2} \left|\slashed{\nabla}^{i_1}W^{i_2}(Z^\star)^{i_3}\left(\eta\chi_+^2 F\right)\right|^2,\\
	\end{aligned}
	\end{equation}
	for any~$\tau_1\leq \tau\leq \tau_2$, with
	\begin{equation}\label{eq: cor: higher order G estimate, eq 2}
	\begin{aligned}
	E_{j}[\psi](\tau)	&	=  \sum_{1 \leq i_1+i_2+i_3\leq j}\int_{\{t^\star=\tau\}}\left|\slashed{\nabla}^{i_1}\partial_{t^\star}^{i_2}\left(Z^\star\right)^{i_3}\psi\right|^2dg_{\{t^\star=\tau\}},\\	
	E_{\mathcal{G}_\chi,j}[\psi](\tau) &   =\int_{\{t^\star=\tau\}}\sum_{0 \leq i_1+i_2\leq j-1} \Delta\left|\slashed{\nabla}^{i_1}W^{i_2}Z^\star\mathcal{G}_\chi\psi\right|^2+\sum_{1 \leq i_1+i_2\leq j-1}\left|\slashed{\nabla}^{i_1}W^{i_2}\mathcal{G}_\chi\psi\right|^2 \\
	&	\quad\quad\quad\quad\quad\quad\quad+\sum_{1 \leq i_1+i_2+i_3\leq j-1}\left|\slashed{\nabla}^{i_1}W^{i_2}(Z^\star)^{i_3}\psi\right|^2,
	\end{aligned}
	\end{equation} 
where~$W=\partial_t+\frac{a\Xi}{r^2+a^2}\partial_{\varphi}$. Moreover, for the spacetime domain~$\mathcal{M}_\delta$ see Section~\ref{subsec: causal domains}.
\end{customCorollary}
\begin{proof}
	See Section~\ref{sec: proof of Theorem in extended region}.
\end{proof}

\begin{remark}
	Note that at top oder in~\eqref{eq: cor: higher order G estimate, eq 1}, we only include ~$W,\slashed{\nabla}$ derivatives of~$\mathcal{G}_\chi\psi$, because of the bad behavior of high~$Z^\star$ derivatives of~$\mathcal{G}_\chi
	\psi$. Specifically, note from Proposition~\ref{prop: subsec: sec: G, subsec 1, prop 1} that~$(Z^\star)^2\mathcal{G}_\chi\psi$ can only be defined in a distributional sense and therefore~$\int\int |(Z^\star)^2\mathcal{G}_\chi\psi|^2$ is not well defined. 
\end{remark}

\section{The fixed frequency ode estimate:~Theorem~\ref{thm: subsec: sec: proof of Theorem 2, subsec 4.1, thm 1}}\label{sec: fixed frequency ode estimates}

The main result of the present Section is Theorem~\ref{thm: subsec: sec: proof of Theorem 2, subsec 4.1, thm 1}. As mentioned already in the introduction, at low order, we need the Morawetz estimate of Theorem~\ref{main theorem 1} with the sharp degeneration at~$r_{trap}(\omega,m,\ell)$, where~$\ell$ is the Carter frequency. Note, however, that the functions $g_2(\omega,m,r)$, $f_s(\omega,m,r)$, $h_s(\omega,m,r)$ that will be used in the present Section, see Definition~\ref{def: subsec: sec: G, subsec 1, def 1} and~\eqref{eq: proof: prop: for v, eq 5.01},~\eqref{eq: lem: sec: exp decay, subsec: energy estimate, lem 1, energy estimate, eq 3.9} respectively, do not depend on the Carter frequency~$\ell$.

For convenience we will use the notation 
\begin{equation*}
	(^{\prime} \:\dot{=}\:\frac{d}{d r^\star}).
\end{equation*}
In the main estimates of the present Section we track the weights in~$r$ at the top order of our estimates, for the readers interested in the asymptotically flat case.

\subsection{The constants}\label{subsec: sec: fixed frequency ode estimates, subsec 1}

We use the constants
\begin{equation}
	b(a,M,l,\mu_{\textit{KG}})>0,\qquad B(a,M,l,\mu_{\textit{KG}})>0
\end{equation}
with the algebra of constants
\begin{equation}
	b+b=b,\qquad b\cdot b=b,\qquad B+B=B,\qquad B\cdot B=B
\end{equation}

\subsection{Carter's radial ode with outgoing boundary conditions}

In the present Section we study smooth solutions of the inhomogeneous radial ode
\begin{equation}\label{eq: sec: proof of Thm: rel-nondeg, eq 2}
	u^{\prime\prime}+\left(\omega^2-V\right)u=H
\end{equation}
that satisfy outgoing boundary conditions
\begin{equation}\label{eq: subsec: sec: proof of Thm: rel-nondeg, subsec 2, eq 2}
	\begin{aligned}
		&   u^\prime=-i\left(\omega-\frac{am\Xi}{r_+^2+a^2}\right) u,\qquad u^\prime=i\left(\omega-\frac{am\Xi}{\bar{r}_+^2+a^2}\right) u,
	\end{aligned}
\end{equation}
at~$r^\star=-\infty$,~$r^\star=\infty$ respectively. We also assume that~$H(r)$ is a smooth function in~$(r_+,\bar{r}_+)$.

\subsection{The commuted quantity~\texorpdfstring{$v$}{v}}\label{subsec: sec: proof of Theorem 2, subsec 1.1}

Note the following definition
\begin{definition}\label{def: subsec: sec: proof of Theorem 2, subsec 1.1, def 1}
		Let~$l>0$ and~$(a,M)\in \mathcal{B}_l$. Let~$(\omega,m)\in\mathbb{R}\times\mathbb{Z}$. Let~$u$ be a smooth function. Moreover, let~$g_1(r^\star),g_2(\omega,m,r^\star)$ be as in Definition~\ref{def: subsec: sec: G, subsec 1, def 1}. Then, we define the following
	\begin{equation}\label{eq: v}
		v=\left(g_1(r^\star)\frac{d}{d r^\star}+i g_2(\omega,m,r^\star)\right)u.
	\end{equation}
\end{definition}

In view of the definitions of~$g_1,g_2$, see Definition~\ref{def: subsec: sec: G, subsec 1, def 1}, then we obtain that for a smooth~$u$ satisfying outgoing boundary condition~\eqref{eq: subsec: sec: proof of Thm: rel-nondeg, subsec 2, eq 2}, we also conclude that 
\begin{equation}\label{eq: subsec: sec: proof of Theorem 2, subsec 1.1, eq 2}
	v|_{\mathcal{H}^+}=v|_{\bar{\mathcal{H}}^+}=0,
\end{equation}
see Proposition~\ref{prop: subsec: sec: G, subsec 1, prop 1}.

\subsection{The main fixed frequency ode result, Theorem~\ref{thm: subsec: sec: proof of Theorem 2, subsec 4.1, thm 1}}

We prove the following result

\begin{theorem}\label{thm: subsec: sec: proof of Theorem 2, subsec 4.1, thm 1}
	
	Let~$l>0$ and~$(a,M)\in \mathcal{B}_l$ and~$\mu^2_{KG}\geq 0$.
	Then, there exists a sufficiently small
	\begin{equation}
		\epsilon(a,M,l,\mu_{KG})>0
	\end{equation}
	such that for any~$\omega\in \mathbb{R},m\in\mathbb{Z},\ell\in \mathbb{Z}_{\geq |m|}$ the following holds.

	 Let~$u$ be a smooth solution of the inhomogeneous radial ode~\eqref{eq: sec: proof of Thm: rel-nondeg, eq 2}, with smooth inhomogeneity~$H$, that moreover satisfies the outgoing boundary conditions~\eqref{eq: subsec: sec: proof of Thm: rel-nondeg, subsec 2, eq 2}. Let~$v$ be as in Definition~\ref{def: subsec: sec: proof of Theorem 2, subsec 1.1, def 1}. Then, the following energy estimate holds
		\begin{equation}
		\begin{aligned}
			&   \int_{r_+}^{r_++\epsilon}(r-r_+)^{-1}\left|v^\prime+i(\omega-\omega_+m) v\right|^2dr+\int_{\bar{r}_+-\epsilon}^{\bar{r}_+}(\bar{r}_+-r)^{-1}\left|v^\prime-i(\omega-\bar{\omega}_+m) v\right|^2dr\\
			&   + \int_{r_+}^{\bar{r}_+}\left( \frac{\tilde{\lambda}}{r^3} |v|^2+\frac{r^2+a^2}{r\Delta}\left(\omega-\frac{am\Xi}{r^2+a^2}\right)^2|v|^2+\frac{1}{r}|v^\prime|^2 +\Delta^2 \omega^2  |u|^2\right) dr \\
			&	\qquad\qquad \leq B \Bigg( \int_{r_+}^{r_++\epsilon}\frac{1}{\Delta^2}|u^\prime+i(\omega-\omega_+m)u|^2dr +\int_{\bar{r}_+-\epsilon}^{\bar{r}_+}\frac{1}{\Delta^2}|u^\prime-(\omega-\bar{\omega}_+m)u|^2 dr\\
			&	\qquad\qquad\qquad\qquad+ \int_{r_+}^{\bar{r}_+} \left(1_{\{|m|>0\}}|u|^2+ \mu^2_{\textit{KG}}|u|^2+ |u^\prime|^2+\left(1-\frac{r_{\textit{trap}}(\omega,m,\ell)}{r}\right)^2(\omega+\tilde{\lambda})|u|^2\right)dr\Bigg)\\		
			&	\qquad\qquad\qquad  + B\left(\omega^2+m^2\right)|u|^2(-\infty)+  B\left(\omega^2+m^2\right)|u|^2(+\infty)\\
			&	\qquad\qquad\qquad  + B \int_{\mathbb{R}} \left( \frac{1}{\Delta}\left|H\right|^2+r\left|\left(g_1\frac{d}{dr^\star}+ig_2\right)H\right|^2\right)dr^\star,
		\end{aligned}
	\end{equation}
	where~$\tilde{\lambda}=\lambda^{(a\omega)}_{m\ell}+(a\omega)^2$, see Definition~\ref{def: subsec: sec: frequencies, subsec 3, def 0}, where~$\omega_+=\frac{a\Xi}{r_+^2+a^2},~\bar{\omega}_+=\frac{a\Xi}{\bar{r}_+^2+a^2}$. For the trapping parameter~$r_{trap}=r_{trap}(\omega,m,\ell)$ see the Morawetz estimate of Theorem~\ref{main theorem 1}. 
\end{theorem}

\begin{proof}
	See Section~\ref{subsec: sec: proof of Theorem 2, subsec 4.1}. 
\end{proof}

To prove Theorem~\ref{thm: subsec: sec: proof of Theorem 2, subsec 4.1, thm 1} we need several preparatory definitions and results, which are detailed in the following Sections.

\subsection{Preparatory identities}\label{subsec: sec: proof of Theorem 2, subsec 2}

We have the following Lemmata

\begin{lemma}\label{prop: equation of v}
		Let~$l>0$ and~$(a,M)\in \mathcal{B}_l$ and~$\mu^2_{KG}\geq 0$. Let~$\omega\in \mathbb{R},m\in\mathbb{Z},\ell\in \mathbb{Z}_{\geq |m|}$. Let~$u$ be a smooth solution of the inhomogeneous radial ode~\eqref{eq: sec: proof of Thm: rel-nondeg, eq 2} that satisfies the outgoing boundary conditions~\eqref{eq: subsec: sec: proof of Thm: rel-nondeg, subsec 2, eq 2}. 
		
		Then, the radial function~$v$, see Definition~\ref{def: subsec: sec: proof of Theorem 2, subsec 1.1, def 1}, is a~$C^0(-\infty,+\infty)$ function and the derivative~$v^\prime$ is a piesewise~$C^1(-\infty,+\infty)$ function. Specifically, if~$(\omega,m)\in\mathcal{SF}$ then~$v$ is not differentiable at~$r=r_s(\omega,m)$, where~$am\omega=\frac{a^2m^2\Xi}{r_s^2+a^2}$.

		Moreover, the following equation holds 
	\begin{equation}\label{eq: equation of v}
		\begin{aligned}
			v^{\prime\prime} +\left(\omega^2-V\right)v =2i\frac{g_2^\prime}{g_1}v+\left(g_1^{\prime\prime} u^\prime +i g_2^{\prime\prime}u+\frac{v}{g_1}(V_{\textit{SL}}+V_{\mu_{\textit{KG}}})^\prime\right)+\left(2g_1^\prime H +(g_1\frac{d}{d r^\star}+ig_2)H\right)
		\end{aligned}
	\end{equation}
	in a distributional sense.
\end{lemma}
\begin{proof}

We recall from Proposition~\ref{prop: subsec: sec: G, subsec 1, prop 1} that for the superradiant frequencies~$\mathcal{SF}$ that the function~$g_2(\omega,m,r)$ is~$C^0(r_+,\bar{r}_+)$ and the derivative~$g_2^\prime$ is piecewise~$C^0$. Specifically, the function~$g_2^\prime$ is not differentiable only at the point~$r=r_s\in (r_+,\bar{r}_+)$, where~$am\omega=\frac{a^2m^2\Xi}{r_s^2+a^2}$, see Proposition~\ref{prop: subsec: sec: G, subsec 1, prop 1}. It is easy to see from the Definition~\ref{def: subsec: sec: proof of Theorem 2, subsec 1.1, def 1} that, for the superradiant frequencies~$\mathcal{SF}$, the function~$v$ is also a~$C^0(-\infty,+\infty)$ function, where its derivative is piecewise~$C^0$.

For the non superradiant frequencies~$(\omega,m)\in(\mathcal{SF})^c$ we have that~$g_2(\omega,m,\cdot)\in C^1(r_+,\bar{r}_+)$, see Proposition~\ref{prop: subsec: sec: G, subsec 1, prop 1}. It is easy to see from Definition~\ref{def: subsec: sec: proof of Theorem 2, subsec 1.1, def 1} that, for the non-superradiant frequencies~$(\mathcal{SF})^c$, we have ~$v\in C^1(-\infty,+\infty)$.

Now, by recalling the definitions of~$g_1,g_2$, see Definition~\ref{def: subsec: sec: G, subsec 1, def 1}, we compute 
	\begin{equation}
		\begin{aligned}
			v^{\prime\prime} -\left(\omega^2-V\right)v = &\left(g_1^{\prime\prime} u^\prime +ig_2^{\prime\prime}u+\frac{v}{g_1}(V_{\textit{SL}}+V_{\mu_{\textit{KG}}})^\prime\right)+\left(2g_1^\prime H +(g_1\frac{d}{d r^\star}+ig_2)H\right)-\frac{v}{g_1}\left(g_1^2\left(\omega-\frac{am\Xi}{r^2+a^2}\right)^2-g_2^2\right)^\prime\\
			&   \quad +2i\frac{g_2^\prime}{g_1}v,
		\end{aligned}
	\end{equation}
	where the above is to be understood in a  distributional sense, since~$v$ is not differentiable in~$r$ for certain frequencies~$(\omega,m,\ell)$.   
\end{proof}

We need the following energy identity for~$v$.

\begin{lemma}\label{lem: for v, 0}
			Let~$l>0$ and~$(a,M)\in \mathcal{B}_l$ and~$\mu^2_{KG}\geq 0$. Let
			\begin{equation}
				(\omega,m)\in\mathcal{SF}.
			\end{equation}
			
			Let~$u$ be a smooth solution of the inhomogeneous radial ode~\eqref{eq: sec: proof of Thm: rel-nondeg, eq 2} that satisfies the outgoing boundary conditions~\eqref{eq: subsec: sec: proof of Thm: rel-nondeg, subsec 2, eq 2}. Let~$v$ be as in Definition~\ref{def: subsec: sec: proof of Theorem 2, subsec 1.1, def 1}.

Then, we have
	\begin{equation}\label{eq: prop: for v, eq 2}
		\begin{aligned}
			2\int_{\mathbb{R}}\frac{g_2^\prime}{g_1}f_s(r)|v|^2	dr^\star&	=
			- \int_{\mathbb{R}}f_s(r)\Im \left(\bar{v}\left(g_1^{\prime\prime}u^\prime+ig_2^{\prime\prime}u\right)\right)dr^\star\\
			&	\quad -\int_{\mathbb{R}} f^\prime_s(r) \Im\left(\bar{v}v^\prime\right) dr^\star\\
			&	\quad -\int_{\mathbb{R}} f_s(r)\Im\left(\bar{v}\left(2g_1^\prime H+\left(g_1\frac{d}{dr^\star}+ig_2\right)H\right)\right)dr^\star,
		\end{aligned}
	\end{equation} 
	for any~$f_s\in C^\infty [r_+,\bar{r}_+]$ with~$f_s(r_s)=0$, where~$r_s(\omega,m)$ is such that~$am\omega=\frac{a^2m^2\Xi}{r_s^2+a^2}$. 
\end{lemma}
\begin{proof}

	We multiply the equation satisfied by~$v$, see Lemma~\ref{prop: equation of v}, with
	\begin{equation}
		f_s(r)\bar{v},\qquad f_s \in C^\infty_r[r_+,\bar{r}_+].
	\end{equation}
	We obtain 
	\begin{equation}\label{eq: proof: prop: for v, eq 1}
		\begin{aligned}
			f_s(r)\bar{v}\left(v^{\prime\prime} -\left(\omega^2-V\right)v\right) &=2i\frac{g_2^\prime}{g_1}f_s(r)|v|^2+i g_2^{\prime\prime}f_s(r)\bar{v}u+ f_s(r)g_1^{\prime\prime}\bar{v} u^\prime+ f_s(r)\frac{|v|^2}{g_1}(V_{\textit{SL}}+V_{\mu_{\textit{KG}}})^\prime\\
			&	\quad+f_s(r)\bar{v}\left(2g_1^\prime H +(g_1\frac{d}{d r^\star}+ig_2)H\right).
		\end{aligned}
	\end{equation}
	Then, we take imaginary parts to~\eqref{eq: proof: prop: for v, eq 1} and obtain 
	\begin{equation}\label{eq: proof: prop: for v, eq 2}
		\begin{aligned}
			\Im \left( f_s(r)\bar{v} v^{\prime\prime} \right)&	=2\frac{g_2^\prime}{g_1}f_s(r)|v|^2+ g_2^{\prime\prime}f_s(r)\Im\left(iu\bar{v}\right)+g_1^{\prime\prime} f_s(r)\Im \left(\bar{v}u^\prime\right)\\
			&	\quad +f_s(r)\Im\left(\bar{v}\left(2g_1^\prime H+\left(g_1\frac{d}{dr^\star}+ig_2\right)H\right)\right).
		\end{aligned}
	\end{equation} 
	Now, we integrate~\eqref{eq: proof: prop: for v, eq 2} with measure~$dr^\star$ and after the appropriate integration by parts we obtain 
	\begin{equation}\label{eq: proof: prop: for v, eq 3}
		\begin{aligned}
			\Im\left[f_s(r)\bar{v}v^\prime\right]^\infty_{-\infty}+\Im\left[f_s(r)\bar{v}v^\prime\right]^{r_s+}_{r_s-}-\int_{\mathbb{R}} f^\prime_s(r) \Im\left(\bar{v}v^\prime\right)dr^\star &	=2\int_{\mathbb{R}}dr^\star\frac{g_2^\prime}{g_1}f_s(r)|v|^2dr^\star+ \int_{\mathbb{R}} g_2^{\prime\prime}f_s(r)\Im\left(iu\bar{v}\right)dr^\star\\
			&	\quad +\int_{\mathbb{R}} g_1^{\prime\prime} f_s(r)\Im \left(\bar{v}u^\prime\right)dr^\star\\
			&	\quad +\int_{\mathbb{R}} f_s(r)\Im\left(\bar{v}\left(2g_1^\prime H+\left(g_1\frac{d}{dr^\star}+ig_2\right)H\right)\right)dr^\star.
		\end{aligned}
	\end{equation} 
	In view of that~$u$ satisfies outgoing boundary conditions~\eqref{eq: subsec: sec: proof of Thm: rel-nondeg, subsec 2, eq 2} then we have 
	\begin{equation}
		\left[f_s(r(r^\star))\bar{v} v^\prime\right]^{\infty}_{-\infty}=0
	\end{equation}
	see Proposition~\ref{prop: subsec: sec: G, subsec 1, prop 1}. Moreover, by the assumption that~$f_s(r_s)=0$ we have that~$	\left[f_s(r)\bar{v} v^\prime\right]^{r_s+}_{r_s-}=0$. Therefore, we obtain that equation~\eqref{eq: proof: prop: for v, eq 3} implies 
	\begin{equation}\label{eq: proof: prop: for v, eq 4}
		\begin{aligned}
			2\int_{\mathbb{R}}\frac{g_2^\prime}{g_1}f_s(r)|v|^2	dr^\star&	=
			- \int_{\mathbb{R}} g_2^{\prime\prime}f_s(r)\Im\left(iu\bar{v}\right)dr^\star-\int_{\mathbb{R}} f^\prime_s(r) \Im\left(\bar{v}v^\prime\right)dr^\star \\
			&	\quad -\int_{\mathbb{R}} g_1^{\prime\prime} f_s(r)\Im \left(\bar{v}u^\prime\right)dr^\star\\
			&	\quad -\int_{\mathbb{R}} f_s(r)\Im\left(\bar{v}\left(2g_1^\prime H+\left(g_1\frac{d}{dr^\star}+ig_2\right)H\right)\right)dr^\star .
		\end{aligned}
	\end{equation} 
	We conclude~\eqref{eq: prop: for v, eq 2} and the proof.

\end{proof}

\subsection{Preparatory energy estimate for~\texorpdfstring{$v$}{PDFstring} for the non superradiant frequencies~\texorpdfstring{$(\mathcal{SF})^c$}{PDFstring}}

We need the following estimate for the non-superradiant frequencies

\begin{lemma}\label{lem: for v, non super}
				Let~$l>0$ and~$(a,M)\in \mathcal{B}_l$ and~$\mu^2_{KG}\geq 0$. Let
					\begin{equation}
					(\omega,m)\in (\mathcal{SF})^c.
				\end{equation}
				Let~$u$ be a smooth solution of the inhomogeneous radial ode~\eqref{eq: sec: proof of Thm: rel-nondeg, eq 2} that satisfies the outgoing boundary conditions~\eqref{eq: subsec: sec: proof of Thm: rel-nondeg, subsec 2, eq 2}. Let~$v$ be as in Definition~\ref{def: subsec: sec: proof of Theorem 2, subsec 1.1, def 1}.

	We have
	\begin{equation}\label{eq: prop: for v, eq 1}
		\begin{aligned}
		&	b(a,M,l)\int_{r_+}^{\bar{r}_+}\frac{r^2+a^2}{\Delta}\left(\left(\omega-\frac{am\Xi}{r^2+a^2}\right)^2 +\frac{\Delta}{r^3} \left(\omega^2+\frac{(am\Xi)^2}{r^2}\right) \right) |v|^2dr \\
		&	\qquad\qquad \leq  \int_{\mathbb{R}} \left(|\omega|+|am\Xi|\right) |v| \left|g_1^{\prime\prime} u^\prime +ig_2^{\prime\prime}u\right|dr^\star  \\
		&   \qquad\qquad\qquad\qquad +\int_{\mathbb{R}}  r\left|2g_1^\prime H +(g_1\frac{d}{dr^\star}+ig_2)H\right|^2dr^\star .
		\end{aligned}
	\end{equation}
\end{lemma}

\begin{proof}
	We multiply the equation that~$v$ satisfies, see Lemma~\ref{prop: equation of v}, namely 
	\begin{equation}
		\begin{aligned}
			v^{\prime\prime} +\left(\omega^2-V\right)v =2i\frac{g_2^\prime}{g_1}v+\left(g_1^{\prime\prime} u^\prime +i g_2^{\prime\prime}u+\frac{v}{g_1}(V_{\textit{SL}}+V_{\mu_{\textit{KG}}})^\prime\right)+\left(2g_1^\prime H +(g_1\frac{d}{d r^\star}+ig_2)H\right),
		\end{aligned}
	\end{equation}
	with 
	\begin{equation}
		- \text{sign}\left(\frac{g_2^\prime}{g_1}\right)\bar{v}
	\end{equation}
	in view of that~$|\frac{dg_2}{dr^\star}|\geq\Delta (|\omega+|am\Xi|)$ for the non-superradiant frequencies, see Proposition~\ref{prop: subsec: sec: G, subsec 1, prop 1}, and then take imaginary parts. Finally, we integrate and obtain 
	\begin{equation}
		\begin{aligned}
			&  -  \text{sign}\left(\frac{g_2^\prime}{g_1}\right)\Im [\bar{v}v^\prime]^\infty_{-\infty} +2\int_{\mathbb{R}}\left|\frac{g_2^\prime}{g_1}\right||v|^2dr^\star \leq \int_{\mathbb{R}}\left( |v| \left|g_1^{\prime\prime} u^\prime +ig_2^{\prime\prime}u\right|+|v|\left|2g_1^\prime H +(g_1\frac{d}{d r^\star}+ig_2)H\right|\right)dr^\star.
		\end{aligned}
	\end{equation}

	Since~$u$ satisfies the outgoing boundary conditions~\eqref{eq: subsec: sec: proof of Thm: rel-nondeg, subsec 2, eq 2} we can immediately compute that~$v$ also satisfies the outgoing boundary conditions
	\begin{equation}
		v^\prime=
		\begin{cases}
			-i\left(\omega-\frac{am\Xi}{r_+^2+a^2}\right)v,\:\:r^\star=-\infty\\
			i\left(\omega-\frac{am\Xi}{\bar{r}_+^2+a^2}\right)v,\:\:r^\star=\infty.
		\end{cases}    
	\end{equation}
	Therefore, we compute 
	\begin{equation}\label{eq: prop: for v, eq 1, boundary conditions}
		\begin{aligned}
			\Im[\bar{v}v^\prime]^\infty_{-\infty}= \left(\omega-\frac{am\Xi}{\bar{r}_+^2+a^2}\right) |v|^2(\infty)+\left(\omega-\frac{am\Xi}{r_+^2+a^2}\right)|v|^2(-\infty)
		\end{aligned}
	\end{equation}
	where note that~$|v|(\pm \infty)=0$, by Proposition~\ref{prop: subsec: sec: G, subsec 1, prop 1}. We obtain
	\begin{equation}\label{eq: proof: prop: for v, eq -10}
		\begin{aligned}
			2\int_{\mathbb{R}}\left|\frac{g_2^\prime}{g_1}\right||v|^2 dr^\star&	\leq \int_{\mathbb{R}}\left( |v| \left|g_1^{\prime\prime} u^\prime +ig_2^{\prime\prime}u\right|+|v|\left|2g_1^\prime H +(g_1\frac{d}{d r^\star}+ig_2)H\right|\right)dr^\star.\\
		\end{aligned}
	\end{equation}

	We note from Proposition~\ref{prop: subsec: sec: G, subsec 1, prop 1} that the following holds
	\begin{equation}\label{eq: proof: lem: sec: exp decay, subsec: energy estimate, lem 1, energy estimate, eq 0.1}
		\frac{1}{g_1}\left|\frac{d g_2}{d r^\star}\right| \geq b\left(\left|\omega-\frac{am\Xi}{r^2+a^2}\right|+\Delta \left(|\omega|+\frac{|am\Xi|}{r^2}\right) \right). 
	\end{equation}

	We multiply both sides of~\eqref{eq: proof: prop: for v, eq -10} with 
	\begin{equation}
		|\omega|+|am\Xi|
	\end{equation}
	namely 
	\begin{equation}\label{eq: proof: lem: sec: exp decay, subsec: energy estimate, lem 1, energy estimate, eq 1.1}
		\begin{aligned}
			b \int_{\mathbb{R}} \left(|\omega|+|am\Xi|\right)\Big|\frac{g_2^\prime}{g_1}\Big||v|^2dr^\star	&	\leq \int_{\mathbb{R}}\Big( \left(|\omega|+|am\Xi|\right)|v| \left|g_1^{\prime\prime} u^\prime +ig_2^{\prime\prime}u\right|\\
			&	\qquad\qquad\qquad +\left(|\omega|+|am\Xi|\right)|v|\left|2g_1^\prime H +(g_1\frac{d}{d r^\star}+ig_2)H\right|\Big)dr^\star.
		\end{aligned}
	\end{equation}
	Now, in view of~\eqref{eq: proof: lem: sec: exp decay, subsec: energy estimate, lem 1, energy estimate, eq 0.1} inequality~\eqref{eq: proof: lem: sec: exp decay, subsec: energy estimate, lem 1, energy estimate, eq 1.1} implies the following
	\begin{equation}\label{eq: proof: lem: sec: exp decay, subsec: energy estimate, lem 1, energy estimate, eq 2.1}
		\begin{aligned}
			&	b\int_{r_+}^{\bar{r}_+}\frac{r^2+a^2}{\Delta}\left(\left(\omega-\frac{am\Xi}{r^2+a^2}\right)^2 +\frac{\Delta}{r^3} \left(\omega^2+\frac{(am\Xi)^2}{r^2}\right) \right) |v|^2 dr	\\
			&	\qquad\qquad \leq  \int_{\mathbb{R}}\Big( \left(|\omega|+|am\Xi|\right) |v| \left|g_1^{\prime\prime} u^\prime +ig_2^{\prime\prime}u\right| \\
			&   \qquad\qquad\qquad\qquad +\int_{\mathbb{R}}dr^\star  r\left|2g_1^\prime H +(g_1\frac{d}{dr^\star}+ig_2)H\right|^2\Big)dr^\star ,
		\end{aligned}
	\end{equation}
	after appropriate Young's inequalities on the right hand side of~\eqref{eq: proof: lem: sec: exp decay, subsec: energy estimate, lem 1, energy estimate, eq 1.1}. Now, in view of the boundary conditions of~\eqref{eq: subsec: sec: proof of Thm: rel-nondeg, subsec 2, eq 2}, inequality~\eqref{eq: proof: lem: sec: exp decay, subsec: energy estimate, lem 1, energy estimate, eq 2.1} implies the desired estimate~\eqref{eq: lem: sec: exp decay, subsec: energy estimate, lem 1, energy estimate, eq 1}, for a sufficiently small~$\epsilon>0$, after appropriate Young's inequalities.

\end{proof}

\subsection{Preparatory energy estimate for~\texorpdfstring{$v$}{PDFstring} for the superradiant frequencies~\texorpdfstring{$\mathcal{SF}$}{PDFstring}}

We need the following Lemma for the superradiant frequencies~$\mathcal{SF}$. Note that the proof of the following Lemma~\ref{lem: for v, super} is significantly more involved than the proof of the proof of Lemma~\ref{lem: for v, non super}. Specifically, the technical difficulty of non-differentiability of the function~$g_2(\omega,m,\cdot)$, see Proposition~\ref{prop: subsec: sec: G, subsec 1, prop 1}, is resolved with the use of quantitative non-trapping at of the superradiant frequencies, after using Lemma~\ref{lem: for v, 0} for an appropriately chosen smooth function~$f=f_s$.

\begin{lemma}\label{lem: for v, super}
		Let~$l>0$ and~$(a,M)\in \mathcal{B}_l$ and~$\mu^2_{KG}\geq 0$ and let~$\epsilon>0$ be sufficiently small. Let
		\begin{equation}
			(\omega,m)\in \mathcal{SF}. 
		\end{equation}
		
		Let~$u$ be a smooth solution of the inhomogeneous radial ode~\eqref{eq: sec: proof of Thm: rel-nondeg, eq 2} that satisfies the outgoing boundary conditions~\eqref{eq: subsec: sec: proof of Thm: rel-nondeg, subsec 2, eq 2}. Let~$v$ be as in Definition~\ref{def: subsec: sec: proof of Theorem 2, subsec 1.1, def 1}.

		Then, we have 
			\begin{equation}\label{eq: lem: sec: exp decay, subsec: energy estimate, lem 1, energy estimate, eq 1}
			\begin{aligned}
				&	b(a,M,l)\int_{r_+}^{\bar{r}_+}\left( \frac{r^2+a^2}{\Delta}\frac{1}{r} \left(\omega-\frac{am\Xi}{r^2+a^2}\right)^2 |v|^2 +\frac{\Delta}{r^3} \left(\omega^2+\frac{(am\Xi)^2}{r^2}\right) |v|^2\right)dr \\
				&\qquad\qquad \leq \int_{-\infty}^{(r_++\epsilon)^\star} \frac{r^2+a^2}{\Delta (r^2+a^2)^2}\left|u^\prime +i(\omega-\omega_+m)u\right|^2dr^\star\\
				&	\qquad\qquad\quad +\int_{(\bar{r}_+-\epsilon)^\star}^{+\infty}\frac{r^2+a^2}{\Delta (r^2+a^2)^2}\left|u^\prime -i(\omega-\bar{\omega}_+m) u\right|^2dr^\star\\
				&	\qquad\qquad\quad +\int_{\mathbb{R}}\left(\frac{\Delta}{(r^2+a^2)^2}\left(\omega^2+\frac{(am\Xi)^2}{r^2+a^2}+\mu^2_{KG}\right)|u|^2+\Delta |u^\prime|^2\right)dr^\star \\
					&	\qquad\qquad\quad  + \left(\omega^2+m^2\right)|u|^2(-\infty)+  \left(\omega^2+m^2\right)|u|^2(+\infty)\\
				&	\qquad\qquad\quad  + \int_{\mathbb{R}}\left( \frac{1}{\Delta}\left|H\right|^2+r\left|\left(g_1\frac{d}{dr^\star}+ig_2\right)H\right|^2\right)dr^\star.
			\end{aligned}
		\end{equation}
\end{lemma}

\begin{proof}

The main technical difficulty is that we need to use Lemma~\ref{lem: for v, 0}, for a non constant function~$f_s(r)$, which produces `bad' terms on the RHS of our inequalities. There terms however can be absorbed by using an inequality that is a quantitative manifestation of the non-trapping of superradiant frequencies, see already~\eqref{eq: lem: sec: exp decay, subsec: energy estimate, lem 1, energy estimate, eq 4.9}.

We note from Definition~\ref{def: subsec: sec: G, subsec 1, def 1} that the following holds
\begin{equation}
	g_2(r)=-\left|\omega-\frac{am\Xi}{r^2+a^2}\right|\sqrt{\frac{(r^2+a^2)^2}{\Delta}}=-|am\Xi|\frac{|r_s^2-r^2|}{\sqrt{\Delta}}=-|am\Xi||r-r_s|\frac{r_s+r}{\sqrt{\Delta}},
\end{equation}
where~$\omega=\frac{am\Xi}{r_s^2+a^2}$. Note that~$g_2(\omega,m,r)$ is not differentiable in~$r$ at~$r_s\in (r_+,\bar{r}_+)$. Furthermore, recall from Proposition~\ref{prop: subsec: sec: G, subsec 1, prop 0} that for the superradiant frequencies~$(\omega,m)\in \mathcal{SF}$ we have
\begin{equation*}
	r_s(\omega,m)=r_{crit}(\omega,m). 
\end{equation*}
In the present proof we will use~$r_s$ instead of~$r_{crit}$, for notational brevity.

Note that the following hold in a distributional sense 
\begin{equation}\label{eq: proof: prop: for v, eq -2}
	\begin{aligned}
		\frac{d g_2}{dr}(r)	&	=  -\text{sign} \left(\omega-\frac{am\Xi}{r^2+a^2}\right)\cdot \frac{d}{dr}\left(\left(\omega-\frac{am\Xi}{r^2+a^2}\right)\frac{(r^2+a^2)}{\sqrt{\Delta}}\right)\\
		&	= -\text{sign} \left(r-r_s\right)\cdot \frac{d}{dr}\left(\left(\omega-\frac{am\Xi}{r^2+a^2}\right)\frac{(r^2+a^2)}{\sqrt{\Delta}}\right),\\
		\frac{d^2 g_2}{dr^2} (r)	&	=-|am\Xi|\left(2\delta_{r_s}\cdot\frac{r_s+r}{\sqrt{\Delta}}+2\text{sign}(r-r_s)\frac{r+r_s}{\sqrt{\Delta}}+|r-r_s|\frac{d^2 }{dr^2}\left(\frac{r+r_s}{\sqrt{\Delta}}\right)\right),
	\end{aligned}
\end{equation}
where~$\delta_{r_s}$ is the Dirac delta centered at~$r_s$.

Now, we note from the energy estimate of Lemma~\ref{lem: for v, 0} for the superradiant frequencies that the following holds
\begin{equation}\label{eq: proof: prop: for v, eq -1.5}
	\begin{aligned}
		2\int_{\mathbb{R}}\frac{g_2^\prime}{g_1}f_s(r)|v|^2dr^\star	&	=
		- \int_{\mathbb{R}} g_2^{\prime\prime}f_s(r)\Im\left(iu\bar{v}\right)dr^\star -\int_{\mathbb{R}} g_1^{\prime\prime} f_s(r)\Im \left(\bar{v}u^\prime\right)dr^\star \\
		&	\quad  -\int_{\mathbb{R}} f^\prime_s(r) \Im\left(\bar{v}v^\prime\right)dr^\star\\
		&	\quad -\int_{\mathbb{R}} f_s(r)\Im\left(\bar{v}\left(2g_1^\prime H+\left(g_1\frac{d}{dr^\star}+ig_2\right)H\right)\right)dr^\star ,
	\end{aligned}
\end{equation} 
where~$f_s$ is a smooth function that we define now.

\begin{center}
	\textbf{The definition of~$f_s(\omega,m,r)$}
\end{center}

Let~$\epsilon>0$ be sufficiently small~(see also inequality~\eqref{eq: lem: sec: exp decay, subsec: energy estimate, lem 1, energy estimate, eq 4.9} which is a quantitative manifestation of that the superradiant frequencies are non trapped). Moreover, we define the following function	
\begin{equation}\label{eq: proof: prop: for v, eq 5.02}
	\epsilon_s(\omega,m)= \epsilon \cdot  \min \{\left|1-\frac{r_s(\omega,m)}{r_+}\right|,\left|1-\frac{r_s(\omega,m)}{\bar{r}_+}\right|\}.
\end{equation}

It is easy to now construct a smooth (frequency dependent) decreasing function~$f_s: \mathbb{R} \times \mathbb{Z} \times [r_+,\bar{r}_+] \rightarrow \mathbb{R}$, with formula
\begin{equation}\label{eq: proof: prop: for v, eq 5.01}
	f_s(\omega,m,r)=
	\begin{cases}
		-\left(\omega-\frac{am\Xi}{r^2+a^2}\right)= am\Xi \frac{(r-r_s)(r_s+r)}{(r^2+a^2)(r_s^2+a^2)},\quad r\in (r_s-\epsilon_s(\omega,m),r_s+\epsilon_s(\omega,m)),\\
		-\left(\omega-\frac{am\Xi}{r_+^2+a^2}\right) = am\Xi \frac{(r_+-r_s)(r_s+r_+)}{(r_+^2+a^2)(r_s^2+a^2)},\quad r > r_s+2\epsilon_s(\omega,m),\\
		-\left(\omega-\frac{am\Xi}{\bar{r}_+^2+a^2}\right)= am\Xi \frac{(\bar{r}_+-r_s)(r_s+\bar{r}_+)}{(\bar{r}_+^2+a^2)(r_s^2+a^2)},\quad r < r_s-2\epsilon_s (\omega,m),
	\end{cases}
\end{equation}
such that for~$r\in (r_s-2\epsilon_s,r_s-\epsilon_s)\cup (r_s+\epsilon_s,r_s+2\epsilon_s)$ we have that 
\begin{equation}\label{eq: proof: prop: for v, eq 5.001}
	|\frac{d f_s}{dr}(\omega,m,r)|\leq B |am\Xi|,
\end{equation}
where the constant~$B$ is frequency independent. Furthermore, for any~$(\omega,m)\in\mathbb{R}\times\mathbb{Z},r\in [r_+,\bar{r}_+]$ and for any~$r_s(\omega,m)\in (r_+,\bar{r}_+)$ it is easy to see from the definition of the function~$f_s(r)$, see~\eqref{eq: proof: prop: for v, eq 5.01}, that the inequality~\eqref{eq: proof: prop: for v, eq 5.001} holds.

\begin{center}
	\textbf{Upper bounds for the first two terms on the RHS of~\eqref{eq: proof: prop: for v, eq -1.5}}
\end{center}

Now, by recalling the definition of the tortoise coordinate~$r^\star$, see Section~\ref{subsec: tortoise coordinate}, and the distributional derivatives~\eqref{eq: proof: prop: for v, eq -2}, we note that the first term on the right hand side of~\eqref{eq: proof: prop: for v, eq -1.5} is equal to 
\begin{equation}\label{eq: proof: prop: for v, eq 5}
	\begin{aligned}
		&	-\int_{\mathbb{R}} g_2^{\prime\prime}f_s(r)\Im\left(iu\bar{v}\right)dr^\star\\	&	\quad	=-\int_{r_+}^{\bar{r}_+} \frac{d}{dr} \left(\frac{\Delta}{r^2+a^2}\frac{d}{dr}g_2\right)f_s(r)\Im (iu\bar{v})dr  \\
		&	\quad = -\int_{r_+}^{\bar{r}_+} \left(\frac{d}{dr}\left(\frac{\Delta}{r^2+a^2}\right)\frac{d g_2}{dr} + \frac{\Delta}{r^2+a^2}\frac{d^2 g_2}{dr^2} \right) f_s(r) \Im (iu\bar{v})dr \\
		&	\quad = -\int_{r_+}^{\bar{r}_+} \frac{d}{dr}\left(\frac{\Delta}{r^2+a^2}\right)\left(|am\Xi|\cdot \text{sign} \left(r-r_s\right)\cdot \frac{d}{dr}\left(\left(\omega-\frac{am\Xi}{r^2+a^2}\right)\frac{(r^2+a^2)}{\sqrt{\Delta}}\right)\right)f_s(r)\Im (iu\bar{v})dr \\
		&	\qquad -	\int_{r_+}^{\bar{r}_+} \frac{\Delta}{r^2+a^2}|am\Xi|\left(2\delta_{r_s}\cdot\frac{r_s+r}{\sqrt{\Delta}}+2\text{sign}(r-r_s)\frac{r+r_s}{\sqrt{\Delta}}+|r-r_s|\frac{d^2 }{dr^2}\left(\frac{r+r_s}{\sqrt{\Delta}}\right)\right) f_s(r)\Im (iu\bar{v})dr \\
		&	\quad = -\int_{r_+}^{\bar{r}_+} \frac{d}{dr}\left(\frac{\Delta}{r^2+a^2}\right)\left(|am\Xi|\cdot \text{sign} \left(r-r_s\right)\cdot \frac{d}{dr}\left(\left(\omega-\frac{am\Xi}{r^2+a^2}\right)\frac{(r^2+a^2)}{\sqrt{\Delta}}\right)\right)f_s(r)\Im (iu\bar{v})dr\\
		&	\qquad -	\int_{r_+}^{\bar{r}_+} \frac{\Delta}{r^2+a^2}|am\Xi|\left(2\text{sign}(r-r_s)\frac{r+r_s}{\sqrt{\Delta}}+|r-r_s|\frac{d^2 }{dr^2}\left(\frac{r+r_s}{\sqrt{\Delta}}\right)\right) f_s(r)\Im (iu\bar{v})dr. \\
	\end{aligned}
\end{equation}

From~\eqref{eq: proof: prop: for v, eq 5} and from the definition of~$f_s$, see~\eqref{eq: proof: prop: for v, eq 5.01}, we note that the first two terms on the right hand side of~\eqref{eq: proof: prop: for v, eq -1.5} are bounded by 
\begin{equation}\label{eq: proof: prop: for v, eq 5.1}
	\begin{aligned}
		&	\left| \int_{\mathbb{R}} \left(g_2^{\prime\prime}f_s(r)\Im\left(iu\bar{v}\right) + g_1^{\prime\prime} f_s(r)\Im \left(\bar{v}u^\prime\right)\right|\right)dr^\star\\
		&\qquad  \leq  \epsilon B\int_{r_+}^{\bar{r}_+} \frac{(am\Xi)^2}{r^2}|v|^2 dr \\
		&	\qquad\qquad  +\frac{B}{\epsilon}\int_{-\infty}^{(r_++\epsilon)^\star}\frac{r^2+a^2}{\Delta (r^2+a^2)^2}\left|u^\prime +i(\omega-\omega_+m)u\right|^2dr^\star +\frac{B}{\epsilon}\int_{(\bar{r}_+-\epsilon)^\star}^{+\infty}\frac{r^2+a^2}{\Delta (r^2+a^2)^2}\left|u^\prime -i(\omega-\bar{\omega}_+m) u\right|^2dr^\star \\
		&	\qquad\qquad +\frac{B}{\epsilon}\int_{\mathbb{R}}\left(\frac{\Delta}{(r^2+a^2)^2}\left(\omega^2+\frac{(am\Xi)^2}{r^2+a^2}\right)|u|^2+\Delta |u^\prime|^2\right)dr^\star , \\
	\end{aligned}
\end{equation}
for a sufficiently small~$\epsilon>0$ independent of the frequencies~$\omega,m,\ell$.

\begin{center}
	\textbf{Upper bound for the third term on the RHS of~\eqref{eq: proof: prop: for v, eq -1.5}}
\end{center}

The third term on the right hand side of~\eqref{eq: proof: prop: for v, eq -1.5} admits the following bound 
\begin{equation}\label{eq: proof: prop: for v, eq 5.2}
	\int_{\mathbb{R}} |f^\prime_s (r)\bar{v}v^\prime|dr^\star  \leq B\int_{\mathbb{R}} \frac{\Delta}{r^3(r^2+a^2)}\left(\epsilon (am\Xi)^2|v|^2+\frac{1_{\supp f_s^\prime}}{\epsilon}|v^\prime|^2\right)dr^\star, 
\end{equation}
where we used the bound~\eqref{eq: proof: prop: for v, eq 5.001} and appropriate Young's inequalities.  

\begin{center}
	\textbf{An integrated estimate with degeneration at~$r=r_s$}
\end{center}

By recalling the distributional derivative~\eqref{eq: proof: prop: for v, eq -2}, the exam form of the function~$f_s(r)$,~see~\eqref{eq: proof: prop: for v, eq 5.01}, and the bounds~\eqref{eq: proof: prop: for v, eq 5.1},~\eqref{eq: proof: prop: for v, eq 5.2} we obtain from~\eqref{eq: proof: prop: for v, eq -1.5} the following
\begin{equation}\label{eq: proof: prop: for v, eq 6}
	\begin{aligned}
		&	\int_{\mathbb{R}} |am\Xi|^2 \left|1-\frac{r_s}{r}\right|\frac{1}{r}|v|^2dr^\star\\
		&	\qquad \leq  \epsilon B\int_{r_+}^{\bar{r}_+} \frac{(am\Xi)^2}{r^3}|v|^2 dr \\
		&	\qquad \qquad +\frac{B}{\epsilon}\int_{-\infty}^{(r_++\epsilon)^\star} \frac{r^2+a^2}{\Delta (r^2+a^2)^2}\left|u^\prime +i(\omega-\omega_+m)u\right|^2dr^\star+\frac{B}{\epsilon}\int_{(\bar{r}_+-\epsilon)^\star}^{+\infty} \frac{r^2+a^2}{\Delta (r^2+a^2)^2}\left|u^\prime -i(\omega-\bar{\omega}_+m) u\right|^2dr^\star\\
		&	\qquad\qquad +\frac{B}{\epsilon}\int_{\mathbb{R}}\left(\frac{\Delta}{(r^2+a^2)^2}\left(\omega^2+\frac{(am\Xi)^2}{r^2+a^2}\right)|u|^2+\Delta |u^\prime|^2 \right)dr^\star\\
		&	\qquad\qquad + \int_{\mathbb{R}}\left( |f^\prime_s(r)\bar{v}v^\prime|+ |am\Xi|f_s(r)\Im\left(\bar{v}\left(2g_1^\prime H+\left(g_1\frac{d}{dr^\star}+ig_2\right)H\right)\right)\right)dr^\star\\
		&	\qquad \leq  \epsilon B\int_{r_+}^{\bar{r}_+} \frac{(am\Xi)^2}{r^3}|v|^2 dr \\
		&	\qquad \qquad +\frac{B}{\epsilon}\int_{-\infty}^{(r_++\epsilon)^\star} \frac{r^2+a^2}{\Delta (r^2+a^2)^2}\left|u^\prime +i(\omega-\omega_+m)u\right|^2dr^\star+\frac{B}{\epsilon}\int_{(\bar{r}_+-\epsilon)^\star}^{+\infty}\frac{r^2+a^2}{\Delta (r^2+a^2)^2}\left|u^\prime -i(\omega-\bar{\omega}_+m) u\right|^2dr^\star \\
		&	\qquad\qquad +\frac{B}{\epsilon}\int_{\mathbb{R}}\left(\frac{\Delta}{(r^2+a^2)^2}\left(\omega^2+\frac{(am\Xi)^2}{r^2+a^2}\right)|u|^2+\Delta |u^\prime|^2\right)dr^\star \\
		&	\qquad\qquad + B\int_{\mathbb{R}}\left( \frac{\Delta}{r^3(r^2+a^2)}\left(\epsilon (am\Xi)^2|v|^2+\frac{1_{\supp f_s^\prime}}{\epsilon}|v^\prime|^2\right)+f_s(r)\Im\left(\bar{v}\left(2g_1^\prime H+\left(g_1\frac{d}{dr^\star}+ig_2\right)H\right)\right)\right)dr^\star.
	\end{aligned}
\end{equation}

Now, we want to estimate the term 
\begin{equation}
	\int_{\mathbb{R}} \frac{\Delta}{r^3(r^2+a^2)}\frac{1_{\supp f_s^\prime}}{\epsilon}|v^\prime|^2dr^\star
\end{equation}
of the right hand side of~\eqref{eq: proof: prop: for v, eq 6} and more importantly \underline{`cover' the degeneration at~$r_s$} on the left hand side of~\eqref{eq: proof: prop: for v, eq 6}, which as discussed will be achieved by using the \underline{quantitative~ manifestation} that the superradiant frequencies are non trapped, see already inequality~\eqref{eq: lem: sec: exp decay, subsec: energy estimate, lem 1, energy estimate, eq 4.9}. We proceed as follows.

\begin{center}
	\textbf{The use of quantitative non trapping of superradiant frequencies}
\end{center}

Let~$\tilde{h}_s(\omega,m,r)\geq 0$ be a smooth function in~$[r_+,\bar{r}_+]$ such that~$\supp(\tilde{h}_s(\omega,m,\cdot))=\supp(f_s^\prime(\omega,m,\cdot))$, where~$\tilde{h}_s(\omega,m,\cdot)=0$ only at~the boundary~$\partial (\supp(\tilde{h}_s))$ and~$\tilde{h}_s(\omega,m,r)>0$ for any point~$r\in(\partial\supp(f^\prime_s(\omega,m,\cdot)))^o$. Then, we define the function
\begin{equation}\label{eq: lem: sec: exp decay, subsec: energy estimate, lem 1, energy estimate, eq 3.9}
	h_s(\omega,m,r)  =\frac{1}{r(r^2+a^2)} \tilde{h}_s(\omega,m,r). 
\end{equation}
We use a current of the form~$Q^{h_s}$, see Definition~\ref{def: currents}, to the equation that~$v$ satisfies, see Lemma~\ref{prop: equation of v}, in view of that~$v$ is a~$C^0(r_+,\bar{r}_+)$, the derivative~$v^\prime$ is piecewise~$C^1(r_+,\bar{r}_+)$, see Lemma~\ref{prop: equation of v} and the RHS of~\eqref{eq: equation of v} of Lemma~\ref{prop: equation of v} is a distribution. We here note that in our companion~\cite{mavrogiannis4} the multipliers that were used in the current~$Q^h$ also depend on the frequency~$\ell$, but this is not the case in the present paper. We obtain
\begin{equation}\label{eq: lem: sec: exp decay, subsec: energy estimate, lem 1, energy estimate, eq 4}
	\begin{aligned}
		\int_{\mathbb{R}} \left(h_s|v^\prime|^2+\left(h_s(V-\omega^2)-\frac{1}{2}h_s^{\prime\prime}\right)|v|^2	\right)dr^\star&	=Q^{h_s}[v](\infty)+Q^{h_s}(r_s+)-Q^{h_s}(r_s-)-Q^{h_s}[v](-\infty) \\
		&\quad -\Re \int_{\mathbb{R}}\Bigg(h_s v\left(-2i\frac{g_2^\prime}{g_1}\bar{v}\right)+h_s v\left(g_1^{\prime\prime}\bar{u}^\prime+ig_2^{\prime\prime} \bar{u}+\frac{\bar{v}}{g_1}(V_{\textit{SL}}+V_{\mu_{\textit{KG}}})^\prime\right)\\
		&\qquad\qquad\quad+v h_s\left(2g_1^\prime\bar{H}+\overline{\left(g_1\frac{d}{d r^\star}+ig_2\right)H}\right)\Bigg)dr^\star\\
		&=[Q^{h_s}[v]]_{r_s-}^{r_s+}-\Re \int_{\mathbb{R}}\Bigg( v h_s\left(g_1^{\prime\prime}\bar{u}^\prime+ig_2^{\prime\prime} \bar{u}+\frac{\bar{v}}{g_1}(V_{\textit{SL}}+V_{\mu_{\textit{KG}}})^\prime\right)\\
		&\qquad\qquad\qquad\qquad+v h_s\left(2g_1^\prime\bar{H}+\overline{\left(g_1\frac{d}{d r^\star}+ig_2\right)H}\right)\Bigg)dr^\star
	\end{aligned}
\end{equation}
where we used that
\begin{equation}
\Re\left( i\frac{g_2^\prime}{g_1}|v|^2\right)=0,\qquad Q^{h_s}[v](+\infty)=0,\qquad Q^{h_s}[v](-\infty)=0,
\end{equation}
where that last two identities are proved in view of that~$v|_{\mathcal{H}^+}=v|_{\bar{\mathcal{H}}^+}=0$, see Section~\ref{subsec: sec: proof of Theorem 2, subsec 1.1}.

In view of the definition of the current~$Q^h$, see Definition~\ref{def: currents}, we have that the boundary terms at the RHS of~\eqref{eq: lem: sec: exp decay, subsec: energy estimate, lem 1, energy estimate, eq 4} can be bounded as follows 
\begin{equation}
	\begin{aligned}
			[Q^{h_s}[v]]_{(r_s-)^\star}^{(r_s+)^\star}	&	= [h_sg_2^\prime g_1\Re \left(iu \bar{u}^\prime\right)]_{(r_s-)^\star}^{(r_s+)^\star}= -[h_sg_2^\prime g_1\Im \left(u \bar{u}^\prime\right)]_{(r_s-)^\star}^{(r_s+)^\star}= -[h_sg_2^\prime g_1\Im \left(u \bar{u}^\prime\right)]_{(r_s-)^\star}^{(r_s+)^\star}\\
			&	\leq B \left( (\omega^2+(am)^2)|u|^2(-\infty)  + (\omega^2+(am)^2)|u|^2(+\infty)+ \int_{\mathbb{R}}(|\omega|+|am|)|u\bar{H}|dr^\star\right),
	\end{aligned}
\end{equation}
where we used the two identities~$\Im [u\bar{u}^\prime]_{-\infty}^{(r_s-)^\star}=\Im \int_{-\infty}^{(r_s)^\star}u\bar{H}$ and~$\Im [u\bar{u}^\prime]_{(r_s+)^\star}^{(+\infty}=\Im \int_{(r_s+)^\star}^{+\infty}u\bar{H}$.

Now, we find a lower bound for the second integrand term on the left hand side of inequality~\eqref{eq: lem: sec: exp decay, subsec: energy estimate, lem 1, energy estimate, eq 4}. Specifically, note that
\begin{equation}
	V(r_s)-\omega^2 =V_{\textit{SL}}(r_s)+V_{\textit{KG}}(r_s)+V_0(r_s)-\omega^2=V_{\textit{SL}}(r_s)+V_{\textit{KG}}(r_s)+\frac{\Delta}{(r^2+a^2)^2}(r=r_s)\left(\tilde{\lambda}-2m\omega a \Xi\right),
\end{equation}
where for~$V_0,V_{\textit{SL}}$ see Section~\ref{sec: carter separation, radial}, and therefore for a sufficiently small positive constant~$\epsilon>0$, we have
\begin{equation}
	\epsilon_s(\omega,m) = \epsilon \cdot \min \{\left|1-\frac{r_s}{r_+}\right|,\left|1-\frac{r_s}{\bar{r}_+}\right|\}
\end{equation}
such that
\begin{equation}\label{eq: lem: sec: exp decay, subsec: energy estimate, lem 1, energy estimate, eq 4.9}
	V_0(r)-\omega^2\geq b\Delta \tilde{\lambda},\qquad r\in (r_s-2\epsilon_s,r_s+2\epsilon_s).
\end{equation}
Note that inequality~\eqref{eq: lem: sec: exp decay, subsec: energy estimate, lem 1, energy estimate, eq 4.9} \underline{is a quantitative manifestation that the superradiant frequencies are non-trapped}, see~\cite{mavrogiannis4} as well.

Therefore,~\eqref{eq: lem: sec: exp decay, subsec: energy estimate, lem 1, energy estimate, eq 4} implies the following 
\begin{equation}\label{eq: lem: sec: exp decay, subsec: energy estimate, lem 1, energy estimate, eq 5}
	\begin{aligned}
		&    b\int_{\mathbb{R}} h_s \cdot|v^\prime|^2dr^\star+ b\int_{\left((r_s-2\epsilon_{s})^\star,(r_s+2\epsilon_{s})^\star\right)} h_s\cdot|am\Xi|^2|v|^2 dr^\star +b\int_{\mathbb{R}} h_s\frac{\Delta}{(r^2+a^2)^2}\left(\tilde{\lambda}-2m\omega a \Xi \right) |v|^2 dr^\star\\
		&   \quad \leq  \int_{\mathbb{R}} |\frac{1}{2}h_s^{\prime\prime}| |v|^2dr^\star-\Re \int_{\mathbb{R}} vh_s\left(g_1^{\prime\prime} \bar{u}^\prime+ g_2^{\prime\prime} i\bar{u} +\frac{\bar{v}}{g_1}\left(V_{\textit{SL}}+V_{\mu_{\textit{KG}}}\right)^\prime\right)dr^\star\\
		&	\qquad\qquad  +  (\omega^2+m^2)|u|^2(-\infty)  + (\omega^2+m^2)|u|^2(+\infty)+ \int_{\mathbb{R}}(|\omega|+|m|)|u\bar{H}|dr^\star \\
		&\qquad\qquad+\Re \int_{\mathbb{R}}v h_s\left(2g_1^\prime\bar{H}+\overline{\left(g_1\frac{d}{d r^\star}+ig_2\right)H}\right)dr^\star .
	\end{aligned}
\end{equation}
For the third integral term on the right hand side of~\eqref{eq: lem: sec: exp decay, subsec: energy estimate, lem 1, energy estimate, eq 5} we note that 
\begin{equation}\label{eq: lem: sec: exp decay, subsec: energy estimate, lem 1, energy estimate, eq 6}
	\begin{aligned}
		\int_{\mathbb{R}} h_s\cdot g_2^{\prime\prime}\Re\left( i\bar{u}v\right)dr^\star &	= \int_{\mathbb{R}} \left(h_s\cdot g_2^{\prime\prime}\right) \Im (\bar{u}v)dr^\star=\int_{r_+}^{\bar{r}_+} \frac{1}{r} \frac{1}{r^2+a^2}\tilde{h}_s\frac{d}{dr}\left(\frac{\Delta}{r^2+a^2}\frac{d}{dr}g_2 \right)\Im (\bar{u}v)dr\\
		&	=\int_{r_+}^{\bar{r}_+}  \frac{1}{r} \frac{1}{r^2+a^2}\tilde{h}_s\cdot\left(\frac{d}{dr}\frac{\Delta}{r^2+a^2}\cdot\frac{d g_2}{dr}+\frac{\Delta}{r^2+a^2}\frac{d^2 g_2}{dr^2} \right)\Im (\bar{u}v)dr,
	\end{aligned}
\end{equation}
where we used~$h_s=\frac{1}{r(r^2+a^2)}\tilde{h}_s$.

By recalling the distributional derivative~\eqref{eq: proof: prop: for v, eq -2} of~$g_2$ , equation~\eqref{eq: lem: sec: exp decay, subsec: energy estimate, lem 1, energy estimate, eq 6} implies the following
\begin{equation}\label{eq: lem: sec: exp decay, subsec: energy estimate, lem 1, energy estimate, eq 7}
	\begin{aligned}
		\text{RHS of }\eqref{eq: lem: sec: exp decay, subsec: energy estimate, lem 1, energy estimate, eq 6} &	\leq B \int_{r_+}^{\bar{r}_+} \frac{1}{r(r^2+a^2)}\tilde{h}_s |am\Xi \bar{u}v|dr +|am\Xi|\frac{\tilde{h}_s}{r}\Delta\left(\frac{1}{r^2+a^2}\right)^2  \frac{r_s+r}{\sqrt{\Delta}}\Im(\bar{u}v)\Big|_{r^\star=r_s^\star}\\
		&	= B\int_{r_+}^{\bar{r}_+} \frac{1}{r(r^2+a^2)} \tilde{h}_s|am\Xi \bar{u}v|dr +|am\Xi|\frac{\tilde{h}_s}{r}\Delta\left(\frac{1}{r^2+a^2}\right)^2  \frac{r_s+r}{\sqrt{\Delta}}\Im(v\bar{u})\Big|_{r^\star=r^\star_s}\\
		&	= B\int_{r_+}^{\bar{r}_+} \frac{1}{r(r^2+a^2)}\tilde{h}_s |am\Xi \bar{u}v|dr +2\cdot\left(\frac{\sqrt{\Delta} \tilde{h}_s}{r(r^2+a^2)}\right)(r_s)|am\Xi|\Im(v\bar{u})\Big|_{r^\star=r^\star_s}\\
		&	=	B\int_{r_+}^{\bar{r}_+} \frac{1\cdot \tilde{h}_s}{r(r^2+a^2)}|am\Xi \bar{u}v|dr \\
		&	\qquad +B\left( 2  |am\Xi|\Im \left(\frac{\sqrt{\Delta} \cdot \tilde{h}_s}{r(r^2+a^2)}\cdot v\bar{u}\right)\Big|_{r^\star=-\infty}+2\Im\int_{-\infty}^{r_s^\star} |am\Xi| \frac{d}{dr^\star}\left( \frac{\sqrt{\Delta}\cdot \tilde{h}_s}{r(r^2+a^2)}v\bar{u}\right)dr^\star\right),\\
	\end{aligned}	
\end{equation}		
where we used the fundamental theorem of calculus for the function
\begin{equation}
	\Im\left(\frac{\sqrt{\Delta}\cdot \tilde{h}_s}{r(r^2+a^2)} \cdot v\bar{u}\right).
\end{equation}
Therefore, in view of that 
\begin{equation}
	\Im\left(\frac{\sqrt{\Delta}\cdot \tilde{h}_s}{r(r^2+a^2)} \cdot v\bar{u}\right)(r^\star=-\infty)=0
\end{equation}
then~\eqref{eq: lem: sec: exp decay, subsec: energy estimate, lem 1, energy estimate, eq 7} implies the following 
\begin{equation}\label{eq: lem: sec: exp decay, subsec: energy estimate, lem 1, energy estimate, eq 7.1}		
	\begin{aligned}
		\text{RHS of }\eqref{eq: lem: sec: exp decay, subsec: energy estimate, lem 1, energy estimate, eq 6}	&	\leq 	\int_{r_+}^{\bar{r}_+} \frac{\sqrt{\Delta}\cdot \tilde{h}_s}{r(r^2+a^2)} |am\Xi \bar{u}v|dr \\
		&	\quad+2|am\Xi|\cdot\Im\int_{-\infty}^{r_s^\star} \Bigg(\left(\frac{\sqrt{\Delta}\cdot \tilde{h}_s }{r(r^2+a^2)}\right)^\prime v\bar{u}+\frac{\sqrt{\Delta}\cdot \tilde{h}_s}{r(r^2+a^2)}\bar{u}^\prime v\\
		&	\qquad\qquad\qquad\qquad\qquad\qquad\qquad+\frac{\sqrt{\Delta}\cdot \tilde{h}_s}{r(r^2+a^2)}\bar{u}\Big(g_1^\prime u^\prime +ig_2^\prime u +ig_2 u^\prime +g_1 u^{\prime\prime}\Big)\Bigg)dr^\star ,\\ 
		&	\leq 	\int_{r_+}^{\bar{r}_+} \frac{\sqrt{\Delta}\cdot \tilde{h}_s}{r(r^2+a^2)} |am\Xi \bar{u}v|dr \\
		&	\quad+2|am\Xi|\cdot\Im\int_{-\infty}^{r_s^\star}\Bigg(\left(\frac{\sqrt{\Delta}\cdot \tilde{h}_s}{r(r^2+a^2)}\right)^\prime v\bar{u}+\frac{\sqrt{\Delta}\cdot \tilde{h}_s}{r(r^2+a^2)}\bar{u}^\prime v\\
		&	\qquad\qquad\qquad\qquad\qquad\qquad\qquad+\frac{\sqrt{\Delta}\cdot \tilde{h}_s}{r(r^2+a^2)}\bar{u}\left(g_1^\prime u^\prime +ig_2^\prime u +ig_2 u^\prime +g_1 H\bar{u}\right)\Bigg)dr^\star  ,\\ 
		&	\leq 	B\int_{r_+}^{\bar{r}_+}\frac{1}{r^3}\left(\epsilon(am\Xi)^2|v|^2+\frac{1}{\epsilon}(am\Xi)^2|u|^2+ (am\Xi)^2|u|^2 +|u^\prime|^2+|am\Xi|^2|u|^2+|g_1 H|^2\right)dr.\\ 
	\end{aligned}
\end{equation}
Therefore, in view of~\eqref{eq: lem: sec: exp decay, subsec: energy estimate, lem 1, energy estimate, eq 6},~\eqref{eq: lem: sec: exp decay, subsec: energy estimate, lem 1, energy estimate, eq 7.1}, we note that inequality~\eqref{eq: lem: sec: exp decay, subsec: energy estimate, lem 1, energy estimate, eq 5} implies that 
\begin{equation}\label{eq: lem: sec: exp decay, subsec: energy estimate, lem 1, energy estimate, eq 8}
	\begin{aligned}
		&    b\int_{\mathbb{R}} h_s|v^\prime|^2dr^\star+ b\int_{\left((r_s-2\epsilon_{s})^\star,(r_s+2\epsilon_{s})^\star\right)} h_s|am\Xi|^2|v|^2 dr^\star +b\int_{\mathbb{R}} h_s\frac{\Delta}{(r^2+a^2)^2}\left(\tilde{\lambda}-2m\omega a \Xi \right)|v|^2dr^\star  \\
		&   \quad \leq  \int_{\mathbb{R}} |\frac{1}{2}h_s^{\prime\prime}| |v|^2dr^\star-\Re \int_{\mathbb{R}} vh_s\left(g_1^{\prime\prime} \bar{u}^\prime +\frac{\bar{v}}{g_1}\left(V_{\textit{SL}}+V_{\mu_{\textit{KG}}}\right)^\prime\right)dr^\star\\
		&\qquad\qquad  + B\int_{r_+}^{\bar{r}_+}\frac{1}{r^3}\left(\epsilon(am\Xi)^2|v|^2+\frac{1}{\epsilon}(am\Xi)^2|u|^2+ (am\Xi)^2|u|^2 +|u^\prime|^2+|am\Xi|^2|u|^2+|g_1 H|^2\right)dr	\\
		&	\qquad\qquad + \left(\omega^2+m^2\right)|u|^2(-\infty)+  \left(\omega^2+m^2\right)|u|^2(+\infty)+ \int_{r_+}^{\bar{r}_+}\frac{1}{\Delta^2}|H|^2dr\\
		&\qquad\qquad+\Re \int_{\mathbb{R}} v h_s\left(2g_1^\prime\bar{H}+\overline{\left(g_1\frac{d}{d r^\star}+ig_2\right)H}\right)dr^\star,
	\end{aligned}
\end{equation}
where recall~$h_s=\frac{1}{r^3(r^2+a^2)}\tilde{h}_s$.

By combining inequalities~\eqref{eq: proof: prop: for v, eq 6} and~\eqref{eq: lem: sec: exp decay, subsec: energy estimate, lem 1, energy estimate, eq 8} we conclude
\begin{equation}\label{eq: lem: sec: exp decay, subsec: energy estimate, lem 1, energy estimate, eq 9}
	\begin{aligned}
		b\int_{\mathbb{R}} |am\Xi|^2 \left( \frac{\Delta}{(r^2+a^2)^2} + \left|1-\frac{r_s}{r}\right|\frac{1}{r}\right)|v|^2  dr^\star&	\leq  \int_{-\infty}^{(r_++\epsilon)^\star} \frac{r^2+a^2}{\Delta (r^2+a^2)^2}\left|u^\prime +i(\omega-\omega_+m)u\right|^2dr^\star\\
		&	\qquad\qquad+\int_{(\bar{r}_+-\epsilon)^\star}^{+\infty} \frac{r^2+a^2}{\Delta (r^2+a^2)^2}\left|u^\prime -i(\omega-\bar{\omega}_+m) u\right|^2dr^\star\\
		&	\qquad\qquad +\int_{\mathbb{R}}\frac{\Delta}{(r^2+a^2)^2}\left(\omega^2+\frac{(am\Xi)^2}{r^2+a^2}\right)|u|^2+\Delta |u^\prime|^2 dr^\star\\ 
		&	\qquad\qquad  + \left(\omega^2+m^2\right)|u|^2(-\infty)+  \left(\omega^2+m^2\right)|u|^2(+\infty)\\
		&	\qquad\qquad   + \int_{\mathbb{R}}\left( \frac{1}{\Delta}\left|H\right|^2+r\left|\left(g_1\frac{d}{dr^\star}+ig_2\right)H\right|^2\right)dr^\star,
	\end{aligned}
\end{equation}
for a sufficiently small~$\epsilon>0$, after appropriate Young's inequalities. For the superradiant frequencies
\begin{equation}
	(\omega,m)\in \mathcal{SF}
\end{equation}
the following inequality holds uniformly
\begin{equation}
	\frac{\Delta}{r^3} \left(\omega^2+\frac{(am\Xi)^2}{r^2}\right)+\frac{1}{r}\left|\omega-\frac{am\Xi}{r^2+a^2}\right|^2 \lesssim  |am\Xi|^2 \left( \frac{\Delta}{(r^2+a^2)^2} + \left|1-\frac{r_s}{r}\right|\frac{1}{r}\right).
\end{equation}
Therefore, we control the following term on the left hand side of~\eqref{eq: lem: sec: exp decay, subsec: energy estimate, lem 1, energy estimate, eq 9} 
\begin{equation}\label{eq: lem: sec: exp decay, subsec: energy estimate, lem 1, energy estimate, eq 10}
	\begin{aligned}
		b\int_{\mathbb{R}}\left(	\frac{\Delta}{r^3} \left(\omega^2+\frac{(am\Xi)^2}{r^2}\right)|v|^2 + \frac{1}{r}\left(\omega-
		\frac{am\Xi}{r^2+a^2}\right)^2 |v|^2\right)dr^\star   &	\leq  \int_{-\infty}^{(r_++\epsilon)^\star} \frac{r^2+a^2}{\Delta (r^2+a^2)^2}\left|u^\prime +i(\omega-\omega_+m)u\right|^2dr^\star\\
		&	\qquad\qquad +\int_{(\bar{r}_+-\epsilon)^\star}^{+\infty} \frac{r^2+a^2}{\Delta (r^2+a^2)^2}\left|u^\prime -i(\omega-\bar{\omega}_+m) u\right|^2dr^\star\\
		&	\qquad\qquad +\int_{\mathbb{R}}\left(\frac{\Delta}{(r^2+a^2)^2}\left(\omega^2+\frac{(am\Xi)^2}{r^2+a^2}\right)|u|^2+\Delta |u^\prime|^2\right)dr^\star \\ 
			&	\qquad\qquad  + \left(\omega^2+m^2\right)|u|^2(-\infty)+  \left(\omega^2+m^2\right)|u|^2(+\infty)\\
		&	\qquad\qquad  + \int_{\mathbb{R}}\left( \frac{1}{\Delta}\left|H\right|^2+r\left|\left(g_1\frac{d}{dr^\star}+ig_2\right)H\right|^2\right)dr^\star,
	\end{aligned}
\end{equation} 
for any sufficiently small~$\epsilon>0$, which concludes the proof. 
\end{proof}

\subsection{Preparatory energy estimate for~\texorpdfstring{$v$}{PDFstring} for all~\texorpdfstring{$(\omega,m)\in\mathbb{R}\times\mathbb{Z}$}{PDFstring}}\label{subsec: sec: proof of Theorem 2, subsec 4}

Now, we prove our main estimate for~$v$. Moreover, we prove a Poincare type  inequality, see Lemma~\ref{lem: control of the law order derivatives with v}, which we use to prove that at the low order there is no loss of derivatives at trapping.

\begin{proposition}\label{prop: all the derivatives of v}
	
	Let~$l>0$ and~$(a,M)\in \mathcal{B}_l$ and~$\mu^2_{KG}\geq 0$. Moreover, let~$(\omega,m)\in\mathbb{R}\times\mathbb{Z}$. Then, for any sufficiently small~$\epsilon>0$ we have the following.

	Let~$u$ be a smooth solution of the inhomogeneous radial ode~\eqref{eq: sec: proof of Thm: rel-nondeg, eq 2} that satisfies the outgoing boundary conditions~\eqref{eq: subsec: sec: proof of Thm: rel-nondeg, subsec 2, eq 2}. Let~$v$ be as in Definition~\ref{def: subsec: sec: proof of Theorem 2, subsec 1.1, def 1}. Then, the following energy estimate holds
	\begin{equation}\label{eq: all the derivatives of v}
		\begin{aligned}
			&   \int_{r_+}^{r_++\epsilon}(r-r_+)^{-1}\left|v^\prime+i(\omega-\omega_+m) v\right|^2dr+\int_{\bar{r}_+-\epsilon}^{\bar{r}_+}(\bar{r}_+-r)^{-1}\left|v^\prime-i(\omega-\bar{\omega}_+m) v\right|^2dr\\
			&   + \int_{r_+}^{\bar{r}_+}\left( \frac{\tilde{\lambda}}{r^3} |v|^2+\frac{r^2+a^2}{r\Delta}\left(\omega-\frac{am\Xi}{r^2+a^2}\right)^2|v|^2+\frac{1}{r}|v^\prime|^2\right) dr \\
			&   \quad \leq  B \int_{-\infty}^{(r_++\epsilon)^\star} \frac{r^2+a^2}{\Delta (r^2+a^2)^2}\left|u^\prime +i(\omega-\omega_+m)u\right|^2dr^\star+B\int_{(\bar{r}_+-\epsilon)^\star}^{+\infty} \frac{r^2+a^2}{\Delta (r^2+a^2)^2}\left|u^\prime -i(\omega-\bar{\omega}_+m) u\right|^2 dr^\star\\
			&	\qquad +B\int_{\mathbb{R}}\left(\frac{\Delta}{(r^2+a^2)^2}\left(\omega^2+\frac{(am\Xi)^2}{r^2+a^2}+\mu^2_{KG}\right)|u|^2+\Delta |u^\prime|^2\right)dr^\star \\
				&	\qquad + \left(\omega^2+m^2\right)|u|^2(-\infty)+  \left(\omega^2+m^2\right)|u|^2(+\infty)\\
			& \qquad + B\int_{\mathbb{R}}\left( \frac{1}{\Delta}\left|H\right|^2+r\left|\left(g_1\frac{d}{dr^\star}+ig_2\right)H\right|^2\right)dr^\star,
		\end{aligned}
	\end{equation}
	where~$\tilde{\lambda}=\lambda^{(a\omega)}_{m\ell}+(a\omega)^2$, see Definition~\ref{def: subsec: sec: frequencies, subsec 3, def 0}, and for~$\omega_+=\frac{a\Xi}{r_+^2+a^2},~\bar{\omega}_+=\frac{a\Xi}{\bar{r}_+^2+a^2}$.
\end{proposition}

\begin{proof}

First, we easily combine Lemmata~\ref{lem: for v, non super},~\ref{lem: for v, super} to obtain that for any sufficiently small~$\epsilon$ and all~$(\omega,m)\in\mathbb{R}\times \mathbb{Z},\ell\in\mathbb{Z}_{\geq |m|}$ the following holds 
\begin{equation}\label{eq: proof: lem: estimate for v prime, eq 1}
		\begin{aligned}
			&	b\int_{r_+}^{\bar{r}_+}\left( \frac{r^2+a^2}{\Delta}\frac{1}{r} \left(\omega-\frac{am\Xi}{r^2+a^2}\right)^2 |v|^2 +\frac{\Delta}{r^3} \left(\omega^2+\frac{(am\Xi)^2}{r^2}\right) |v|^2 \right)dr\\
			&\qquad\qquad \leq \int_{-\infty}^{(r_++\epsilon)^\star} \frac{r^2+a^2}{\Delta (r^2+a^2)^2}\left|u^\prime +i(\omega-\omega_+m)u\right|^2 dr^\star\\
			&	\qquad\qquad\quad +\int_{(\bar{r}_+-\epsilon)^\star}^{+\infty} \frac{r^2+a^2}{\Delta (r^2+a^2)^2}\left|u^\prime -i(\omega-\bar{\omega}_+m) u\right|^2dr^\star\\
			&	\qquad\qquad\quad +\int_{\mathbb{R}}\left(\frac{\Delta}{(r^2+a^2)^2}\left(\omega^2+\frac{(am\Xi)^2}{r^2+a^2}+\mu^2_{KG}\right)|u|^2+\Delta |u^\prime|^2\right)dr^\star \\
				&	\qquad\qquad\quad  + \left(\omega^2+m^2\right)|u|^2(-\infty)+  \left(\omega^2+m^2\right)|u|^2(+\infty)\\
			&	\qquad\qquad\quad  + \int_{\mathbb{R}}\left( \frac{1}{\Delta}\left|H\right|^2+r\left|\left(g_1\frac{d}{dr^\star}+ig_2\right)H\right|^2\right)dr^\star,
		\end{aligned}
	\end{equation}
where~$\omega_+=\frac{a\Xi}{r_+^2+a^2},~\bar{\omega}_+=\frac{a\Xi}{\bar{r}_+^2+a^2}$.

	Specifically, we have from~\eqref{eq: proof: lem: estimate for v prime, eq 1} the following 
	\begin{equation}\label{eq: proof: lem: estimate for v prime, eq 2}
		\begin{aligned}
			b \int_{r_+}^{\bar{r}_+}\frac{r^2+a^2}{ \Delta}\frac{1}{r}\left(\omega-\frac{am\Xi}{r^2+a^2}\right)^2|v|^2 dr	\leq &\int_{-\infty}^{(r_++\epsilon)^\star} \frac{r^2+a^2}{\Delta (r^2+a^2)^2}\left|u^\prime +i(\omega-\omega_+m)u\right|^2 dr^\star\\
			&	\quad +\int_{(\bar{r}_+-\epsilon)^\star}^{+\infty}\frac{r^2+a^2}{\Delta (r^2+a^2)^2}\left|u^\prime -i(\omega-\bar{\omega}_+m) u\right|^2 dr^\star \\
			&	\quad +\int_{\mathbb{R}}\left(\frac{\Delta}{(r^2+a^2)^2}\left(\omega^2+\frac{(am\Xi)^2}{r^2+a^2}+\mu^2_{KG}\right)|u|^2+\Delta |u^\prime|^2\right)dr^\star \\
			&	\quad  + \left(\omega^2+m^2\right)|u|^2(-\infty)+  \left(\omega^2+m^2\right)|u|^2(+\infty)\\
			&	\quad  +  \int_{\mathbb{R}}\left( \frac{1}{\Delta}\left|H\right|^2+r\left|\left(g_1\frac{d}{dr^\star}+ig_2\right)H\right|^2\right)dr^\star,
		\end{aligned}
	\end{equation}
	for any sufficiently small~$\epsilon>0$.

	Now, we multiply the equation that is satisfied by~$v$ namely~\eqref{eq: equation of v} with~$\frac{1}{r}\bar{v}$ and then take real parts. We obtain 
	\begin{equation}\label{eq: proof: lem: estimate for v prime, eq 3}
		\begin{aligned}
			&	\Re \left(\frac{1}{r}\bar{v}v^{\prime\prime}\right)+\left(-\frac{\Delta}{(r^2+a^2)^2}\left(\tilde{\lambda}-2m\omega a\Xi\right)+\left(\omega-\frac{am\Xi}{r^2+a^2}\right)^2\right)\frac{1}{r}|v|^2 \\
			&	\qquad =\Re \left(\frac{1}{r}\bar{v}i g_2^{\prime\prime} u\right) +\Re \left(\frac{1}{r}\bar{v}g_1^{\prime\prime} u^\prime\right)+|v|^2\frac{1}{g_1 r}\left(V_{\textit{SL}}+V_{\mu_{\textit{KG}}}\right)^\prime+\Re \left(\frac{1}{r}\bar{v}\left(2g_1^\prime H +\left(g_1\frac{d}{dr^\star}+ig_2\right)H\right)\right).
		\end{aligned}
	\end{equation}
	Then, we integrate~\eqref{eq: proof: lem: estimate for v prime, eq 3} and after appropriate integration by parts and rearranging we obtain 
	\begin{equation}\label{eq: proof: lem: estimate for v prime, eq 4}
		\begin{aligned}
			&	\int_{\mathbb{R}} \left(\frac{1}{r}|v^\prime|^2 +\frac{\Delta}{(r^2+a^2)^2}\frac{1}{r}\left(\tilde{\lambda}-2m\omega a \Xi\right)|v|^2\right)dr^\star\\
			&	\quad =\int_{\mathbb{R}}\frac{1}{r}\left(\omega-\frac{am\Xi}{r^2+a^2}\right)^2|v|^2dr^\star-\int_{\mathbb{R}} \frac{1}{r}\left(\Re \left(g_1^{\prime\prime}\bar{v} u^\prime+ig_2^{\prime\prime}\bar{v}u\right)+ |v|^2 \frac{1}{g_1}\left(V_{\textit{SL}}+V_{\mu_{\textit{KG}}}\right)^\prime \right)dr^\star \\
			&	\qquad-\int_{\mathbb{R}}\frac{1}{r^2}\Re\left(\bar{v}v^\prime\right)dr^\star -\int_{\mathbb{R}}\frac{1}{r}\left( \bar{v}2g_1^\prime H+\bar{v}\left(g_1\frac{d}{dr^\star}+ig_2\right)H \right)dr^\star.
		\end{aligned}
	\end{equation}

We note that for a sufficiently small~$\epsilon>0$ we obtain
	\begin{equation}\label{eq: proof: lem: estimate for v prime, eq 5}
		\begin{aligned}
			\int_{\mathbb{R}} \frac{1}{r}|v^\prime|^2dr^\star &	= \int_{-\infty}^{r^\star(r_++\epsilon)} \frac{1}{r}\left|v^\prime -i\left(\omega-\frac{am\Xi}{r_+^2+a^2}\right)v+i\left(\omega-\frac{am\Xi}{r_+^2+a^2}\right)v \right|^2 dr^\star +\int_{r^\star(r_++\epsilon)}^{r^\star(\bar{r}_+-\epsilon)} \frac{1}{r}|v^\prime|^2 dr^\star\\
			&	\qquad+\int_{r^\star (\bar{r}_+-\epsilon)}^\infty \frac{1}{r}\left|v^\prime +i\left(\omega-\frac{am\Xi}{\bar{r}_+^2+a^2}\right)v-i\left(\omega-\frac{am\Xi}{\bar{r}_+^2+a^2}\right)v\right|^2 dr^\star.
		\end{aligned}
	\end{equation}
	By appropriate Young's inequalities on~\eqref{eq: proof: lem: estimate for v prime, eq 5}, and in view of that that~$u$ satisfies the outgoing boundary conditions:	
	\begin{equation}
		\begin{aligned}
			&   \frac{d u}{d r^\star}=-i\left(\omega-\frac{am\Xi}{r_+^2+a^2}\right)u,\qquad \frac{d u}{d r^\star} =i\left(\omega-\frac{am\Xi}{\bar{r}_+^2+a^2}\right)u,
		\end{aligned}
	\end{equation}
	at~$r^\star=-\infty,r^\star=+\infty$ respectively, we obtain 
	\begin{equation}\label{eq: proof: lem: estimate for v prime, eq 6} 
		\begin{aligned}
			&	\int_{-\infty}^{r^\star(r_++\epsilon)} \frac{1}{r}\left|v^\prime -i\left(\omega-\frac{am\Xi}{r_+^2+a^2}\right)v \right|^2 dr^\star+\int_{r^\star (\bar{r}_+-\epsilon)}^\infty \frac{1}{r}\left|v^\prime +i\left(\omega-\frac{am\Xi}{\bar{r}_+^2+a^2}\right)v\right|^2 dr^\star \\
			&	\qquad \leq B \int_{\mathbb{R}}\frac{1}{r}\left(|v^\prime|+\left(\omega-\frac{am\Xi}{r^2+a^2}\right)^2|v|^2\right)dr^\star.
		\end{aligned}
	\end{equation}
	
	Finally, to control the term 
	\begin{equation}
		\int_{\mathbb{R}} \frac{1}{r}\Re \left(g_1^{\prime\prime}\bar{v} u^\prime+ig_2^{\prime\prime}\bar{v}u\right)dr^\star
	\end{equation}
	on the right hand side of~\eqref{eq: proof: lem: estimate for v prime, eq 4}, we use an estimate similar to the estimate~\eqref{eq: proof: prop: for v, eq 5.1} of Lemma~\ref{lem: for v, super} which recall controls a similar term.

	Therefore, we combine~\eqref{eq: proof: lem: estimate for v prime, eq 1},~\eqref{eq: proof: lem: estimate for v prime, eq 4},~\eqref{eq: proof: lem: estimate for v prime, eq 5},~\eqref{eq: proof: lem: estimate for v prime, eq 6} and conclude the result of the Proposition after using appropriate Young's inequalities. 
\end{proof}

We need the following Lemma, whose estimate is a Poincare type inequality.

\begin{lemma}\label{lem: control of the law order derivatives with v}
	
		Let~$l>0$ and~$(a,M)\in \mathcal{B}_l$. For any~$\epsilon_p,\epsilon>0$ sufficiently small we have the following. Let~$u$ be a smooth function in~$\mathbb{R}$, that satisfies the outgoing boundary conditions~\eqref{eq: subsec: sec: proof of Thm: rel-nondeg, subsec 2, eq 2}~(not necessarily a solution of Carter's radial ode~\eqref{eq: sec: proof of Thm: rel-nondeg, eq 2}). Let~$v$ be as in Definition~\ref{def: subsec: sec: proof of Theorem 2, subsec 1.1, def 1}. 
		
		Then, the following holds
	\begin{equation}
		\begin{aligned}
		b\frac{1}{\epsilon_p}\int_{r_+}^{\bar{r}_+}  \frac{1}{r^4}\left(\frac{\Delta}{r^2+a^2}\right)^2 |u|^2 dr &   \leq \frac{1}{\epsilon_p}\int_{r_+}^{\bar{r}_+}  \left|\frac{d}{dr}\frac{\Delta}{r^2+a^2}\right|\frac{\left|r-r_{\textit{trap}}\right|}{r^4}\frac{\Delta}{r^2+a^2} |u|^2 dr\\
		&   \quad+\frac{1}{\epsilon_p^2}\int_{r_+}^{\bar{r}_+}  \frac{1}{r^2}\left(\frac{r-r_{\textit{trap}}}{r^2}\right)^2 \frac{1}{g_1}\frac{\Delta}{r^2+a^2} |u|^2dr \\
		&   \quad+\int_{r_+}^{\bar{r}_+}  \frac{1}{g_1}\frac{1}{r^2}\frac{\Delta}{r^2+a^2}|(g_1\frac{d}{d r^\star}+ig_2)u|^2dr .
	\end{aligned}
	\end{equation}
\end{lemma}

\begin{proof}
	
We note
	\begin{equation}\label{eq: proof lem: control of the law order derivatives with v, eq 1}
		\int_{r_+}^{\bar{r}_+} \frac{1}{r^4} \left(\frac{\Delta}{r^2+a^2}\right)^2 |u|^2 dr^\star = \int_{r_+}^{\bar{r}_+}\frac{1}{r^4} \left(\frac{\Delta}{r^2+a^2}\right)^2\frac{d}{dr}(r-r_{\textit{trap}})|u|^2 dr,
	\end{equation}
	where for trapping parameter~$r_{\textit{trap}}(\omega,m,\ell)$ see Theorem~\ref{main theorem 1}. We integrate by parts, equation~\eqref{eq: proof lem: control of the law order derivatives with v, eq 1}, to obtain  
	\begin{equation}
		\begin{aligned}
			&- \int_{r_+}^{\bar{r}_+}\Bigg( 2\frac{1}{r^4}\frac{\Delta}{r^2+a^2}\frac{d}{dr}\left(\frac{\Delta}{r^2+a^2}\right)(r-r_{\textit{trap}})|u|^2-\frac{4}{r^5}\left(\frac{\Delta}{r^2+a^2}\right)^2(r-r_{\textit{trap}})|u|^2\\
			&	\qquad+\left(\frac{\Delta}{r^2+a^2}\right)^2\frac{(r-r_{\textit{trap}})}{r^4}\left(\bar{u}\frac{d}{dr}u+u\frac{d}{dr}\bar{u}\right) \Bigg)dr\\ 
			&  = - \int_{r_+}^{\bar{r}_+} \Bigg(2\frac{\Delta}{r^2+a^2}\frac{d}{dr}\left(\frac{\Delta}{r^2+a^2}\right)\frac{(r-r_{\textit{trap}})}{r^4}|u|^2-4\left(\frac{\Delta}{r^2+a^2}\right)^2\frac{(r-r_{\textit{trap}})}{r^5}|u|^2\\
			&   \quad\quad\quad\quad +\frac{1}{g_1(r^\star)}\left(\frac{\Delta}{(r^2+a^2)}\right)^2\frac{(r-r_{\textit{trap}})}{r^4}\left(\bar{u}g_1\frac{d}{dr}u+ug_1\frac{d}{dr}\bar{u}\right)\Bigg)dr\\
			&= - \int_{r_+}^{\bar{r}_+} \Big(2 \frac{\Delta}{r^2+a^2} \frac{d}{dr}\left(\frac{\Delta}{r^2+a^2}\right)\frac{(r-r_{\textit{trap}})}{r^4}|u|^2-4\left(\frac{\Delta}{r^2+a^2}\right)^2\frac{(r-r_{\textit{trap}})}{r^5}|u|^2\Big)dr\\
			&   \quad - \int_{r_+}^{\bar{r}_+} \frac{1}{g_1}\frac{(r-r_{\textit{trap}})}{r^4}\frac{\Delta}{r^2+a^2}\left( \bar{u}\left(g_1 u^\prime +ig_2(\omega,m,r^\star)u\right)+u \overline{\left(g_1 u^\prime +ig_2(\omega,m,r^\star)u\right)} \right)dr.
		\end{aligned}
	\end{equation}

	We finally use Young's inequality to obtain 
	\begin{equation}\label{eq: proof lem: control of the law order derivatives with v, eq 2}
		\begin{aligned}
			b\frac{1}{\epsilon}\int_{r_+}^{\bar{r}_+}  \frac{1}{r^4}\left(\frac{\Delta}{r^2+a^2}\right)^2 |u|^2 dr &   \leq \frac{1}{\epsilon_p}\int_{r_+}^{\bar{r}_+}  \left|\frac{d}{dr}\frac{\Delta}{r^2+a^2}\right|\frac{\left|r-r_{\textit{trap}}\right|}{r^4}\frac{\Delta}{r^2+a^2} |u|^2dr \\
			&   \quad+\frac{1}{\epsilon_p^2}\int_{r_+}^{\bar{r}_+}  \frac{1}{r^2}\left(\frac{r-r_{\textit{trap}}}{r^2}\right)^2 \frac{1}{g_1}\frac{\Delta}{r^2+a^2} |u|^2 dr \\
			&   \quad+\int_{r_+}^{\bar{r}_+}  \frac{1}{g_1}\frac{1}{r^2}\frac{\Delta}{r^2+a^2}|(g_1\frac{d}{d r^\star}+ig_2)u|^2dr .
		\end{aligned}
	\end{equation}
	We conclude the following 
	\begin{equation}\label{eq: proof lem: control of the law order derivatives with v, eq 3}
		\begin{aligned}
			b\frac{1}{\epsilon_p}\int_{r_+}^{\bar{r}_+}  \frac{1}{r^4}\left(\frac{\Delta}{r^2+a^2}\right)^2 |u|^2 dr &   \leq \frac{1}{\epsilon_p}\int_{r_+}^{\bar{r}_+}  \left|\frac{d}{dr}\frac{\Delta}{r^2+a^2}\right|\frac{\left|r-r_{\textit{trap}}\right|}{r^4}\frac{\Delta}{r^2+a^2} |u|^2 dr\\
			&   \quad+\frac{1}{\epsilon_p^2}\int_{r_+}^{\bar{r}_+}  \frac{1}{r^2}\left(\frac{r-r_{\textit{trap}}}{r^2}\right)^2 \frac{1}{g_1}\frac{\Delta}{r^2+a^2} |u|^2dr \\
			&   \quad+\int_{r_+}^{\bar{r}_+}  \frac{1}{g_1}\frac{1}{r^2}\frac{\Delta}{r^2+a^2}|(g_1\frac{d}{d r^\star}+ig_2)u|^2 dr ,
		\end{aligned}
	\end{equation}
	and the Lemma. 
\end{proof}

\subsection{The proof of Theorem~\ref{thm: subsec: sec: proof of Theorem 2, subsec 4.1, thm 1}}\label{subsec: sec: proof of Theorem 2, subsec 4.1}

\begin{proof}
	First, we recall the energy estimate of Proposition~\ref{prop: all the derivatives of v}
	\begin{equation}\label{eq: proof: prop: energy estimate for v, 1, eq 1}
		\begin{aligned}
			&   \int_{r_+}^{r_++\epsilon}(r-r_+)^{-1}\left|v^\prime+i(\omega-\omega_+m) v\right|^2dr+\int_{\bar{r}_+-\epsilon}^{\bar{r}_+}(\bar{r}_+-r)^{-1}\left|v^\prime-i(\omega-\bar{\omega}_+m) v\right|^2dr\\
			&   + \int_{r_+}^{\bar{r}_+}\left( \frac{\tilde{\lambda}}{r^3} |v|^2+\frac{r^2+a^2}{r\Delta}\left(\omega-\frac{am\Xi}{r^2+a^2}\right)^2|v|^2+\frac{1}{r}|v^\prime|^2\right) dr \\
			&   \quad \leq B\Bigg(\int_{-\infty}^{(r_++\epsilon)^\star} \frac{1}{\Delta}\left|u^\prime +i(\omega-\omega_+m)u\right|^2dr^\star+\int_{(\bar{r}_+-\epsilon)^\star}^{+\infty} \frac{1}{\Delta}\left|u^\prime -i(\omega-\bar{\omega}_+m) u\right|^2dr^\star\\
			&	\qquad\qquad\qquad\qquad +\int_{\mathbb{R}}\left(\frac{\Delta}{(r^2+a^2)^2}\left(\omega^2+\frac{(am\Xi)^2}{r^2+a^2}+\mu^2_{KG}\right)|u|^2+\Delta |u^\prime|^2\right)dr^\star\Bigg) \\
				&	\qquad + B\left(\omega^2+m^2\right)|u|^2(-\infty)+  B\left(\omega^2+m^2\right)|u|^2(+\infty)\\
			&	\qquad + B \int_{\mathbb{R}} \left( \frac{1}{\Delta}\left|H\right|^2+r\left|\left(g_1\frac{d}{dr^\star}+ig_2\right)H\right|^2\right)dr^\star ,
		\end{aligned}
	\end{equation}
	for a sufficiently small~$\epsilon>0$, where~$\tilde{\lambda}=\lambda^{(a\omega)}_{m\ell}+(a\omega)^2$, see Definition~\ref{def: subsec: sec: frequencies, subsec 3, def 0}.

	Now, we estimate the first bulk term on the right hand side of~\eqref{eq: proof: prop: energy estimate for v, 1, eq 1} by using the result of the Poincare type estimate of Lemma~\ref{lem: control of the law order derivatives with v} for any sufficiently small~$\epsilon_p>0$ to obtain 
		\begin{equation}\label{eq: proof: prop: energy estimate for v, 1, eq 2}
		\begin{aligned}
			&   \int_{r_+}^{r_++\epsilon}(r-r_+)^{-1}\left|v^\prime+i(\omega-\omega_+m) v\right|^2dr+\int_{\bar{r}_+-\epsilon}^{\bar{r}_+}(\bar{r}_+-r)^{-1}\left|v^\prime-i(\omega-\bar{\omega}_+m) v\right|^2dr\\
			&   + \int_{r_+}^{\bar{r}_+}\left( \frac{\tilde{\lambda}}{r^3} |v|^2+\frac{r^2+a^2}{r\Delta}\left(\omega-\frac{am\Xi}{r^2+a^2}\right)^2|v|^2+\frac{1}{r}|v^\prime|^2\right)dr+\int_{r_+}^{r_+}  \frac{1}{r^4}\omega^2 1_{\{(\omega,m,\tilde{\lambda}):r_{trap}\neq 0\}} \left(\frac{\Delta}{r^2+a^2}\right)^2 |u|^2 dr \\
			&	\qquad\qquad \leq B\Bigg( \int_{r_+}^{r_++\epsilon}\frac{1}{\Delta^2}|u^\prime+i(\omega-\omega_+m)u|^2 dr +\int_{\bar{r}_+-\epsilon}^{\bar{r}_+}\frac{1}{\Delta^2}|u^\prime-(\omega-\bar{\omega}_+m)u|^2dr \\
			&	\qquad\qquad\qquad\qquad+ \int_{r_+}^{\bar{r}_+} \left(1_{\{|m|>0\}}|u|^2+ \mu^2_{\textit{KG}}|u|^2+ |u^\prime|^2+\left(1-\frac{r_{\textit{trap}}(\omega,m,\tilde{\lambda})}{r}\right)^2(\omega+\lambda^{(a\omega)}_{m\ell})|u|^2\right)dr\Bigg)\\		
				&	\qquad\qquad\qquad  + B\left(\omega^2+m^2\right)|u|^2(-\infty)+  B\left(\omega^2+m^2\right)|u|^2(+\infty)\\
			&	\qquad\qquad\qquad  + B \int_{\mathbb{R}} \left( \frac{1}{\Delta}\left|H\right|^2+r\left|\left(g_1\frac{d}{dr^\star}+ig_2\right)H\right|^2\right)dr^\star ,
		\end{aligned}
	\end{equation}
	which is the desired result. Note that to use the Poincare estimate of Lemma~\ref{lem: control of the law order derivatives with v} to obtain~\eqref{eq: proof: prop: energy estimate for v, 1, eq 2} we used that the LHS of~\eqref{eq: proof: prop: energy estimate for v, 1, eq 1} in conjunction with the LHS of Theorem~\ref{main theorem 1} control the following 
	\begin{equation}
		\int_{r_+}^{\bar{r}_+}\Delta\left(\omega^2+a^2m^2\right)|u|^2dr.
	\end{equation}
	We conclude. 
\end{proof}

\section{Proof of Theorem \ref{main theorem, relat. non-deg}}\label{sec: proof of Thm: rel-nondeg}

The main ingredients in the proof of Theorem~\ref{main theorem, relat. non-deg} is the fixed frequency Theorem~\ref{thm: subsec: sec: proof of Theorem 2, subsec 4.1, thm 1} in conjunction with the Morawetz estimate of Theorem~\ref{main theorem 1}, proved in our companion~\cite{mavrogiannis4}, and several pseudodifferential commutation arguments on the Kerr--de~Sitter manifold, see Section~\ref{subsec: sec: proof of Theorem 2, subsec 5.1}.

\subsection{Assumption on the smoothness of~\texorpdfstring{$\psi$}{g}}\label{subsec: sec: proof of Thm: rel-nondeg, subsec 2}

Let~$T\geq 3$ and let
\begin{equation}
	0\leq \tilde{T}\leq \tau_1<\tau_1+T.
\end{equation}

In view of general density arguments, in the present Section, it suffices to make the assumption that the solution~$\psi$ of the Klein--Gordon equation~\eqref{eq: kleingordon} 
\begin{equation}\label{eq: subsec: sec: proof of Thm: rel-nondeg, subsec 2, eq 1}
\textit{arises~from~smooth~initial~data~on~the~hypersurface~$\{t^\star=\tau_1-\tilde{T}\}$}.
\end{equation}
This assumption will allow us to use the fixed frequency energy estimates of Theorem~\ref{thm: subsec: sec: proof of Theorem 2, subsec 4.1, thm 1}~(which recall refer to smooth solutions of Carter's radial ode satisfying outgoing boundary conditions), in view of Carter's Proposition~\ref{prop: Carters separation, radial part}.

\subsection{The constants and bad behavior in the limit~\texorpdfstring{$a\rightarrow 0$}{g}}\label{sec: subsec: sec: proof of Thm: rel-nondeg, subsec 2, sec 0}

We use the constants
\begin{equation}
	b(a,M,l,\mu_{\textit{KG}})>0,\qquad B(a,M,l,\mu_{\textit{KG}})>0
\end{equation}
with the algebra of constants
\begin{equation}
	b+b=b,\qquad b\cdot b=b,\qquad B+B=B,\qquad B\cdot B=B.
\end{equation}

Starting from Section~\ref{subsec: sec: proof of Theorem 2, subsec 5.5} the constant~$b$ will degenerate in the limit~$a\rightarrow 0$, while the constant~$B$ will blow up in the same limit. Note, however, that the constants~$b,B$ are finite and non zero for the Schwarzschild--de~Sitter case~$a=0$. This bad behavior of our constants stems from the commutation estimates and specifically from the bahavior of the symbol~$\widetilde{g}_2$, see Lemma~\ref{lem: proof thm 3 exp decay, lem 0}.

\subsection{Frequency localization}\label{subsec: sec: proof of Thm: rel-nondeg, subsec 2.1}

Let~$\psi$ be the smooth solution of the Klein--Gordon equation~\eqref{eq: kleingordon} in the domain~$D(\tau_1-\tilde{T},\tau_2)$ where~$\tau_2=\tau_1+T^2$, arising from the initial data~\eqref{eq: subsec: sec: proof of Thm: rel-nondeg, subsec 2, eq 1}. We define 
\begin{equation}\label{eq: sec: proof of Thm: rel-nondeg, eq 0}
\tilde{\psi}_{\tau_1,\tau_1+T^2}=\eta^{(T)}_{\tau_1}\chi^2_{\tau_1,\tau_1+T^2}\psi\: \dot{=}\:\eta\chi_+^2\psi,
\end{equation}
where for the smooth cut-offs~$\eta^{(T)}_{\tau_1},\chi_{\tau_1,\tau_1+T^2}$ see Section~\ref{subsec: sec: carter separation, subsec 1}.

In view of Proposition~\ref{prop: Carters separation, radial part} we frequency localize the function~\eqref{eq: sec: proof of Thm: rel-nondeg, eq 0} and obtain 
\begin{equation}\label{eq: sec: proof of Thm: rel-nondeg, eq 1}
	u^{(a\omega)}_{m\ell}=\sqrt{r^2+a^2}\Psi^{(a\omega)}_{m\ell},\qquad \Psi=\sqrt{r^2+a^2}\tilde{\psi}_{\tau_1,\tau_1+T^2}
\end{equation}
and note that the radial function~$u$ solves the following radial ode
\begin{equation}\label{eq: sec: proof of Thm: rel-nondeg, eq 1.1}
	u^{\prime\prime}+\left(\omega^2-V\right)u=H^{(a\omega)}_{m\ell}
\end{equation}
for all~$\omega,m,\ell$, with 
\begin{equation}\label{eq: sec: proof of Thm: rel-nondeg, eq 1.2}
	H^{(a\omega)}_{m\ell}=\frac{\Delta}{(r^2+a^2)^{3/2}} \left(\rho^2 F\right)^{(a\omega)}_{m\ell},\qquad F\:\dot{=}\:(\Box\eta)\cdot \left(\chi_+^2\cdot\psi\right) +2\nabla^a\eta\nabla_a \left(\chi_+^2\cdot\psi\right).
\end{equation}

\subsection{Asymptotic behavior}\label{subsec: sec: proof of Thm: rel-nondeg, subsec 3}

We note from Carter's Proposition~\ref{prop: Carters separation, radial part} that for~$u$ as in~\eqref{eq: sec: proof of Thm: rel-nondeg, eq 1} then for all~$\omega\in\mathbb{R},m\in\mathbb{Z},\ell\in\mathbb{Z}_{\geq |m|}$ the following outgoing boundary conditions hold
\begin{equation}\label{eq: subsec: sec: proof of Thm: rel-nondeg, subsec 3, eq 2}
	\begin{aligned}
		&   u^\prime=-i\left(\omega-\frac{am\Xi}{r_+^2+a^2}\right) u,\qquad u^\prime=i\left(\omega-\frac{am\Xi}{\bar{r}_+^2+a^2}\right) u,
	\end{aligned}
\end{equation}
in a limiting sense, at~$r^\star=-\infty$,~$r^\star=\infty$ respectively.

\subsection{The function~$v$ and its asymptotic behavior}\label{subsec: sec: proof of Thm: rel-nondeg, subsec 4}

Let $u^{(a\omega)}_{m\ell}$ be as in Section~\ref{subsec: sec: proof of Thm: rel-nondeg, subsec 2.1}. Then, in accordance to Definition~\ref{def: subsec: sec: proof of Theorem 2, subsec 1.1, def 1} of the previous Section, in the present Section we define
\begin{equation}
	v=\left(g_1\partial_{r^\star}+ig_2(\omega,m,r)\right)u,
\end{equation}
where for~$g_1,g_2$ see Definition~\ref{def: subsec: sec: G, subsec 1, def 1}.

In view of the assumptions of Section~\ref{subsec: sec: proof of Thm: rel-nondeg, subsec 2} and the asymptotic behavior of~$u$, see Section~\ref{subsec: sec: proof of Thm: rel-nondeg, subsec 3}, we use the definition of~$g_1,g_2$, see Definition~\ref{def: subsec: sec: G, subsec 1, def 1}, to conclude
\begin{equation}\label{eq: subsec: sec: proof of Thm: rel-nondeg, subsec 4, eq 2}
	\lim_{r^\star \rightarrow \pm \infty} v =0.
\end{equation}

\subsection{The control of the derivatives of \texorpdfstring{$\mathcal{G}\left(\eta\chi_+^2\psi\right)$}{psi}}\label{subsec: sec: proof of Theorem 2, subsec 5}

The following is an energy estimate for~$\mathcal{G}(\eta\chi_+^2\psi)$.

\begin{proposition}\label{prop: energy estimate for v, 1}

Let the assumptions of Theorem~\ref{main theorem, relat. non-deg} hold. Moreover, assume that~$\psi$ arises from smooth initial data, in accordance to the assumptions of Section~\ref{subsec: sec: proof of Thm: rel-nondeg, subsec 2}.

Then, we obtain the following estimate 
\begin{equation}\label{eq: prop: energy estimate for v, 1, eq 1}
    \begin{aligned}
        &   \int\int_{\mathcal{M}}dg\left( \frac{\Delta}{r(r^2+a^2)}|Z^\star\mathcal{G}(\eta\chi_+^2\psi)|^2+\frac{r^2+a^2}{r\Delta}|W\mathcal{G}(\eta\chi_+^2\psi)|^2+\frac{1}{r}|\slashed{\nabla}\mathcal{G}(\eta\chi_+^2\psi)|^2+1_{D(\tau_1+T,\tau_2)}J^n_\mu[\psi] n^\mu\right)\\
        &   \quad \leq \int_\mathbb{R}\int_{\mathbb{S}^2}\int_{r_+}^{\bar{r}_+}\frac{r^2+a^2}{\Delta}\Big( \Delta^{-1}\left|\frac{\Delta}{(r^2+a^2)^{3/2}}\rho^2 F\right|^2+r\left|\mathcal{G}\left(\frac{\Delta}{(r^2+a^2)^{3/2}}\rho^2 F\right)\right|^2 \Big)\\
        &\qquad +B\int_{\{t^\star=\tau_1\}} J_\mu^n[\psi]n^\mu+|\psi|^2,
    \end{aligned}
\end{equation}
for any~$\tau_1\geq 0$, where for the error term~$F$ see~\eqref{eq: sec: proof of Thm: rel-nondeg, eq 1.2}. For the cut-offs~$\eta,\chi_+$ see~\eqref{eq: sec: proof of Thm: rel-nondeg, eq 0}. For the regular vector field~$Z^\star$ see Section~\ref{subsec: boldsymbol partial r}. Also recall~$W=\partial_t+\frac{a\Xi}{r^2+a^2}\partial_{\varphi}$.   
\end{proposition}

\begin{proof}

In view of the assumptions on~$\psi$ then we recall from Carter's Proposition~\ref{prop: Carters separation, radial part} that for any~$\omega\in \mathbb{R},m\in\mathbb{Z},\ell\in \mathbb{Z}_{\geq |m|}$ the frequency localized radial function
\begin{equation}
	u^{(a\omega)}_{m\ell}(r),
\end{equation}
is smooth and satisfies the outgoing boundary conditions~\eqref{eq: subsec: sec: proof of Thm: rel-nondeg, subsec 3, eq 2}.

Therefore, we can use fixed frequency Theorem~\ref{thm: subsec: sec: proof of Theorem 2, subsec 4.1, thm 1} to obtain 
	\begin{equation}\label{eq: proof prop: energy estimate for v, 1, eq 1}
	\begin{aligned}
		&   \int_{\mathbb{R}}d\omega\sum_{m\ell}\Bigg(\int_{r_+}^{r_++\epsilon}(r-r_+)^{-1}\left|v^\prime+i(\omega-\omega_+m) v\right|^2dr+\int_{\bar{r}_+-\epsilon}^{\bar{r}_+}(\bar{r}_+-r)^{-1}\left|v^\prime-i(\omega-\bar{\omega}_+m) v\right|^2dr\\
		&   \qquad\qquad\qquad + \int_{r_+}^{\bar{r}_+}\left( \frac{\tilde{\lambda}}{r^3} |v|^2+\frac{r^2+a^2}{r\Delta}\left(\omega-\frac{am\Xi}{r^2+a^2}\right)^2|v|^2+\frac{1}{r}|v^\prime|^2\right)dr\Bigg)  \\
		&	\qquad\qquad \leq B \int_{\mathbb{R}}\sum_{m\ell} \Bigg( \int_{r_+}^{r_++\epsilon}\frac{1}{\Delta^2}|u^\prime+i(\omega-\omega_+m)u|^2dr +\int_{\bar{r}_+-\epsilon}^{\bar{r}_+}\frac{1}{\Delta^2}|u^\prime-(\omega-\bar{\omega}_+m)u|^2 dr\\
		&	\qquad\qquad\qquad\qquad\qquad+ \int_{r_+}^{\bar{r}_+} \left(1_{\{|m|>0\}}m^2|u|^2+ \mu^2_{\textit{KG}}|u|^2+ |u^\prime|^2+\left(1-\frac{r_{\textit{trap}}(\omega,m,\ell)}{r}\right)^2(\omega+\tilde{\lambda})|u|^2\right)dr\Bigg)\\		
		&	\qquad\qquad\qquad +B\int_{\mathbb{R}}d\omega \sum_{m\ell}\left(\omega^2+m^2\right)|u|^2(-\infty)+B\int_{\mathbb{R}}d\omega \sum_{m\ell}\left(\omega^2+m^2\right)|u|^2(+\infty)\\
		&	\qquad\qquad\qquad  + B \int_{\mathbb{R}}d\omega\sum_{m\ell} \int_{\mathbb{R}} \left( \Delta^{-1}\left|H\right|^2+r\left|\left(g_1\frac{d}{dr^\star}+ig_2\right)H\right|^2\right)dr^\star ,
	\end{aligned}
\end{equation}
for a sufficienty small~$\epsilon>0$, where we integrated~$\int_{\mathbb{R}}d\omega \sum_{m\ell}$. For the cut-off term~$H$ see~\eqref{eq: sec: proof of Thm: rel-nondeg, eq 1.2}.

We use the Parseval identities of Section~\ref{subsec: sec: morawetz estimate, subsec 6.1} to rewrite the top order terms of the LHS of~\eqref{eq: proof prop: energy estimate for v, 1, eq 1} as follows 
\begin{equation}\label{eq: proof prop: energy estimate for v, 1, eq 2}
	\begin{aligned}
		&   \int_{\mathbb{R}}d\omega\sum_{m\ell}\Bigg(\int_{r_+}^{r_++\epsilon}(r-r_+)^{-1}\left|v^\prime+i(\omega-\omega_+m) v\right|^2dr+\int_{\bar{r}_+-\epsilon}^{\bar{r}_+}(\bar{r}_+-r)^{-1}\left|v^\prime-i(\omega-\bar{\omega}_+m) v\right|^2dr\\
		&   \qquad\qquad\qquad + \int_{r_+}^{\bar{r}_+}\left( \frac{\tilde{\lambda}}{r^3} |v|^2+\frac{r^2+a^2}{r\Delta}\left(\omega-\frac{am\Xi}{r^2+a^2}\right)^2|v|^2+\frac{1}{r}|v^\prime|^2\right)dr\Bigg) \\
		&= \int_{\mathbb{R}}\int_{\mathbb{S}^2}\Bigg( \int_{r_+}^{r_++\epsilon}(r-r_+)^{-1}|\big(\partial_{r^\star}-(\partial_t+\omega_+ \partial_\varphi)\big)\mathcal{G}(\sqrt{r^2+a^2}\eta\chi_+^2\psi)|^2\\
		&	\qquad\qquad\qquad +\int_{\bar{r}_+-\epsilon}^{\bar{r}_+}(r-\bar{r}_+)^{-1}|\big(\partial_{r^\star}+(\partial_t+\bar{\omega}_+ \partial_\varphi)\big)\mathcal{G}(\sqrt{r^2+a^2}\eta\chi_+^2\psi)|^2\Bigg)\\
		&	+\int_{\mathbb{R}}\int_{\mathbb{S}^2}\int_{r_+}^{\bar{r}_+}\Bigg( \frac{1}{r^3}\left|^{d\sigma}\nabla\mathcal{G}(\sqrt{r^2+a^2}\eta\chi_+^2\psi)\right|^2+ \frac{r^2+a^2}{r\Delta}\left|\left(\partial_t-\frac{a\Xi}{r^2+a^2}\partial_\varphi\right)\mathcal{G}(\sqrt{r^2+a^2}\eta\chi_+^2\psi)\right|^2\\
		&	\qquad\qquad\qquad\qquad\qquad+\frac{1}{r}\left|\partial_{r^\star}\mathcal{G}(\sqrt{r^2+a^2}\eta\chi_+^2\psi)\right|^2 \Bigg).
	\end{aligned}
\end{equation}

Moreover, in view of the Parseval identities of Section~\ref{subsec: sec: carter separation, radial, subsec 4} we rewrite the horizon terms of the RHS of~\eqref{eq: proof prop: energy estimate for v, 1, eq 1} as follows
\begin{equation}\label{eq: proof prop: energy estimate for v, 1, eq 3}
	\begin{aligned}
		& \int_{\mathbb{R}}d\omega \sum_{m\ell}\left(\omega^2+m^2\right)|u|^2(-\infty)+\int_{\mathbb{R}}d\omega \sum_{m\ell}\left(\omega^2+m^2\right)|u|^2(+\infty)\\
		&	= \lim_{r^\star\rightarrow -\infty}\int_\mathbb{R}\int_{\mathbb{S}^2}\left(|\partial_t\sqrt{r^2+a^2}(\eta\chi_+^2\psi)|^2+|\partial_\varphi\sqrt{r^2+a^2}(\eta\chi_+^2\psi)|^2\right)\\
		&	\qquad+ \lim_{r^\star\rightarrow +\infty}\int_\mathbb{R}\int_{\mathbb{S}^2}\left(|\partial_t\sqrt{r^2+a^2}(\eta\chi_+^2\psi)|^2+|\partial_\varphi\sqrt{r^2+a^2}(\eta\chi_+^2\psi)|^2\right)
	\end{aligned}
\end{equation}
which note can be bounded by
\begin{equation}\label{eq: proof prop: energy estimate for v, 1, eq 4}
\int_{\mathcal{H}^+\cap D(\tau_1,\infty)}J^n[\eta\chi_+^2\psi]n^\mu +\int_{\bar{\mathcal{H}}^+\cap D(\tau_1,\infty)}J^n[\eta\chi_+^2\psi]n^\mu. 
\end{equation}

Therefore, in view of~\eqref{eq: proof prop: energy estimate for v, 1, eq 3},~\eqref{eq: proof prop: energy estimate for v, 1, eq 4}, we note that the lower order terms on the RHS of~\eqref{eq: proof prop: energy estimate for v, 1, eq 2} can be controlled from the Morawetz estimate of Theorem~\ref{main theorem 1}. Therefore, in view of the Parseval identity~\eqref{eq: proof prop: energy estimate for v, 1, eq 2} we obtain from~\eqref{eq: proof prop: energy estimate for v, 1, eq 1} the following 
\begin{equation}\label{eq: proof prop: energy estimate for v, 1, eq 5}
\begin{aligned}
&   \int\int_{\mathcal{M}}dg\left( \frac{\Delta}{r(r^2+a^2)}|Z^\star\mathcal{G}(\eta\chi_+^2\psi)|^2+\frac{r^2+a^2}{r\Delta}|W\mathcal{G}(\eta\chi_+^2\psi)|^2+\frac{1}{r}|\slashed{\nabla}\mathcal{G}(\eta\chi_+^2\psi)|^2+J^n_\mu[\eta\chi_+^2\psi] n^\mu \right)\\
&   \quad \leq B \int_{\mathbb{R}}dr^\star \int_{\mathbb{R}}d\omega\sum_{m\ell}\left(\Delta^{-1}|H|^2 +r\left|(g_1\frac{d}{d r^\star}+ig_2)H\right|^2\right) +B\int_{\{t^\star=\tau_1\}} J_\mu^n[\psi]n^\mu+|\psi|^2,
\end{aligned}
\end{equation}
where we used that~$\tilde{\psi}_{\tau_1,\tau_1+T^2}=\eta\chi_+^2\psi$, see~\eqref{eq: sec: proof of Thm: rel-nondeg, eq 0}.

Furthermore, it is easy to conclude the following 
\begin{equation}
	\begin{aligned}
		\int\int_\mathcal{M} J^n_\mu[\eta\chi_+^2\psi]n^\mu &	\geq b \int\int_\mathcal{M}\left(\left|\partial_{t^\star}(\eta\chi_+^2\psi)\right|^2+|\slashed{\nabla}(\eta\chi_+^2\psi)|^2\right)= \int\int_{\mathcal{M}}\left(\left|\partial_{t^\star}(\eta\chi_+^2)\psi+\eta\chi_+^2\partial_{t^\star}\psi\right|^2+|\eta\chi_+^2\slashed{\nabla}(\psi)|^2\right)
	\end{aligned}
\end{equation}
and therefore, by using that~$|\partial_{t^\star}(\eta\chi_+^2)|\lesssim T^{-1}$ in conjunction with the boundedness estimate of Theorem~\ref{main theorem 1} we conclude that 
\begin{equation}
	\int\int_{\mathcal{M}}(\eta\chi_+^2)^2\left(|\partial_{t^\star}\psi|^2+|\slashed{\nabla}\psi|^2\right)\leq B \int_{t^\star=\tau_1}J^n_\mu[\psi]n^\mu+|\psi|^2+ B\int\int_{\mathcal{M}} J^n_\mu[\eta\chi_+^2\psi]n^\mu
\end{equation}
which, in view of the support of the functions~$\eta,\chi_+$ and by using the boundedness estimate of Theorem~\ref{main theorem 1} one more time, we obtain
\begin{equation}\label{eq: proof prop: energy estimate for v, 1, eq 6}
	\int\int_{D(\tau_1+T,\tau_2)}\left(|\partial_{t^\star}\psi|^2+|\slashed{\nabla}\psi|^2\right)\leq B \int_{\{t^\star=\tau_1\}}J^n_\mu[\psi]n^\mu+|\psi|^2+ B\int\int_{\mathcal{M}} J^n_\mu[\eta\chi_+^2\psi]n^\mu.
\end{equation}

Therefore, by using the above~\eqref{eq: proof prop: energy estimate for v, 1, eq 6} and~\eqref{eq: proof prop: energy estimate for v, 1, eq 5}, in conjunction with the Morawetz estimate of Theorem~\ref{main theorem 1}~(in order to obtain the derivatives~$(Z^\star\psi)^2$ in the desired domain) we obtain the following
\begin{equation}\label{eq: proof prop: energy estimate for v, 1, eq 7}
	\begin{aligned}
		&   \int\int_{\mathcal{M}}dg\left( \frac{\Delta}{r(r^2+a^2)}|Z^\star\mathcal{G}(\eta\chi_+^2\psi)|^2+\frac{r^2+a^2}{r\Delta}|W\mathcal{G}(\eta\chi_+^2\psi)|^2+\frac{1}{r}|\slashed{\nabla}\mathcal{G}(\eta\chi_+^2\psi)|^2+1_{D(\tau_1+T,\tau_2)}J^n_\mu[\psi] n^\mu \right)\\
		&   \quad \leq B \int_{\mathbb{R}}dr^\star \int_{\mathbb{R}}d\omega\sum_{m\ell}\left(\Delta^{-1}|H|^2 +r\left|(g_1\frac{d}{d r^\star}+ig_2)H\right|^2\right) +B\int_{\{t^\star=\tau_1\}} J_\mu^n[\psi]n^\mu+|\psi|^2.
	\end{aligned}
\end{equation}

Finally, in view of the Parseval identities of Sections~\ref{subsec: sec: carter separation, radial, subsec 4},~\ref{subsec: sec: morawetz estimate, subsec 6.1} we have the following identity
\begin{equation}
	\begin{aligned}
		&	\int_{\mathbb{R}}dr^\star\int_{\mathbb{R}}d\omega\sum_{m\ell}\left( \Delta^{-1}|H|^2+r\left|(g_1\frac{d}{d r^\star}+ig_2)H\right|^2\right)\\
		&		\qquad=\int_\mathbb{R}\int_{\mathbb{S}^2}\int_{r_+}^{\bar{r}_+}\frac{r^2+a^2}{\Delta}\Big( \Delta^{-1}\left|\frac{\Delta}{(r^2+a^2)^{3/2}}\rho^2 F\right|^2+r\left|\mathcal{G}\left(\frac{\Delta}{(r^2+a^2)^{3/2}}\rho^2 F\right)\right|^2 \Big),
	\end{aligned}
\end{equation}
where~$\rho^2=r^2+a^2\cos^2\theta$ and for~$F$ see~\eqref{eq: sec: proof of Thm: rel-nondeg, eq 1.2}. We conclude. 
\end{proof}

\subsection{Bounding the inhomogeneities of Proposition~\ref{prop: energy estimate for v, 1}}\label{subsec: sec: proof of Theorem 2, subsec 5.5}

We need the following preparatory proposition whose purpose is to bound the inhomogeneous~$H$ terms on the RHS of Proposition~\ref{prop: energy estimate for v, 1}, where for~$H$ see~\eqref{eq: sec: proof of Thm: rel-nondeg, eq 1.2}. We will appeal to the pseudodifferential estimates of Lemma~\ref{lem: proof thm 3 exp decay, lem 1}.

Starting from the estimate~\eqref{eq: proof thm 3 exp decay, eq 3.2.1}  the constants~$b,B$ will respectively degenerate and blow up in the limit~$a\rightarrow 0$, as discussed in Section~\ref{sec: subsec: sec: proof of Thm: rel-nondeg, subsec 2, sec 0}. Specifically, the constant~$b$ of estimate~\eqref{eq: lem: subsec: sec: proof of Theorem 2, subsec 5.5, lem 1, eq 1} of Proposition~\ref{prop: subsec: sec: proof of Theorem 2, subsec 5.5, prop 1} degenerates in the limit~$a\rightarrow 0$ but, as discussed in Section~\ref{sec: subsec: sec: proof of Thm: rel-nondeg, subsec 2, sec 0}, the constants do not degenerate exactly on the Schwarzschild--de~Sitter case~$a=0$. This is related to the pseudodifferential commutations we need to employ, specifically see the properties of the symbol~$\widetilde{g}_2$ in Lemma~\ref{lem: proof thm 3 exp decay, lem 0}.

\begin{proposition}\label{prop: subsec: sec: proof of Theorem 2, subsec 5.5, prop 1}
	
	Let the assumptions of Theorem~\ref{main theorem, relat. non-deg} hold. Moreover, assume that~$\psi$ is in accordance to the assumptions of Section~\ref{subsec: sec: proof of Thm: rel-nondeg, subsec 2}, specifically~$\psi$ arises from smooth initial data. Then, we obtain the following estimate 
			\begin{equation}\label{eq: lem: subsec: sec: proof of Theorem 2, subsec 5.5, lem 1, eq 1}
				\begin{aligned}
					& b\int\int_{\mathcal{M}}\Bigg( \frac{1}{r}\frac{\Delta}{(r^2+a^2)}|Z^\star\mathcal{G}(\eta\chi_+^2\psi)|^2+\frac{1}{r}\frac{r^2+a^2}{\Delta}|W\mathcal{G}(\eta\chi_+^2\psi)|^2+\frac{1}{r}|\slashed{\nabla}\mathcal{G}(\eta\chi_+^2\psi)|^2+1_{D(\tau_1+T,\tau_2)}J^n_\mu[\psi]n^\mu\Bigg)\\
					\quad &\leq \int_{\{t^\star=\tau_1\}}  \left(J^n_\mu[\psi]n^\mu+|\psi|^2\right)+\frac{1}{T^2}\int\int_{D(\tau_1,\tau_1+T)} J_\mu^{W}[\mathcal{G}(\chi_+^2\psi)]n^\mu,
				\end{aligned}    
			\end{equation}
			for any~$T\geq 3$ and for any~$\tilde{T}< \tau_1<\tau_1+T< \tau_2$. For the cut-offs~$\eta,\chi_+$ see~\eqref{eq: sec: proof of Thm: rel-nondeg, eq 0}. For the regular vector field~$Z^\star$ see Section~\ref{subsec: boldsymbol partial r}, where~$W=\partial_t+\frac{a\Xi}{r^2+a^2}\partial_{\varphi}$.
\end{proposition}

\begin{proof}
	Since the present Proposition satisfies the same assumptions as Proposition~\ref{prop: energy estimate for v, 1}, we immediately have the result~\eqref{eq: prop: energy estimate for v, 1, eq 1} of Proposition~\ref{prop: energy estimate for v, 1}, which we recall here:
	\begin{equation}\label{eq: proof thm 3 exp decay, eq 0}
		\begin{aligned}
			&   b\int\int_{\mathcal{M}}\left( \frac{1}{r}\frac{\Delta}{(r^2+a^2)}|Z^\star\mathcal{G}(\eta\chi_+^2\psi)|^2+\frac{1}{r}\frac{r^2+a^2}{\Delta}|W\mathcal{G}(\eta\chi_+^2\psi)|^2+\frac{1}{r}|\slashed{\nabla}\mathcal{G}(\eta\chi_+^2\psi)|^2+1_{D(\tau_1+T,\tau_2)}J^n_\mu[\psi] n^\mu \right)\\
			\quad &\leq \int_\mathbb{R}\int_{\mathbb{S}^2}\int_{r_+}^{\bar{r}_+}\frac{r^2+a^2}{\Delta}\Big( \Delta^{-1}\left|\frac{\Delta}{(r^2+a^2)^{3/2}}\rho^2 F\right|^2+r\left|\mathcal{G}\left(\frac{\Delta}{(r^2+a^2)^{3/2}}\rho^2 F\right)\right|^2 \Big)\\ &\qquad+\int_{\{t^\star=\tau_1\}} \left(J^n_\mu[\psi]n^\mu+|\psi|^2\right),\\
		\end{aligned}    
	\end{equation}
	where for the error terms~$H$ see~\eqref{eq: sec: proof of Thm: rel-nondeg, eq 1.2}. In what follows we seek to appropriately bound the~$H$ error terms on the RHS of~\eqref{eq: proof thm 3 exp decay, eq 0}.

	\begin{center}
		\textbf{An auxilliary estimate}
	\end{center}

	We will only use the estimate~\eqref{eq: proof thm 3 exp decay, eq 0.1} of the following Lemma to prove inequalities~\eqref{eq: proof thm 3 exp decay, eq 3.2.4-2},~\eqref{eq: proof thm 3 exp decay, eq 3.4-4}.

\begin{lemma}\label{lem: subsec: sec: proof of Theorem 2, subsec 5.5, lem 1}
	Let the assumptions of Theorem~\ref{main theorem, relat. non-deg} hold. Moreover, assume that~$\psi$ is in accordance to the assumptions of Section~\ref{subsec: sec: proof of Thm: rel-nondeg, subsec 2}, specifically~$\psi$ arises from smooth initial data. Then, we obtain the following estimate
	\begin{equation}\label{eq: proof thm 3 exp decay, eq 0.1}
		\begin{aligned}
			\int\int_{\mathcal{M}} J^n_\mu[\eta\chi_+^2\psi]n^\mu &\leq \frac{B}{\epsilon_p}\int_{\{t^\star=\tau_1\}} J^n_\mu[\psi]n^\mu+|\psi|^2 +B\epsilon_p\int_{\mathbb{R}}\int_{\mathbb{S}^2}\int_{r_+}^{\bar{r}_+}\frac{1}{g_1}\frac{1}{r^2}\frac{\Delta}{r^2+a^2}\left|\partial_t\mathcal{G}(\sqrt{r^2+a^2}\eta\chi_+^2\psi)\right|^2\\
			&	\qquad +B\epsilon_p \int_{\mathbb{R}}\int_{\mathbb{S}^2}\int_{r_+}^{\bar{r}_+}\frac{1}{g_1}\frac{1}{r^2}\frac{\Delta}{r^2+a^2}\left|a\Xi\partial_\varphi\mathcal{G}(\sqrt{r^2+a^2}\eta\chi_+^2\psi)\right|^2.
		\end{aligned}
	\end{equation}
	for any sufficiently small~$\epsilon_p>0$. 
\end{lemma}

\begin{proof}
The assumptions of the present Lemma allow us access to the results of Theorem~\ref{main theorem 1}, specifically the Morawetz and boundedness estimates. Next, we use the Poincare type Lemma~\ref{lem: control of the law order derivatives with v} in order to `cover' the lower order degeneratation at~$r_{trap}(\omega,m,\ell)$ at the LHS of the Morawetz estimate, at the expence of introducing an appropriate top order term on the RHS. Specifically, we obtain the following estimate 
	\begin{equation}\label{eq: proof lem: subsec: sec: proof of Theorem 2, subsec 5.5, lem 1, eq 1}
		\begin{aligned}
			&	\int_{\mathbb{R}}d\omega\sum_{m,
				\ell} \Bigg( \int_{r_+}^{r_++\epsilon}\frac{1}{\Delta^2}|u^\prime+i(\omega-\omega_+m)u|^2dr +\int_{\bar{r}_+-\epsilon}^{\bar{r}_+}\frac{1}{\Delta^2}|u^\prime-(\omega-\bar{\omega}_+m)u|^2dr \\
			&	\qquad\qquad\qquad+ \int_{r_+}^{\bar{r}_+} \left(1_{\{|m|>0\}}|u|^2+ \mu^2_{\textit{KG}}|u|^2+ |u^\prime|^2+(\omega^2+\lambda^{(a\omega)}_{m\ell})|u|^2\right)dr\Bigg) \\
			&	\leq \frac{B}{\epsilon_p}\int_\mathbb{R} d\omega\sum_{m,\ell} \int_{r_+}^{\bar{r}_+}\left(1-\frac{r_{trap}}{r}\right)^2|u|^2 +B\epsilon_p\int_{\mathbb{R}}d\omega\sum_{m\ell} \int_{r_+}^{\bar{r}_+} dr  \frac{1}{g_1}\frac{1}{r^2}\frac{\Delta}{r^2+a^2}(\omega^2+(am\Xi)^2)|v|^2.
		\end{aligned}
	\end{equation}
	Recall from Section~\ref{subsec: sec: proof of Thm: rel-nondeg, subsec 4} that~$v=(g_1\partial_{r^\star}+ig_2)u$, and~$u$ is the frequency localization of~$\eta\chi_+^2\psi$, see Section~\ref{subsec: sec: proof of Thm: rel-nondeg, subsec 2.1}. 
	
	Therefore, by using the Morawetz estimate of Theorem~\ref{main theorem 1} one last time to bound the first term on the RHS of the above~\eqref{eq: proof lem: subsec: sec: proof of Theorem 2, subsec 5.5, lem 1, eq 1} we obtain
		\begin{equation}\label{eq: proof lem: subsec: sec: proof of Theorem 2, subsec 5.5, lem 1, eq 2}
		\begin{aligned}
			&	\int_{\mathbb{R}}d\omega\sum_{m,
				\ell} \Bigg( \int_{r_+}^{r_++\epsilon}\frac{1}{\Delta^2}|u^\prime+i(\omega-\omega_+m)u|^2dr +\int_{\bar{r}_+-\epsilon}^{\bar{r}_+}\frac{1}{\Delta^2}|u^\prime-(\omega-\bar{\omega}_+m)u|^2dr \\
			&	\qquad\qquad\qquad+ \int_{r_+}^{\bar{r}_+} \left(1_{\{|m|>0\}}|u|^2+ \mu^2_{\textit{KG}}|u|^2+ |u^\prime|^2+(\omega^2+\lambda^{(a\omega)}_{m\ell})|u|^2\right)dr\Bigg) \\
			&	\leq \frac{B}{\epsilon_p}\int_{\{t^\star=\tau_1\}}J^n_\mu[\psi]n^\mu+|\psi|^2 +B\epsilon_p\int_{\mathbb{R}}d\omega\sum_{m\ell} \int_{r_+}^{\bar{r}_+} dr  \frac{1}{g_1}\frac{1}{r^2}\frac{\Delta}{r^2+a^2}(\omega^2+(am\Xi)^2)|v|^2.
		\end{aligned}
	\end{equation}

	Now, we use the Parseval identities of Section~\ref{subsec: sec: morawetz estimate, subsec 6.1} to obtain 
	\begin{equation}\label{eq: proof lem: subsec: sec: proof of Theorem 2, subsec 5.5, lem 1, eq 3}
		\begin{aligned}
			\int_{\mathbb{R}}d\omega\sum_{m\ell} \int_{r_+}^{\bar{r}_+} dr  \frac{1}{g_1}\frac{1}{r^2}\frac{\Delta}{r^2+a^2}(\omega^2+(am\Xi)^2)|v|^2	&=\int_{\mathbb{R}}\int_{\mathbb{S}^2}\int_{r_+}^{\bar{r}_+}\frac{1}{g_1}\frac{1}{r^2}\frac{\Delta}{r^2+a^2}\left|\partial_t\mathcal{G}(\sqrt{r^2+a^2}\eta\chi_+^2\psi)\right|^2\\
			&	\qquad +\int_{\mathbb{R}}\int_{\mathbb{S}^2}\int_{r_+}^{\bar{r}_+}\frac{1}{g_1}\frac{1}{r^2}\frac{\Delta}{r^2+a^2}\left|a\Xi\partial_\varphi\mathcal{G}(\sqrt{r^2+a^2}\eta\chi_+^2\psi)\right|^2,
		\end{aligned}
	\end{equation}
	and moreover we use the Parseval identities of Section~\ref{subsec: sec: carter separation, radial, subsec 4} to write the LHS of~\eqref{eq: proof lem: subsec: sec: proof of Theorem 2, subsec 5.5, lem 1, eq 2} in physical space and obtain the desired estimate. 
\end{proof}

	\begin{center}
		\textbf{Upper bound for the bulk of the RHS of}~\eqref{eq: proof thm 3 exp decay, eq 0} 
	\end{center}
	
	To bound the bulk term on the RHS of~\eqref{eq: proof thm 3 exp decay, eq 0}  we will use the pseudodifferential commutator estimates of Lemma~\ref{lem: proof thm 3 exp decay, lem 1} multiple times.

	By using the operator~$\widetilde{\mathcal{G}}$, see Definition~\ref{def: proof thm 3 exp decay, def 1} and the coarea formula we estimate the bulk of the right hand side of~\eqref{eq: proof thm 3 exp decay, eq 0} as follows 
	\begin{equation}\label{eq: proof thm 3 exp decay, eq 3.2}
		\begin{aligned}
			\boxed{\text{Bulk on RHS of }\eqref{eq: proof thm 3 exp decay, eq 0}} &	\leq B \int\int_{\mathcal{M}} r\left|\mathcal{G}\left(\frac{\Delta\rho^2}{(r^2+a^2)^{3/2}} (\Box\eta) \chi_+^2\psi\right)\right|^2+r\left|\mathcal{G}\left(\frac{\Delta\rho^2}{(r^2+a^2)^{3/2}} \nabla^a\eta\nabla_a \left(\chi_+^2\psi\right)\right)\right|^2\\
			&	\qquad\qquad\qquad\qquad\quad +\Delta^{-2}\left|\frac{\Delta\rho^2}{(r^2+a^2)^{3/2}} \left((\Box\eta) \chi_+^2\psi +2\nabla^a\eta\nabla_a (\chi_+^2\psi)\right)\right|^2\\
			&	=B	\int\int_{\mathcal{M}} r\left|(\mathcal{G}-\widetilde{\mathcal{G}}+\widetilde{\mathcal{G}})\left(\frac{\Delta\rho^2}{(r^2+a^2)^{3/2}} (\Box\eta) \chi_+^2\psi\right)\right|^2\\
			&	\qquad\qquad\qquad\qquad\quad+r\left|(\mathcal{G}-\widetilde{\mathcal{G}}+\widetilde{\mathcal{G}})\left(\frac{\Delta\rho^2}{(r^2+a^2)^{3/2}} g^{ab}\partial_a\eta\nabla_b (\chi_+^2\psi)\right)\right|^2\\
			&	\quad +B\int\int_{\mathcal{M}}\Delta^{-2}\left|\frac{\Delta\rho^2}{(r^2+a^2)^{3/2}} \left((\Box\eta)\chi_+^2\psi +2\nabla^a\eta\nabla_a(\chi_+^2 \psi)\right)\right|^2.\\
		\end{aligned}
	\end{equation}
	For the cut-offs~$\eta,\chi_+$ see Section~\ref{subsec: sec: carter separation, subsec 1}.

	Now, by using the pseudodifferential estimate~\eqref{eq: thm 3 exp decay, eq 2.1} of Lemma~\ref{lem: proof thm 3 exp decay, lem 1} we obtain from~\eqref{eq: proof thm 3 exp decay, eq 3.2} the following 
	\begin{equation}\label{eq: proof thm 3 exp decay, eq 3.2.1} 
		\begin{aligned}
			&\boxed{\text{Bulk on RHS of }\eqref{eq: proof thm 3 exp decay, eq 0}}\\
			&	\leq B	\int\int_{\mathcal{M}} r\left|(\mathcal{G}-\widetilde{\mathcal{G}})\left(\frac{\Delta\rho^2}{(r^2+a^2)^{3/2}} (\Box\eta) \chi_+^2\psi\right)\right|^2 + r\left|\widetilde{\mathcal{G}}\left(\frac{\Delta\rho^2}{(r^2+a^2)^{3/2}} (\Box\eta) \chi_+^2\psi\right)\right|^2\\
			&	\qquad\qquad\qquad\qquad\quad+r\left|(\mathcal{G}-\widetilde{\mathcal{G}})\left(\frac{\Delta\rho^2}{(r^2+a^2)^{3/2}} g^{ab}\partial_a\eta\nabla_b (\chi_+^2\psi)\right)\right|^2  +r\left|\widetilde{\mathcal{G}}\left(\frac{\Delta\rho^2}{(r^2+a^2)^{3/2}} g^{ab}\partial_a\eta\nabla_b (\chi_+^2\psi)\right)\right|^2\\
			&	\quad +B\int\int_{\mathcal{M}}\Delta^{-2}\left|\frac{\Delta\rho^2}{(r^2+a^2)^{3/2}} \left((\Box\eta)\chi_+^2\psi +2\nabla^a\eta\nabla_a(\chi_+^2 \psi)\right)\right|^2\\
			&	\leq B	\int\int_{\mathcal{M}} r\left|[\widetilde{\mathcal{G}},\frac{\Delta\rho^2}{(r^2+a^2)^{3/2}} \chi_+^2] \Box\eta\psi\right|^2+B\int\int_{\mathcal{M}} r\left|[\widetilde{\mathcal{G}},\frac{\Delta\rho^2}{(r^2+a^2)^{3/2}} g^{ab}\partial_a\eta\nabla_b] (\chi_+^2\psi)\right|^2 \\
			&	\qquad + \int\int_{\mathcal{M}} r\left|\frac{\Delta\rho^2}{(r^2+a^2)^{3/2}} \chi_+^2 \widetilde{\mathcal{G}}(\Box\eta\psi)\right|^2+B\int\int_{\mathcal{M}} r\left|\left(\frac{\Delta\rho^2}{(r^2+a^2)^{3/2}} g^{ab}\partial_a\eta\nabla_b\widetilde{\mathcal{G}}(\chi_+^2\psi)\right)\right|^2\\
			&	\qquad+	B\int\int_{\mathcal{M}} r\left|\left(\frac{\Delta\rho^2}{(r^2+a^2)^{3/2}} (\Box\eta) \chi_+^2\psi\right)\right|^2 +r\left|\left(\frac{\Delta\rho^2}{(r^2+a^2)^{3/2}} g^{ab}\partial_a\eta\nabla_b (\chi_+^2\psi)\right)\right|^2\\
			&	\qquad +B\int\int_{\mathcal{M}}\Delta^{-2}\left|\frac{\Delta\rho^2}{(r^2+a^2)^{3/2}} \left((\Box\eta)\chi_+^2\psi +2\nabla^a\eta\nabla_a(\chi_+^2 \psi)\right)\right|^2\\
			&	\leq  B	\int\int_{\mathcal{M}} r\left|[\widetilde{\mathcal{G}},\frac{\Delta\rho^2}{(r^2+a^2)^{3/2}} \chi_+^2] \Box\eta\psi\right|^2+B\int\int_{\mathcal{M}} r\left|[\widetilde{\mathcal{G}},\frac{\Delta\rho^2}{(r^2+a^2)^{3/2}} g^{ab}\partial_a\eta\nabla_b] (\chi_+^2\psi)\right|^2 \\
			&	\qquad +B \int\int_{\mathcal{M}} r\left|\frac{\Delta\rho^2}{(r^2+a^2)^{3/2}} \Box\eta\cdot \widetilde{\mathcal{G}}(\chi_+^2\psi)\right|^2\\
			&	\qquad +B\int\int_{\mathcal{M}} r\left|\left(\frac{\Delta\rho^2}{(r^2+a^2)^{3/2}} g^{ab}\partial_a\eta\nabla_b(\widetilde{\mathcal{G}}-\mathcal{G}+\mathcal{G})(\chi_+^2\psi)\right)\right|^2\\
			&	\qquad+	B\int\int_{\mathcal{M}} r\left|\left(\frac{\Delta\rho^2}{(r^2+a^2)^{3/2}} (\Box\eta) \chi_+^2\psi\right)\right|^2 +r\left|\left(\frac{\Delta\rho^2}{(r^2+a^2)^{3/2}} g^{ab}\partial_a\eta\nabla_b (\chi_+^2\psi)\right)\right|^2\\
			&	\qquad +B\int\int_{\mathcal{M}}\Delta^{-2}\left|\frac{\Delta\rho^2}{(r^2+a^2)^{3/2}} \left((\Box\eta)\chi_+^2\psi +2\nabla^a\eta\nabla_a(\chi_+^2 \psi)\right)\right|^2.\\
		\end{aligned}
	\end{equation}
	The constants on the RHS of the above estimate blows up in the limit~$a\rightarrow 0$, in view of the estimates of Lemma~\ref{lem: proof thm 3 exp decay, lem 1}. We immediately bound the right hand side of~\eqref{eq: proof thm 3 exp decay, eq 3.2.1} as follows
	\begin{equation}\label{eq: proof thm 3 exp decay, eq 3.2.3} 
		\begin{aligned}
			&	\boxed{\text{Bulk on RHS of }\eqref{eq: proof thm 3 exp decay, eq 0}}\leq \\
			&	\leq B \int\int_{\mathcal{M}} r\left|\left(\frac{\Delta\rho^2}{(r^2+a^2)^{3/2}} g^{ab}\partial_a\eta\nabla_b\mathcal{G}(\chi_+^2\psi)\right)\right|^2\\
			&	\qquad + B \int\int_{\mathcal{M}} r\left|[\widetilde{\mathcal{G}},\frac{\Delta\rho^2}{(r^2+a^2)^{3/2}} g^{ab}\partial_a\eta\nabla_b] (\chi_+^2\psi)\right|^2  + r\left|\left(\frac{\Delta\rho^2}{(r^2+a^2)^{3/2}} g^{ab}\partial_a\eta\nabla_b(\widetilde{\mathcal{G}}-\mathcal{G})(\chi_+^2\psi)\right)\right|^2\\
			& \qquad\qquad\qquad\qquad\qquad + r\left|[\widetilde{\mathcal{G}},\frac{\Delta\rho^2}{(r^2+a^2)^{3/2}} \chi_+^2] \Box\eta\psi\right|^2+ r\left|\frac{\Delta\rho^2}{(r^2+a^2)^{3/2}} \Box\eta\cdot \widetilde{\mathcal{G}}(\chi_+^2\psi)\right|^2\\
			&	\qquad\qquad\qquad\qquad\qquad + r\left|\left(\frac{\Delta\rho^2}{(r^2+a^2)^{3/2}} (\Box\eta) \chi_+^2\psi\right)\right|^2 +r\left|\left(\frac{\Delta\rho^2}{(r^2+a^2)^{3/2}} g^{ab}\partial_a\eta\nabla_b (\chi_+^2\psi)\right)\right|^2\\
			&	\qquad\qquad\qquad\qquad\qquad + \Delta^{-2}\left|\frac{\Delta\rho^2}{(r^2+a^2)^{3/2}} \left((\Box\eta)\chi_+^2\psi +2\nabla^a\eta\nabla_a(\chi_+^2 \psi)\right)\right|^2.\\
		\end{aligned}
	\end{equation}

	Now, by using the support of the functions~$\eta,\chi$ and that
	\begin{equation}\label{eq: proof thm 3 exp decay, eq 3.2.3.0} 
		|Z^\star\eta|+|\partial_t\eta|+|\slashed{\nabla}\eta|\leq  \frac{B}{T},\qquad |\partial_{t^\star}^2\eta|\leq \frac{B}{T^2},\qquad |Z^\star\chi_+|+|\partial_t\chi_+|+|\slashed{\nabla}\chi_+|\leq  \frac{B}{T}
	\end{equation}
	see Section~\ref{subsec: sec: carter separation, subsec 1} and the~$H^1$ boundedness result
	\begin{equation}\label{eq: proof thm 3 exp decay, eq 3.2.3.1} 
		\int_{\{t^\star=\tau_2\}} |\psi|^2 +J^n_\mu[\psi]n^\mu \leq B \int_{\{t^\star=\tau_1\}} |\psi|^2 +J^n_\mu[\psi]n^\mu
	\end{equation}
	of Theorem~\ref{main theorem 1} we \underline{bound the zero order terms} of the right hand side of~\eqref{eq: proof thm 3 exp decay, eq 3.2.3} and obtain the following 
	\begin{equation}\label{eq: proof thm 3 exp decay, eq 3.2.4} 
		\begin{aligned}
			&\boxed{\text{Bulk on RHS of }\eqref{eq: proof thm 3 exp decay, eq 0}}\leq \\
			&	\leq B \int\int_{\mathcal{M}} r\left|\left(\frac{\Delta\rho^2}{(r^2+a^2)^{3/2}} g^{ab}\partial_a\eta\nabla_b\mathcal{G}(\chi_+^2\psi)\right)\right|^2\\
			&	\qquad + B \int\int_{\mathcal{M}} r\left|[\widetilde{\mathcal{G}},\frac{\Delta\rho^2}{(r^2+a^2)^{3/2}} g^{ab}\partial_a\eta\nabla_b] (\chi_+^2\psi)\right|^2  + r\left|\left(\frac{\Delta\rho^2}{(r^2+a^2)^{3/2}} g^{ab}\partial_a\eta\nabla_b(\widetilde{\mathcal{G}}-\mathcal{G})(\chi_+^2\psi)\right)\right|^2\\
			&	\qquad +B \int_{\{t^\star=\tau_1\}}  J^n_\mu[\psi]n^\mu +|\psi|^2\\
			&	\leq \frac{B}{T^2}\int\int_{D(\tau_1,\tau_1+T^2)} J_\mu^{W}[\mathcal{G}(\chi_+^2\psi)]n^\mu\\
			&	\qquad + B \int\int_{\mathcal{M}} r\left|[\widetilde{\mathcal{G}},\frac{\Delta\rho^2}{(r^2+a^2)^{3/2}} g^{ab}\partial_a\eta\nabla_b] (\chi_+^2\psi)\right|^2  +\frac{1}{T^2}\int\int_{\mathcal{M}} J^n_\mu[\chi^2_+\psi]+ |\chi_+^2\psi|^2 \\
			&	\qquad +\frac{B}{T^2}\int\int_{D(\tau_1,\tau_1+T^2)} J_\mu^{W}[\mathcal{G}(\chi_+^2\psi)]n^\mu\\
			&	\leq B \int\int_{\mathcal{M}} r\left|\left(\frac{\Delta\rho^2}{(r^2+a^2)^{3/2}} g^{ab}\partial_a\eta\nabla_b\mathcal{G}(\chi_+^2\psi)\right)\right|^2\\
			&	\qquad + B \int\int_{\mathcal{M}} r\left|[\widetilde{\mathcal{G}},\frac{\Delta\rho^2}{(r^2+a^2)^{3/2}} g^{ab}\partial_a\eta\nabla_b] (\chi_+^2\psi)\right|^2  \\
			&	\qquad +B \int_{\{t^\star=\tau_1\}}  J^n_\mu[\psi]n^\mu +|\psi|^2,\\
		\end{aligned}
	\end{equation}
	where in the second to last inequality of~\eqref{eq: proof thm 3 exp decay, eq 3.2.4}  we use the bound on the derivatives and the support of the cut-offs~\eqref{eq: proof thm 3 exp decay, eq 3.2.3.0} and in the last inequality we use the~$H^1$ boundedness result~\eqref{eq: proof thm 3 exp decay, eq 3.2.3.1}.

	Now, by using the bounds on the derivatives of the cut-offs~\eqref{eq: proof thm 3 exp decay, eq 3.2.3.0} we obtain from~\eqref{eq: proof thm 3 exp decay, eq 3.2.4}  the following 
	\begin{equation}\label{eq: proof thm 3 exp decay, eq 3.2.4-1}
		\begin{aligned}
			\boxed{\text{Bulk on RHS of }\eqref{eq: proof thm 3 exp decay, eq 0}} &	\leq B \int\int_{\mathcal{M}} r\left|[\widetilde{\mathcal{G}},\frac{\Delta\rho^2}{(r^2+a^2)^{3/2}} g^{ab}\partial_a\eta\nabla_b] (\chi_+^2\psi)\right|^2  \\
			&	\qquad + \frac{B}{T^2}\int\int_{D(\tau_1,\tau_1+T^2)} J_\mu^{W}[\mathcal{G}(\chi_+^2\psi)]n^\mu \\
			&	\qquad +B \int_{\{t^\star=\tau_1\}}  J^n_\mu[\psi]n^\mu +|\psi|^2.\\
		\end{aligned}
	\end{equation}

	Now, in view of the pseudodifferential commutator estimates of Lemma~\ref{lem: proof thm 3 exp decay, lem 1} and the coarea formula, see Section~\ref{subsec: coarea formula}, we bound the commutation terms of~\eqref{eq: proof thm 3 exp decay, eq 3.2.4-1} as follows 
	\begin{equation}\label{eq: proof thm 3 exp decay, eq 3.2.4-2}
		\begin{aligned}
			\int\int_{\mathcal{M}} r\left|[\widetilde{\mathcal{G}},\frac{\Delta\rho^2}{(r^2+a^2)^{3/2}} g^{ab}\partial_a\eta\nabla_b] (\chi_+^2\psi)\right|^2 &	\leq B \int\int_{\mathcal{M}} J^n_\mu [\chi_+^2\psi]n^\mu \\
			&	\leq  B \int_{\{t^\star=\tau_1\}} J^n_\mu [\psi] n^\mu +|\psi|^2 + \int\int_{D(\tau_1,\tau_2)} J^n_\mu [\eta\chi_+^2\psi]n^\mu \\
			&	\leq \frac{B}{\epsilon_p}\int_{\{t^\star=\tau_1\}} J^n_\mu[\psi]n^\mu+|\psi|^2\\
			&	\qquad +B\epsilon_p\int_{\mathbb{R}}\int_{\mathbb{S}^2}\int_{r_+}^{\bar{r}_+}\frac{1}{g_1}\frac{1}{r^2}\frac{\Delta}{r^2+a^2}\left|\partial_t\mathcal{G}(\sqrt{r^2+a^2}\eta\chi_+^2\psi)\right|^2\\
			&	\qquad +B\epsilon_p \int_{\mathbb{R}}\int_{\mathbb{S}^2}\int_{r_+}^{\bar{r}_+}\frac{1}{g_1}\frac{1}{r^2}\frac{\Delta}{r^2+a^2}\left|a\Xi\partial_\varphi\mathcal{G}(\sqrt{r^2+a^2}\eta\chi_+^2\psi)\right|^2.
		\end{aligned}
	\end{equation}
	where we used the~$H^1$ boundedness result~\eqref{eq: proof thm 3 exp decay, eq 3.2.3.1} and the estimate~\eqref{eq: proof thm 3 exp decay, eq 0.1} of Lemma~\ref{lem: subsec: sec: proof of Theorem 2, subsec 5.5, lem 1} mentioned in the beginning of the proof.

	From~\eqref{eq: proof thm 3 exp decay, eq 3.2.4-1} and~\eqref{eq: proof thm 3 exp decay, eq 3.2.4-2} we obtain the following
	\begin{equation}\label{eq: proof thm 3 exp decay, eq 3.3}
		\begin{aligned}
			\boxed{\text{Bulk on RHS of }\eqref{eq: proof thm 3 exp decay, eq 0}}	&	\leq  \frac{B}{\epsilon_p}\int_{\{t^\star=\tau_1\}} J^n_\mu[\psi]n^\mu+|\psi|^2\\
			&	\qquad +B\epsilon_p\int_{\mathbb{R}}\int_{\mathbb{S}^2}\int_{r_+}^{\bar{r}_+}\frac{1}{g_1}\frac{1}{r^2}\frac{\Delta}{r^2+a^2}\left|\partial_t\mathcal{G}(\sqrt{r^2+a^2}\eta\chi_+^2\psi)\right|^2\\
			&	\qquad +B\epsilon_p \int_{\mathbb{R}}\int_{\mathbb{S}^2}\int_{r_+}^{\bar{r}_+}\frac{1}{g_1}\frac{1}{r^2}\frac{\Delta}{r^2+a^2}\left|a\Xi\partial_\varphi\mathcal{G}(\sqrt{r^2+a^2}\eta\chi_+^2\psi)\right|^2\\
			& \qquad +\frac{B}{T^2}\int\int_{D(\tau_1,\tau_1+T^2)} J_\mu^{W}[\mathcal{G}(\chi_+^2\psi)]n^\mu, \\
		\end{aligned}
	\end{equation}
	where~$W=\partial_t+\frac{a\Xi}{r^2+a^2}\partial_{\varphi}$. 
	
	\begin{center}
		\textbf{Putting everything together}
	\end{center}

	Therefore, inequality~\eqref{eq: proof thm 3 exp decay, eq 3.3} for a sufficiently small~$\epsilon_p>0$ implies
	\begin{equation}\label{eq: proof thm 3 exp decay, eq 3.4}
		\begin{aligned}
			& b\int\int_{\mathcal{M}}\Bigg( \frac{1}{r}\frac{\Delta}{(r^2+a^2)}(Z^\star\mathcal{G}(\eta\chi_+^2\psi))^2+\frac{1}{r}\frac{r^2+a^2}{\Delta}(W\mathcal{G}(\eta\chi_+^2\psi))^2+\frac{1}{r}|\slashed{\nabla}\mathcal{G}(\eta\chi_+^2\psi)|^2+1_{D(\tau_1+T,\tau_2)} J^n_\mu[\psi] n^\mu\Bigg)\\
			\quad &\leq \int_{\{t^\star=\tau_1\}}  J^n_\mu[\psi]n^\mu+|\psi|^2+\frac{1}{T^2}\int\int_{D(\tau_1,\tau_1+T)} J_\mu^{W}[\mathcal{G}(\chi_+^2\psi)]n^\mu,
		\end{aligned}    
	\end{equation}
	for~$T\geq 3$ and~$\tilde{T}\leq  \tau_1<\tau_1+T<\tau_1+2T<\tau_2$. 
\end{proof}

\subsection{Modifying the LHS of Proposition~\ref{prop: subsec: sec: proof of Theorem 2, subsec 5.5, prop 1}}\label{subsec: sec: proof of Theorem 2, subsec 6}

We want to appropriately modify the LHS of the estimate of Proposition~\ref{prop: subsec: sec: proof of Theorem 2, subsec 5.5, prop 1}. We will do so by appealing to the pseudifferential estimates of Lemma~\ref{lem: proof thm 3 exp decay, lem 1}. 

\begin{proposition}\label{prop: subsec: sec: proof of Theorem 2, subsec 5.5, prop 2}
	Let the assumptions of Theorem~\ref{main theorem, relat. non-deg} hold. Moreover, assume that~$\psi$ is in accordance to the assumptions of Section~\ref{subsec: sec: proof of Thm: rel-nondeg, subsec 2}, specifically~$\psi$ arises from smooth initial data on the hypersurface~$\{t^\star=0\}$. Then, we obtain the following estimate 
\begin{equation}
	\begin{aligned}
		& b\int\int_{D(\tau_1+T,\tau_2)}dg\Bigg( \frac{1}{r}\frac{\Delta}{(r^2+a^2)}|Z^\star\mathcal{G}_{\chi_+}\psi|^2+\frac{1}{r}\frac{r^2+a^2}{\Delta}|W\mathcal{G}_{\chi_+}\psi|^2+\frac{1}{r}|\slashed{\nabla}\mathcal{G}_{\chi_+}\psi|^2+ J^n_\mu[\psi] n^\mu\Bigg)\\
		\quad &\leq \int_{\{t^\star=\tau_1\}} J^n_\mu[\psi]n^\mu+|\psi|^2+\frac{1}{T^2}\int\int_{D(\tau_1,\tau_1+T)} J_\mu^{W}[\mathcal{G}(\chi_+^2\psi)]n^\mu,\\
	\end{aligned}    
\end{equation}
for~$T\geq 3$ and~$\tilde{T}\leq \tau_1<\tau_1+T<\tau_1+2T < \tau_2$. For the cut-offs~$\eta,\chi_+$ see~\eqref{eq: sec: proof of Thm: rel-nondeg, eq 0}. For the regular vector field~$Z^\star$ see Section~\ref{subsec: boldsymbol partial r}. Also, recall~$W=\partial_t+\frac{a\Xi}{r^2+a^2}\partial_{\varphi}$.
\end{proposition}

\begin{proof}
	
	In Proposition~\ref{prop: subsec: sec: proof of Theorem 2, subsec 5.5, prop 1} we proved, with the assumptions of Theorem~\ref{main theorem, relat. non-deg}, that the following holds 
	\begin{equation}\label{eq: proof thm 3 exp decay, eq -1}
		\begin{aligned}
			& b\int\int_{\mathcal{M}}\Bigg( \frac{1}{r}\frac{\Delta}{(r^2+a^2)}|Z^\star\mathcal{G}(\eta\chi_+^2\psi)|^2+\frac{1}{r}\frac{r^2+a^2}{\Delta}|W\mathcal{G}(\eta\chi_+^2\psi)|^2+\frac{1}{r}|\slashed{\nabla}\mathcal{G}(\eta\chi_+^2\psi)|^2+1_{D(\tau_1+T,\tau_2)} J^n_\mu[\psi] n^\mu\Bigg)\\
			\quad &\leq \int_{\{t^\star=\tau_1\}}  J^n_\mu[\psi]n^\mu+|\psi|^2+\frac{1}{T^2}\int\int_{D(\tau_1,\tau_1+T)} J_\mu^{W}[\mathcal{G}(\chi_+^2\psi)]n^\mu,
		\end{aligned}    
	\end{equation}
	for~$T\geq 3$ and~$\tilde{T}\leq \tau_1<\tau_1+T<\tau_1+2T< \tau_2$, where for the cut-offs~$\eta,\chi_+$ see~\eqref{eq: sec: proof of Thm: rel-nondeg, eq 0} and the Figure~\ref{fig: cut-offs 0}.

	We want to modify the LHS of~\eqref{eq: proof thm 3 exp decay, eq -1} in order to appropriately display only the pseudodifferential operator
	\begin{equation}
		\mathcal{G}_{\chi_+}
	\end{equation}
	where we use the notation
	\begin{equation}
		\mathcal{G}_{\chi}=\chi\mathcal{G}\chi,
	\end{equation}
	for a smooth function~$\chi:\mathcal{M}\rightarrow \mathbb{R}$. Specifically, we want to prove that the following hold
	\begin{equation}\label{eq: proof thm 3 exp decay, eq 3.4.0}
		\begin{aligned}
			b \int\int_{\mathcal{M}} \left| \slashed{\nabla}\left(\eta\mathcal{G}_{\chi_+}\psi\right)  \right|^2 &	\leq  \text{LHS of}~\eqref{eq: proof thm 3 exp decay, eq -1}+\text{RHS of}~\eqref{eq: proof thm 3 exp decay, eq -1},\\
			b \int\int_{\mathcal{M}} \Delta\left| Z^\star\left(\eta\mathcal{G}_{\chi_+}\psi\right)  \right|^2 &	\leq  \text{LHS of}~\eqref{eq: proof thm 3 exp decay, eq -1}+\text{RHS of}~\eqref{eq: proof thm 3 exp decay, eq -1},\\
			b \int\int_{\mathcal{M}} \frac{1}{\Delta}\left| W\left(\eta\mathcal{G}_{\chi_+}\psi\right)\right|^2 & \leq  \text{LHS of}~\eqref{eq: proof thm 3 exp decay, eq -1}+\text{RHS of}~\eqref{eq: proof thm 3 exp decay, eq -1}.\\
		\end{aligned}
	\end{equation}

	We only prove the second estimate of~\eqref{eq: proof thm 3 exp decay, eq 3.4.0}, which is the most elaborate, and note that the proofs of the first and second inequality of~\eqref{eq: proof thm 3 exp decay, eq 3.4.0} are similar.  
	
	\begin{center}
		\textbf{Proof of the second estimate of~\eqref{eq: proof thm 3 exp decay, eq 3.4.0}}
	\end{center}

	First, we note the following 
	\begin{equation}\label{eq: proof thm 3 exp decay, eq 3.4-1}
		\begin{aligned}
			\int\int_{\mathcal{M}} \Delta \left| Z^\star \left(\eta \mathcal{G}_{\chi_+} \psi\right) \right|^2 &	= 	\int\int_{\mathcal{M}} \Delta \left| Z^\star \left(\eta \chi_+ \mathcal{G}(\chi_+\psi) \right) \right|^2 \\
			&	=	\int\int_{\mathcal{M}} \Delta \left| Z^\star \left( \mathcal{G}(\eta\chi_+^2\psi) + [\eta\chi_+,\mathcal{G}](\chi_+\psi)  \right) \right|^2\\
			&	\lesssim \int\int_{\mathcal{M}} \Delta \left| Z^\star \mathcal{G}(\eta\chi_+^2\psi) \right|^2	+ \int\int_{\mathcal{M}} \Delta \left| Z^\star \left([\eta\chi_+,\mathcal{G}](\chi_+\psi)\right) \right|^2\\
			&	\lesssim \text{LHS of}~\eqref{eq: proof thm 3 exp decay, eq -1}	+ \int\int_{\mathcal{M}} \Delta \left| Z^\star \left([\eta\chi_+,\mathcal{G}](\chi_+\psi)\right) \right|^2.
		\end{aligned}
	\end{equation}
	Moreover, we bound  the last term on the RHS of~\eqref{eq: proof thm 3 exp decay, eq 3.4-1} as follows
	\begin{equation}\label{eq: proof thm 3 exp decay, eq 3.4-2}
		\begin{aligned}
			&	\int\int_{\mathcal{M}} \Delta \left| Z^\star \left([\eta\chi_+,\mathcal{G}](\chi_+\psi)\right) \right|^2\\
			&	\quad  = \int\int_{\mathcal{M}}\Delta \Bigg| Z^\star(\eta\chi_+)\mathcal{G}(\chi_+\psi) -\mathcal{G}\left(Z^\star(\eta\chi_+)\cdot \chi_+\psi\right)-[\mathcal{G},\eta\chi_+](Z^\star(\chi_+\psi))-[Z^\star,\mathcal{G}](\eta\chi_+^2\psi)+\eta\chi_+ [Z^\star,\mathcal{G}](\chi_+\psi) \Bigg|^2\\
			&	\quad \leq B \int\int_{\mathcal{M}} \Delta (Z^\star(\eta\chi_+))^2(\mathcal{G}(\chi_+\psi))^2 + \Delta\left([\mathcal{G},Z^\star(\eta\chi_+)] (\chi_+\psi)\right)^2+\Delta\left(Z^\star(\eta\chi_+)\cdot\mathcal{G}( \chi_+\psi)\right)^2\\
			&	\qquad\qquad + B\int\int_{\mathcal{M}} \Delta \left([\mathcal{G},\eta\chi_+](Z^\star(\chi_+\psi))\right)^2 + \Delta \left([Z^\star,\mathcal{G}](\eta\chi_+^2\psi)\right)^2 + \Delta \left(\eta\chi_+ [Z^\star,\mathcal{G}](\chi_+\psi)\right)^2\\
			&	\quad \leq B \int_{\{t^\star=\tau_1\}} J^n_\mu[\psi]n^\mu +|\psi|^2 + B\int\int_{\mathcal{M}}J^n_\mu[\eta\chi_+^2\psi]n^\mu,
		\end{aligned}
	\end{equation}
	where the first equality of~\eqref{eq: proof thm 3 exp decay, eq 3.4-2} holds pointwise for the integrands~(after a lengthy computation) and the last estimate of~\eqref{eq: proof thm 3 exp decay, eq 3.4-2} is proved by using the pseudodifferential commutations of Lemma~\ref{lem: proof thm 3 exp decay, lem 1} and by recalling the support of the cut--off functions, see Section~\ref{subsec: sec: carter separation, subsec 1} and the~$H^1$ boundedness result of Theorem~\ref{main theorem 1}.

	Now we want to bound the last bulk term on the RHS of~\eqref{eq: proof thm 3 exp decay, eq 3.4-2}. We will use the estimate~\eqref{eq: proof thm 3 exp decay, eq 0.1} of Lemma~\ref{lem: subsec: sec: proof of Theorem 2, subsec 5.5, lem 1}, for a sufficiently small~$\epsilon_p>0$. We immediately obtain
	\begin{equation}\label{eq: proof thm 3 exp decay, eq 3.4-4}
		b \int\int_{\mathcal{M}} J^n_\mu[\eta\chi_+^2\psi]n^\mu  \leq \text{LHS of}~\eqref{eq: proof thm 3 exp decay, eq -1}+\text{RHS of}~\eqref{eq: proof thm 3 exp decay, eq -1}.\\
	\end{equation}
	Therefore, by~\eqref{eq: proof thm 3 exp decay, eq 3.4-4},~\eqref{eq: proof thm 3 exp decay, eq 3.4-2} we conclude the second inequality of~\eqref{eq: proof thm 3 exp decay, eq 3.4.0}.

	\begin{center}
		\textbf{Putting everything together}
	\end{center}
	
	Now, by using the now proved~\eqref{eq: proof thm 3 exp decay, eq 3.4.0} then our main inequality~\eqref{eq: proof thm 3 exp decay, eq -1} implies
	\begin{equation}
		\begin{aligned}
			& b\int\int_{D(\tau_1+T,\tau_2)}dg\Bigg( \frac{1}{r}\frac{\Delta}{(r^2+a^2)}|Z^\star\mathcal{G}_{\chi_+}\psi|^2+\frac{1}{r}\frac{r^2+a^2}{\Delta}|W\mathcal{G}_{\chi_+}\psi|^2+\frac{1}{r}|\slashed{\nabla}\mathcal{G}_{\chi_+}\psi|^2\Bigg)+b\int\int_{D(\tau_1+T,\tau_2)} J^n_\mu[\psi] n^\mu\\
			\quad &\leq \int_{\{t^\star=\tau_1\}} J^n_\mu[\psi]n^\mu+|\psi|^2+\frac{1}{T^2}\int\int_{D(\tau_1,\tau_1+T)} J_\mu^{W}[\mathcal{G}(\chi_+^2\psi)]n^\mu,\\
		\end{aligned}    
	\end{equation}
	where~$T\geq 3$ and~$\tilde{T}\leq \tau_1<\tau_1+T<\tau_1+2T< \tau_2$. For the timelike vector field field~$W=\partial_t+\frac{a\Xi}{r^2+a^2}\partial_{\varphi}$ see Lemma~\ref{lem: causal vf E,1}. 
\end{proof}

\subsection{Modifying the RHS of Proposition~\ref{prop: subsec: sec: proof of Theorem 2, subsec 5.5, prop 2}}\label{subsec: sec: proof of Theorem 2, subsec 7}

We want to appropriately modify the RHS of the estimate of Proposition~\ref{prop: subsec: sec: proof of Theorem 2, subsec 5.5, prop 2} to display~$\mathcal{G}(\chi_-\psi)$. We will do so by appealing to the pseudodifferential estimates of Lemma~\ref{lem: proof thm 3 exp decay, lem 1}.

\begin{proposition}\label{prop: subsec: sec: proof of Theorem 2, subsec 5.5, prop 3}
	Let the assumptions of Theorem~\ref{main theorem, relat. non-deg} hold. Moreover, assume that~$\psi$ is in accordance to the assumptions of Section~\ref{subsec: sec: proof of Thm: rel-nondeg, subsec 2}, specifically~$\psi$ arises from smooth initial data. Then, we obtain the following estimate 
			\begin{equation}
				\begin{aligned}
					& b\int\int_{D(\tau_1+T,\tau_2)}dg\Bigg( \frac{1}{r}\frac{\Delta}{(r^2+a^2)}|Z^\star\mathcal{G}_{\chi_+}\psi|^2+\frac{1}{r}\frac{r^2+a^2}{\Delta}|W\mathcal{G}_{\chi_+}\psi|^2+\frac{1}{r}|\slashed{\nabla}\mathcal{G}_{\chi_+}\psi|^2 + J^n_\mu[\psi] n^\mu \Bigg)\\
					\quad &\leq \int_{\{t^\star=\tau_1\}}  J^n_\mu[\psi]n^\mu+|\psi|^2+\frac{1}{T^2}\int\int_{D(\tau_1,\tau_1+T)} J_\mu^{W}[\mathcal{G}_{\chi_{-}}\psi]n^\mu,\\
				\end{aligned}    
			\end{equation}
	for~$T\geq 3$ and~$\tilde{T}\leq \tau_1<\tau_1+T<\tau_1+2T<\tau_2$. For the cut-offs~$\eta,\chi_+$ see~\eqref{eq: sec: proof of Thm: rel-nondeg, eq 0}. For the regular vector field~$Z^\star$ see Section~\ref{subsec: boldsymbol partial r}, where~$W=\partial_t+\frac{a\Xi}{r^2+a^2}\partial_{\varphi}$.
\end{proposition}

\begin{proof}
In view of Proposition~\ref{prop: subsec: sec: proof of Theorem 2, subsec 5.5, prop 2}, it suffices to prove the following estimate
	\begin{equation}\label{eq: proof thm 3 exp decay, eq 4}
	\begin{aligned}
		&	\int_{\tau_1}^{\tau_1+T}\int_{\{t^\star=\tau^\prime\}} J_\mu^{W}[\mathcal{G}(\chi_+^2\psi)]n^\mu \leq B \cdot T  \int_{\{t^\star=\tau_1\}} J^n_\mu[\psi]n^\mu+|\psi|^2 + B \int_{\tau_1}^{\tau_1+T}\int_{\{t^\star=\tau^\prime\}}J_\mu^{W}[\mathcal{G}_{\chi_{-}}\psi]n^\mu d\tau^\prime,
	\end{aligned}
\end{equation}
for~$\tilde{T}\leq \tau_1$, where for the cut offs~$\chi_-,\chi_+$ see Section~\ref{subsec: sec: carter separation, subsec 1}, and for their graphic representation see Figure~\ref{fig: cut-offs 0}. Recall the vector field~$W=\partial_t+\frac{a\Xi}{r^2+a^2}\partial_{\varphi}$. 
\begin{figure}[htbp]
	\centering
	\includegraphics[scale=1]{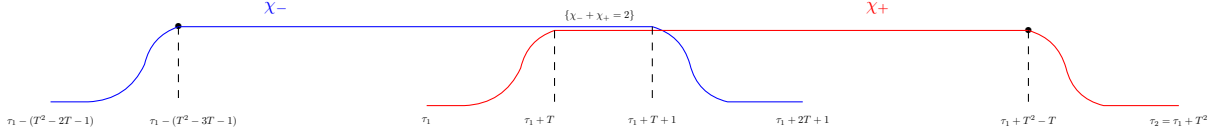}
	\caption{The cut-offs~$\chi_+,\chi_{-}$}
	\label{fig: cut-offs 0}
\end{figure} 	
	
		To prove~\eqref{eq: proof thm 3 exp decay, eq 4} we first note that~$\chi_-\Big|_{\{\tau_1,\tau_1+T\}} \equiv 1$. Therefore, we write
	\begin{equation}\label{eq: proof: lem: proof thm 3 exp decay, lem 2, eq 1}
		\begin{aligned}
			&	\int_{\tau_1}^{\tau_1+T}\int_{\{t^\star=\tau^\prime\}} J_\mu^{W}[\mathcal{G}(\chi_+^2\psi)]n^\mu d\tau^\prime\\
			&	\qquad \leq B \int\int_{\mathcal{M}} 1_{\{\tau_1\leq \tau\leq \tau_1+T\}}\cdot \chi_{-}^4\cdot\Bigg( \Delta\left|Z^\star\mathcal{G}(\chi_+^2\psi)\right|^2+\sum_{i_1+i_2=1}\left|W^{i_1}\slashed{\nabla}^{i_2}\mathcal{G}(\chi_+^2\psi)\right|^2  \Bigg)\\
			&	\qquad \leq  B \int\int_{\mathcal{M}} J^W_\mu [\chi_+^2\psi]+|\chi_+\psi|^2 + 1_{\{\tau_1\leq \tau\leq\tau_1+T\}}J^W_\mu [\mathcal{G}_{\chi_-}\psi]n^\mu, 
		\end{aligned}
	\end{equation}
	where we used the Coifman--Meyer and pseudodifferential estimates of Lemma~\ref{lem: proof thm 3 exp decay, lem 1}, specifically we used the pseudodifferential estimates~\eqref{eq: thm 3 exp decay, eq 5},~\eqref{eq: thm 3 exp decay, eq 5.1},~\eqref{eq: thm 3 exp decay, eq 5.2}. Note that the cut-offed functions~$\chi_+^2\psi,\chi_-\psi$ vanish in the domain for~$\{t^\star<0\}$ and moreover belong in the space~$H^1(\mathcal{M})$ and therefore we can safely appeal to the pseudodifferential estimate of Lemma~\ref{lem: proof thm 3 exp decay, lem 1}.

	Now, in view of the support of the function~$\chi_+$, see Section~\ref{subsec: sec: carter separation, subsec 1} and Figure~\ref{fig: cut-offs 0}, and the boundedness estimate of Theorem~\ref{main theorem 1} we conclude from~\eqref{eq: proof: lem: proof thm 3 exp decay, lem 2, eq 1} that~\eqref{eq: proof thm 3 exp decay, eq 4} holds. 
\end{proof}

\subsection{Completing the proof of Theorem~\ref{main theorem, relat. non-deg}}\label{subsec: sec: proof of Theorem 2, subsec 8}

We use the result of Proposition~\ref{prop: subsec: sec: proof of Theorem 2, subsec 5.5, prop 3} and Theorem~\ref{main theorem 1} to obtain
\begin{equation}\label{eq: proof thm 3 exp decay, eq 7}
	\begin{aligned}
		& b\int_{\{t^\star=\tau\}} J^n_\mu[\psi]n^\mu   \\
		&	+b\int\int_{D(\tau_1+T,\tau_2)}dg\Bigg( \frac{1}{r}\frac{\Delta}{(r^2+a^2)}|Z^\star\mathcal{G}_{\chi_+}\psi|^2+\frac{1}{r}\frac{r^2+a^2}{\Delta}|W\mathcal{G}_{\chi_+}\psi|^2+\frac{1}{r}|\slashed{\nabla}\mathcal{G}_{\chi_+}\psi|^2+  J^n_\mu[\psi] n^\mu \Bigg)\\
		\quad &\leq \int_{\{t^\star=\tau_1\}} J^n_\mu[\psi]n^\mu +|\psi|^2+\frac{1}{T^2}\int\int_{D(\tau_1,\tau_1+T)} J_\mu^{W}[\mathcal{G}_{\chi_{-}}\psi]n^\mu,\\
	\end{aligned}    
\end{equation}
where~$T\geq 3$,~$T^2-2Τ-1\leq \tau_1\leq\tau_1+T\leq\tau \leq \tau_2$.

The estimate~\eqref{eq: proof thm 3 exp decay, eq 7} concludes the proof of Theorem~\ref{main theorem, relat. non-deg} for the case of non zero Klein--Gordon mass~$\mu^2_{KG}>0$.

To conclude the result of Theorem~\ref{main theorem, relat. non-deg} we need to remove the zeroth order term~$|\psi|^2$ from the hypersurface term of the RHS of~\eqref{eq: proof thm 3 exp decay, eq 7}, for the case of the wave equation~$\mu_{KG}=0$.

Recall from Section~\ref{subsec: poincare wirtigner} the following Poincare--Wirtinger inequality
\begin{equation}\label{eq: proof thm 3 exp decay, eq 7.0.1}
		\int_{\{t^\star=\tau_1\}} |\psi-\underline{\psi}(\tau_1)|^2\leq B \int_{\{t^\star=\tau\}} J^n_\mu[\psi]n^\mu,
\end{equation}
where~$\underline{\psi}(\tau)=\frac{1}{|\{t^\star=\tau\}|}\int_{\{t^\star=\tau\}}\psi$.

We will now make an assumption that will be removed at the end of the present proof. Specifically, we make the assumption that the mean of~$\psi$ at the hypersurface~$\{t^\star=\tau_1\}$ vanishes, namely
\begin{equation}\label{eq: proof thm 3 exp decay, eq 7.1}
	\underline{\psi}(\tau_1)=0.
\end{equation}

By using the assumption~\eqref{eq: proof thm 3 exp decay, eq 7.1} in~\eqref{eq: proof thm 3 exp decay, eq 7} we conclude that we have 
\begin{equation}
\begin{aligned}
& b\int_{\{t^\star=\tau\}} J^n_\mu[\psi]n^\mu   \\
&	+b\int\int_{D(\tau_1+T,\tau_2)}dg\Bigg( \frac{1}{r}\frac{\Delta}{(r^2+a^2)}|Z^\star\mathcal{G}_{\chi_+}\psi|^2+\frac{1}{r}\frac{r^2+a^2}{\Delta}|W\mathcal{G}_{\chi_+}\psi|^2+\frac{1}{r}|\slashed{\nabla}\mathcal{G}_{\chi_+}\psi|^2+  J^n_\mu[\psi] n^\mu \Bigg)\\
\qquad &\leq \int_{\{t^\star=\tau_1\}} J^n_\mu[\psi]n^\mu+|\psi-\underline{\psi}(\tau_1)|^2 +\frac{1}{T}\int\int_{D(\tau_1,\tau_1+T)} J_\mu^{W}[\mathcal{G}_{\chi_{-}}\psi]n^\mu\\
\qquad  &\leq \int_{\{t^\star=\tau_1\}} J^n_\mu[\psi]n^\mu +\frac{1}{T^2}\int\int_{D(\tau_1,\tau_1+T)} J_\mu^{W}[\mathcal{G}_{\chi_{-}}\psi]n^\mu,\\
\end{aligned}    
\end{equation}
where we used the Poincare--Wirtinger inequality~\eqref{eq: proof thm 3 exp decay, eq 7.0.1} on the RHS of~\eqref{eq: proof thm 3 exp decay, eq 7}, and therefore we obtain
\begin{equation}\label{eq: proof thm 3 exp decay, eq 8}
	\begin{aligned}
		& b\int_{\{t^\star=\tau\}} J^n_\mu[\psi]n^\mu   \\
		&	+b\int\int_{D(\tau_1+T,\tau_2)}dg\Bigg( \frac{1}{r}\frac{\Delta}{(r^2+a^2)}|Z^\star\mathcal{G}_{\chi_+}\psi|^2+\frac{1}{r}\frac{r^2+a^2}{\Delta}|W\mathcal{G}_{\chi_+}\psi|^2+\frac{1}{r}|\slashed{\nabla}\mathcal{G}_{\chi_+}\psi|^2+  J^n_\mu[\psi] n^\mu \Bigg)\\
		\qquad  &\leq \int_{\{t^\star=\tau_1\}} J^n_\mu[\psi]n^\mu +\frac{1}{T^2}\int\int_{D(\tau_1,\tau_1+T)} J_\mu^{W}[\mathcal{G}_{\chi_{-}}\psi]n^\mu.\\
	\end{aligned}    
\end{equation}
We note the following identity
\begin{equation}\label{eq: proof thm 3 exp decay, eq 8.1}
	[Z^\star, \mathcal{G}_{\chi_-}] \psi = \chi_- \left([Z^\star,\mathcal{G}](\chi_-\psi)-\mathcal{G}[Z^\star,\chi_-]\psi\right) +Z^\star\chi_- \cdot \mathcal{G}(\chi_-\psi).
\end{equation}
Similar, and in fact simpler, identities hold if we replace~$Z^\star$ with the operators~$\partial_t,\partial_{\varphi},\slashed{\nabla}$. Similar identities also hold if we replace~$\mathcal{G}_{\chi_-}$ with~$\mathcal{G}_{\chi_+}$.

Now, by using~\eqref{eq: proof thm 3 exp decay, eq 8.1} in~\eqref{eq: proof thm 3 exp decay, eq 8} in conjunction with the Morawetz and boundedness estimates of Theorem~\ref{main theorem 1}, to bound the terms on the RHS of~\eqref{eq: proof thm 3 exp decay, eq 8.1} generated by the commutation, we obtain the following 
\begin{equation}\label{eq: proof thm 3 exp decay, eq 8.2}
	\begin{aligned}
		& b\int_{\{t^\star=\tau\}} J^n_\mu[\psi]n^\mu   \\
		&	+b\int\int_{D(\tau_1+T,\tau_2)}dg\Bigg( \frac{1}{r}\frac{\Delta}{(r^2+a^2)}|\mathcal{G}_{\chi_+}Z^\star\psi|^2+\frac{1}{r}\frac{r^2+a^2}{\Delta}|\mathcal{G}_{\chi_+}W\psi|^2+\frac{1}{r}|\mathcal{G}_{\chi_+}\slashed{\nabla}\psi|^2+  J^n_\mu[\psi] n^\mu \Bigg)\\
		\qquad  &\leq \int_{\{t^\star=\tau_1\}} J^n_\mu[\psi]n^\mu +\frac{1}{T^2}\int\int_{D(\tau_1,\tau_1+T)} \left|\mathcal{G}_{\chi_{-}}Z^\star\psi\right|^2+ \left|\mathcal{G}_{\chi_{-}}W\psi\right|^2+\left|\mathcal{G}_{\chi_{-}}\slashed{\nabla}\psi\right|^2,
	\end{aligned}    
\end{equation}
for any~$\tau\geq \tau_1$, where note that we commuted all of the derivatives with the operators~$\mathcal{G}_{\chi_\pm}$. Note that in~\eqref{eq: proof thm 3 exp decay, eq 8.2} we used the following abuse of notation, for convenience: when we write $|\mathcal{G}_{\chi_+}\slashed{\nabla}\psi|^2$ we mean~$\slashed{g}^{ab}\mathcal{G}_{\chi_+}\partial_a\psi\mathcal{G}_{\chi_+}\partial_b\psi$, where~$\slashed{g}$ is the metric of the~$r$ sphere, see Section~\ref{subsec: admissible hypersurfaces}.

For the case of the wave equation~$\mu_{KG}=0$ if~$\psi$ is a solution then~$\psi+c$ is a solution, for any constant~$c\in \mathbb{C}$. Therefore, if~$\psi$ is a solution then so is~$\psi-\underline{\psi}(\tau_1)$. Therefore, by replacing~$\psi$ with~$\psi-\underline{\psi}(\tau_1)$ in~\eqref{eq: proof thm 3 exp decay, eq 8.2} we can see that~\eqref{eq: proof thm 3 exp decay, eq 8.2} holds for any initial mean. To prove that~\eqref{eq: proof thm 3 exp decay, eq 8} holds for any initial mean we use the energy estimate~\eqref{eq: proof thm 3 exp decay, eq 8.2} and go back to~\eqref{eq: proof thm 3 exp decay, eq 8} by using again the identity~\eqref{eq: proof thm 3 exp decay, eq 8.1} in conjunction with the Morawetz and boundedness estimates of Theorem~\ref{main theorem 1}.

\begin{remark}
Note that we concluded the proof of the main Theorem~\ref{main theorem, relat. non-deg} from estimate~\eqref{eq: proof thm 3 exp decay, eq 7}, with a constant that blows up as~$\mu_{KG}\rightarrow 0$, since there is a zeroth order term~$|\psi|^2$ on the RHS of~\eqref{eq: proof thm 3 exp decay, eq 7} that we cannot remove. On the other hand, in the case~$\mu^2_{KG}=0$ we can remove the zeroth order term from the RHS of~\eqref{eq: proof thm 3 exp decay, eq 7}, in view of the process described above. Therefore, as already mentioned in Remark~\ref{rem: subsec: sec: intro: subsec 5, rem 1}, the constant of the main Theorem~\ref{main theorem, relat. non-deg} blows up in the limit~$\mu_{KG}\rightarrow 0$, but we obtain the result of the main Theorem~\ref{main theorem, relat. non-deg} for the wave case~$\mu_{KG}=0$ separately. 
\end{remark}

\section{Proof of Corollary~\ref{cor: main theorem, relat. non-deg, cor 3}}\label{sec: cor: main theorem, relat. non-deg, cor 3}

In this Section we prove Corollary~\ref{cor: main theorem, relat. non-deg, cor 3}. 

\subsection{Proof of the energy decay~\eqref{eq: cor: main theorem, relat. non-deg, cor 3, eq 1} of Corollary~\ref{cor: main theorem, relat. non-deg, cor 3}}

The relatively non degenerate estimate of Theorem~\ref{main theorem, relat. non-deg} implies the following
\begin{equation}\label{eq: sec: cor: main theorem, relat. non-deg, cor 3, eq 0}
	\begin{aligned}
		& \int_{\{t^\star=\tau\}} J^n_\mu[\psi]n^\mu+\int\int_{D(\tau_1+T,+\infty)}  J^n_\mu [\psi]n^\mu +\mathcal{E}(\mathcal{G}_{\chi_+}\psi,\psi) \\
		&	\qquad\qquad \leq C\int_{\{t^\star=\tau_1\}} J^n_\mu[\psi]n^\mu+\frac{C}{T}\int\int_{D(\tau_1,\tau_1+T)} \mathcal{E}(\mathcal{G}_{\chi_{-}}\psi,\psi),\\
	\end{aligned}    
\end{equation}
for any~$T\geq 3$,~$3\leq\tilde{T}\leq\tau_1\leq\tau_1+T\leq\tau_2$, and for any~$\tau\geq \tau_1$, where we neglected the bulk integrand term~$\frac{1}{\Delta}(W\mathcal{G}_\chi\psi)^2$ on the LHS. For the cut-offs
\begin{equation}
	\chi_+(t^\star)=\chi_{\tau_1,\tau_1+T^2}(t^\star),\qquad \chi_{-}(t^\star)=\chi_{\tau_1,\tau_1+T^2}(t^\star+\tilde{T})
\end{equation}
see Section~\ref{subsec: sec: carter separation, subsec 1}, and for their graphic representation see Figure~\ref{fig: cut-offs 1}.

We define the sequence
\begin{equation}
 t_i =T+i\tilde{T},\qquad i=0,1,2,\dots.
\end{equation}
In view of the relatively non degenerate estimate~\eqref{eq: sec: cor: main theorem, relat. non-deg, cor 3, eq 0} we note that there exists a constant~$C_m$ such that for any~$T$ sufficiently large and for any~$i=0,1,2,\dots$ the following holds
\begin{equation}\label{eq: sec: cor: main theorem, relat. non-deg, cor 3, eq 0.1}
	\begin{aligned}
		& \int_{\{t^\star=t_i+T\}} J^n_\mu[\psi]n^\mu+\int_{t_i+T}^\infty \int_{\{t^\star=\tau\}}  J^n_\mu [\psi]n^\mu +\mathcal{E}(\mathcal{G}_i\psi,\psi) \\
		&	\qquad\qquad \leq C_m\int_{\{t^\star=t_i\}} J^n_\mu[\psi]n^\mu+\frac{C_m}{T}\int_{t_i}^{t_i+T}\int_{\{t^\star=\tau\}} \mathcal{E}(\mathcal{G}_{i-1}\psi,\psi),\\
	\end{aligned}    
\end{equation}
where 
\begin{equation}
\chi_i=\chi_{\textit{data}}(z-i\tilde{T}),\qquad \chi_{\textit{data}}(t^\star):=\chi_0(t^\star)=\chi_{0,T^2}(t^\star),
\end{equation}
for~$i=0,1,2\dots$ and we denote:
\begin{equation}
	\mathcal{G}_i := \chi_i \mathcal{G}\chi_i.
\end{equation}

Moreover, we recall the boundedness estimate of Theorem~\ref{main theorem 1}
\begin{equation}\label{eq: sec: cor: main theorem, relat. non-deg, cor 3, eq 0.2}
	\int_{\{t^\star=t^{\prime\prime}\}} J^n_\mu [\psi] n^\mu \leq C_b 	\int_{\{t^\star=t^{\prime}\}} J^n_\mu [\psi] n^\mu
\end{equation}
for some constant~$C_b>0$ and for any~$0<t^\prime\leq t^{\prime\prime}$.

By using the mean value theorem for integrals we note that there exists a
\begin{equation}
	\tilde{t}_1\in [t_1+2T,t_2-T]
\end{equation}
such that 
\begin{equation}\label{eq: sec: cor: main theorem, relat. non-deg, cor 3, eq 0.3}
	(t_2-T-(t_1+2T)) \int_{\{t^\star=\tilde{t}_1\}} J^n_\mu [\psi] n^\mu \leq \int_{t_1+T}^\infty \int_{\{t^\star=\tau\}} J^n_\mu [\psi] n^\mu.
\end{equation}

By using~\eqref{eq: sec: cor: main theorem, relat. non-deg, cor 3, eq 0.3} in the relatively non degenerate estimate~\eqref{eq: sec: cor: main theorem, relat. non-deg, cor 3, eq 0.1}, for~$i=1$, we obtain the following 
\begin{equation}\label{eq: sec: cor: main theorem, relat. non-deg, cor 3, eq 0.4}
	\begin{aligned}
		(t_2-t_1-3T) \int_{\{t^\star=\tilde{t}_1\}} J^n_\mu [\psi] n^\mu  \leq C_m\int_{\{t^\star=t_1\}} J^n_\mu[\psi]n^\mu+\frac{C_m}{T}\int_{t_1}^{t_1+T}\int_{\{t^\star=\tau\}} \mathcal{E}(\mathcal{G}_{0}\psi,\psi).
	\end{aligned}
\end{equation}
Now, by using the boundedness estimate~\eqref{eq: sec: cor: main theorem, relat. non-deg, cor 3, eq 0.2}, for~$t^{\prime\prime}=t_2,t^\prime=\tilde{t}_1$ in the LHS of~\eqref{eq: sec: cor: main theorem, relat. non-deg, cor 3, eq 0.4} we obtain 
\begin{equation}\label{eq: sec: cor: main theorem, relat. non-deg, cor 3, eq 0.5}
	(t_2-t_1-3T) \frac{1}{C_b} \int_{\{t^\star=t_2\}} J^n_\mu [\psi] n^\mu  \leq C_m\int_{\{t^\star=t_1\}} J^n_\mu[\psi]n^\mu+\frac{C_m}{T}\int_{t_1}^{t_1+T}\int_{\{t^\star=\tau\}} \mathcal{E}(\mathcal{G}_{0}\psi,\psi).
\end{equation}
By taking~$T>0$ sufficiently large, we note~$(t_2-t_1-3T)=T^2-5T-1>T$ and therefore inequality~\eqref{eq: sec: cor: main theorem, relat. non-deg, cor 3, eq 0.5} implies the following 
\begin{equation}\label{eq: sec: cor: main theorem, relat. non-deg, cor 3, eq 0.6}
	\int_{\{t^\star=t_2\}} J^n_\mu [\psi] n^\mu \leq \frac{ C_m\cdot C_b}{T} \int_{\{t^\star=t_1\}} J^n_\mu[\psi]n^\mu+\frac{C_m\cdot C_b}{T^2}\int_{t_1}^{t_1+T}\int_{\{t^\star=\tau\}} \mathcal{E}(\mathcal{G}_{0}\psi,\psi).
\end{equation}

Finally, we bound the hypersurface term on the RHS of the relatively non degenerate estimate~\eqref{eq: sec: cor: main theorem, relat. non-deg, cor 3, eq 0.1}, for~$i=2$, by using~\eqref{eq: sec: cor: main theorem, relat. non-deg, cor 3, eq 0.6}, and we bound the bulk term on the RHS of~\eqref{eq: sec: cor: main theorem, relat. non-deg, cor 3, eq 0.1} by using the relatively non degenereate estimate~\eqref{eq: sec: cor: main theorem, relat. non-deg, cor 3, eq 0.1} for~$i=1$ to obtain 
\begin{equation}\label{eq: sec: cor: main theorem, relat. non-deg, cor 3, eq 0.7}
	\begin{aligned}
		& \int_{\{t^\star=t_2+T\}} J^n_\mu[\psi]n^\mu+\int_{t_2+T}^\infty \int_{\{t^\star=\tau\}}  J^n_\mu [\psi]n^\mu +\mathcal{E}(\mathcal{G}_2\psi,\psi) \\
		&	\qquad\qquad \leq C_m \Big(\frac{ C_m\cdot C_b}{T} \int_{\{t^\star=t_1\}} J^n_\mu[\psi]n^\mu+\frac{C_m\cdot C_b}{T^2}\int_{t_1}^{t_1+T}\int_{\{t^\star=\tau\}} \mathcal{E}(\mathcal{G}_{0}\psi,\psi) \Big) \\
		&\qquad\qquad\qquad +\frac{C_m}{T} \Big(  C_m\int_{\{t^\star=t_1\}} J^n_\mu[\psi]n^\mu+\frac{C_m}{T}\int_{t_1}^{t_1+T}\int_{\{t^\star=\tau\}} \mathcal{E}(\mathcal{G}_{0}\psi,\psi)\Big). 
	\end{aligned}    
\end{equation}
We use the boundedness estimate~\eqref{eq: sec: cor: main theorem, relat. non-deg, cor 3, eq 0.2} one final time and we obtain 
\begin{equation}\label{eq: sec: cor: main theorem, relat. non-deg, cor 3, eq 0.8}
	\begin{aligned}
		& \int_{\{t^\star=t_2+T\}} J^n_\mu[\psi]n^\mu+\int_{t_2+T}^\infty \int_{\{t^\star=\tau\}}  J^n_\mu [\psi]n^\mu +\mathcal{E}(\mathcal{G}_2\psi,\psi) \\
		&	\qquad\qquad \leq C_m \Big(\frac{ C_m\cdot C_b^2}{T} \int_{\{t^\star=t_0\}} J^n_\mu[\psi]n^\mu+\frac{C_m\cdot C_b}{T^2}\int_{t_1}^{t_1+T}\int_{\{t^\star=\tau\}} \mathcal{E}(\mathcal{G}_{0}\psi,\psi) \Big) \\
		&\qquad\qquad\qquad +\frac{C_m}{T} \Big(  C_m\cdot C_b\int_{\{t^\star=t_0\}} J^n_\mu[\psi]n^\mu+\frac{C_m}{T}\int_{t_1}^{t_1+T}\int_{\{t^\star=\tau\}} \mathcal{E}(\mathcal{G}_{0}\psi,\psi)\Big) \\
		&	\qquad\qquad =  \left(\frac{ C_m^2\cdot C_b^2}{T} +\frac{C_m^2C_b}{T} \right) \int_{\{t^\star=t_0\}} J^n_\mu[\psi]n^\mu+\left(\frac{C_m^2\cdot C_b}{T^2}+ \frac{C_m^2}{T^2} \right) \int_{t_1}^{t_1+T}\int_{\{t^\star=\tau\}} \mathcal{E}(\mathcal{G}_{0}\psi,\psi). \\
	\end{aligned}    
\end{equation}

We now take~$T$ sufficiently large so that 
\begin{equation}
	\frac{C_m^2C_b^2}{T} + \frac{C_m^2C_b}{T}+\frac{C_m^2C_b}{T^2}+\frac{C_m^2}{T^2}<\frac{1}{20}.
\end{equation}
Therefore, if
\begin{equation}
	\int_{\{t^\star=t_0\}} J^n_\mu[\psi]n^\mu\leq E_{data},\qquad \int_{t_1}^{t_1+T}\int_{\{t^\star=\tau\}} \mathcal{E}(\mathcal{G}_{0}\psi,\psi)\leq E_{data}
\end{equation}
for some constant~$E_{data}$ that will be determined later, then the estimate~\eqref{eq: sec: cor: main theorem, relat. non-deg, cor 3, eq 0.8} implies that 
\begin{equation}
	 \int_{\{t^\star=t_2+T\}} J^n_\mu[\psi]n^\mu+\int_{t_2+T}^\infty \int_{\{t^\star=\tau\}}  J^n_\mu [\psi]n^\mu +\mathcal{E}(\mathcal{G}_2\psi,\psi) \leq \frac{E_{data}}{2}
\end{equation}
We inductively obtain that for any even number~$i$ we have that 
\begin{equation}\label{eq: sec: cor: main theorem, relat. non-deg, cor 3, eq 0.9}
	 \int_{\{t^\star=t_i+T\}} J^n_\mu[\psi]n^\mu+\int_{t_i+T}^\infty \int_{\{t^\star=\tau\}}  J^n_\mu [\psi]n^\mu +\mathcal{E}(\mathcal{G}_i\psi,\psi) \leq \frac{E_{data}}{2^{i-1}},\qquad i=2,4,6\dots
\end{equation}

By using the relatively non degenerate estimate~\eqref{eq: sec: cor: main theorem, relat. non-deg, cor 3, eq 0.1}, for odd numbers~$i=1,3,5,\dots$, and the established estimate~\eqref{eq: sec: cor: main theorem, relat. non-deg, cor 3, eq 0.9}, it is easy to extend the estimate~\eqref{eq: sec: cor: main theorem, relat. non-deg, cor 3, eq 0.9} for any natural number~$i$ and obtain that the following holds  
\begin{equation}\label{eq: sec: cor: main theorem, relat. non-deg, cor 3, eq 0.10}
	\int_{\{t^\star=t_i+T\}} J^n_\mu[\psi]n^\mu+\int_{t_i+T}^\infty \int_{\{t^\star=\tau\}}  J^n_\mu [\psi]n^\mu +\mathcal{E}(\mathcal{G}_i\psi,\psi) \leq \frac{E_{data}}{2^{i-1}},\qquad i=1,2,3,\dots
\end{equation}

Now, we note that we can take 
\begin{equation}
	E_{data}=  C \Bigg(\int_{\{t^\star=0\}}  J^n_\mu[\psi]n^\mu+ \int_{\supp \chi_{\textit{data}}}d\tau^\prime  \int_{\{t^\star=\tau^\prime\}} \mathcal{E}(\mathcal{G}_{\chi_{\textit{data}}}\psi,\psi)\Bigg)
\end{equation}
for some constant~$C=C(a,M,l,\mu^2_{KG})$. 

Therefore, we obtain the desired energy decay estimate~\eqref{eq: cor: main theorem, relat. non-deg, cor 3, eq 1} from~\eqref{eq: sec: cor: main theorem, relat. non-deg, cor 3, eq 0.10} and the boundedness estimate~\eqref{eq: sec: cor: main theorem, relat. non-deg, cor 3, eq 0.2}. We can include~$\tau\geq 0$ by taking~$C$ larger if necessary.

\subsection{Proof of the pointwise estimate~\eqref{eq: cor: main theorem, relat. non-deg, cor 3, eq 2} of Corollary~\ref{cor: main theorem, relat. non-deg, cor 3}}

It is convenient to use the algebra of constants already discussed in Section~\ref{subsec: sec: fixed frequency ode estimates, subsec 1}. First, we will need an appropriate exponential decay energy estimate at the level of~$H^2$. Specifically, we can commute the Klein--Gordon equation with~$\partial_t,\partial_\varphi,N$ appropriately many times obtain from~\eqref{eq: cor: main theorem, relat. non-deg, cor 3, eq 1} that the following holds
\begin{equation}\label{eq: sec: cor: main theorem, relat. non-deg, cor 3, eq 1.1}
	\int_{\{t^\star=\tau\}} J^n_\mu [\psi]n^\mu +J^n_\mu [n\psi]n^\mu +J^n[n^2\psi]n^\mu \leq B e^{-c\tau} E
\end{equation}
where
\begin{equation}
	E=\int_{\{t^\star=0\}} J^n_\mu [\psi]n^\mu+J^n_\mu [n\psi]n^\mu +J^n[n^2\psi]n^\mu+J^n[n^3\psi]n^\mu.
\end{equation}

Now, for
\begin{equation}
	\underline{\psi}(\tau) = \frac{1}{|\{t^\star=\tau\}|}  \int_{\{t^\star=\tau\}} \psi(\tau,r,\theta,\varphi) v(r,\theta)dr d\theta d\varphi
\end{equation}
it is straightforward to note that 
\begin{equation}\label{eq: sec: cor: main theorem, relat. non-deg, cor 3, eq 1}
\begin{aligned}
|\partial_\tau \underline{\psi}(\tau)|^2 &	=\left| \partial_\tau \frac{1}{|\{t^\star=\tau\}|}  \int_{\{t^\star=\tau\}} \psi(\tau,r,\theta,\varphi) v(r,\theta)dr d\theta d\varphi \right|^2 \leq B  \left| \int_{\{t^\star=\tau\}} \partial_\tau \psi(\tau,r,\theta,\varphi) v(r,\theta)dr d\theta d\varphi \right|^2 \\
&	\leq B \int_{\{t^\star=\tau\}} (\partial_t \psi)^2,
\end{aligned}
\end{equation}
where in the last estimate we used a H\"older inequality for integrals, in view of that for any~$\tau\geq 0 $ we have~$|\{t^\star=\tau\}|=\int_{r_+}^{\bar{r}_+}\int_{\mathbb{S}^2}dr \sin\theta d\theta d\varphi \sim 1$, where also note that~$|\{t^\star=\tau\}|$ is independent of~$\tau$. Therefore, in view of~\eqref{eq: sec: cor: main theorem, relat. non-deg, cor 3, eq 1} and of the exponential decay of the energy~\eqref{eq: cor: main theorem, relat. non-deg, cor 3, eq 1} we obtain
\begin{equation}\label{eq: sec: cor: main theorem, relat. non-deg, cor 3, eq 2}
|\partial_\tau \underline{\psi}(\tau)|^2 \leq B \Bigg(\int_{\{t^\star=0\}}  J^n_\mu[\psi]n^\mu+ \int_{\supp \chi_{\textit{data}}}d\tau^\prime  \int_{\{t^\star=\tau^\prime\}} \mathcal{E}(\mathcal{G}_{\chi_{\textit{data}}}\psi,\psi)\Bigg) e^{-c\tau} = B\cdot E_0 e^{-c\tau}.
\end{equation}
for some~$c=c(a,M,l,\mu_{KG})>0$, where
\begin{equation}
	E_0 := \int_{\{t^\star=0\}}  J^n_\mu[\psi]n^\mu+ \int_{\supp \chi_{\textit{data}}}d\tau^\prime  \int_{\{t^\star=\tau^\prime\}} \mathcal{E}(\mathcal{G}_{\chi_{\textit{data}}}\psi,\psi)\leq B \int_{\{t^\star=0\}} J^n_\mu [\psi] n^\mu + J^n_\mu [n\psi] n^\mu.
\end{equation}

Second, we prove that the function~$\underline{\psi}(\tau)$ converges for~$\tau\rightarrow \infty$. Specifically, by using the fundamental theorem of calculus we obtain the following 
\begin{equation}\label{eq: sec: cor: main theorem, relat. non-deg, cor 3, eq 3}
\begin{aligned}
|\underline{\psi}(t_2) -\underline{\psi}(t_1)|\leq \int_{t_1}^{t_2} |\partial_\tau \underline{\psi}(\tau)|d\tau \leq \int_{t_1}^\infty |\partial_\tau \underline{\psi}(\tau)|d\tau \leq B \sqrt{E_0} e^{- c  t_1}
\end{aligned}
\end{equation}
for any~$t_1\leq t_2$, where in the last inequality we used the energy decay estimate~\eqref{eq: cor: main theorem, relat. non-deg, cor 3, eq 1}. In view of inequality~\eqref{eq: sec: cor: main theorem, relat. non-deg, cor 3, eq 3} we apply a Cauchy sequence limiting argument and conclude that the function~$\underline{\psi}(\tau)$ indeed converges pointwise with an exponential rate to a constant
\begin{equation}
	\underline{\psi}
\end{equation}
namely the following holds
\begin{equation}\label{eq: sec: cor: main theorem, relat. non-deg, cor 3, eq 3.1}
	|\underline{\psi}(\tau)- \underline{\psi}|\leq B \sqrt{E_0} e^{-c \tau}
\end{equation}
for all~$\tau\geq 0$. Moreover, from~\eqref{eq: sec: cor: main theorem, relat. non-deg, cor 3, eq 3} and in view of that~$\underline{\psi}(\tau)\rightarrow \underline{\psi}$ we obtain that 
\begin{equation}
	|\underline{\psi}| \leq |\underline{\psi}(0)|+ B \sqrt{E_0}.
\end{equation}
Note that in the case where the Klein--Gordon mass is non zero~$\mu^2_{KG}>0$ we have
\begin{equation}
	\underline{\psi}=0. 
\end{equation}

By the Poincare--Wirtinger inequality of Section~\ref{subsec: poincare wirtigner}, we obtain
\begin{equation}\label{eq: sec: cor: main theorem, relat. non-deg, cor 3, eq 4}
\begin{aligned}
\int_{\{t^\star=\tau\}} \left|\psi -\underline{\psi}(\tau) \right|^2 \leq B\int_{\{t^\star=\tau\}} J^n_\mu [\psi]n^\mu 
\end{aligned}
\end{equation}
where in the last inequality we used the energy decay~\eqref{eq: cor: main theorem, relat. non-deg, cor 3, eq 1} and the Poincare--Wirtinger inequality of Section~\ref{subsec: poincare wirtigner}.

Now, for any~$s>\frac{n}{2}$ we recall the classical Sobolev embedding
\begin{equation}\label{eq: sec: cor: main theorem, relat. non-deg, cor 3, eq 5.1}
	\|\phi\|_{L^\infty(I)}\lesssim \|\phi\|_{H^s(I)}.
\end{equation}
where~$I\subset R^n$ is compact, see the book by Brezis~\cite{Brezis}. We use the Sobolev embedding~\eqref{eq: sec: cor: main theorem, relat. non-deg, cor 3, eq 5.1} for
\begin{equation}
	\phi=\psi-\underline{\psi}(\tau),\qquad n=3,\qquad s=2,\qquad I=\{t^\star=\tau\}
\end{equation}
in conjunction with~\eqref{eq: sec: cor: main theorem, relat. non-deg, cor 3, eq 1} to obtain the following
\begin{equation}\label{eq: sec: cor: main theorem, relat. non-deg, cor 3, eq 6}
\begin{aligned}
\sup_{\{t^\star=\tau\}}\left|\psi -\underline{\psi}(\tau) \right|^2  \leq B\int_{\{t^\star=\tau\}} J^n_\mu [\psi]n^\mu +J^n_\mu[n\psi]n^\mu + J^n_\mu[n^2\psi]n^\mu,
\end{aligned}
\end{equation}
where we also used the Poincare--Wirtinger~\eqref{eq: sec: cor: main theorem, relat. non-deg, cor 3, eq 4}.

Now, we note that 
\begin{equation}\label{eq: sec: cor: main theorem, relat. non-deg, cor 3, eq 7}
	\begin{aligned}
		\sup_{\{t^\star=\tau\}}\left|\psi -\underline{\psi}\right|^2 &	\leq B \left(  \sup_{\{t^\star=\tau\}}\left|\psi -\underline{\psi}(\tau) \right|^2 + \sup_{\{t^\star=\tau\}}\left|\underline{\psi}(\tau) -\underline{\psi}\right|^2 \right) \\
		&	\leq B\int_{\{t^\star=\tau\}}\left( J^n_\mu [\psi]n^\mu +J^n_\mu[n\psi]n^\mu + J^n_\mu[n^2\psi]n^\mu\right)+B E_0 e^{-c \tau}
	\end{aligned}
\end{equation}
where in the last inequality we used~\eqref{eq: sec: cor: main theorem, relat. non-deg, cor 3, eq 6} and~\eqref{eq: sec: cor: main theorem, relat. non-deg, cor 3, eq 3.1}. 

Finally, we recall the exponential decay~\eqref{eq: sec: cor: main theorem, relat. non-deg, cor 3, eq 1.1} and therefore we obtain from~\eqref{eq: sec: cor: main theorem, relat. non-deg, cor 3, eq 7} that the following holds 
\begin{equation}
	\begin{aligned}
			\sup_{\{t^\star=\tau\}}\left|\psi -\underline{\psi}\right|^2 \leq B e^{-c\tau} \int_{\{t^\star=0\}} J^n_\mu [\psi]n^\mu +J^n_\mu[n\psi]n^\mu+J^n[n^2\psi]n^\mu+J^n[n^3\psi]n^\mu.
	\end{aligned}
\end{equation}

We conclude~\eqref{eq: cor: main theorem, relat. non-deg, cor 3, eq 2} and the proof of the Corollary.

\section{Proof of Theorem~\ref{thm: main thm extended region} and Corollary~\ref{cor: thm: main thm extended region, cor 2}}\label{sec: proof of Theorem in extended region}

In Theorem~\ref{thm: main thm extended region} we study the inhomogeneous wave equation
\begin{equation}\label{eq: sec: proof of Theorem in extended region, eq 1}
	\Box_{g_{a,M,l}}\psi=F,
\end{equation}
where~$F$ is a sufficiently regular function. 

Let~$T\geq 3$ and let~$3\leq\tilde{T}\leq \tau_1<\tau_1+T\leq \tau_2$. Let
\begin{equation}
\eta(t^\star)=\eta^{(T)}_{\tau_1}(t^\star),\qquad \chi_+(t^\star) =\chi_{\tau_1,\tau_1+T^2}(t^\star),\qquad \chi_{-}(t^\star)=\chi_{\tau_1,\tau_1+T^2}(t^\star+\tilde{T}),	
\end{equation}
see Section~\ref{sec: carter separation, radial}. We appropriately cut-off the solutions of the inhomogeneous wave equation~\eqref{eq: sec: proof of Theorem in extended region, eq 1} and we study the following equation
\begin{equation}\label{eq: sec: proof of Theorem in extended region, eq 2}
\Box_{g_{a,M,l}}\left(\eta\chi_+^2\psi\right)=\eta \chi_+^2 F+ \chi_+^2\psi\Box_{g_{a,M,l}}\eta+2\nabla^\alpha\eta\nabla_\alpha(\chi_+^2\psi).
\end{equation}

We use Carter's separation of variables, see Proposition~\ref{prop: Carters separation, radial part}, for the cut-offed equation~\eqref{eq: sec: proof of Theorem in extended region, eq 2} and obtain the fixed frequency ode
\begin{equation}\label{eq: sec: proof of Theorem in extended region, eq 3}
	u^{\prime\prime}+\left(\omega^2-V\right)u=H,
\end{equation}
where~$u$ here is the frequency localization of the function~$\eta\chi_+^2\psi$, as explained in Section~\ref{sec: carter separation, radial}, and
\begin{equation}\label{eq: sec: proof of Theorem in extended region, eq 4}
	H^{(a\omega)}_{m\ell} =\frac{\Delta}{(r^2+a^2)^{3/2}}\left(\rho^2 \tilde{F}\right)^{(a\omega)}_{m\ell},\qquad \tilde{F}=\eta\chi_+^2 F+ \chi_+^2\psi\Box_{g_{a,M,l}}\eta+2\nabla^\alpha\eta\nabla_\alpha(\chi_+^2\psi). 
\end{equation}

\begin{proof}[\textbf{Proof of Theorem~\ref{thm: main thm extended region}}]

Since the function~$u$ is a smooth solution of the inhomogeneous ode~\eqref{eq: sec: proof of Theorem in extended region, eq 3}, and the boundary conditions~\eqref{eq: BC} are satisfied, see Proposition~\ref{prop: Carters separation, radial part}, we can use the fixed frequency result of Theorem~\ref{thm: subsec: sec: proof of Theorem 2, subsec 4.1, thm 1} and obtain 
\begin{equation}\label{eq: proof thm: main thm extended region, eq 0}
	\begin{aligned}
		&   \int_{r_+}^{r_++\epsilon}(r-r_+)^{-1}\left|v^\prime+i(\omega-\omega_+m) v\right|^2dr+\int_{\bar{r}_+-\epsilon}^{\bar{r}_+}(\bar{r}_+-r)^{-1}\left|v^\prime-i(\omega-\bar{\omega}_+m) v\right|^2dr\\
		&   + \int_{r_+}^{\bar{r}_+}\left( \frac{\tilde{\lambda}}{r^3} |v|^2+\frac{r^2+a^2}{r\Delta}\left(\omega-\frac{am\Xi}{r^2+a^2}\right)^2|v|^2+\frac{1}{r}|v^\prime|^2 +\Delta^2 \omega^2 1_{\{(\omega,m,\tilde{\lambda}):r_{trap}\neq 0\}} |u|^2\right)dr  \\
		&	\qquad\qquad \leq B \Bigg( \int_{r_+}^{r_++\epsilon}\frac{1}{\Delta^2}|u^\prime+i(\omega-\omega_+m)u|^2dr +\int_{\bar{r}_+-\epsilon}^{\bar{r}_+}\frac{1}{\Delta^2}|u^\prime-i(\omega-\bar{\omega}_+m)u|^2 dr \\
		&	\qquad\qquad\qquad\qquad+ \int_{r_+}^{\bar{r}_+} \left(1_{\{|m|>0\}}|u|^2+ \mu^2_{\textit{KG}}|u|^2+ |u^\prime|^2+\left(1-\frac{r_{\textit{trap}}(\omega,m,\ell)}{r}\right)^2(\omega+\lambda^{(a\omega)}_{m\ell})|u|^2\right)dr\Bigg)\\		
		&	\qquad\qquad\qquad +B\left(\omega^2+m^2\right)|u|^2(-\infty)+B\left(\omega^2+m^2\right)|u|^2(+\infty)\\
		&	\qquad\qquad\qquad  + B \int_{\mathbb{R}} \left( \frac{1}{\Delta}\left|H\right|^2+r\left|\left(g_1\frac{d}{dr^\star}+ig_2\right)H\right|^2\right)dr^\star ,
	\end{aligned}
\end{equation}
where for~$H$ see~\eqref{eq: sec: proof of Theorem in extended region, eq 4}. We sum~\eqref{eq: proof thm: main thm extended region, eq 0} over~$\int_{\mathbb{R}}d\omega \sum_{m,\ell}$ and follow verbatim the steps of the proof of Theorem~\ref{main theorem, relat. non-deg}, see Section~\ref{sec: proof of Thm: rel-nondeg}, to obtain the folllowing 
\begin{equation}\label{eq: proof thm: main thm extended region, eq 0.1}
	\begin{aligned}
		& \int_{\{t^\star=\tau_2\}} J^n_\mu[\psi]n^\mu \\
		&	+\int\int_{D(\tau_1+T,\tau_2)}\Bigg(\frac{1}{r}\frac{r^2+a^2}{\Delta}\left|W\mathcal{G}_{\chi_+}\psi\right|^2+ \frac{1}{r}\frac{\Delta}{(r^2+a^2)}|Z^\star\mathcal{G}_{\chi_+}\psi|^2+\frac{1}{r}\frac{r^2+a^2}{\Delta}|W\mathcal{G}_{\chi_+}\psi|^2+\frac{1}{r}|\slashed{\nabla}\mathcal{G}_{\chi_+}\psi|^2\\
		&	\qquad\qquad\qquad\qquad\qquad+ J^n_\mu[\psi] n^\mu\Bigg)\\
		\quad &\leq B\int_{\{t^\star=\tau_1\}} J^n_\mu[\psi]n^\mu+\frac{B}{T}\int\int_{D(\tau_1,\tau_1+T)} J_\mu^{W}[\mathcal{G}_{\chi_{-}}\psi]n^\mu \\
		&	\qquad + B\int\int_{\mathcal{M}}  |\eta\chi_+^2F|^2+\Delta|\mathcal{G}\left(\eta\chi_+^2 F\right)|^2,
	\end{aligned}    
\end{equation} 
where note that in the above~\eqref{eq: proof thm: main thm extended region, eq 0.1} we kept track of the errors terms. Recall from the similar steps of Section~\ref{sec: proof of Thm: rel-nondeg}  that, in order to obtain~\eqref{eq: proof thm: main thm extended region, eq 0.1}, we also need to use the Morawetz and boundedness estimates of Theorem~\ref{main theorem 1}.

Now, we want to extend~\eqref{eq: proof thm: main thm extended region, eq 0.1} in the domain~$D_\delta(\tau_1,+\infty)$, see Section~\ref{subsec: causal domains}, for some sufficiently small $\delta(a,M,l,\mu_{KG},j)>0$. We recall from Definition~\ref{def: sec: G, def 1} that~$\mathcal{G}$ identically vanishes in~$\{r\leq r_+\}\cup\{\bar{r}_+\leq r\}$, so the terms of~\eqref{eq: proof thm: main thm extended region, eq 0.1} that include the operator~$\mathcal{G}$ are easy to extend to~$D_\delta(\tau_1,+\infty)$. To extend the low order terms to the desired domain~$D_\delta (\tau_1,+\infty)$ we use Corollary~\ref{cor: main theorem 1, cor 1}.

We conclude that for~$l>0$ and~$(a,M)\in\mathcal{B}_l$ there exists a sufficiently small~$\delta(a,M,l,j)>0$ such that for~$T>0$ sufficiently large and~$\tilde{T}< \tau_1<\tau_1+T< \tau_2$ we obtain that 
\begin{equation}\label{eq: proof thm: main thm extended region, eq 1}
\begin{aligned}
& b\int_{\{t^\star=\tau_2\}}J^n_\mu[\psi]n^\mu  \\
&	+b\int\int_{D_\delta(\tau_1+T,\tau_2)}dg\Bigg( \frac{1}{r}\frac{\Delta}{(r^2+a^2)}|Z^\star\mathcal{G}_{\chi_+}\psi|^2+\frac{1}{r}\frac{r^2+a^2}{\Delta}|W\mathcal{G}_{\chi_+}\psi|^2+\frac{1}{r}|\slashed{\nabla}\mathcal{G}_{\chi_+}\psi|^2+ J^n_\mu[\psi] n^\mu\Bigg)\\
\quad &\leq \int_{\{t^\star=\tau_1\}} J^n_\mu[\psi]n^\mu+|\psi|^2+\frac{1}{T}\int\int_{D_\delta(\tau_1,\tau_1+T)} J_\mu^{W}[\mathcal{G}_{\chi_{-}}\psi]n^\mu\\
&\qquad\qquad +\int\int_{\mathcal{M}_\delta} |\eta\chi_+^2 F|^2+\Delta|\mathcal{G}\left(\eta\chi_+^2 F \right)|^2
\end{aligned}    
\end{equation}
which concludes the proof of Theorem~\ref{thm: main thm extended region}. 
\end{proof}

Now, we prove the higher order Corollary~\ref{cor: thm: main thm extended region, cor 2}
\begin{proof}[\textbf{Proof of Corollary~\ref{cor: thm: main thm extended region, cor 2}}]

This is follows easily as a Corollary of Theorem~\ref{thm: main thm extended region} and of the higher order Corollary~\ref{cor: main theorem 1, cor 1}~(a corollary of the Morawetz estimate of Theorem~\ref{main theorem 1}). 
\end{proof}

\section{The axisymmetric case: A physical space proof}\label{sec: proof of Thm: rel-nondeg, axisymmetry}

We use the generic constants
\begin{equation}
	b(a,M,l,\mu_{\textit{KG}})>0,\qquad B(a,M,l,\mu_{\textit{KG}})>0.
\end{equation}
with the algebra of constants
\begin{equation}
	b+b=b\qquad b\cdot b=b,\qquad B+B=B,\qquad B\cdot B=B. 
\end{equation}

Recall from Section~\ref{subsec: sec: morawetz estimate, subsec 5} that
\begin{equation}
	\mathcal{G}=\mathcal{G}\big|_{m=0}
\end{equation}
is a physical space vector field, and it takes the form
\begin{equation}
	\mathcal{G}:=\mathcal{G}_{m=0}=\: g_1(r)\partial_{r^\star}+G_2(r)\partial_{t^\star}
\end{equation}
where 
\begin{equation}
	G_2=
	\begin{cases}
		\sqrt{G_2^2},\qquad r\geq r_{\Delta,\textit{frac}}\\
		-\sqrt{G_2^2},\qquad r\leq r_{\Delta,\textit{frac}},
	\end{cases}
\end{equation}
with~$G_2^2=\frac{(r^2+a^2)^2}{\Delta}-\min_{[r_+,\bar{r}_+]}\frac{(r^2+a^2)^2}{\Delta}$.

We need the following Lemma

\begin{lemma}\label{lem: sec: expo decay on axisymmetry, lem 1}
	Let~$l>0$,~$(a,M)\in\mathcal{B}_l$. 
	
	We have that 
	\begin{equation}
		\frac{2V_{\textit{axi}}G_2^\prime}{g_1(r)}\geq b \frac{r^2+a^2}{r\Delta},
	\end{equation}
	where~$V_{axi}=\frac{(r^2+a^2)^2}{\Delta},~g_1(r)=\frac{r^2+a^2}{\sqrt{\Delta}}$, and moreover~$G_2$ is twice differentiable at~$r_{\Delta,frac}$.
\end{lemma}
\begin{proof}
	Direct from the definitions of Section~\ref{subsec: sec: morawetz estimate, subsec 5}.
\end{proof}

We now state the following theorem, for axisymmetric solutions~$\psi$ of the Klein--Gordon equation~\eqref{eq: kleingordon}. Note that there is no need to assume the assumption~(MS).

\begin{theorem}

Let~$l>0$,~$(a,M)\in\mathcal{B}_l$ and~$\mu^2_{\textit{KG}}\geq 0$. Then, there exists a constant 
\begin{equation}
	C(a,M,l,\mu^2_{\textit{KG}})>0
\end{equation}
such the following holds.

Let~$\psi$ satisfy the Klein--Gordon equation~\eqref{eq: kleingordon} in~$D(0,\infty)$. Then, we have the following
	\begin{equation}
	\begin{aligned}
	\int_{\{t^\star=\tau\}}\mathcal{E}(\mathcal{G}\psi,\psi)+ \int\int_{D(\tau_1,\tau_2)} \mathcal{E}(\mathcal{G}\psi,\psi) \leq C\int_{\{t^\star=\tau_1\}} \mathcal{E}(\mathcal{G}\psi,\psi),
	\end{aligned}
\end{equation}
for any~$0\leq \tau_1\leq \tau\leq \tau_2$, where for the operator~$\mathcal{G}$ in the case of axisymmetry see Section~\ref{sec: G}. For the energy density~$\mathcal{E}(\cdot,\cdot)$ see~\eqref{eq: new energy}. 
\end{theorem}

\begin{proof}

	We sketch the next steps, since they are similar to the arguments of \cite{mavrogiannis}. 	
	
	Let~$\psi$ be an axisymmetric solution of the Klein--Gordon equation \eqref{eq: kleingordon}, namely 
	\begin{equation}\label{eq: the equation for psi in kerr de sitter for axisymmetry}
		\Box_{g_{a,M,l}}\psi -\mu_{\textit{KG}}^2\psi=0,\qquad \partial_{\varphi^\star} \psi=0.
	\end{equation}
	Note that for axisymmetric solutions~$\psi$ of the Klein--Gordon equation~\eqref{eq: kleingordon} the following holds 
	\begin{equation}
		\begin{aligned}
			\Box \phi \equiv -\frac{(r^2+a^2)^2}{\rho^2 \Delta}\left(\partial^2_{t^\star}{\phi} -(r^2+a^2)^{-1}\partial_{r^\star}\left((r^2+a^2)\partial_{r^\star}\phi\right) \right) +\frac{a^2\sin^2\theta}{\rho^2 \Delta_\theta}\partial_{t^\star}^2{\phi}+\frac{1}{\rho^2 \sin\theta}\partial_\theta\left(\Delta_\theta \sin\theta \partial_\theta \phi\right).
		\end{aligned}
	\end{equation}
	Therefore, the Klein--Gordon equation~\eqref{eq: kleingordon} in the case of axisymmetry, reduces to the following equation
	\begin{equation}
		\begin{aligned}
			\mathcal{R}\psi=\rho^2\mu_{\textit{KG}}^2\psi,
		\end{aligned}
	\end{equation}
	\begin{equation}\label{eq: proof energy estimate for axisymmetry exp decay, 2}
		\mathcal{R}\Psi \equiv -\frac{(r^2+a^2)^2}{\Delta}\left(\partial_{t^\star}^2\Psi -(r^2+a^2)^{-1}\partial_{r^\star}\left((r^2+a^2)\partial_{r^\star}\Psi\right) \right) +\frac{a^2\sin^2\theta}{\Delta_\theta}\partial_{t^\star}^2{\Psi}+\frac{1}{\sin\theta}\partial_\theta\left(\Delta_\theta \sin\theta \partial_\theta \Psi\right),
	\end{equation}
	where~$\rho^2=r^2+a^2\cos^2\theta$. 
	
	Now, we note the following commutation 
	\begin{equation}\label{eq: proof energy estimate for axisymmetry exp decay, 3}
		\begin{aligned}
			\left[\mathcal{R},g_1\partial_{r^\star}+G_2\partial_t\right]\psi &= \partial_{r^\star}^2\psi \left(2V_{\textit{axi}} g_1^\prime -g_1 V_{\textit{axi}}^\prime\right)+2V_{\textit{axi}} G_2^\prime \frac{\partial_{t^\star}{\mathcal{G}\psi}}{g_1}+\partial_{t^\star}^2{\psi}\left(-2 G_2^\prime\frac{G_2}{g_1}V_{\textit{axi}}+g_1 V_{\textit{axi}}^\prime\right)\\
			&  \quad+\partial_{t^\star}{\psi}\left(V_{\textit{axi}}G_2^{\prime\prime}+V_{\textit{axi}}(r^2+a^2)^{-1}(r^2+a^2)^\prime  G_2^\prime\right)\\
			&  \quad +\partial_{r^\star}\psi\Big(V_{\textit{axi}} g_1^{\prime\prime}+V_{\textit{axi}}(r^2+a^2)^{-1}(r^2+a^2)^\prime g_1^\prime \\
			&\quad\quad\quad\quad\quad -g_1 V_{\textit{axi}}^\prime (r^2+a^2)^{-1} (r^2+a^2)^2-g_1 V_{\textit{axi}} \left((r^2+a^2)^{-1}(r^2+a^2)^\prime\right)^\prime\Big),
		\end{aligned}
	\end{equation}
	where~$V_{\textit{axi}}	=\frac{(r^2+a^2)^2}{\Delta}. $

	Τhe components of the second order terms~$\partial_{t^\star}^2{\psi},\partial_{r^\star}^2\psi$ in~\eqref{eq: proof energy estimate for axisymmetry exp decay, 3} are identically zero, see Section~\ref{sec: G}. We obtain
	\begin{equation}\label{eq: proof energy estimate for axisymmetry exp decay, 4}
		\begin{aligned}
			\left[\mathcal{R},g_1\partial_{r^\star}+G_2\partial_t\right]\psi &= 2V_{\textit{axi}}G_2^\prime \frac{\partial_{t^\star}{\mathcal{G}\psi}}{g_1}+\partial_{t^\star}{\psi}E_1(r^\star) +\partial_{r^\star}\psi E_2(r^\star),
		\end{aligned}
	\end{equation}
	where 
	\begin{equation}
		\begin{aligned}
			E_1(r^\star) & =\left(V_{\textit{axi}}G_2^{\prime\prime}+V_{\textit{axi}}(r^2+a^2)^{-1}(r^2+a^2)^\prime G_2^\prime\right),\\
			E_2(r^\star) &  = \Big(V_{\textit{axi}} g_1^{\prime\prime}+V_{\textit{axi}}(r^2+a^2)^{-1}(r^2+a^2)^\prime g_1^\prime \\
			&\quad\quad\quad\quad -g_1 V_{\textit{axi}}^\prime (r^2+a^2)^{-1} (r^2+a^2)^2-g_1 V_{\textit{axi}}\left((r^2+a^2)^{-1}(r^2+a^2)^\prime\right)^\prime\Big).
		\end{aligned}
	\end{equation}

	Therefore, in view of~\eqref{eq: proof energy estimate for axisymmetry exp decay, 4} we compute 
	\begin{equation}\label{eq: commutation in axisymmetry}
		\begin{aligned}
			\Box\mathcal{G}\psi	&	=\frac{1}{\rho^2}\left[\mathcal{R},\mathcal{G}\right]\psi-\mathcal{G}\left(\frac{1}{\rho^2}\right) \mathcal{R}\psi+\mathcal{G}\Box\psi\\
			&	=\frac{1}{\rho^2}\left[\mathcal{R},\mathcal{G}\right]\psi-\mathcal{G}\left(\frac{1}{\rho^2}\right)\rho^2 \mu_{\textit{KG}}^2\psi +\mu_{\textit{KG}}^2 \mathcal{G}\psi\\
			&	=\frac{1}{\rho^2}\Bigg(2V_{\textit{axi}}G_2^\prime \frac{\partial_{t^\star}{\mathcal{G}\psi}}{g_1}+\partial_{t^\star}{\psi}\left(V_{\textit{axi}} G_2^{\prime\prime}+V_{\textit{axi}}(r^2+a^2)^{-1}(r^2+a^2)^\prime  G_2^\prime\right)\\
			&  \qquad\quad +\frac{d}{d r^\star}\psi\Big(V_{\textit{axi}}  g_1^{\prime\prime}+V_{\textit{axi}}(r^2+a^2)^{-1}(r^2+a^2)^\prime  g_1^\prime\\
			&  \qquad\qquad\qquad\quad-g_1 V_{\textit{axi}}^\prime (r^2+a^2)^{-1} (r^2+a^2)^2 -g_1 V_{\textit{axi}}\left((r^2+a^2)^{-1}(r^2+a^2)^\prime\right)^\prime\Big)\Bigg)\\
			&	\quad -\mathcal{G}\left(\frac{1}{\rho^2}\right)\rho^2 \mu_{\textit{KG}}^2\psi +\mu_{\textit{KG}}^2 \mathcal{G}\psi.\\
			&
		\end{aligned}
	\end{equation}
	We use the divergence Theorem, see Lemma~\ref{lem: divergence theorem}, with multiplier~$W=\partial_{t^\star}+\frac{a\Xi}{r^2+a^2} \partial_{\varphi^\star}$, see Lemma~\ref{lem: causal vf E,1}, to the equation~\eqref{eq: commutation in axisymmetry}, where recall~$\rm{K}^{\partial_{t^\star}+\frac{a\Xi}{r^2+a^2} \partial_{\varphi^\star}}=0$, see Lemma~\ref{lem: causal vf E,1}, in view of axisymmetry
	\begin{equation}
		\partial_{\varphi^\star}\psi=0.
	\end{equation}
	We conclude that
	\begin{equation}\label{eq: energy estimate for axisymmetry exp decay, 1}
		\begin{aligned}
			\int_{\{t^\star=\tau_2\}}J^{W}_\mu[\mathcal{G}\psi]n^\mu+ \int\int_{D(\tau_1,\tau_2)} 2\frac{V_{\textit{axi}} G_2^\prime}{\rho^2 g_1}\left|\partial_{t^\star}{\mathcal{G}\psi}\right|^2  &  =\int_{\{t^\star=\tau_1\}} J^{W}_\mu[\mathcal{G}\psi]n^\mu\\
			&   +\int\int_{D(\tau_1,\tau_2)}\textit{l.o.t.},
		\end{aligned}
	\end{equation}
	for all~$0\leq\tau_1\leq\tau_2$ where
	\begin{equation}\label{eq: energy estimate for axisymmetry exp decay, 2}
		\begin{aligned}
			\textit{l.o.t.}&	=\frac{1}{\rho^2}\left(E_1(r^\star)\partial_t\psi+E_2\partial_{r^\star}\psi\right)\left(\partial_{t^\star}+\frac{a\Xi}{r^2+a^2} \partial_{\varphi^\star}\right)\mathcal{G}\psi-\mathcal{G}\left(\frac{1}{\rho^2}\right)\rho^2 \mu_{\textit{KG}}^2\psi +\mu_{\textit{KG}}^2 \mathcal{G}\psi
		\end{aligned}
	\end{equation}
	with~$\rho^2=r^2+a^2\cos^2\theta$. Moreover, note that 
	\begin{equation}
		\int\int_{D(\tau_1,\tau_2)}\textit{l.o.t.} \leq \frac{B}{\epsilon}\int\int_{D(\tau_1,\tau_2)} |\partial_{t^\star}\psi|^2+|Z^\star\psi|^2 +\mu^2_{\textit{KG}}|\psi|^2+B\epsilon \int\int_{D(\tau_1,\tau_2)} \frac{r^2+a^2}{\Delta r} |W\mathcal{G}\psi|^2 ,
	\end{equation}
	where for the vector field~$Z^\star$ see Section~\ref{subsec: boldsymbol partial r},

	Now we use~\eqref{eq: energy estimate for axisymmetry exp decay, 1} and Lemma~\ref{lem: sec: expo decay on axisymmetry, lem 1} in conjunction with a Poincare type estimate, which we do not present here (for a fixed frequency such result see Lemma~\ref{lem: control of the law order derivatives with v}) to obtain the following inequality
	\begin{equation}\label{eq: energy estimate for axisymmetry exp decay, 1.5}
		\begin{aligned}
			\int_{\{t^\star=\tau_2\}}J^{W}_\mu[\mathcal{G}\psi]n^\mu+ \int\int_{D(\tau_1,\tau_2)} \frac{r^2+a^2}{\Delta r}\left|\partial_{t^\star}{\mathcal{G}\psi}\right|^2&  \leq B\int_{\{t^\star=\tau_1\}} J^W_\mu[\mathcal{G}\psi]n^\mu\\
		\end{aligned}
	\end{equation}
	for all~$0\leq\tau_1\leq\tau_2$.

	We use elliptic estimates to generate the following derivatives
	\begin{equation}
		\Delta|Z^\star\mathcal{G}\psi|^2+|\slashed{\nabla}\mathcal{G}\psi|^2
	\end{equation}
	on the left hand side of equation~\eqref{eq: energy estimate for axisymmetry exp decay, 1.5}, by multiplying the equation that~$\mathcal{G}_{m=0}\psi$ satisfies, see~\eqref{eq: commutation in axisymmetry}, with~$\mathcal{G}_{m=0}\psi$, and then using integration by parts. Then, we can generate a large non-degenerate~$\Delta(\partial_{t^\star}\psi)^2$ integrand term on the left hand side of equation~\eqref{eq: energy estimate for axisymmetry exp decay, 1} by using the positivity of the term
	\begin{equation}
		\frac{r^2+a^2}{\Delta r}\left|\partial_{t^\star}{\mathcal{G}\psi}\right|^2. 
	\end{equation}

	We obtain the following
	\begin{equation}\label{eq: energy estimate for axisymmetry exp decay, 3}
		\begin{aligned}
			&   \int_{\{t^\star=\tau_2\}}J_\mu^{W}[\mathcal{G}\psi]n^\mu +J^{n}_\mu[\psi]n^\mu\\
			&   \quad+\int\int_{D(\tau_1,\tau_2)} \frac{r^2+a^2}{\Delta r}|\partial_{t^\star}\mathcal{G}\psi|^2+\frac{\Delta}{r (r^2+a^2)}|Z^\star\mathcal{G}\psi|^2+\frac{1}{r}|\slashed{\nabla}\mathcal{G}\psi|^2+|\partial_{t^\star}\psi|^2+|Z^\star\psi|^2+|\slashed{\nabla}\psi|^2\\
			&   \quad\quad \leq B\int_{\{t^\star=\tau_1\}}J_\mu^W[\mathcal{G}\psi]n^\mu + J^{n}_\mu[\psi]n^\mu,
		\end{aligned}
	\end{equation}
	which immediately concludes 
	\begin{equation}
		\begin{aligned}
			\int_{\{t^\star=\tau\}}\mathcal{E}(\mathcal{G}\psi,\psi)+ \int\int_{D(\tau_1,\tau_2)} \mathcal{E}(\mathcal{G}\psi,\psi) \leq C\int_{\{t^\star=\tau_1\}} \mathcal{E}(\mathcal{G}\psi,\psi),
		\end{aligned}
	\end{equation}
	for all~$\tau_1\leq\tau \leq\tau_2$. 
\end{proof}

It is easy to use an iteration argument, see~\cite{mavrogiannis}, to obtain 
\begin{equation}
	\begin{aligned}
		&\int_{\{t^\star=\tau\}} J^n_\mu[\psi]n^\mu+  \int\int_{\{t^\star\geq \tau\}} \mathcal{E}(\mathcal{G}\psi,\psi)  \leq C \Bigg(\int_{\{t^\star=0\}}  J^n_\mu[\psi]n^\mu+\mathcal{E}(\mathcal{G}\psi,\psi) \Bigg) \cdot e^{-c\tau}.\\
	\end{aligned}
\end{equation}
for some constants~$c,C>0$. We can also obtain pointwise bounds from the above.

\appendix

\section{A classical Coifman--Meyer commutation estimate}\label{sec: appendix, pseudodifferential commutation}

We define the symbol class and the operator class we use in this paper. 

\begin{definition}\label{def: sec: appendix, pseudodifferential commutation, def 1}
We define the symbol class 	
\begin{equation}\label{eq: sec: appendix, pseudodifferential commutation, eq 1}
S^1=\{\sigma:\mathbb{R}\times \mathbb{Z}\rightarrow \mathbb{R}:\: \exists C>0~\forall 0\leq i_1+i_2\leq 1~ \left|\partial^{i_1}_\omega\Delta^{i_2}_m \sigma(\omega,m) \right|\leq C\left(1+|\omega|\right)^{1-i_1-i_2}\left(1+|m|\right)^{1-i_1-i_2} \},
\end{equation}
where~$\Delta_m$ is the finite difference operator. 
\end{definition}

Now, we define the following function space
\begin{equation}
	X= \{\Psi: \mathbb{R}\times [0,2\pi)\rightarrow \mathbb{C}:~\Psi~\text{is smooth},~\supp\Psi \subset (a,b)\times [0,2\pi)~\text{for some}~|a|+|b|<\infty\}.
\end{equation}

Let~$\sigma\in S^1$ , see Definition~\ref{eq: sec: appendix, pseudodifferential commutation, eq 1}. We define the following Fourier multiplier operator
\begin{equation}\label{eq: def: sec: appendix, pseudodifferential commutation, def 3}
	\left(\textit{Op}(\sigma)\Psi\right)(t,\varphi):=\int_{\mathbb{R}}\sum_m  g(\omega,m) e^{i\omega t}e^{-im\varphi} \mathcal{F}_{\omega,m} (\Psi) d\omega,
\end{equation}
for any~$\Psi\in X$, where here we define
\begin{equation}
	\begin{aligned}
	\mathcal{F}_{\omega,m}(\Psi)(r,\theta) 	= \int_{\mathbb{R}}\int_0^{2\pi}e^{-i\omega t} e^{im \varphi}\Psi (t,\varphi,r,\theta) d\varphi dt.
	\end{aligned}
\end{equation}
The function~$Op(\sigma)\Psi:\mathbb{R}\times[0,2\pi)\rightarrow\mathbb{C}$ is also smooth.

We define the operator class
\begin{equation}
	OPS^1=\{Op(\sigma):X\rightarrow C^\infty(\mathbb{R}\times[0,2\pi)\rightarrow\mathbb{C}),~\sigma\in S^1\}. 
\end{equation}

The following result is a classical pseudodifferential commutator estimate

\begin{lemma}\label{lem: sec: appendix, pseudodifferential com, lem 0}
	Let~$P\in \textit{OPS}^1$. Then, there exists a constant~$C(P)>0$ such that if~$\chi:\mathbb{R}\rightarrow \mathbb{R}$ is a smooth function, then we have the following. 
	
	For any smooth~$f\in X$ we have the following pseudodifferential commutation
 \begin{equation}\label{eq: lem: sec: appendix, pseudodifferential com, lem 0, eq 1}
 	\|[P,\chi]f\|_{L^2(\mathbb{R}\times[0,2\pi))}\leq C \|\chi\|_{\text{Lip}}\cdot \|f\|_{L^2(\mathbb{R}\times [0,2\pi))},
 \end{equation}
where~$\|\chi\|_{\text{Lip}}=\sup_{t\in\mathbb{R}} \sup_{h\not{=}0}\frac{|\chi(t+h)-\chi(t)|}{|h|}$. 
\end{lemma}
\begin{proof}
The proof follows by slighlty adapting the proofs found in the book of Taylor~[\cite{TaylorPseudodifferentialBook},~Chapter 4]~(also see~\cite{Dyatlov5}), in conjunction with pseudodifferential estimates on the torus~$[0,2\pi)$, see~\cite{Ruzhansky} and specifically see [Theorem 5.2,\cite{Ruzhansky}]. See also the Appendix of~\cite{zworskibook}, where in fact the authors discuss pseudodifferential estimates on general smooth manifolds.
\end{proof}

\begin{remark}
	We emphasize that the constant~$C=C(P)$, on the RHS of the pseudodifferential estimate~\eqref{eq: lem: sec: appendix, pseudodifferential com, lem 0, eq 1}, does not depend on the functions~$\chi,f$. 
\end{remark}

\section{Pseudodifferential commutation estimates on the Kerr--de~Sitter spacetime}\label{subsec: sec: proof of Theorem 2, subsec 5.1}

In Lemma~\ref{lem: proof thm 3 exp decay, lem 1} we prove several commutation properties for the operators~$\mathcal{G},\widetilde{\mathcal{G}}$, see Definitions~\ref{def: sec: G, def 1}, \ref{def: proof thm 3 exp decay, def 1} respectively.

We will prove the estimates of Lemma~\ref{lem: proof thm 3 exp decay, lem 1} by using the standard theory for pseudodifferential operators~(which we discussed in Section~\ref{sec: appendix, pseudodifferential commutation}) and Lemma~\ref{lem: proof thm 3 exp decay, lem 0}.

\begin{lemma}\label{lem: proof thm 3 exp decay, lem 1}
	Let~$l>0$ and~$(a,M)\in\mathcal{B}_l$. Let~$\mathcal{G},\widetilde{\mathcal{G}}$ be as in Definitions~\ref{def: sec: G, def 1}, \ref{def: proof thm 3 exp decay, def 1} respectively.

	Then, there exist constants
	\begin{equation*}
		B(a,M,l)>0,\qquad C(a,M,l)>0
	\end{equation*}
	where the constant~$C=C(a,M,l)>0$ blows up in the limit~$a\rightarrow 0$ but is finite for the value~$a=0$, such that the following hold. 
	
	Let $\chi$ be either equal to~$\chi^2_+$ or~$\eta\chi_+^2$, see Section~\ref{subsec: sec: carter separation, subsec 1}. Let~$f:~\mathcal{M}\rightarrow\mathbb{R}$ be smooth such that
	\begin{equation}
		\supp f\subset \{0\leq t^\star <+\infty\}
	\end{equation}
	and either~$f\in L^2(\mathcal{M})$ or~$f\in H^1(\mathcal{M})$, depending on the particular inequality below. Then, we have the following:
	\begin{equation}\label{eq: thm 3 exp decay, eq 2}
		\int\int_{\mathcal{M}} \left|[\widetilde{\mathcal{G}},\chi]f \right|^2 \leq C	\int\int_{\mathcal{M}} |f|^2,
	\end{equation} 
	\begin{equation}\label{eq: thm 3 exp decay, eq 2.1}
		\int\int_{\mathcal{M}} |(\mathcal{G}-\widetilde{\mathcal{G}}) f|^2\leq B	\int\int_{\mathcal{M}} |f|^2,
	\end{equation}
	\begin{equation}\label{eq: thm 3 exp decay, eq 3}
		\begin{aligned}
			[\partial_t,\mathcal{G}]f=0,\quad [\partial_{\varphi},\mathcal{G}]f=0,\quad  [\partial_\theta,\mathcal{G}]f=0,
		\end{aligned}
	\end{equation}	
	\begin{equation}\label{eq: thm 3 exp decay, eq 4}
		\begin{aligned}
			\int\int_{\mathcal{M}} |[\partial_{r^\star},\mathcal{G}]f|^2\leq B \int\int_{\mathcal{M}} J^n_\mu[f]n^\mu,\qquad \int\int_{\mathcal{M}} \Delta|[Z^\star,\mathcal{G}]f|^2\leq B \int\int_{\mathcal{M}} J^n_\mu[f]n^\mu, 
		\end{aligned}
	\end{equation}	
	\begin{equation}\label{eq: thm 3 exp decay, eq 5}
		\begin{aligned}
			\int\int_{\mathcal{M}} \left|[\partial_\theta\widetilde{\mathcal{G}},\chi]f \right|^2 \leq C 	\int\int_{\mathcal{M}} J^n_\mu [f]n^\mu,
		\end{aligned}
	\end{equation}	
	\begin{equation}\label{eq: thm 3 exp decay, eq 5.1}
		\int\int_{\mathcal{M}} \frac{1}{\Delta}\left| [W\widetilde{\mathcal{G}},\chi] f\right|^2 \leq C\int\int_{\mathcal{M}} J^n_\mu [f]n^\mu +|f|^2,
	\end{equation}
	\begin{equation}\label{eq: thm 3 exp decay, eq 5.2}
		\int\int_{\mathcal{M}} \Delta\left|[Z^\star\mathcal{G},\chi]f \right|^2 \leq B \int\int_{\mathcal{M}} J^n_\mu [f]n^\mu+|f|^2,
	\end{equation}
	where for the vector field~$Z^\star$ see Section~\ref{subsec: boldsymbol partial r}. All the spacetimes volume forms of the spacetime domains of the present Lemma are to be understood as the volume form~\eqref{eq: subsec: volume forms of spacelike hypersurfaces, eq 0} defined in Section~\ref{subsec: admissible hypersurfaces}. 
\end{lemma}

\begin{proof}
	
	In the present proof the constant~$B$, will only depend on~$a,M,l$ and moreover will not blow up in the limit~$a\rightarrow 0$.

	Furthrmore, we will use~$C$ to denote constants that only depend on~$a,M,l,\mu_{KG}$, specifically they do not depend on the coordinate~$r$, and will blow up in the limit~$a\rightarrow 0$, but are finite for~$a=0$.

	Let~$\alpha>0$ be sufficiently small. We first prove inequality~\eqref{eq: thm 3 exp decay, eq 2}, where note that we divide our analysis $\alpha$--close the horizons~$\mathcal{H}^+,\bar{\mathcal{H}}^+$ and away from the horizons.

We note that for~$\chi(t^\star)=\chi(t^\star(t,r))$ as in the assumptions of the present Lemma we have that there exists a constant~$B(a,M,l)>0$ such that for any~$r\in [r_+,\bar{r}_+]$ we have that
	\begin{equation}\label{eq: lem: proof thm 3 exp decay, lem 1, eq 1.0}
		\|\chi\|_{\text{Lip}}(r)\leq B,
	\end{equation}
	since there exists a constant~$B$ such that for any~$r\in [r_+,\bar{r}_+]$ and for any~$t\in \mathbb{R}$ we have~$|\partial_t\chi(t,r)|\leq B$. We used the Lipschitz norm~$\|\chi\|_{\text{Lip}}(r)=\sup_{t\in\mathbb{R}} \sup_{h\not{=}0}\frac{|\chi(t+h,r)-\chi(t,r)|}{|h|}$. Note that~$\partial_\varphi\chi \equiv 0$.

	\begin{center}
		\textbf{The inequality}~\eqref{eq: thm 3 exp decay, eq 2} \textbf{away from the horizons}~$\mathcal{H}^+,~\bar{\mathcal{H}}^+$
	\end{center}

	First we prove inequality~\eqref{eq: thm 3 exp decay, eq 2} away from the horizons. We recall from Lemma~\ref{lem: proof thm 3 exp decay, lem 0}, that there exists a constant~$C>0$ such that for any~$0\leq i_1+i_2\leq 1$ and for any~$r\in(r_+,\bar{r}_+)$ we obtain that
	\begin{equation}
		\left|\partial^{i_1}_\omega\Delta^{i_2}_m \widetilde{g_2}(\omega,m,r) \right|\leq C\left(1+|\omega|\right)^{1-i_1-i_2}\left(1+|m|\right)^{1-i_1-i_2}, 
	\end{equation}
	for~$i_1+i_2\leq 1$ and therefore
	\begin{equation}\label{eq: lem: proof thm 3 exp decay, lem 1, eq 1.02}
		\widetilde{g_2}(\cdot,\cdot,r)\in	S^1=\{\sigma:\mathbb{R}\times \mathbb{Z}\rightarrow \mathbb{R}:\: \exists C>0~\forall 0\leq i_1+i_2\leq 1~ \left|\partial^{i_1}_\omega\Delta^{i_2}_m \sigma(\omega,m) \right|\leq C\left(1+|\omega|\right)^{1-i_1-i_2}\left(1+|m|\right)^{1-i_1-i_2} \}.
	\end{equation}
	Note that the constant~$C>0$ of~\eqref{eq: lem: proof thm 3 exp decay, lem 1, eq 1.02} does not depend on~$r$ and blows up in the limit~$a\rightarrow 0$ but is finite for the value~$a=0$. Therefore, in view of the assumptions on the functions~$f,\chi$, we can use the Coifman--Meyer commutator of Lemma~\ref{lem: sec: appendix, pseudodifferential com, lem 0} and conclude that 
	\begin{equation}\label{eq: lem: proof thm 3 exp decay, lem 1, eq 1}
		\begin{aligned}
			\int_{r_++\alpha}^{\bar{r}_+-\alpha}\int_0^\pi \int_{-\infty}^{\infty}\int_0^{2\pi}\left|[\widetilde{\mathcal{G}},\chi]f \right|^2  \sin\theta d\varphi dt d\theta  dr &	\leq C \int_{r_++\alpha}^{\bar{r}_+-\alpha} \|\chi\|_{Lip}(r) \int_0^\pi\int_{-\infty}^\infty \int_0^{2\pi}  |f|^2\sin\theta d\varphi dt d\theta  dr\\
			& \leq C \int\int_{\mathcal{M}}|f|^2,
		\end{aligned}
	\end{equation}
	where in the last inequality we used the bound~\eqref{eq: lem: proof thm 3 exp decay, lem 1, eq 1.0}.

	\begin{center}
		\textbf{The inequality}~\eqref{eq: thm 3 exp decay, eq 2} \textbf{near the horizons}~$\mathcal{H}^+,~\bar{\mathcal{H}}^+$
	\end{center}

	Now, we prove inequality~\eqref{eq: thm 3 exp decay, eq 2} near the horizons. We will use the result of Lemma~\ref{lem: proof thm 3 exp decay, lem 0}. Note that there exist smooth functions~$h_1(r),h_2(r)$ such that the following holds 
	\begin{equation}\label{eq: lem: proof thm 3 exp decay, lem 1, eq 1.1}
		\partial_{r^\star}=h_1(r)\left(\partial_t +\frac{a\Xi}{r^2+a^2}\partial_{\varphi}\right) +h_2(r) X
	\end{equation}
	with~$h_1(r)=1+\mathcal{O}(\Delta)$ and~$h_2(r)=\mathcal{O}(\Delta)$, where~$X$ here is a smooth vector field on the domain~$\{r_+\leq r\leq \bar{r}_+\}$ where~$[X,\partial_t]=[X,\partial_\varphi]=0$.

	Therefore, by using~\eqref{eq: lem: proof thm 3 exp decay, lem 1, eq 1.1} we note the following 
	\begin{equation}\label{eq: lem: proof thm 3 exp decay, lem 1, eq 1.2}
		\begin{aligned}
			\int_{r_+}^{r_++\alpha} \int_0^\pi\int_{\mathbb{R}}\int_0^{2\pi} |[\widetilde{\mathcal{G}},\chi]f|^2& = \int_{r_+}^{r_++\alpha} \int_0^\pi\int_{\mathbb{R}}\int_0^{2\pi} |[g_1\partial_{r^\star}+i\textit{Op}(\widetilde{g_2}),\chi]f|^2\\
			&	 =\int_{r_+}^{r_++\alpha}\int_0^\pi \int_{\mathbb{R}}\int_0^{2\pi} \left|[g_1 h_1(r)\left(\partial_t +\frac{a\Xi}{r^2+a^2}\partial_{\varphi}\right) +i\textit{Op}(\widetilde{g_2})+g_1 h_2(r) X,\chi]f\right|^2\\
			&	 \leq B \int_{r_+}^{r_++\alpha}\int_0^\pi \int_{\mathbb{R}} \int_0^{2\pi} \left|[g_1 h_1(r)\left(\partial_t +\frac{a\Xi}{r^2+a^2}\partial_{\varphi}\right) +i\textit{Op}(\widetilde{g_2}),\chi]f\right|^2\\
			&	\qquad   + B\int_{r_+}^{r_++\alpha}\int_0^\pi \int_{\mathbb{R}} \int_0^{2\pi}\left|[g_1 h_2(r) X,\chi]f\right|^2,
		\end{aligned} 
	\end{equation}
	where in~\eqref{eq: lem: proof thm 3 exp decay, lem 1, eq 1.2} we suppressed the volume form~$v(r,\theta) d\varphi dt d\theta dr $ for brevity, where for~$v(r,\theta)$ see Section~\ref{subsec: admissible hypersurfaces}.

	The second term on the last line of~\eqref{eq: lem: proof thm 3 exp decay, lem 1, eq 1.2} is easily bounded by a physical space commutation
	\begin{equation}
		\int_{r_+}^{r_++\alpha}\int_0^\pi \int_{\mathbb{R}}\int_0^{2\pi} \left|[g_1 h_2(r) X,\chi]f\right|^2 \leq B \int\int_{\mathcal{M}}|f|^2.
	\end{equation}
	To bound the first term on the last line of~\eqref{eq: lem: proof thm 3 exp decay, lem 1, eq 1.2} we use Lemma~\ref{lem: proof thm 3 exp decay, lem 0} where recall that we proved that for any~$r\in [r_+,r_+]$ the symbol
	\begin{equation}
		- g_1(r)i\left(\omega-\frac{am\Xi}{r_+^2+a^2}\right) + i\widetilde{g_2}(\omega,m,r),
	\end{equation}
	belongs in the symbol class
	\begin{equation}\label{eq: lem: proof thm 3 exp decay, lem 1, eq 1.2.1}
		S^1=\{\sigma:\mathbb{R}\times \mathbb{Z}\rightarrow \mathbb{R}:\: \exists C>0~\forall 0\leq i_1+i_2\leq 1~ \left|\partial^{i_1}_\omega\Delta^{i_2}_m \sigma(\omega,m) \right|\leq C\left(1+|\omega|\right)^{1-i_1-i_2}\left(1+|m|\right)^{1-i_1-i_2} \},
	\end{equation}
	where the constant~$C>0$ of~\eqref{eq: lem: proof thm 3 exp decay, lem 1, eq 1.2.1} blows up in the limit~$a\rightarrow 0$ but is finite for the value~$a=0$. Therefore, in view of the assumptions for the functions~$f,\chi$~(see~\eqref{eq: lem: proof thm 3 exp decay, lem 1, eq 1.0}) we can use the Coifman--Meyer commutator estimate of Lemma~\ref{lem: sec: appendix, pseudodifferential com, lem 0} to obtain
	\begin{equation}\label{eq: lem: proof thm 3 exp decay, lem 1, eq 1.3}
		\int_{r_+}^{r_++\alpha} \int_0^\pi\int_{-\infty}^\infty\int_0^{2\pi} \left|[g_1 h_1(r)\left(\partial_t +\frac{a\Xi}{r^2+a^2}\partial_{\varphi}\right) +i\textit{Op}(\widetilde{g_2}),\chi]f\right|^2 \leq C \int\int_{\mathcal{M}}|f|^2. 
	\end{equation}

	By following similar considerations we also conclude that 
	\begin{equation}\label{eq: lem: proof thm 3 exp decay, lem 1, eq 1.4}
		\int_{\bar{r}_+-\alpha}^{\bar{r}_+} \int_0^\pi\int_{\mathbb{R}} \int_0^{2\pi} |[\widetilde{\mathcal{G}},\chi]f|^2 \leq C \int_{\bar{r}_+-\alpha}^{\bar{r}_+} \int_{\mathbb{R}} \int_0^{2\pi} |f|^2.
	\end{equation}
	Therefore, from~\eqref{eq: lem: proof thm 3 exp decay, lem 1, eq 1},~\eqref{eq: lem: proof thm 3 exp decay, lem 1, eq 1.3},~\eqref{eq: lem: proof thm 3 exp decay, lem 1, eq 1.4} we conclude that 
	\begin{equation}\label{eq: lem: proof thm 3 exp decay, lem 1, eq 2}
		\int_{r_+}^{\bar{r}_+} \int_0^\pi\int_{\mathbb{R}} \int_0^{2\pi} \left|[\widetilde{\mathcal{G}},\chi]f \right|^2 dr dt d\slashed{g} \leq C \int\int_{\mathcal{M}}|f|^2.
	\end{equation}
	Therefore,  we conclude inequality~\eqref{eq: thm 3 exp decay, eq 2}. 
	
	\begin{center}
		\textbf{The inequality}~\eqref{eq: thm 3 exp decay, eq 2.1} 
	\end{center}

	The second inequality of~\eqref{eq: thm 3 exp decay, eq 2} is immediate by noting the following 
	\begin{equation}\label{eq: lem: proof thm 3 exp decay, lem 1, eq 4}
		\begin{aligned}
			\int\int_{\mathcal{M}}\left|(\mathcal{G}-\widetilde{\mathcal{G}})f\right|^2&	= \int\int_{\mathcal{M}}\left|(\textit{Op}(g_2)-\textit{Op}(\widetilde{g_2}))f\right|^2 \\
			&	\leq B \int_{r_+}^{\bar{r}_+}dr\int_0^\pi \sin\theta d\theta \int_{\mathbb{R}}d\omega\sum_m \left| \left(g_2-\widetilde{g_2}\right)\mathcal{F}_{\omega,m}(f)\right|^2.
		\end{aligned}
	\end{equation}
	where note that above we used the coarea formula, see Section~\ref{subsec: coarea formula}. It is a straightforward computation to conclude inequality~\eqref{eq: thm 3 exp decay, eq 2.1} from~\eqref{eq: lem: proof thm 3 exp decay, lem 1, eq 4} after using the property~\eqref{eq: lem: proof thm 3 exp decay, lem 0, eq 2} from Lemma~\ref{lem: proof thm 3 exp decay, lem 0} for the difference~$\widetilde{g_2}-g_2$.

	\begin{center}
		\textbf{The identities}~\eqref{eq: thm 3 exp decay, eq 3} 
	\end{center}
	
	The identities~\eqref{eq: thm 3 exp decay, eq 3} follow from pointwise considerations, in view of the definition of~$\mathcal{G}$, see Definition~\ref{def: sec: G, def 1}.

	\begin{center}
		\textbf{The inequalities of}~\eqref{eq: thm 3 exp decay, eq 4}
	\end{center}
	
	We now prove the estimates~\eqref{eq: thm 3 exp decay, eq 4}. We only proceed to prove the last, which is the most elaborate, and note that the remaining inequality can be proved similarly. We proceed as follows 
	\begin{equation}\label{eq: lem: proof thm 3 exp decay, lem 1, eq 5}
		\begin{aligned}
			&	\int\int_\mathcal{M} \Delta |[Z^\star,\mathcal{G}]f|^2 =	\int\int_\mathcal{M} \Delta |Z^\star\mathcal{G} f-\mathcal{G}Z^\star f|^2	\\
			&	\qquad \sim \int_{\mathbb{R}}dr^\star\int_0^\pi \int_{\mathbb{R}}d\omega \sum_m \Delta^2 \left| Z^\star(g_1\partial_{r^\star} +ig_2 )\mathcal{F}_{\omega,m} (f) - (g_1\partial_{r^\star} +ig_2 )Z^\star \mathcal{F}_{\omega,m} (f) \right|^2\\
			&	\qquad = 	\int_{\mathbb{R}}dr^\star \int_0^\pi \int_{\mathbb{R}}d\omega \sum_m \Delta^2 \left| Z^\star (g_1) \partial_{r^\star}\mathcal{F}_{\omega,m} (f)+iZ^\star(g_2) \mathcal{F}_{\omega,m} (f)\right|^2\\
			&	\qquad = \int_{\mathbb{R}}dr^\star \int_0^\pi\int_{\mathbb{R}}d\omega \sum_m \Delta^2 \left| \frac{d g_1}{dr}  \partial_{r^\star}\mathcal{F}_{\omega,m} (f)+i\frac{d g_2}{dr} \mathcal{F}_{\omega,m} (f)\right|^2\\
			&	\qquad =	 \int_{\mathbb{R}}dr^\star\int_0^\pi \int_{\mathbb{R}}d\omega \sum_m 1_{\mathcal{SF}}\Delta^2 \left| \frac{d g_1}{dr}  \partial_{r^\star}\mathcal{F}_{\omega,m} (f)+i\frac{d g_2}{dr} \mathcal{F}_{\omega,m} (f)\right|^2 \\
			&	\qquad\qquad +  \int_{\mathbb{R}}dr^\star \int_0^\pi\int_{\mathbb{R}}d\omega \sum_m 1_{(\mathcal{SF})^c}\Delta^2 \left| \frac{d g_1}{dr}  \partial_{r^\star}\mathcal{F}_{\omega,m} (f)+i\frac{d g_2}{dr} \mathcal{F}_{\omega,m} (f)\right|^2
		\end{aligned}
	\end{equation}
	where in the first similarity we used the coarea formula, see Section~\ref{subsec: coarea formula}, and Parseval identities. For~$\mathcal{F}_{\omega,m}$ see~\eqref{eq: subsec: sec: carter separation, radial, eq -2}.

	By recalling the definitions of~$g_1,g_2$, see Definition~\ref{def: subsec: sec: G, subsec 1, def 1}, we bound the \underline{non-superradiant frequency terms} of the last line of~\eqref{eq: lem: proof thm 3 exp decay, lem 1, eq 5}, namely the last term of the last line of~\eqref{eq: lem: proof thm 3 exp decay, lem 1, eq 5}. We note that the following holds 
	\begin{equation}\label{eq: lem: proof thm 3 exp decay, lem 1, eq 6}
		\begin{aligned}
			&	\int_{\mathbb{R}}dr^\star \int_0^\pi\int_{\mathbb{R}}d\omega\sum_m 1_{(\mathcal{SF})^c} \Delta^2\left| \frac{d g_1}{dr}  \partial_{r^\star}\mathcal{F}_{\omega,m} (f)+i\frac{d g_2}{dr} \mathcal{F}_{\omega,m} (f)\right|^2 \\
			&	\qquad\qquad \leq B 		\int_{\mathbb{R}}dr^\star\int_0^\pi\int_{\mathbb{R}}\sum_m 1_{(\mathcal{SF})^c}\frac{\Delta^2}{\Delta^3} \left| \left(\partial_{r^\star}-i\left(\omega-\frac{am\Xi}{r_+^2+a^2}\right)\right) \mathcal{F}_{\omega,m} (f)\right|^2\\
			&	\qquad\qquad \leq B 	\int_{\mathbb{R}}dr^\star\int_0^\pi\int_{\mathbb{R}}\sum_m 1_{(\mathcal{SF})^c}\Delta \left( |\omega\mathcal{F}_{\omega,m} (f)|^2+|am\mathcal{F}_{\omega,m} (f)|^2 +|Z^\star \mathcal{F}_{\omega,m} (f)|^2 \right),
		\end{aligned}
	\end{equation}
	where in the last inequality we used that~$\left(\partial_{r^\star} - i\left(\omega-\frac{am\Xi}{r_+^2+a^2}\right)\right) \mathcal{F}_{\omega,m}(f) \sim \mathcal{O}(\Delta) \left(\omega^2+(am)^2+Z^\star\right) \mathcal{F}_{\omega,m}(f) $ near~$r_+,\bar{r_+}$, since~$\supp f\subset \{0\leq t^\star<+\infty\}$.

	Finally, we bound the \underline{superradiant frequency term} of the last line of~\eqref{eq: lem: proof thm 3 exp decay, lem 1, eq 5}, namely the second to last term on the lat line of~\eqref{eq: lem: proof thm 3 exp decay, lem 1, eq 5}. We recall from~\eqref{eq: proof: prop: for v, eq -2} that for the superradiant frequencies we obtain that~$\frac{d}{dr}(g_2)= -\text{sign} \left(r-r_s\right)\cdot \frac{d}{dr}\left(\left(\omega-\frac{am\Xi}{r^2+a^2}\right) g_1(r)\right)$. We bound
	\begin{equation}\label{eq: lem: proof thm 3 exp decay, lem 1, eq 6.1}
		\begin{aligned}
			&		 \int_{\mathbb{R}}dr^\star\int_0^\pi\int_{\mathbb{R}} d\omega\sum_m 1_{\mathcal{SF}}\Delta^2\left| \frac{d g_1}{dr}  \partial_{r^\star}\mathcal{F}_{\omega,m} (f)-i \text{sign} \left(r-r_s\right)\cdot \frac{d}{dr}\left(\left(\omega-\frac{am\Xi}{r^2+a^2}\right) g_1(r)\right) \mathcal{F}_{\omega,m} (f)\right|^2 \\ 
			&	\leq B \int_{\mathbb{R}}dr^\star\int_0^\pi\int_{\mathbb{R}} \sum_m 1_{\mathcal{SF}} \Delta^2\left| \frac{d g_1}{dr} \left( \partial_{r^\star}-i \text{sign} \left(r-r_s\right)\cdot \left(\omega-\frac{am\Xi}{r^2+a^2} \right) \right)\mathcal{F}_{\omega,m} (f)\right|^2  + 	\\
			&	\qquad\qquad +  B\int_{\mathbb{R}}dr^\star\int_0^\pi\int_{\mathbb{R}}\sum_m 1_{\mathcal{SF}} \Delta^2 \left|  \frac{d}{dr}\left(\omega-\frac{am\Xi}{r^2+a^2}\right)g_1(r) \mathcal{F}_{\omega,m} (f)\right|^2 \\ 
			&	\leq B \int_{\mathbb{R}}dr^\star\int_0^\pi\int_{\mathbb{R}} \sum_m 1_{\mathcal{SF}} \sqrt{\Delta}\left| \left( \partial_{r^\star}-i \text{sign} \left(r-r_s\right)\cdot \left(\omega-\frac{am\Xi}{r^2+a^2} \right) \right)\mathcal{F}_{\omega,m} (f)\right|^2  \\
			&	\qquad\qquad + 	  B\int_{\mathbb{R}}dr^\star\int_0^\pi\int_{\mathbb{R}}\sum_m 1_{\mathcal{SF}} \Delta^2 \left|  \frac{d}{dr}\left(\omega-\frac{am\Xi}{r^2+a^2}\right)g_1(r) \mathcal{F}_{\omega,m} (f)\right|^2, \\ 
		\end{aligned}
	\end{equation}
	where recall that
	\begin{equation}
		\begin{aligned}
			&  \int_{\mathbb{R}}dr^\star\int_0^\pi\int_{\mathbb{R}} \sum_m 1_{\mathcal{SF}} \left|\left( \partial_{r^\star}-i \text{sign} \left(r-r_s\right)\cdot \left(\omega-\frac{am\Xi}{r^2+a^2} \right) \right)\mathcal{F}_{\omega,m} (f)\right|^2 \\
			&	\qquad  \leq B    \int_{\mathbb{R}}dr^\star\int_0^\pi\int_{\mathbb{R}} \sum_m 1_{\mathcal{SF}} \Delta \left( |\omega\mathcal{F}_{\omega,m} (f)|^2+|am\mathcal{F}_{\omega,m} (f)|^2 +|Z^\star \mathcal{F}_{\omega,m} (f)|^2 \right),
		\end{aligned}
	\end{equation}
	where, again, in the last inequality we used that~$\left(\partial_{r^\star} - i\left(\omega-\frac{am\Xi}{r_+^2+a^2}\right)\right) \mathcal{F}_{\omega,m}(f) \sim \mathcal{O}(\Delta) \left(\omega^2+(am)^2+Z^\star\right) \mathcal{F}_{\omega,m}(f) $ near~$r_+,\bar{r}_+$, since~$\supp f\subset \{0\leq t^\star <+\infty\}$.

	Therefore, by~\eqref{eq: lem: proof thm 3 exp decay, lem 1, eq 6},~\eqref{eq: lem: proof thm 3 exp decay, lem 1, eq 6.1} we conclude that for the superradiant frequencies the following holds pointwise
	\begin{equation}\label{eq: lem: proof thm 3 exp decay, lem 1, eq 7}
		\begin{aligned}
			&			 \int_{\mathbb{R}}dr^\star \int_0^\pi\int_{\mathbb{R}} \sum_m 1_{\mathcal{SF}} \Delta^2\left| \frac{d g_1}{dr}  \partial_{r^\star}\mathcal{F}_{\omega,m} (f)-i \text{sign} \left(r-r_s\right)\cdot \frac{d}{dr}\left(\left(\omega-\frac{am\Xi}{r^2+a^2}\right) g_1(r)\right) \mathcal{F}_{\omega,m} (f)\right|^2 \\ 
			&	\qquad\leq B \int_{\mathbb{R}}dr^\star \int_0^\pi\int_{\mathbb{R}} \sum_m 1_{\mathcal{SF}} \Delta \left( |\omega\mathcal{F}_{\omega,m} (f)|^2+|am\mathcal{F}_{\omega,m} (f)|^2 +|Z^\star \mathcal{F}_{\omega,m} (f)|^2 \right).
		\end{aligned}
	\end{equation}
	Therefore, by using~\eqref{eq: lem: proof thm 3 exp decay, lem 1, eq 5} and Parseval identities, in view of the pointwise bounds~\eqref{eq: lem: proof thm 3 exp decay, lem 1, eq 6},~\eqref{eq: lem: proof thm 3 exp decay, lem 1, eq 7} we conclude the last estimate of~\eqref{eq: thm 3 exp decay, eq 4}.

	\begin{center}
		\textbf{The inequalities}~\eqref{eq: thm 3 exp decay, eq 5},~\eqref{eq: thm 3 exp decay, eq 5.1},~\eqref{eq: thm 3 exp decay, eq 5.2}
	\end{center}

	The inequality~\eqref{eq: thm 3 exp decay, eq 5} is immediate since
	\begin{equation}
		[\mathcal{G},\partial_\theta]=0
	\end{equation}
	and therefore
	\begin{equation}
		\begin{aligned}
			&	\int\int_{\mathcal{M}} \left|[\partial_\theta\widetilde{\mathcal{G}},\chi]f \right|^2 = 	\int\int_{\mathcal{M}} \frac{1}{r^2}\left|[\widetilde{\mathcal{G}},\chi]\partial_\theta f \right|^2\leq B \int\int_{\mathcal{M}} J^n_\mu[\psi]n^\mu,
		\end{aligned}
	\end{equation}
	where and in the last inequality we used the pseudodifferential commutation~\eqref{eq: thm 3 exp decay, eq 2} already proved earlier.

	For the inequality~\eqref{eq: thm 3 exp decay, eq 5.2} we proceed as follows
	\begin{equation}\label{eq: lem: proof thm 3 exp decay, lem 1, eq 8}
		\begin{aligned}
			\int\int_{\mathcal{M}}\Delta \left|  \left[Z^\star \mathcal{G},\chi\right] f  \right|^2 &	= 	\int\int_{\mathcal{M}}\Delta \left|  \left[[Z^\star,\mathcal{G}]+\mathcal{G}Z^\star,\chi\right] f  \right|^2\\
			&	\leq B \int\int_{\mathcal{M}}\Delta \left|  \left[[Z^\star,\mathcal{G}],\chi\right] f  \right|^2 + \int\int_{\mathcal{M}}\Delta \left|  \left[\mathcal{G}Z^\star,\chi\right] f  \right|^2 \\
			&	=	 B\int\int_{\mathcal{M}}\Delta \left|  \left[[Z^\star,\mathcal{G}],\chi\right] f  \right|^2 + B\int\int_{\mathcal{M}} \Delta \left| \mathcal{G}(f Z^\star \chi)+ \chi [\mathcal{G},Z^\star]f  \right|^2 \\
			&	\leq B \int\int_{\mathcal{M}}\Delta \left|  \left[[Z^\star,\mathcal{G}],\chi\right] f  \right|^2 + \left|\chi [\mathcal{G},Z^\star]f  \right|^2 +  \Delta |\mathcal{G}(fZ^\star \chi)|^2\\
			&	\leq B \int\int_{\mathcal{M}} J^n_\mu [f]n^\mu +|f|^2,
		\end{aligned}
	\end{equation}
	where in the last inequality we use the already proved~\eqref{eq: thm 3 exp decay, eq 4} to bound the first integral, and note that the last integral is bounded easily. The inequality~\eqref{eq: thm 3 exp decay, eq 5.1} follows similarly. 
\end{proof}

\bibliographystyle{plain}
\bibliography{MyBibliography}

\end{document}